\documentclass{amsart}

\usepackage{amsmath, amsfonts, amssymb}  
\usepackage{enumerate}
\usepackage{enumitem}
\usepackage{hyperref}
\hypersetup{ 
  colorlinks   = true,
  urlcolor     = blue,
  linkcolor    = blue, 
  citecolor   = red 
}
\usepackage{graphicx}
\usepackage[all]{xy}
\usepackage{wrapfig}
\usepackage{fancyvrb}
\usepackage[mathscr]{euscript}
\usepackage[T1]{fontenc}
\usepackage{listings}
\usepackage{tikz}
\usepackage{amsthm}
\usepackage{dsfont}
\usepackage{tikz-cd}
\usepackage{adjustbox}
\usepackage{centernot}
\usepackage{mathtools}
\usepackage[noadjust]{cite}
\usepackage{bbm}
\usepackage{stmaryrd}

\usepackage[margin=3cm]{geometry}

\makeatletter
\DeclareRobustCommand{\cev}[1]{%
  {\mathpalette\do@cev{#1}}%
}
\newcommand{\do@cev}[2]{%
  \vbox{\offinterlineskip
    \sbox\z@{$\m@th#1 x$}%
    \ialign{##\cr
      \hidewidth\reflectbox{$\m@th#1\vec{}\mkern4mu$}\hidewidth\cr
      \noalign{\kern-\ht\z@}
      $\m@th#1#2$\cr
    }%
  }%
}
\makeatother

\let\oldtocsection=\tocsection

\let\oldtocsubsection=\tocsubsection

\let\oldtocsubsubsection=\tocsubsubsection

\renewcommand{\tocsection}[2]{\hspace{0em}\oldtocsection{#1}{#2}}
\renewcommand{\tocsubsection}[2]{\hspace{1em}\oldtocsubsection{#1}{#2}}
\renewcommand{\tocsubsubsection}[2]{\hspace{2em}\oldtocsubsubsection{#1}{#2}}

\title{Soergel calculus for monodromic Hecke categories}
\author{Colton Sandvik}
\date{}

\newcommand{\C}{\mathbb{C}}
\newcommand{\Q}{\mathbb{Q}}
\newcommand{\Z}{\mathbb{Z}}
\newcommand{\F}{\mathbb{F}}

\renewcommand{\O}{\mathbb{O}}

\newcommand{\K}{\mathbb{K}}
\newcommand{\A}{\mathbb{A}}

\renewcommand{\k}{\mathbbm{k}}
\newcommand{\R}{\mathbb{R}}
\renewcommand{\H}{\mathbb{H}}

\newcommand{\bfX}{\mathbf{X}}
\newcommand{\bfY}{\mathbf{Y}}

\newcommand{\scrF}{\mathcal{F}}
\newcommand{\scrG}{\mathcal{G}}
\newcommand{\scrH}{\mathcal{H}}
\newcommand{\scrC}{\mathcal{C}}

\newcommand{\scrE}{\mathcal{E}}

\newcommand{\scrA}{\mathcal{A}}
\newcommand{\scrK}{\mathcal{K}}

\newcommand{\scrL}{\mathcal{L}}

\newcommand{\scrR}{\mathcal{R}}

\DeclareMathOperator{\cons}{c}

\DeclareMathOperator{\Semis}{Semis}
\DeclareMathOperator{\Loc}{Loc}

\DeclareMathOperator{\Hom}{Hom}
\newcommand{\TwoHom}{2\textnormal{-}\Hom}
\newcommand{\TwoEnd}{2\textnormal{-}\End}

\DeclareMathOperator{\End}{End}
\DeclareMathOperator{\Free}{Free}
\DeclareMathOperator{\Sym}{Sym}
\DeclareMathOperator{\Ob}{Ob}
\DeclareMathOperator{\Ch}{Ch}

\renewcommand{\mod}[1]{#1\textnormal{-mod}}

\newcommand{\grrmod}[1]{\textnormal{mod}^{\Z}\textnormal{-}#1}
\newcommand{\grbim}[1]{#1\textnormal{-bim}^{\Z}}
\newcommand{\bim}[1]{#1\textnormal{-bim}}

\DeclareMathOperator{\supp}{supp}
\newcommand{\DD}{\mathbb{D}}
\DeclareMathOperator{\IC}{IC}

\newcommand{\RHom}{\textbf{R}\mathcal{H}om}
\newcommand{\uk}{\underline{\k}}

\DeclareMathOperator{\For}{For}
\DeclareMathOperator{\pt}{pt}
\DeclareMathOperator{\id}{id}

\DeclareMathOperator{\rk}{rk}
\DeclareMathOperator{\Fun}{Fun}

\newcommand{\fr}[1]{\mathfrak{#1}}
\DeclareMathOperator{\op}{op}

\DeclareMathOperator{\RW}{RW}

\DeclareMathOperator{\rex}{rex}

\DeclareMathOperator{\ev}{ev}

\DeclareMathOperator{\fg}{fg}

\DeclareMathOperator{\CM}{CM}

\DeclareMathOperator{\diag}{diag}

\DeclareMathOperator{\alg}{alg}
\DeclareMathOperator{\Abe}{A}
\DeclareMathOperator{\mon}{mon}

\DeclareMathOperator{\BS}{BS}

\newcommand{\eFl}{\widetilde{\mathcal{F}\ell}}
\newcommand{\BGB}{B \backslash G/B}
\newcommand{\UGU}{U \backslash G/U}
\newcommand{\DME}[2]{\:_{#1 \smallleftdash} D_{/ #2} }
\newcommand{\DEE}[2]{\:_{#1 \backslash} D_{/ #2} }
\DeclareMathOperator{\geom}{geom}
\newcommand{\DGEE}[2]{\:_{#1 \backslash} D_{/ #2}^{m, \geom} }
\newcommand{\DGME}[2]{\:_{#1 \smallleftdash} D_{/ #2}^{m, \geom}}

\newcommand{\PME}[2]{\:_{#1 \smallleftdash} \textnormal{Par}_{/ #2} }
\newcommand{\PEE}[2]{\:_{#1 \backslash} \textnormal{Par}_{/ #2} }

\newcommand{\ForME}[1]{\:_{#1 \backslash} \textnormal{For}}

\newcommand{\unabla}{\underline{\nabla}}
\newcommand{\uDelta}{\underline{\Delta}}

\newcommand{\TLcat}{2\mathcal{T}\mathcal{L}}
\DeclareMathOperator{\TL}{TL}

\DeclareMathOperator{\JW}{JW}
\DeclareMathOperator{\LL}{LL}
\newcommand{\dLL}{\mathbb{L}\mathbb{L}}

\newcommand{\EW}{\mathcal{D}}
\newcommand{\EWmon}[2]{\:_{#1 \backslash} \mathcal{D}_{/ #2} }
\newcommand{\EWmonME}[2]{\:_{#1 \smallleftdash} \mathcal{D}_{/ #2} }

\newcommand{\grEWmon}[2]{\:_{#1 \backslash} \!\! {}^{\circ} \mathcal{D}_{/ #2} }
\newcommand{\grEWmonME}[2]{\:_{#1 \smallleftdash} \!\!{}^{\circ} \mathcal{D}_{/ #2} }

\newcommand{\DDBE}[2]{\:_{#1 \backslash} D_{/ #2}^{m, \diag}}
\newcommand{\DDME}[2]{\:_{#1 \smallleftdash} D_{/ #2}^{m, \diag}}
\newcommand{\DDBEI}[3]{\:_{#2 \backslash} D_{#1 / #3}^{m, \diag} }

\newcommand{\BSBim}{\textbf{BS}\textnormal{Bim}}
\newcommand{\SBim}{\mathbb{S}\textnormal{Bim}}
\newcommand{\Cmon}[2]{\:_{#1 \backslash} \mathcal{C}_{/ #2} }
\newcommand{\CmonQ}[2]{\:_{#1 \backslash} \mathcal{C}_{Q / #2} }
\newcommand{\Amon}[2]{\:_{#1 \backslash} \mathcal{A}_{/ #2} }
\newcommand{\grAmon}[2]{\:_{#1 \backslash} \!\! {}^{\circ} \mathcal{A}_{/ #2} }

\newcommand{\uH}{\underline{H}}
\DeclareMathOperator{\Seq}{Seq}
\DeclareMathOperator{\lab}{lab}
\DeclareMathOperator{\dec}{dec}
\DeclareMathOperator{\Exp}{Exp}
\DeclareMathOperator{\Subexp}{Sub}
\newcommand{\SGraph}{\mathbb{S}\textnormal{Gr}}
\newcommand{\VSGraph}{\mu\mathbb{S}\textnormal{Gr}}
\newcommand{\W}[2]{{}_{#1}W_{#2}}
\newcommand{\uW}[2]{{}_{#1}\underline{W}_{#2}}
\newcommand{\ue}{\underline{e}}
\newcommand{\uf}{\underline{f}}
\newcommand{\uw}{\underline{w}}
\newcommand{\ub}{\underline{b}}
\newcommand{\ux}{\underline{x}}
\newcommand{\uy}{\underline{y}}
\newcommand{\us}{\underline{s}}
\DeclareMathOperator{\Bl}{Bl}

\newcommand{\grk}{\textnormal{rank}^{\Z}}
\renewcommand{\emptyset}{\varnothing}
\renewcommand{\binom}[2]{\begin{bmatrix} #1 \\ #2 \end{bmatrix}}
\newcommand{\naivequotient}{/\!\!/}
\newcommand{\abs}[1]{\lvert #1 \rvert}

\newcommand{\bfslash}{\mathbin{\fatbslash}}

\makeatletter
\newcommand{\customlabel}[2]{%
   \protected@write \@auxout {}{\string \newlabel {#1}{{#2}{\thepage}{#2}{#1}{}} }%
   \hypertarget{#1}{}
}
\makeatother

\makeatletter

\newcommand*\smallleftdash{\bfslash}

\makeatother

\newtheorem{corollary}[subsubsection]{Corollary}
\newtheorem{lemma}[subsubsection]{Lemma}
\newtheorem{proposition}[subsubsection]{Proposition}
\newtheorem{assumption}[subsubsection]{Assumption}

\newtheorem{theorem}[subsubsection]{Theorem}
\newtheorem{claim}[subsubsection]{Claim}

\theoremstyle{definition}
\newtheorem{definition}[subsubsection]{Definition}
\newtheorem{example}[subsubsection]{Example}
\newtheorem{remark}[subsubsection]{Remark}

\newtheorem{algorithm}[subsubsection]{Algorithm}
\newtheorem{convention}[subsubsection]{Convention}

\newenvironment{midsecproof}[1]{\vspace{\topsep} \noindent \textit{Proof of #1.}}{\hfill $\square$}

\numberwithin{equation}{subsection}

\usepackage{tikzit}
\tikzstyle{EW_s}=[fill=red, draw=red, shape=circle, inner sep=0pt, minimum size=3pt]
\tikzstyle{box}=[fill=white, draw=black, shape=rectangle]
\tikzstyle{circle_region}=[fill={rgb,255: red,255; green,205; blue,155}, draw=black, shape=circle, dashed, minimum width=6em]
\tikzstyle{JW}=[fill=white, draw=black, shape=rectangle, minimum width=1.75cm, minimum height=0.5cm]
\tikzstyle{EW_t}=[fill=blue, draw=blue, shape=circle, inner sep=0pt, minimum size=3pt]
\tikzstyle{black_dot}=[fill=black, draw=black, shape=circle, inner sep=0pt, minimum size=3pt]
\tikzstyle{small_none}=[fill=none, draw=none, shape=circle, scale=0.5]
\tikzstyle{gnode}=[fill=white, draw=black, shape=circle]
\tikzstyle{g2node}=[fill=cyan, draw=black, shape=circle]
\tikzstyle{small_dot_s}=[fill=red, draw=red, shape=circle, scale=0.5]
\tikzstyle{small_dot_t}=[fill=blue, draw=blue, shape=circle, scale=0.5]
\tikzstyle{EW_u}=[fill={rgb,255: red,128; green,0; blue,128}, draw={rgb,255: red,128; green,0; blue,128}, shape=circle, inner sep=0pt, minimum size=3pt]
\tikzstyle{mor_box}=[fill=white, draw=black, shape=rectangle, minimum width=1cm, minimum height=0.3cm]
\tikzstyle{small_mor_box}=[fill=white, draw=black, shape=rectangle, minimum width=0.6cm, minimum height=0.3cm]
\tikzstyle{med_mor_box}=[fill=white, draw=black, shape=rectangle, minimum width=1.5cm, minimum height=0.3cm]
\tikzstyle{large_mor_box}=[fill=white, draw=black, shape=rectangle, minimum width=2.3cm, minimum height=0.3cm]
\tikzstyle{small_white_node}=[fill=white, draw=black, shape=circle, scale=0.3]
\tikzstyle{3_strand_box}=[fill=white, draw=black, shape=rectangle, minimum height=0.3cm, minimum width=0.75cm]
\tikzstyle{4_strand_box}=[fill=white, draw=black, shape=rectangle, minimum height=0.3cm, minimum width=1cm]
\tikzstyle{5_strand_box}=[fill=white, draw=black, shape=rectangle, minimum height=0.3cm, minimum width=1.25cm]
\tikzstyle{6_strand_box}=[fill=white, draw=black, shape=rectangle, minimum height=0.3cm, minimum width=1.5cm]
\tikzstyle{7_strand_box}=[fill=white, draw=black, shape=rectangle, minimum height=0.3cm, minimum width=1.75cm]
\tikzstyle{8_strand_box}=[fill=white, draw=black, shape=rectangle, minimum height=0.3cm, minimum width=2cm]
\tikzstyle{9_strand_box}=[fill=white, draw=black, shape=rectangle, minimum height=0.3cm, minimum width=2.25cm]
\tikzstyle{small_mon_label}=[fill=none, draw=none, shape=circle, scale=0.5]
\tikzstyle{dot_green}=[fill=green, draw=green, shape=circle, inner sep=0pt, minimum size=3pt]
\tikzstyle{big_bracket_6}=[fill=none, draw=none, shape=circle, scale=2.5]

\tikzstyle{row}=[-, draw=black, thick]
\tikzstyle{dashed_row}=[-, draw=black, dashed]
\tikzstyle{red_edge}=[-, draw=red, thick]
\tikzstyle{blue edge}=[-, draw=blue, thick]
\tikzstyle{purple_edge}=[-, draw={rgb,255: red,128; green,0; blue,128}, thick]
\tikzstyle{green_edge}=[-, draw=green, thick]
\tikzstyle{fill_orange}=[-, fill={rgb,255: red,255; green,205; blue,155}, draw=none]
\tikzstyle{fill_purple}=[-, fill={rgb,255: red,225; green,137; blue,225}, draw=none]
\tikzstyle{fill_green}=[-, fill={rgb,255: red,155; green,255; blue,155}, draw=none]
\tikzstyle{fill_yellow}=[-, fill={rgb,255: red,255; green,255; blue,155}, draw=none]
\tikzstyle{fill_blue}=[-, fill={rgb,255: red,155; green,255; blue,255}, draw=none]
\tikzstyle{fill_coral}=[-, fill={rgb,255: red,255; green,205; blue,205}, draw=none]
\tikzstyle{fill_brown}=[-, fill={rgb,255: red,120; green,90; blue,20}, draw=none]
\tikzstyle{fill_cyan}=[-, fill={rgb,255: red,0; green,255; blue,255}, draw=none]
\tikzstyle{fill_magenta}=[-, fill={rgb,255: red,255; green,0; blue,255}, draw=none]
\tikzstyle{fill_darkgreen}=[-, fill={rgb,255: red,0; green,165; blue,55}, draw=none]
\tikzstyle{fill_lavender}=[-, fill={rgb,255: red,230; green,170; blue,255}, draw=none]
\tikzstyle{fill_gold}=[-, fill={rgb,255: red,255; green,180; blue,0}, draw=none]
\tikzstyle{fill_dark_blue}=[-, fill=blue, draw=none]
\tikzstyle{fill_red}=[-, fill=red, draw=none]
\tikzstyle{fill_white}=[-, fill=white, draw=none]
\tikzstyle{red_dashed_edge}=[-, fill=none, draw=red, dashed, thick]
\tikzstyle{blue_dashed_edge}=[-, draw=blue, dashed, thick]
\tikzstyle{purple_dashed_edge}=[-, draw={rgb,255: red,128; green,0; blue,128}, thick, dashed]
\tikzstyle{fill_light_gray}=[-, fill={rgb,255: red,198; green,198; blue,198}, draw=none]
\tikzstyle{fill_dark_gray}=[-, fill={rgb,255: red,109; green,109; blue,109}, draw=none]
\tikzstyle{fill_sky}=[-, fill={rgb,255: red,230; green,234; blue,255}, draw=none]
\tikzstyle{fill_bright_green}=[-, fill={rgb,255: red,66; green,255; blue,79}, draw=none]
\tikzstyle{red_tl_edge}=[-, draw={rgb,255: red,255; green,128; blue,128}]
\tikzstyle{blue_tl_edge}=[-, draw={rgb,255: red,128; green,128; blue,255}]
\tikzstyle{red_tl_region}=[-, fill={rgb,255: red,255; green,128; blue,128}, draw=none]
\tikzstyle{blue_tl_region}=[-, draw=none, fill={rgb,255: red,128; green,128; blue,255}]
\tikzstyle{purple_tl_region}=[-, draw=none, fill={rgb,255: red,220; green,135; blue,220}]
\tikzstyle{orange_line}=[-, draw={rgb,255: red,255; green,205; blue,155}]
\tikzstyle{purple_line}=[-, draw={rgb,255: red,225; green,137; blue,225}]
\tikzstyle{implies_arrow}=[double, double equal sign distance, -implies, ->]
\tikzstyle{densely_dashed}=[-, densely dashed]
\tikzstyle{mapsto}=[{|->}]
\tikzstyle{arrow}=[->]

\begin{document}
\begin{abstract}
        We introduce two 2-categories which categorify the monodromic Hecke algebra. 
        The first is algebraic in nature and generalizes Abe's theory of Soergel bimodules. 
        The second is a diagrammatic category defined via generators and relations which generalizes the Elias--Williamson diagrammatic calculus. 
        As our first main result, we prove that these algebraic and diagrammatic categorifications are equivalent, extending an earlier theorem of Abe. 
        Furthermore, we relate these new categorifications to a third categorification via parity sheaves which was previously studied by the author. 
        More precisely, we provide a monodromic analogue of a theorem of Riche and Williamson to show that the diagrammatic category is equivalent to the monodromic Hecke category of parity sheaves associated to a reductive group. 
        Finally, we show that these monodromic Hecke categories can be described by unipotent Hecke categories associated to endoscopic Coxeter groups.
\end{abstract}

	\maketitle

        \begingroup
        \hypersetup{hidelinks}
	\tableofcontents
        \endgroup

    \section{Introduction}
    \subsection{Hecke Algebras}

Let $G$ be a connected reductive group over $\F = \overline{\F}_p$. Pick a prime $\ell \neq p$ and a power $q = p^n$.
Let $B \subset G$ be a Borel subgroup with split maximal torus $T \subset B$.
The Hecke algebra 
\[ \scrH (\F_q) \coloneq \Fun_{B (\F_q) \times B (\F_q)} (G (\F_q), \overline{\Q}_\ell)\] 
is the algebra of $B (\F_q)$-biinvariant $\overline{\Q}_\ell$-valued functions on $G (\F_q)$ with multiplication given by convolution.
The Hecke algebra was originally introduced to study unipotent $\overline{\Q}_\ell$-representations of $G (\F_q)$; however, since their introduction, Hecke algebras have become ubiquitous throughout many areas of representation theory. 

In the study of the non-unipotent representation theory of $G (\F_q)$, Lusztig \cite{Lus19} introduced a generalization of the Hecke algebra which is now called the \emph{monodromic Hecke algebra}.
Let $\chi, \chi' : T (\F_q) \to \overline{\Q}_{\ell}^{\times}$ be two characters of the maximal torus.
We can regard $\chi$ and $\chi'$ as characters $B (\F_q) \to \overline{\Q}_{\ell}^{\times}$ by the projection $B \to T$.
We can then consider the $\overline{\Q}_{\ell}$-vector space of $(B (\F_q) \times B (\F_q), \chi \times \chi')$-invariant  $\overline{\Q}_\ell$-valued functions on $G (\F_q)$,
\begin{align*}
    {}_{\chi} \scrH_{\chi'}^{\mon} (\F_q) &\coloneq \Fun_{(B (\F_q) \times B (\F_q), \chi \times \chi')} (G (\F_q), \overline{\Q}_{\ell}) \\
    &= \{ f : G (\F_q) \to  \overline{\Q}_{\ell}^{\times} \mid f(b_1g b_2) = \chi (b_1) f(g) \chi' (b_2) \text{ for } b_1, b_2 \in B, g \in G\}.
\end{align*}
When $\chi = \chi'$, ${}_{\chi} \scrH_{\chi}^{\mon} (\F_q)$ is also an algebra under convolution. When $\chi \neq \chi'$,  ${}_{\chi} \scrH_{\chi'}^{\mon} (\F_q)$ is no longer an algebra.
Instead, convolution gives suitably associative maps
\[{}_{\chi} \scrH_{\chi'}^{\mon} (\F_q) \times {}_{\chi'} \scrH_{\chi''}^{\mon} (\F_q) \to {}_{\chi} \scrH_{\chi''}^{\mon} (\F_q).\]
If we fix a $W$-orbit $\fr{o}$ of characters $T (\F_q) \to \overline{\Q}_{\ell}$, we can define the \emph{monodromic Hecke algebra} as the direct sum of these vector spaces
\[\underline{\scrH}^{\mon} (\F_q) = \bigoplus_{\chi, \chi' \in \fr{o}} {}_{\chi} \scrH_{\chi'}^{\mon} (\F_q)\]
which becomes an algebra under convolution and the relation ${}_{\chi} \scrH_{\chi'}^{\mon} (\F_q) \cdot {}_{\chi''} \scrH_{\chi'''}^{\mon} (\F_q) = 0$ if $\chi' \neq \chi''$.

Iwahori gave a presentation of the Hecke algebra which is independent of $\F_q$ and only depends on the Weyl group $W$ \cite{Iwa}.
In particular, Iwahori defined a $\Z [v^{\pm}]$-algebra $\scrH$ such that $\scrH (\F_q) \cong \scrH \otimes_{\Z [v^{\pm}]} \overline{\Q}_{\ell}$ along the specialization $v^{-1} \mapsto \sqrt{q} \in \overline{\Q}_{\ell}$.
This presentation allowed for the generalization of the Hecke algebra to any Coxeter group and not just the crystallographic ones.
We write $\scrH (W)$ for the (non-specialized) Hecke algebra with Coxeter group $W$. 

Lusztig has given a presentation of the monodromic Hecke algebra closely inspired by Iwahori's presentation \cite{Lus16}. 
In particular, Lusztig defined a $\Z [v^{\pm}]$-algebra $\underline{\scrH}^{\mon} (\fr{o})$ which specializes to $\underline{\scrH}^{\mon} (\F_q)$ along $v^{-1} \mapsto \sqrt{q}$.
Note that this presentation still depends on $q$ since $\fr{o}$ depends on $q$; however, this presentation is still favorable for the purposes of categorification.
To remove the restriction to a fixed $W$-orbit $\fr{o}$ of characters, we instead define the \emph{monodromic Hecke algebroid}.
Instead of an algebra, the monodromic Hecke algebroid $\scrH^{\mon} (\F_q)$ is a category whose objects are characters $T (\F_q) \to \overline{\Q}_{\ell}^{\times}$ and whose morphisms are the vector spaces ${}_{\chi} \scrH_{\chi'}^{\mon} (\F_q)$.
The composition in $\scrH^{\mon} (\F_q)$ is given by convolution. 
We give a presentation of this algebroid in \S\ref{subsec:mha} by means of generators and relations.
This presentation allows us to define the monodromic Hecke algebroid over an arbitrary Coxeter group $W$ and allows us to replace the set of characters $T (\F_q) \to \overline{\Q}_{\ell}^{\times}$ by any $W$-set $\fr{o}$.

Lusztig has shown that the monodromic Hecke algebra is closely related to the ordinary Hecke algebra \cite{Lus19}.
For simplicity, assume that $G$ has connected center.
For each character $\chi$ of $T (\F_q)$, one can associate a reductive group $H_{\chi}^{\circ}$, called the \emph{endoscopic group}. We denote its Weyl group by $W_{\chi}^{\circ}$.
In \emph{loc. cit.}, it is then proved that there is an isomorphism of algebras
\[{}_{\chi} \scrH_{\chi}^{\mon} (\F_q, W) \cong \scrH (\F_q, W_{\chi}^{\circ}).\]
This isomorphism is called the \emph{monodromic-endoscopic equivalence}.
Moreover, a similar description of all of $\scrH^{\mon}$ can be given in terms of ordinary Hecke algebras for the Weyl groups of endoscopic groups.
This is closely related with Lusztig's classification of $G(\F_q)$-representations where the non-unipotent representations are parameterized by unipotent representations of the endoscopic groups.  

\subsection{Motivation from Geometry}

In their foundational work \cite{KL}, Kazhdan and Lusztig categorified the Hecke algebra $\scrH$ via the sheaf-function correspondence. 
In particular, one can define the geometric Hecke category as the category $\Semis (\BGB, \overline{\Q}_{\ell})$
of semisimple étale constructible sheaves on the stack $\BGB$. 
The geometric Hecke category is a monoidal category under convolution. 
The split Grothendieck group $[\Semis (\BGB, \overline{\Q}_\ell)]_{\oplus}$ is then an algebra. In \emph{loc. cit.}, it is then shown that there is an isomorphism of algebras
\[[ \Semis (\BGB, \overline{\Q}_\ell) ]_{\oplus} \cong \scrH.\]
Moreover, this isomorphism takes the simple IC-sheaves to the Kazhdan--Lusztig basis. 

The geometric Hecke category is closely related with the category of unipotent $\overline{\Q}_\ell$-representations of finite groups of Lie type (cf., \cite{BZN, BFO, Lus05}).
This poses the question: if one replaces $\overline{\Q}_\ell$-coefficients with $\overline{\F}_{\ell}$-coefficients, does the geometric Hecke category give information about unipotent $\overline{\F}_\ell$-representations?
The answer is yes, but with some degree of caution.
Namely, with $\overline{\F}_{\ell}$-coefficients, the decomposition theorem fails and the category of semisimple complexes is no longer the correct object of study.
There are two common replacements in the literature:
\begin{enumerate}
    \item One can instead work with the derived category of constructible sheaves on $\BGB$ instead of just the semisimple complexes.
    This category has many nice formal properties and in \cite{Zhu} it is shown that its categorical center is closely related with unipotent $\overline{\F}_\ell$-representations of finite groups of Lie type.
    However, the category is hard to compute with. Moreover, it no longer categorifies $\scrH$, but instead just the group algebra $\Z [W]$.
    \item Another option is to work with the category of parity sheaves $\textnormal{Par} (\BGB, \overline{\F}_{\ell})$ introduced in \cite{JMW}.
        While it is less clear what applications this category has to representations of finite groups of Lie type, it does provide a novel categorification of $\scrH$.
        Moreover, the indecomposable parity sheaves get sent to the ``$p$-Kazhdan--Lusztig basis'' which has found numerous applications throughout modular representation theory.
\end{enumerate}

One can ask how much of this story transfers over to the monodromic setting.
In \cite{LY}, Lusztig and Yun gave a geometric categorification of the monodromic Hecke algebra.
In this setup, a character $\chi : T (\F_q) \to \overline{\Q}_\ell^{\times}$ get associated to a character sheaf $\scrL_{\chi} \in \Loc (T, \overline{\Q}_\ell)$.
One can then define the geometric $(\scrL_{\chi}, \scrL_{\chi'})$-monodromic Hecke category as the category 
\[ \Semis_{(T \times T, \scrL_{\chi} \boxtimes \scrL_{\chi'}^{-1}) } (\UGU, \overline{\Q}_\ell)\] 
of $(T \times T,\scrL_{\chi} \boxtimes \scrL_{\chi'}^{-1}) $-equivariant semisimple  complexes on the stack $\UGU$.
This is a monoidal category under convolution of sheaves.

For simplicity, we again assume that $G$ has connected center.
The main result of \emph{loc. cit.} upgrades the description of the monodromic-endoscopic equivalence of Hecke algebras to Hecke categories.
Note that $H_{\chi}^{\circ}$ is defined so that $T$ is a maximal torus for $H_{\chi}^{\circ}$. A Borel subgroup $B$ of $G$ containing $T$ determines a Borel subgroup $B_{\chi}$ of $H_{\chi}^{\circ}$ containing $T$ as well. 
Lusztig and Yun \cite{LY} proved that there is an equivalence of monoidal categories
\[\Semis_{(T \times T, \scrL_{\chi} \boxtimes \scrL_{\chi})} (\UGU, \overline{\Q}_\ell) \cong \Semis (B_{\chi} \backslash H_{\chi}^{\circ} / B_{\chi}, \overline{\Q}_\ell).\]
This equivalence can also be extended to the case of unequal parameters $\chi, \chi'$ and $G$ having a disconnected center.

In \cite{Sandvik}, we extended the result of \cite{LY} to positive characteristic coefficients.
We defined categories of parity sheaves 
\[ \textnormal{Par}_{(T \times T, \scrL_{\chi} \boxtimes \scrL_{\chi'}^{-1} )} (\UGU, \overline{\F}_\ell).\] 
We then showed that there is an equivalence of monoidal categories
\[\textnormal{Par}_{(T \times T, \scrL_{\chi} \boxtimes \scrL_{\chi}^{-1} )} (\UGU, \overline{\F}_\ell) \cong \textnormal{Par} (B_{\chi} \backslash H_{\chi}^{\circ} / B_{\chi}, \overline{\F}_\ell).\]
These categories of parity sheaves can be assembled into a 2-category denoted by
\[\PEE{}{} (G, \overline{\F}_{\ell}) \]
We further showed in \cite[Theorem 3.8.4]{Sandvik} that the split Grothendieck algebroid of $\PEE{}{} (G,  \overline{\F}_{\ell})$ is isomorphic to the monodromic Hecke algebroid.
More generally, the $\overline{\F}_{\ell}$-coefficients appearing here can be replaced by $\k$-coefficients where $\k$ is a field or a complete local ring. 

\subsection{Soergel Bimodules}

In \cite{So90, So92, So00}, Soergel constructed another categorification of the Hecke algebra using bimodules over the equivariant cohomology ring
\[R = H_T^{\bullet} (\pt; \C) \cong \Sym (\fr{h}^*)\]
where $\fr{h}$ is the Lie algebra of $T$. Note that $R$ is a graded algebra over $\k$ with $\fr{h}^*$ in degree 2.
The Weyl group $W$ acts on $R$, and from this, one considers the subalgebra of $s$-invariants, denoted $R^s$, for $s \in S$.
Soergel bimodules are then generated by taking direct sums and direct summands from graded $R$-bimodules of the form
\[R \otimes_{R^{s_1}} R \otimes_{R^{s_2}} \ldots \otimes_{R^{s_k}} R (n)\]
where $s_1, \ldots, s_k \in S$ and $(n)$ denotes a shift in grading by $n \in \Z$.
The resulting category of Soergel bimodules, denoted $\SBim (\fr{h}, W)$, again categorifies the Hecke algebra.

The key advantage with using Soergel bimodules is that they can be easily generalized to non-crystallographic Coxeter groups.
In \cite{So07}, Soergel extends the eponymous theory to arbitrary Coxeter systems $W$ and ``reflection faithful'' representations $\fr{h}$ of $W$.
This affords a categorification of Hecke algebras for non-crystallographic Coxeter groups.

Unfortunately, Soergel bimodules starts to become ill-behaved whenever $W$ is infinite or $\k$ has small characteristic (in which case, reflection faithful representations are hard to find).
More recently, Abe \cite{Abe19} was able to rectify this undesirable behavior by enhancing the category of Soergel bimodules.
Roughly speaking, in Abe's incarnation of Soergel bimodules, one also keeps track of the various localization data associated to a given bimodule.
This localization data helps to shrink the Hom spaces and ensures that Soergel bimodules associated to distinct elements of $W$ stay distinct in the Grothendieck group.
As a consequence, Abe's enhancement of Soergel bimodules, denoted $\scrA (\fr{h}, W)$, again categorifies the Hecke algebra.

Building off of work in \cite{LY}, we will develop an analogous theory for the monodromic Hecke algebroid associated to a Coxeter group $W$ with a $W$-set $\fr{o}$ and a $W$-realization $\fr{h}$.
In the monodromic theory, one considers two types of building blocks for Soergel bimodules associated to a simple reflection $s \in S$.
The first is the non-monodromic Soergel bimodule $C_s \coloneq R \otimes_{R^s} R (1)$ associated to $s$ which we use whenever $s$ stabilizes the monodromy parameter $\scrL$.
The second is the $s$-twisted standard bimodule $R_s$ which we use whenever $s$ does not stabilize $\scrL$.
The two different types of building blocks can then be used together to produce a 2-category of monodromic Soergel bimodules over $\fr{o}$, denoted $\Amon{}{} (\fr{h}, W, \fr{o})$.

Monodromic Soergel bimodules have a convenient feature compared to parity sheaves. 
Since the category of monodromic Soergel bimodules is a subcategory of $R$-bimodules, it is quite straightforward to prove a categorified version of the monodromic-endoscopic equivalence.
For $\scrL \in \fr{o}$, write $\Amon{\scrL}{\scrL} (\fr{h}, W, \fr{o})$ for the endomorphism category of $\scrL$ viewed as an object of $\Amon{}{} (\fr{h}, W, \fr{o})$.
There is a decomposition of $\Amon{\scrL}{\scrL} (\fr{h}, W, \fr{o})$ into a direct sum of full subcategories called blocks. We write $\Amon{\scrL}{\scrL}^{\circ} (\fr{h}, W, \fr{o})$ for the block containing the monoidal unit $R$. 
There is a normal subgroup $W_{\scrL}^{\circ}$ of the stabilizer of $\scrL$ in $W$ called the \emph{endoscopic Coxeter group}. This recovers the Weyl group of the endoscopic group when we are in the setting from before.

\begin{proposition}[Proposition \ref{prop:abe_endo}]
    There is an equivalence of monoidal categories
    \[\Amon{\scrL}{\scrL}^{\circ} (\fr{h}, W, \fr{o}) \cong \scrA (\fr{h}, W_{\scrL}^{\circ}).\]
\end{proposition}

This can also be extended to an equivalence of the 2-category $\Amon{}{} (\fr{h}, W, \fr{o})$ with a 2-category of Abe--Soergel bimodules constructed by gluing the categories $\scrA (\fr{h}, W_{\scrL}^{\circ})$ together as $\scrL \in \fr{o}$ varies (see Theorem \ref{thm:alg_endoscopy}).

\subsection{Diagrammatic Hecke Category}

There is one final categorification of the Hecke algebra we will consider. Elias and Williamson asked, and answered, the following question: is it possible to present the category of Soergel bimodules by means of generators and relations \cite{EW}?
The morphisms of the Elias--Williamson diagrammatic Hecke category $\EW (\fr{h}, W)$ are defined using certain isotopy classes of decorated planar graphs.
We will not review these generators and relations at this time, and instead defer the construction to later on.

One key result from \cite{EW} is that the Grothendieck group of the diagrammatic Hecke category is isomorphic to the Hecke algebra.
The strategy for proving this result is rather involved. Namely, they define a basis for the Hom spaces $\EW (\fr{h}, W)$ called the double leaves basis which is closely based on Libedinsky's light leaves basis \cite{Lib08} for Soergel bimodules.
The hard part is showing that a given (potentially massive) diagram can be simplified via the defining relations into a linear combination of double leaves.

We will perform an analogous construction in the monodromic setting. Namely, we will give a presentation of monodromic Soergel bimodules in terms of generators and relations.
Since the categories involved are 2-categories, the resulting generators turn out to be planar graphs with faces labeled by monodromy parameters.
We also construct a double leaves basis for the monodromic diagrammatic Hecke category following ideas of \cite{EW}. 
The cellular structure of the double leaves basis along with a mixed-sheaf formalism allows us to deduce that the monodromic diagrammatic Hecke category categorifies the monodromic Hecke algebroid regardless of the base ring of the realization.

\subsection{Main Results}

We have at this point reviewed three different incarnations of the Hecke category.
Since they all have isomorphic Grothendieck groups (when they are all properly defined), one can ask whether the categories themselves are equivalent.
Two key theorems give answers in the affirmative. First, Abe \cite{Abe19} has shown that his incarnation of Soergel bimodules is equivalent to the Elias--Williamson diagrammatic category.
Abe's result extends a similar result by \cite{EW} when $\fr{h}$ is reflection faithful.
Second, Riche and Williamson \cite{RW} have shown that the parity sheaf incarnation of the Hecke category is equivalent to the Elias--Williamson diagrammatic category.

We are now ready to state our main results which are monodromic variants of the Abe and the Riche--Williamson theorems.

\begin{theorem}[Monodromic Abe]\label{thm:intro_Abe}
    There is an equivalence of 2-categories
    \[\Upsilon_{\textnormal{A}} : \EWmon{}{} (\fr{h}, W, \fr{o}) \stackrel{\sim}{\to} \Amon{}{} (\fr{h}, W, \fr{o}).\]
\end{theorem}

\begin{theorem}[Monodromic Riche--Williamson]\label{thm:intro_RW}
    Assume that $\k$ is either a finite extension of $\Z_{\ell}$, the residue field of such an extension, or the field of fractions of such an extension.
    Let $G$ be a connected complex reductive group with Weyl group $W$ and denote by $\fr{h}$ the Kac--Moody realization for $G$.
    Let $\fr{o}$ be the $W$-orbit of a torsion character sheaf $\scrL$ on the maximal torus $T$.
    There is an equivalence of 2-categories
    \[\Upsilon_{\RW} : \EWmon{}{} (\fr{h}, W, \fr{o}) \stackrel{\sim}{\to} \PEE{}{} (G, \k).\]
\end{theorem}

\begin{remark}
    These theorems along with the monodromic-endoscopic equivalence for Soergel bimodules gives rise to monodromic-endoscopic equivalences for the diagrammatic category.
\end{remark}

\begin{remark}
    One may hope to extend the monodromic Riche--Williamson theorem to arbitrary Kac--Moody groups or for more general rings $\k$.
    The only obstruction is in defining the image under $\Upsilon_{\RW}$ of a simple reflection in $s \in S$ when $\scrL s \neq \scrL$.
    In this case, the minimal IC sheaf associated to $\scrL$ and $s$ is only well-defined up to a choice of scalar.
    It is hard to check whether these choices can be made in a sufficiently coherent manner to ensure the functor is well-defined.
\end{remark}

\subsection{Organization}

In \S\ref{sec:coxeter}, we will study systems of monodromy parameters for arbitrary Coxeter groups. Associated to a monodromy parameter, we will construct the endoscopic Coxeter group.
We will further discuss realizations of Coxeter groups and the necessary 2-colored quantum combinatorics needed to have well-behaved theories of monodromic Abe--Bott--Samelson bimodules.

In \S\ref{sec:abe}, we introduce the 2-category of monodromic Abe--Bott--Samelson bimodules. As part of their study, we will establish an algebraic version of the monodromic-endoscopic equivalence. We continue the study of the algebraic category in \S\ref{sec:ll} where we develop a version of the double leaves basis for monodromic Abe--Bott--Samelson bimodules.

In \S\ref{sec:diag}, we introduce the main character of the paper-- the diagrammatic monodromic Hecke category. We will then define a functor from the diagrammatic category to the algebraic category.
Finally, we will provide a diagrammatic incarnation of the double leaves basis, and deduce an equivalence between the diagrammatic monodromic Hecke category and the category of Abe--Bott--Samelson bimodules (Theorem \ref{thm:intro_Abe}).
We delay the proof that double leaves span until \S\ref{sec:spanning}.

In \S\ref{sec:geom}, we recall that geometric monodromic Hecke category of \cite{Sandvik}. We then construct a functor from the geometric category to the diagrammatic category and show that this functor is an equivalence (Theorem \ref{thm:intro_RW}).

The paper includes two appendices. In \S\ref{apdx:three_color}, we prove that the functor from the diagrammatic Hecke category to the category of Bott--Samelson bimodules is invariant under 3-color relations.
This effort is largely computational and depends on an exhaustive search of the possible systems of monodromy parameters. In \S\ref{apdx:mdc}, we develop a general theory of ``mixed derived categories'' for object-adapted cellular categories closely based on \cite{ARV}.   

\subsection{Acknowledgements}

The author thanks Pramod Achar for numerous helpful discussions on this project.
The author also thanks Ben Elias for fostering an appreciation for diagrammatic calculus and for illuminating conversations regarding double leaves.
The author was partially supported by NSF Grant DMS-2231492.

    \section{Monodromic Coxeter Systems}\label{sec:coxeter}
    \subsection{Endoscopic Coxeter Group}

\subsubsection{Notation}

Let $(W,S)$ be a Coxeter system and let $e \in W$ denote the identity.
By definition $W$ is the group generated by the set of simple reflections $S$ subject to the relations:
\begin{equation}
    s^2 = e \text{ for all } s \in S,
\end{equation}
\begin{equation}
    \underbrace{sts\ldots}_{m_{s,t}} = \underbrace{tst\ldots}_{m_{s,t}} \text{ for all } s\neq t \in S.
\end{equation}
The numbers $m_{s,t} \in \Z_{> 0} \cup \{\infty\}$ associated to each pair of distinct simple reflections determine $W$. 
The group $W$ is equipped with the Bruhat order $\leq$ and the length function $\ell : W \to \Z_{\geq 0}$.
If $W$ is finite, then it admits the longest element $w_0$ under the Bruhat order.
The rank of $W$ is the cardinality of $S$ which we allow to be (countably) infinite.

For any subset $I \subset S$, the corresponding parabolic subgroup $W_I$ of $W$ is the subgroup generated by $s \in I$.
It follows that $(W_I, I)$ is also Coxeter system.

We write $\Exp (W)$ for the set of expressions in $W$. These consist of sequences $\uw = (s_1, \ldots, s_k)$ of simple reflections.
We write $\overline{\uw} = (s_k, \ldots, s_1)$ for the reversed expression of $\uw$.
The length of $\uw$ is $\ell (\uw) \coloneq k$. Note that $\ell (\uw) \geq \ell (w)$ where $w=s_1 \ldots s_k$ is the product, and equality holds if and only if $\uw$ is a reduced expression.
We adopt the convention of \cite{EW} and shorten ``reduced expression'' to \emph{rex} with the plural form being \emph{rexes}. 

A subexpression $\ue$ of an expression $\uw = (s_1, \ldots, s_k)$ is a tuple $(e_1, \ldots, e_k)$ where $e_i \in \{ 0, 1\}$ for all $i=1,\ldots, k$.
Given such a subexpression, we can consider the evaluation of $\uw$ at $\ue$ which is the element $\uw^{\ue} \coloneq s_1^{e_1} \ldots s_k^{e_k} \in W$.
The \emph{Bruhat stroll} associated to $\ue$ is the sequence $w_0, \ldots, w_k$ in $W$ defined by $w_i = s_1^{e_1} \ldots s_i^{e_i}$.
We denote the set of all subexpressions of $\uw$ by $\Subexp (\uw)$.

Suppose we have a pair of distinct simple reflections $s,t$. 
For $k \in \Z_{> 1}$, we denote the expressions given by alternating sequences starting with $s$ (resp. $t$) by 
\[{}_s \underline{k} \coloneq \underbrace{(s,t,s,\ldots)}_{k\text{ terms}} \qquad \text{and} \qquad {}_t \underline{k} \coloneq \underbrace{(t,s,t,\ldots)}_{k\text{ terms}}.\]
We write ${}_s k$ and ${}_t k$ for the elements in $W$ obtained by evaluating ${}_s \underline{k}$ and ${}_t \underline{k}$ respectively.
We could also consider alternating expressions ending in $s$ or $t$, denoted $\underline{k}_s = \overline{{}_s\underline{k}}$ and $\underline{k}_t = \overline{{}_t\underline{k}}$. Their evaluations will similarly be denoted by $k_s$ and $k_t$ respectively.

A \emph{reflection} in $W$ is an element of the form $wsw^{-1}$ for $w \in W$ and $s\in S$. Every reflection is clearly an involution. We denote the set of reflections in $W$ by $\mathscr{R} (W)$.

\subsubsection{Basic Definitions}

Let $(W,S)$ be a Coxeter system.

\begin{definition}
    A set of \emph{monodromy parameters} is a set $\fr{o}$ equipped with a right $W$-action on $\fr{o}$.
    Given monodromy parameters $\scrL, \scrL' \in \fr{o}$, we define
    \[\W{\scrL}{\scrL'} \coloneq \{w \in W \mid \scrL w = \scrL'\}.\]
    If $\scrL = \scrL'$, we will often write $W_{\scrL} \coloneq \W{\scrL}{\scrL}$.

    From $\fr{o}$, we can construct a groupoid $M^{\fr{o}} (W)$ whose object set is $\fr{o}$ and whose set of morphisms from $\scrL$ to $\scrL'$ is given by $\W{\scrL}{\scrL'}$.
    Composition in $M^{\fr{o}} (W)$ is given by multiplication in $W$.
\end{definition}

\begin{example}\label{ex:mon_params}
    \begin{enumerate}
        \item Let $\fr{o}$ be a set with a single object. Then $W$ acts trivially on $\fr{o}$. In this case, $M^{\fr{o}} (W)$ is isomorphic as a groupoid to $W$. We often call this example the \emph{non-monodromic} or \emph{unipotent} case. 
        We will sometimes write $\fr{o} = 1$ to refer to $\fr{o}$ having only one object which we also denote by $1$. 
        \item Let $\fr{o} = W$ and have $W$ act on $\fr{o}$ by right multiplication. In this case, $M^{\fr{o}} (W)$ is equivalent as a groupoid to the trivial group.
        \item The running example for a set of monodromy parameters comes from geometry. Namely, let $G$ be a connected reductive group (or more generally a Kac--Moody group) over $\C$ with maximal torus $T$ and Weyl group $W$.
        Let $\k$ be a commutative ring. We can then consider $T_{\k}^{\vee}$, the Langland's dual torus of $T$.
        In this case, $\fr{o}$ is the $\k$-points of $T_{\k}^{\vee}$.
        \item Let $I \subset S$ and consider $W_I \subseteq W$ the parabolic subgroup generated by $I$. We define $\fr{o} = W_I \backslash W$ and have $W$ act via right multiplication.
        Let $e$ denote the identity coset in $\fr{o}$. Then $W_e = W_I$. More generally, for $x,y \in W_I \backslash W$, then $\W{x}{y} = xW_I y^{-1}$. 
        \item Example \ref{ex:mon_params} (4) admits a mild generalization. Let $W' \subseteq W$ be a subgroup. Define $\fr{o} = W' \backslash W$ and have $W$ act via right multiplication.
        Let $e$ denote the identity coset. Then $W_e = W'$.
    \end{enumerate}
\end{example}

The set $W_{\scrL}$ is a subgroup of $W$, but it need not be a parabolic subgroup. In fact, $W_{\scrL}$ need not admit a Coxeter structure whatsoever.
However, we will show that there is a normal subgroup $W_{\scrL}^{\circ}$ of $W_{\scrL}$ which (naturally) can be endowed with a Coxeter structure.

Note that $\Exp (W)$ is a monoid under concatenation with identity $\emptyset$-- the empty expression.
There is an antiinvolution of $\Exp (W)$ taking $\uw \in \Exp (W)$ to the reversed expression $\overline{\uw}$.
Moreover, there is a map of monoids $\ev : \Exp (W) \to W$ given by evaluating an expression in $W$.
If $\fr{o}$ is a set of monodromy parameters for $W$, then there is a monoid action of $\Exp (W)$ on $\fr{o}$ via the evaluation map $\ev$.

Let $\uw = (s_1, \ldots, s_k) \in \Exp (W)$ and $\scrL \in \fr{o}$. For each $0 \leq i \leq k$, we write $\uw_{\leq i}$ for the expression given by the first $i$ terms of $\uw$.
By definition $\uw_{\leq 0} = \emptyset$. We then define
\[K(\uw, \scrL) \coloneq \{ 1 \leq i \leq k \mid \scrL \uw_{\leq i} = \scrL \uw_{\leq i-1} \}.\]
Associated to each $\uw$ and $\scrL \in \fr{o}$ is a certain subexpression $\ue = (e_1, \ldots, e_k)$ defined as follows:
\begin{equation*}
    e_i = \begin{cases} 1 & i \notin K (\uw, \scrL), \\ 0 & i \in K (\uw, \scrL). \end{cases}
\end{equation*}
This procedure defines a map $\ub_{\scrL} : \Exp (W) \to \Exp (W)$ taking $\uw$ to $\uw^{\ue}$. We also define $b_{\scrL} \coloneq \ev \circ \ub_{\scrL}$.

\begin{lemma}\label{lem:block_prelims}
    Let $\ux, \uy \in \Exp (W)$ and $\scrL \in \fr{o}$.
    \begin{enumerate}
        \item $b_{\scrL} (\ux \uy) = b_{\scrL} (\ux) b_{\scrL \ux} (\uy)$;
        \item $b_{\scrL} (\overline{\ux}) = b_{\scrL \overline{\ux}} (\ux)^{-1}$;
        \item $b_{\scrL} (\ux) = b_{\scrL} (\uy)$ if $\ux, \uy$ are both expressions for some $w \in W$.
    \end{enumerate}
    In particular, the maps $b_{\scrL}$ factor through a morphism of groupoids
    \[\Bl^{\min} : M^{\fr{o}} (W) \to M^{\fr{o}} (W).\]
\end{lemma}
\begin{proof}
    Statements (1) and (2) are clear from definitions.
    We will now prove (3).
    In general, if $\ux$ and $\uy$ are both expressions for some $w \in W$, then $\ux$ and $\uy$ differ by a series of relations of the form
    \[{}_s \underline{m} = \underbrace{(s,t,\ldots)}_{m_{s,t} \text{ terms}} \leftrightarrow \underbrace{(t,s,\ldots)}_{m_{s,t} \text{ terms}} = {}_t \underline{m} \qquad \text{and} \qquad (s,s) \leftrightarrow \emptyset.\]
    By (1), it suffices to show that $b_{\scrL}$ is invariant under these relations. It is clear from definitions that $b_{\scrL} ( (s,s)) = e = b_{\scrL} (\emptyset)$.
    
    We will now show that $b_{\scrL} ({}_s \underline{m}) = b_{\scrL} ({}_t \underline{m})$.
    Suppose that $\scrL s = \scrL$. 
    Then we have that 
    \[\scrL {}_t \underline{m} = \scrL {}_s \underline{m} = \scrL {}_t \underline{m-1}.\]
    We can then compute that
    \[b_{\scrL} ({}_s \underline{m}) = b_{\scrL} ({}_t \underline{m-1}) = b_{\scrL} ({}_t \underline{m}).\]
    The same argument shows that if $\scrL t = \scrL$, then $b_{\scrL} ({}_s \underline{m}) = b_{\scrL} ({}_t \underline{m})$.

    If $\# K ({}_s \underline{m}, \scrL) = 0$, then one can easily check that $b_{\scrL} ({}_s \underline{m}) = w_0$ where $w_0$ is the longest element of $\langle s,t \rangle$.
    Likewise, we have that $b_{\scrL} ({}_t \underline{m}) = w_0$.
    As a result, we may assume that $\scrL {}_s \underline{k} = \scrL {}_s \underline{k-1}$ for some $1 \leq k \leq m$. 
    Write $\scrL' = \scrL {}_s \underline{k}$ and set $u = s$ if $m$ is even and $u = t$ if $m$ is odd.
    By the previous case, we have that $b_{\scrL'} ({}_s \underline{m}) = b_{\scrL'} ({}_t \underline{m})$.
    We can then multiply on the left by $b_{\scrL} ({}_s \underline{k})$ and on the right by $b_{\scrL' {}_s \underline{m}} ({}_u \underline{k-1})^{-1}$.
    By iterated applications of (1) and (2), we can conclude that $b_{\scrL} ({}_s \underline{m}) = b_{\scrL} ({}_t \underline{m})$.
\end{proof}

Lemma \ref{lem:block_prelims} gives a groupoid morphism $\Bl^{\min} : M^{\fr{o}} (W) \to M^{\fr{o}} (W)$. In particular, this restricts to a map ${}_{\scrL} \Bl^{\min}_{\scrL'} : \W{\scrL}{\scrL'} \to \W{\scrL}{\scrL'}$ for $\scrL, \scrL' \in \fr{o}$.
We define the set of \emph{blocks} $\uW{\scrL}{\scrL'}$ as the set of nonempty fibers of ${}_{\scrL} \Bl^{\min}_{\scrL'}$.
Associated to each block $\beta \in \uW{\scrL}{\scrL'}$ is an element $w^\beta = {}_{\scrL} \Bl^{\min}_{\scrL w} (w)$ for any $w \in \beta$.

There is a monodromic length function on $\Exp (W)$ defined by
\[\ell_{\scrL} : \Exp (W) \to \Z_{\geq 0}, \qquad \uw \mapsto \# K (\uw, \scrL).\]
We also define $\ell_{\scrL} : W \to \Z_{\geq 0}$ by $\ell_{\scrL} (w) = \ell_{\scrL} (\uw)$ where $\uw$ is any reduced expression of $w$.
It is easy to check for any expression $\ux$ that $\ell_{\scrL} (\ux) = \ell (\ux) - \ell (b_{\scrL} (\ux))$.
By Lemma \ref{lem:block_prelims}, it then follows that $\ell_{\scrL}$ is independent of the choice of reduced expression.

If $\beta \in \uW{\scrL}{\scrL'}$ is a block, then $\ell_{\scrL} (w^\beta) = 0$ and $w^{\beta}$ is the unique element in $\beta$ with monodromic length 0.
We also note that $\ell_{\scrL}$ does not necessarily agree with the length in $W$.
However, we always have that $\ell_{\scrL} (w) \leq \ell (w)$.

One of the most important examples of blocks are the \emph{neutral blocks} which have minimal element $e \in W$.
We define the neutral block $W_{\scrL}^{\circ} \coloneq ({}_{\scrL} \Bl^{\min}_{\scrL})^{-1} (\{e\})$.
By Lemma \ref{lem:block_prelims}, the neutral block $W_{\scrL}^\circ$ is a normal subgroup of $W_{\scrL}$.

\begin{proposition}\label{prop:endo_gps_are_coxeter_gps}
    Let $S_{\scrL}^\circ = \ell_{\scrL}^{-1} (1) \cap W_{\scrL}^{\circ}$. Then $(W_{\scrL}^\circ, S_{\scrL}^\circ)$ is a Coxeter group with length function $\ell_{\scrL}$.
\end{proposition}
\begin{proof}
    We first claim that $S_{\scrL}^{\circ}$ generates $W_{\scrL}^{\circ}$.
    Let $w \in W_{\scrL}^{\circ}$. We will argue by induction on $\ell_{\scrL} (w)$ that $w \in \langle S_{\scrL}^{\circ} \rangle$.
    If $\ell_{\scrL} (w) = 1$, the claim is obvious. Suppose $\ell_{\scrL} (w) > 1$. We can then write $w = xsy$ such that $\ell_{\scrL} (x) = \ell_{\scrL} (w) - 1$, $s \in S_{\scrL x}^{\circ}$, and $\ell_{\scrL x} (y) = 0$.
    Write $w = (xy)(y^{-1} s y)$ and observe that $xy \in W_{\scrL}^{\circ}$ and $y^{-1} sy \in S_{\scrL}^{\circ}$. Note that $\ell_{\scrL} (xy) = \ell_{\scrL} (w) - 1$. The claim then follows via induction.
    
    Next, we will show that the elements of $S_{\scrL}^{\circ}$ are reflections. 
    Let $s \in S_{\scrL}^{\circ}$. We can then find blocks $\beta \in \uW{\scrL}{\scrL'}$ and $\gamma \in \uW{\scrL'}{\scrL}$ and a simple reflection $t \in S$ such that $s = w^{\beta} t w^{\gamma}$ and $t \in S_{\scrL'}^{\circ}$.
    By definition of $W_{\scrL}^{\circ}$, we must have that $w^{\beta} w^{\gamma} = e$; therefore, $s = w^{\beta} t w^{\beta, -1}$.

    We have now shown that $W_{\scrL}^{\circ}$ is a reflection subgroup. By a classical theorem of Dyer \cite{Dyer} and Deodhar \cite{Deodhar}, $W_{\scrL}^{\circ}$ is a Coxeter group with a canonical set of generators $T_{\scrL}^{\circ}$.
    We can describe $T_{\scrL}^{\circ}$ as follows.
    For each $x \in W$, define $N(x) \coloneq \{r \in \mathscr{R} (W) \mid \ell (xr) < \ell (x) \}$. Then $T_{\scrL}^{\circ} = \{ t \in \mathscr{R} (W) \mid N(t) \cap W_{\scrL}^{\circ} = \{t\} \}$.
    
    We claim that $T_{\scrL}^{\circ} = S_{\scrL}^{\circ}$.
    First, let $x \in W_{\scrL}^{\circ}$. Let  $(s_1, \ldots, s_k)$ be a reduced expression of $x$. If $\ell_{\scrL} (x) > 1$, we can find some index $1 \leq i \leq k$ such that $\scrL s_1 \ldots s_i =  \scrL s_{1} \ldots s_{i+1}$ and that $i$ is minimal with respect to this property.
    Let $r = s_1 \ldots s_{i-1} s_i s_{i-1} \ldots s_1 \in \mathscr{R} (W) \cap W_{\scrL}^{\circ}$. By construction, $\ell (rx) < \ell (x)$; therefore, $x \notin T_{\scrL}^{\circ}$.
    As a result, $T_{\scrL}^{\circ} \subseteq S_{\scrL}^{\circ}$.
    For the other inclusion, let $s \in S_{\scrL}^{\circ}$ and $r \in \mathscr{R} (W) \cap W_{\scrL}^{\circ}$ such that $\ell (sr) < \ell (s)$. Since $s \in S_{\scrL}^{\circ}$, we have that $\ell_{\scrL} (sr) \leq \ell_{\scrL} (s) = 1$. 
    If $\ell_{\scrL} (sr) = 1$, then $\ell_{\scrL} (r) = 0$ and hence $r=e$ which is a contradiction. Therefore, $\ell_{\scrL} (sr) = 0$.
    As a result, $sr = e$, and hence, $N (s) \cap W_{\scrL}^{\circ} = \{s\}$.
\end{proof}

In light of Proposition \ref{prop:endo_gps_are_coxeter_gps}, the group $W_{\scrL}^{\circ}$ is also called the \emph{endoscopic Coxeter group} for $\scrL$.
We refer to the elements in $S_{\scrL}^{\circ}$ as the \emph{endosimple reflections}.

The following lemma gives an easy way to calculate the endoscopic Coxeter group.

\begin{lemma}\label{lem:calculating_endoscopy_groups}
    For all $\scrL \in \fr{o}$, one has that $W_{\scrL}^{\circ} = \langle \mathscr{R} (W) \cap W_{\scrL} \rangle$. 
\end{lemma}
\begin{proof}
    Proposition \ref{prop:endo_gps_are_coxeter_gps} implies that $W_{\scrL}^{\circ} \subseteq \langle \mathscr{R} (W) \cap W_{\scrL} \rangle$ since $W_{\scrL}^{\circ}$ is a reflection subgroup.
    It remains to check the opposite inclusion. Let $r \in \mathcal{R} (W) \cap W_{\scrL}$ and write $r = wsw^{-1}$ for some $w \in W$ and $s \in S$.
    Note that $s \in W_{\scrL w}^{\circ}$.
    By Lemma \ref{lem:block_prelims}, we have that 
    \[ {}_{\scrL} \Bl^{\min}_{\scrL} (r) = \left( {}_{\scrL} \Bl^{\min}_{\scrL w} (w) \right) \left( {}_{\scrL w} \Bl^{\min}_{\scrL w} (s) \right) \left( {}_{\scrL w} \Bl^{\min}_{\scrL} (w^{-1}) \right) = \left( {}_{\scrL} \Bl^{\min}_{\scrL w} (w) \right) \left( {}_{\scrL} \Bl^{\min}_{\scrL} (w)^{-1} \right) = e.\]
    Therefore, $r \in W_{\scrL}^{\circ}$ which completes the proof.
\end{proof}

\begin{example}\label{ex:endoscopic_groups}
    \begin{enumerate}
        \item Let $\fr{o} = 1$. Then $W_{1}^{\circ} = W$.
        \item Let $\fr{o} = W$ and have $W$ act on $\fr{o}$ by right multiplication. Then for all $x \in W$, $W_x^{\circ} = \{e\}$. 
        \item Let $G$ be a connected reductive group with maximal torus $T$ and Weyl group $W$ and set $\fr{o} = T_{\k}^{\vee}$. 
        Let $\Phi$ denote the root system of $G$ with respect to $T$ and $\Phi^{\vee}$ the set of coroots.
        Each coroot in $\Phi^{\vee}$ defines a character of $T_{\k}^{\vee}$. We can then set 
        \[\Phi_t^{\vee} \coloneq \{ \alpha^{\vee} \in \Phi^{\vee} \mid \alpha^{\vee} (t) = 1 \}.\]
        Let $\Phi_t$ denote the subset of $\Phi$ dual to $\Phi_t^{\vee}$. The subset $\Phi_t$ is a sub-root system of $\Phi$ with positive roots $\Phi_t \cap \Phi^+$.
        It follows from \cite[\S 4]{LY} that the Weyl group of $\Phi_t$ with respect to $\Phi_t^+$ is equal to $W_t^{\circ}$.
        \item Let $J \subset S$ and $\fr{o} = W_I \backslash W$. Then $W_{e}^{\circ} = W_I$. More generally, for all cosets $x \in W_I \backslash W$ with representative $\dot{x} \in W$, $W_x^{\circ} = \dot{x}W_I \dot{x}^{-1}$.
        \item  Let $W' \subseteq W$ be a reflection subgroup. Let $\fr{o} = W' \backslash W$. As in Example \ref{ex:mon_params}, we have that $W_e = W'$.
        It then follows from Lemma \ref{lem:calculating_endoscopy_groups} that $W_e^{\circ} = W'$. In particular, every reflection subgroup may arise as an endoscopic Coxeter group.
    \end{enumerate}
\end{example}

\begin{example}
    \begin{enumerate}
        \item Endoscopic Coxeter subgroups need not be parabolic subgroups of $W$. 
            Let $W$ be the Coxeter group of type $B_2$, with generators $s$ and $t$. Let $W'$ be the reflection subgroup generated by the reflections $s$ and $tst$.
            Note that $W'$ is of type $A_1 \times A_1$ and clearly is not a parabolic subgroup of $W$. We can then take $\fr{o} = W' \backslash W$.
            By Example \ref{ex:endoscopic_groups} (5), we have that $W_e^{\circ} = W'$.
        \item Endoscopic Coxeter subgroups need not be finite rank Coxeter groups even when $W$ has finite rank. 
            Let $W$ be the universal Coxeter group on three generators $s_1, s_2, s_3$.
            Consider the element $x = s_1 s_2$, which has infinite order. Define an infinite sequence of reflections $r_n = x^n s_3 x^{-n}$ for all $n \in \Z_{\geq 0}$.
            Let $W'$ be the reflection subgroup of $W$ generated by $\{r_n\}_{n \geq 0}$.
            Here we have that $W'$ is isomorphic to the free product of countably infinitely many copies of $\Z/2\Z$. It can then be easily verified that $W'$ is not isomorphic to a Coxeter group with finitely many simple reflections.
            Again, we can take $\fr{o} = W' \backslash W$ and Example \ref{ex:endoscopic_groups} (5) shows that $W_e^{\circ} = W'$.
    \end{enumerate}
\end{example}

\begin{definition}
    An expression $\uw \in \Exp (W)$ of $w \in W_{\scrL}^{\circ}$ is said to be \emph{endo-reduced} if $\ell_{\scrL} (\uw) = \ell_{\scrL} (w)$.
\end{definition}

\begin{definition}
    An \emph{endosimple expansion datum} for $\scrL$ is a function
    \[\iota : S_{\scrL}^{\circ} \to \Exp (W)\]
    such that $\iota (s)$ is an endo-reduced expression for $s \in W$. 
\end{definition}

Given an endosimple expansion datum $\iota$ for $\scrL$, then $\iota$ can be extended uniquely to all $W_{\scrL}^{\circ}$-expressions,
\[\iota : \Exp (W_{\scrL}^{\circ}) \to \Exp (W).\]
Moreover, $\iota$ satisfies that if $\ux$ is a reduced expression for some $x \in W_{\scrL}^{\circ}$, then $\iota (\ux)$ is an endo-reduced expression in $W$.

\subsection{Realizations}

For both Bott--Samelson bimodules and the diagrammatic category, the starting point will be the data of a realization of a Coxeter system.
In order for the resulting categories to be well-defined, one has to impose some rather technical combinatorial constraints.
We will cover the necessary combinatorics in this section.

\subsubsection{Two-Colored Quantum Numbers}

Let $A = \Z [x,y]$ be the integral polynomial ring in two variables.
We define a collection of elements of $A$ called the \emph{two-colored quantum numbers} as follows.
First set $[0]_x = 0$, $[0]_y = 0$, $[1]_x = 1$, $[1]_y = 1$, $[2]_x = x$, and $[2]_y = y$.
Define the remaining quantum numbers recursively by the formulas
\begin{equation}
    [n+1]_x = [2]_x [n]_y - [n-1]_x \qquad \text{and} \qquad [n+1]_y = [2]_y [n]_x - [n-1]_y.
\end{equation}
One can use the same formulas to define $[n]_x$ and $[n]_y$ for $n \leq 0$. It is easy to check that $[-n]_x = -[n]_x$ and $[-n]_y = - [n]_y$.
For all $d \in \Z$, we define the \emph{$d$-twisted variables}, 
\[ x(d) = [d+1]_x - [d-1]_x \qquad \text{and} \qquad y(d) = [d+1]_y - [d-1]_y.\]
For all $n \in \Z$, we will write $[n]_{x(d)}$ (resp. $[n]_{y(d)}$) for the element of $A$ obtained by substituting $x$ with $x(d)$ and $y$ with $y(d)$ in $[n]_{x}$ (resp. $[n]_y$).

We now recall some standard facts about two-colored quantum numbers (cf., \cite{Elib, EW, EW17}).

\begin{lemma}\label{lem:facts_about_quantum_numbers}
    Let $n, m \in \Z$.
    \begin{enumerate}
        \item If $n$ is odd, then $[n]_x = [n]_y$. As a result, if $d$ is even, then $x(d) = y(d)$.
        \item If either $n$ is even or both $n, m$ are odd, then
        \[[n]_x [m]_y = \sum_{k=0}^{m-1} [n+m-1-2k]_x.\]
        \item If $n$ divides $m$, then $[n]_x$ divides $[m]_x$.
        \item We have that $[2n]_x = [n]_x ([n+1]_x - [n-1]_x)$.
    \end{enumerate}
\end{lemma}

It will be useful to give a recursive formula for computing the $d$-twisted variables similar to the two-colored quantum numbers.
We write $x'(0) = 2$, $y'(0) = 2$, $x'(1) = x$, and $y'(1) = y$. We can then define inductively for $d \geq 2$,
\[ x'(d) = \begin{cases} -x'(d-2) + x \cdot x'(d-1) & \text{if } d \text{ odd}, \\ -x'(d-2) + y \cdot x'(d-1) & \text{if } d \text{ even},\end{cases}\]
\[ y'(d) = \begin{cases} -y'(d-2) + y \cdot y'(d-1) & \text{if } d \text{ odd}, \\ -y'(d-2) + x \cdot y'(d-1) & \text{if } d \text{ even}.\end{cases} \]

\begin{lemma}\label{lem:recursive_formula_for_twisted_vars}
    For all $d \geq 0$, we have that $x'(d) = x(d)$ and $y'(d) = y(d)$. 
\end{lemma}
\begin{proof}
    We will just prove that $x' (d) = x(d)$. The other statements follow from an entirely symmetric argument. We will argue by induction on $d$.
    If $d = 0$, then $x(0) = 1-(-1) = 2$. If $d=1$, then $x(1) = [2]_x = x$.
    Assume that $x' (k) = x(k)$ for $k < d$ and $d \geq 2$. 
    If $d$ is even, then
    \begin{align*}
        x'(d) &= -x' (d-2) + y \cdot x' (d-1) \\
        &= -x (d-2) + y \cdot x (d-1) \\
       &= -[d-1]_x - [d-3]_x + y [d]_x - y[d-2]_x \\
       &= (y[d]_x - [d-1]_y) - (y[d-2]_x - [d-3]_y) \\
       &= [d+1]_y - [d-1]_y \\
       &= [d+1]_x - [d-1]_x.
    \end{align*}
   The case of $d$ odd is similar, and thus omitted.
\end{proof}

The twisted variables are intimately related with quotients of two-colored polynomials as explained by the following proposition.

\begin{proposition}\label{prop:twisted_vars_and_two_color}
    Let $n,v \in \Z$ such that $v \mid n$ and $v \neq 0$. Write $d = n/v$.
    \begin{equation}\label{eq:prop_twisted_vars_and_two_color_1} 
        [v]_{x(d)} = \frac{[n]_x}{[d]_x} \qquad \text{and} \qquad [v]_{y(d)} = \frac{[n]_y}{[d]_y}. 
    \end{equation}
\end{proposition}
\begin{proof}
    We will argue by induction on $v$. If $v= 0$ or $v=1$, there is nothing to show.
    The case of $v=2$ follows immediately from Lemma \ref{lem:facts_about_quantum_numbers} (4).
    Suppose $v > 2$. We will just consider the case when $v$ is even and $d$ is odd. The other cases are entirely similar.
    Then we can compute via induction and Lemma \ref{lem:facts_about_quantum_numbers} (2), 
    \begin{align*}
        [v]_{x(d)} [d]_x &= ([v-1]_{y (d)} [2]_{x(d)} - [v-2]_{x(d)}) [d]_x \\
        &= [v-1]_{y(d)} [2d]_{x} - [dv-2d]_x \\
        &= \frac{[d(v-1)]_y [2d]_x - [dv-2d]_x [d]_y}{[d]_y} \\
        &= \frac{1}{[d]_y} \left( \sum_{k=0}^{2d-1} [dv+d-1-2k]_x -  \sum_{k=0}^{d-1} [dv-d-1-2k]_x \right)  \\
        &= \frac{1}{[d]_y} \sum_{k=0}^{d-1} [dv+d-1-2k]_x \\
        &= \frac{[dv]_x [d]_y}{[d]_y} \\
        &= [dv]_x.
    \end{align*}
\end{proof}

\begin{definition}
    Let $n,d \in \Z_{\geq 0}$ and $0 \leq k \leq n$. Define the \emph{two-colored quantum binomial coefficients} as
    \begin{equation}
        \binom{n}{k}_x \coloneq \frac{[n]_x [n-1]_x \ldots [n-k+1]_x}{[1]_x [2]_x \ldots [k]_x}.
    \end{equation}
    We also define the \emph{$d$-twisted two-colored quantum binomial coefficients} as
    \begin{equation}
        \binom{n}{k}_{x(d)} \coloneq \frac{[n]_{x(d)} [n-1]_{x(d)} \ldots [n-k+1]_{x(d)}}{[1]_{x(d)} [2]_{x(d)} \ldots [k]_{x(d)}}.
    \end{equation}
    We can also define $\binom{n}{k}_x$ (resp. $\binom{n}{k}_{x(d)}$) by switching all occurrences of $x$ (resp. $x(d)$) by $y$ (resp. $y(d)$) in the above formulas.
\end{definition}
Note that $A$ is a UFD, and one can use standard arguments to prove that the ($d$-twisted) quantum binomial coefficients are elements of $A$.
In light of Proposition \ref{prop:twisted_vars_and_two_color}, the $d$-twisted coefficients are also given by the following equation,
\begin{equation}\label{eq:cor_of_twisted_vars_and_two_color}
    \binom{n}{k}_{x(d)} =\frac{[d(n)]_{x} [d(n-1)]_{x} \ldots [d(n-k+1)]_{x}}{[d]_{x} [2d]_{x} \ldots [kd]_{x}}.
\end{equation}

\subsubsection{Rank 2 Reflection Subgroups of Dihedral Groups}

Let $m \in \Z_{\geq 0} \cup \{ \infty\}$.
Let $W_m$ denote the dihedral group of order $2m$. If $m \geq 2$, we regard $W_m$ as a Coxeter group via the presentation $\langle s,t \mid s^2 = t^2 = (st)^m \rangle$.
More generally, we interpret $W_0$ as the trivial group and $W_1$ as the type $A_1$ Coxeter group. 

The following lemma is well-known and can be proved using basic group theory.

\begin{lemma}\label{lem:class_of_rk_2_subgroups_of_dihedral_groups}
    Assume $m \in \Z_{\geq 0}  \cup \{ \infty\}$.
    Let $W'$ be a rank $\leq 2$ reflection subgroup of $W_m$. If $m < \infty$, then $W'$ is isomorphic as a Coxeter group to $W_v$ for some $v$ dividing $m$ or $v = 0$.
    If $m = \infty$, then $W'$ is an infinite dihedral group.
\end{lemma}

If $\fr{o}$ is a $W_m$-set and $\scrL \in \fr{o}$, the endoscopic subgroup $W_{\scrL}^{\circ}$ is a reflection subgroup.
We define $v_{s,t}^{\scrL} \in \Z_{\geq 0} \cup \{\infty\}$ where $W_{\scrL}^{\circ}$ is of type $I_2 (v_{s,t}^{\scrL})$ (where $I_2 (0)$ is defined to be trivial and $I_2 (1)$ is defined to be $A_1$). We call $v_{s,t}^{\scrL}$ the \emph{endoscopic order} of $st$.

\subsubsection{Abe Realizations}

Let $\k$ be an integral domain. Let $a,b \in \k$. Consider the ring homomorphism $A \to \k [x,y]/(x-a, y-b) \cong \k$ defined by the canonical map $\Z \to \k$.
Given $p \in A$, we will write $p(a,b)$ for the element of $\k$ obtained via this homomorphism.
At times, when the choice of $a,b$ are obvious, we will simply refer to $p$ as an element of $\k$.

\begin{definition}\label{def:realization}
    Let $\k$ be an integral domain. A \emph{pre-realization} of a Coxeter group $(W,S)$ over $\k$ consists of a free, finite rank $\k$-module $\fr{h}$ along with subsets
    \[ \{ \alpha_s^{\vee} : s \in S \} \subset \fr{h} \qquad\text{and}\qquad \{ \alpha_s : s \in S \} \subset \fr{h}^*, \]
    called the simple coroots and simple roots respectively, such that
    \begin{enumerate}
        \item $\langle \alpha_s^{\vee}, \alpha_s \rangle = 2$ for all $s \in S$;
        \item the assignment
            \[s (v) = v - \langle v, \alpha_s \rangle \alpha_s^{\vee}\]
            for all $s \in S$ and $v \in \fr{h}$ defines a representation of $W$ on $\fr{h}$;
        \item $\alpha_s : \fr{h} \to \k$ is surjective for all $s \in S$.
    \end{enumerate}
    Here we are writing $\langle -, - \rangle : \fr{h} \times \fr{h}^* \to \k$ for the canonical pairing. The third condition is called \emph{Demazure surjectivity}.

    To every pre-realization $\fr{h}$, there is a \emph{Cartan matrix} $(a_{s,t})_{s,t \in S}$ where $a_{s,t} \coloneq \langle \alpha_s^{\vee}, \alpha_t \rangle$.
    We will further call a pre-realization $\fr{h}$ an \emph{Abe realization} if for all distinct $s,t \in S$ such that $st$ has order $m_{s,t} < \infty$, we have
        \[\binom{m_{s,t}}{k}_x (a_{s,t}, a_{t,s}) = \binom{m_{s,t}}{k}_y (a_{s,t}, a_{t,s}) = 0\]
        for all integers $1 \leq k \leq m_{s,t} - 1$.
\end{definition}

Hazi has shown that in the absence of parabolic subgroups of type $H_3$, the diagrammatic Hecke category is well-defined if and only if the underlying realization is an Abe realization \cite{Hazi}.
We will later see that the diagrammatic monodromic Hecke category may not be well-defined, even if the underlying realization is an Abe realization.
To resolve this, we introduce a collection of modified versions of the Abe realization condition.

\begin{definition}\label{def:endo_abe_realization}
    Let $\fr{h}$ be a pre-realization over $\k$ of a Coxeter group $(W,S)$. Let $\fr{o}$ be a right $W$-set.
    We call $\fr{h}$ an \emph{endoscopic Abe realization} if $\fr{h}$ is an Abe realization of $(W_{\scrL}^{\circ}, S_{\scrL}^{\circ})$ for all $\scrL \in \fr{o}$.
\end{definition}

\begin{definition}\label{def:reflection_abe_realization}
    Let $\fr{h}$ be a pre-realization over $\k$ of a Coxeter group $(W,S)$. Let $\fr{o}$ be a right $W$-set.
    We call $\fr{h}$ a \emph{reflection-stable Abe realization} if for all distinct $s,t \in S$ with $m_{s,t} < \infty$ and $v \in \Z_{> 1}$ dividing $m_{s,t}$, we have that
    \[\binom{v}{k}_{x(m_{s,t}/v)} (a_{s,t},a_{t,s}) = \binom{v}{k}_{y(m_{s,t}/v)} (a_{s,t}, a_{t,s}) = 0\]
    for all integers $1 \leq k \leq v - 1$.
\end{definition}

\begin{definition}\label{def:monodromic_abe_realization}
    Let $\fr{h}$ be a pre-realization over $\k$ of a Coxeter group $(W,S)$. Let $\fr{o}$ be a right $W$-set.
    We call $\fr{h}$ a \emph{monodromic Abe realization} if for all distinct $s,t \in S$ with $m_{s,t} < \infty$ and $\scrL \in \fr{o}$ such that $0 < v_{s,t}^{\scrL} < \infty$, we have that
    \[\binom{v_{s,t}^{\scrL}}{k}_{x(m_{s,t}/v_{s,t}^{\scrL})} (a_{s,t}, a_{t,s}) = \binom{v_{s,t}^{\scrL}}{k}_{y(m_{s,t}/v_{s,t}^{\scrL})} (a_{s,t}, a_{t,s}) = 0\]
    for all integers $1 \leq k \leq v_{s,t}^{\scrL} - 1$.
\end{definition}

The notion of a monodromic Abe realization is sufficient to have a well-behaved category of monodromic Abe--Soergel bimodules.
The author suspects that any monodromic Abe realization is also an endoscopic Abe realization; however, the author does not currently know a way to prove this.
It is clear from their definitions that any reflection-stable Abe realization is also a monodromic Abe realization.
We will show in Corollary \ref{cor:refl_abe_are_endo} that reflection-stable Abe realizations are also endoscopic Abe realizations.
As a result, they are a sufficiently restrictive to give well-behaved monodromic Hecke categories and endoscopic Hecke categories.
We summarize the above in the following diagram:
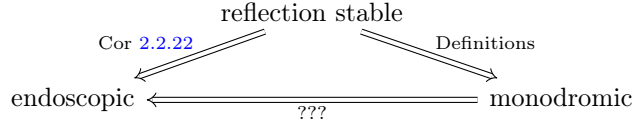
\[\begin{tikzcd}
                  & \text{reflection stable} \arrow[ld, "\text{Cor \ref{cor:refl_abe_are_endo}}"', Rightarrow] \arrow[rd, "\text{Definitions}", Rightarrow] &                                                 \\
\text{endoscopic} &                                                                                                           & \text{monodromic} \arrow[ll, "???", Rightarrow]
\end{tikzcd}\]
Fortunately, asking that a pre-realization $\fr{h}$ be a reflection-stable Abe realization is not too restrictive. 
For example, we show in Proposition \ref{prop:crystallographic_abe} that if $\fr{h}$ arises via a root system of a Kac--Moody group, then it is automatically reflection-stable. 

\begin{remark}\label{rem:necessity_for_refln_datums}
Let $r \in W$ be a reflection. We can then write $r = wsw^{-1}$ for some $s \in S$ and $w \in W$.
We can then define elements $\alpha_r^{\vee} = w \alpha_s^{\vee} \in \fr{h}$ and $\alpha_r = w \alpha_s \in \fr{h}^*$.
Unfortunately, $\alpha_r^{\vee}$ and $\alpha_r$ are not actually well-defined. Namely, they depend on the choice of $w$ and $s$.
However, by \cite[Lemma 2.1]{Abe19}, $\alpha_{r}$ and $\alpha_r^{\vee}$ is independent of the choice of $w$ and $s$ up to a scalar in $\k^{\times}$. 
Let $r = xtx^{-1}$ for some other $t \in S$ and $x \in W$.
We now have that $w \alpha_s = \lambda x \alpha_t$ for some $\lambda \in \k^{\times}$. It can be easily checked that the constraint from Definition \ref{def:realization} (1) ensures that $w \alpha_s^{\vee} = \lambda^{-1} x \alpha_t$.
As a consequence, for two reflections $r,r' \in W$, we have that $a_{r,r'} a_{r',r} \coloneq \langle \alpha_r^{\vee}, \alpha_{r'} \rangle \langle \alpha_{r'}^{\vee}, \alpha_{r} \rangle$ is well-defined and independent of all choices.
\end{remark}

\begin{definition}
    Let $(W,S)$ be a Coxeter group and $(W', S')$ be a reflection subgroup.
    A \emph{reflection datum} for $W'$ consists of a pair of functions $w_{-} : S' \to W$ and $r_{-} : S' \to S$ such that for all $s' \in S'$, we have that $s' = w_{s'} r_{s'} w_{s'}^{-1}$ and $w_{s'} r_{s'} > w_{s'}$.
\end{definition}

\begin{lemma}\label{lem:pre_realizations_and_reflection_subgroups}
    Let $(W,S)$ be a Coxeter group and $W'$ is a reflection subgroup of $W$ with canonical generators $S'$.
    Let $\fr{h}$ be a pre-realization for $W$.
    \begin{enumerate}
        \item Let $w_{-} : S' \to W$ and $r_{-} : S' \to S$ be a reflection datum for $W'$. Then 
        \[ (\fr{h}, \{w_{s'} \alpha_{r_{s'}}\}_{s' \in S'}, \{w_{s'} \alpha_{r_{s'}}^{\vee}\}_{s' \in S'} )\] 
        is a pre-realization for $W'$.
        \item The condition that $\fr{h}$ be an Abe realization for $W'$ is independent of the choice of reflection datum for $W'$.
    \end{enumerate}
\end{lemma}
\begin{proof}
    Let $w_{-} : S' \to W$ and $r_{-} : S' \to S$ be a reflection datum for $W'$. Let $s' \in S'$. 
    We define $\alpha_{s'} = w_{s'} \cdot \alpha_{r_{s'}}$ and $\alpha_{s'}^{\vee} = w_{s'} \cdot \alpha_{r_{s'}}^{\vee}$.
    It then follows that 
    \[ \langle \alpha_{s'}^{\vee}, \alpha_{s'} \rangle = \langle w_{s'} \cdot \alpha_{r_{s'}}^{\vee},  w_{s'} \cdot \alpha_{r_{s'}} \rangle = \langle \alpha_{r_{s'}}^{\vee}, \alpha_{r_{s'}} \rangle =  2.\]
    Let $v \in \fr{h}$.
    \begin{align*}
        v - \langle v, w_{s'} \cdot \alpha_{r_{s'}} \rangle w_{s'} \cdot \alpha_{r_{s'}}^{\vee} &= v - \langle w_{s'}^{-1} v, \alpha_{r_{s'}} \rangle w_{s'} \cdot \alpha_{r_{s'}}^{\vee} \\
        &= w_{s'} \cdot \left( w_{s'}^{-1} v - \langle w_{s'}^{-1} v, \alpha_{r_{s'}} \rangle \alpha_{r_{s'}}^{\vee} \right) \\
        &= w_{s'} ( r_{s'} (w_{s'}^{-1} v)) \\
        &= s' (v).
    \end{align*}
    As a result, the $W'$ action on $\fr{h}$ given by restricting the $W$ action coincides on $S'$ with the formula from Definition \ref{def:realization} (2). 
    Finally, Demazure surjectivity for $\fr{h}$ follows from $\langle -, - \rangle$ being $W$-invariant. 

    It remains to check that whether $\fr{h}$ is an Abe realization for $W'$ is independent of the choice of reflection datum. 
    Let $y_{-} : S' \to W$ and $u_{-} : S' \to S$ be another reflection datum for $W'$. 
    We define $\alpha_{s'}' = y_{s'} \cdot \alpha_{u_{s'}}$ and $(\alpha_{s'}')^{\vee} = y_{s'} \cdot \alpha_{u_{s'}}^{\vee}$. Denote by $a_{\bullet, \bullet}$ (resp. $a_{\bullet, \bullet}'$) the Cartan matrix for the restriction of $\fr{h}$ to $W'$ with respect to the reflection datum $(w_{-}, r_{-})$ (resp. $(y_{-}, u_{-})$).
    By Remark \ref{rem:necessity_for_refln_datums}, we have that 
    \begin{equation}
        a_{s', t'} a_{t',s'} \coloneq \langle \alpha_{s'}^{\vee}, \alpha_{t'} \rangle \langle \alpha_{t'}^{\vee}, \alpha_{s'} \rangle =  \langle (\alpha_{s'}')^{\vee}, \alpha_{t'}' \rangle \langle (\alpha_{t'}')^{\vee}, \alpha_{s'}' \rangle \eqcolon a_{s',t'}' a_{t',s'}'.
    \end{equation}
    Let $n \in \Z$ be odd. Then $[n] \in \Z [xy]$. As a result, evaluating $[n]$ at $x= a_{s',t'}$ and $y = a_{t',s'}$ is the same as evaluating $[n]$ at $x= a_{s',t'}'$ and $y = a_{t',s'}'$.
    If $n \in \Z$ is even, then $[n]_x \in x\Z[xy]$ and $[n]_y \in y \Z[xy]$.
    As a result, 
    \[[n]_x (a_{s',t'}, a_{t',s'}) = \lambda [n]_x (a_{s',t'}' , a_{t',s'}')\]
    for some $\lambda \in \k^{\times}$ and similarly for $[n]_y$.
    As a result, the vanishing of the quantum binomial coefficients is independent of the choice of reflection datum.
\end{proof}

Since we are primarily concerned with knowing when the realization $\fr{h}$ for $W$ restricts to an Abe realization for a reflection subgroup $W'$, we do not particularly care what choice of simple roots and coroots is needed.
This in addition with Lemma \ref{lem:pre_realizations_and_reflection_subgroups} (2) empowers us to suppress the reflection datum when discussing restricted pre-realizations to $W'$.
In particular, we will call any of the pre-realizations arising via Lemma \ref{lem:pre_realizations_and_reflection_subgroups} (1) \emph{the} restricted pre-realization even though it does depend on a choice.
This abuse of terminology will be fully rectified in Lemma \ref{lem:well_defined_pos_roots_balancedness}.

The following two lemmas easily follow from the definitions.

\begin{lemma}\label{lem:parabolic_subgroups_and_Abe_realizations}
    Let $(W,S)$ be a Coxeter system and let $\fr{h}$ be a pre-realization for $W$. 
    Let $P$ be one of the following properties of a (pre-)realization: monodromic Abe, reflection-stable Abe, or Abe.
    The following are equivalent
    \begin{enumerate}
        \item $\fr{h}$ is a $P$ realization for $W_{\{s,t\}}$ for distinct $s,t \in S$;
        \item $\fr{h}$ is a $P$ realization for $W$;
        \item $\fr{h}$ is a $P$ realization for $W_I$ for all $I \subseteq S$;
        \item $\fr{h}$ is a $P$ realization for ${}^x W_I \coloneq x W_I x^{-1}$ for all $x \in W$ and $I \subseteq S$.
    \end{enumerate}
\end{lemma}

\begin{lemma}\label{lem:refln_subgroups_conjugation_and_Abe}
    Let $(W,S)$ be a Coxeter system and let $\fr{h}$ be a pre-realization for $W$. 
    Let $(W', S')$ be a reflection subgroup of $W$ and let $x \in W$.
    Then $\fr{h}$ is an Abe realization for $(W',S')$ if and only if $\fr{h}$ is an Abe realization for $(xW'x^{-1}, xS'x^{-1})$.
\end{lemma}

\begin{lemma}\label{lem:parabolic_closures_of_refln_subgroups}
    Let $(W,S)$ be a Coxeter group and $W'$ is a reflection subgroup of $W$ with canonical generators $S'$.
    Assume that $W'$ is a finite Coxeter group of rank $r$. Then there exists a rank $r$ parabolic subgroup ${}^w W_I$ of $W$ such that $W' \subseteq {}^w W_I$. 
\end{lemma}
\begin{proof}
    It is a classical result that $W'$ is contained in some finite parabolic subgroup of $W$ (see \cite[Page 130]{Bourbaki}).
    As a result, without loss of generality, we may assume that $W$ itself is finite.
    The result then follows from \cite[Lemma 2.1]{DPR13}. 
\end{proof}

\begin{lemma}\label{lem:twisted_vars_and_pairings}
    Let $W = W_m = \langle s,t \rangle$ be a finite dihedral group.
    For each $k \in \Z_{>1}$, define 
    \[\alpha_k \coloneq \begin{cases} ({}_t (k-1)) \cdot \alpha_t & \text{if } k \text{ odd}, \\ ({}_t (k-1)) \cdot \alpha_s & \text{if } k \text{ even},\end{cases} \qquad\qquad \alpha_k^{\vee} \coloneq \begin{cases} ({}_t (k-1)) \cdot \alpha_t^{\vee} & \text{if } k \text{ odd}, \\ ({}_t (k-1)) \cdot \alpha_s^{\vee} & \text{if } k \text{ even}.\end{cases}\]
    We also set $\alpha_0 = -\alpha_s$ and $\alpha_0^{\vee} = -\alpha_s^{\vee}$.
    We then have that
    \[\langle \alpha_s^{\vee}, \alpha_{k} \rangle = (-1)^{k+1} \left( x(k) \right) (a_{s,t}, a_{t,s}) \qquad\text{and}\qquad \langle \alpha_{k}^{\vee}, \alpha_s  \rangle = (-1)^{k+1} \left( x(k) \right) (a_{s,t}, a_{t,s}).\]
\end{lemma}
\begin{proof}
    We will prove the lemma by induction on $k$.
    The base cases when $k = 0$ and $k=1$ are obvious. We will just show the first equality and the case when $k$ is even. The other cases are similar.
    \begin{align*}
        \langle \alpha_s^{\vee}, \alpha_k \rangle &= \langle \alpha_s^{\vee}, {}_t (k-2) \alpha_s \rangle - \langle \alpha_t^{\vee} ,\alpha_s \rangle \langle \alpha_s^{\vee}, {}_t (k-2) \alpha_s \rangle \\
        &= -\langle \alpha_s^{\vee}, {}_t (k-3) \alpha_s \rangle - \langle \alpha_t^{\vee} ,\alpha_s \rangle \langle \alpha_s^{\vee}, {}_t (k-2) \alpha_s \rangle \\
        &= (x(k-2))(a_{s,t}, a_{t,s}) - a_{t,s} (x(k-1))(a_{s,t}, a_{t,s}) \\
        &= - (x(k)) (a_{s,t}, a_{t,s}).
    \end{align*}
    The first equality is the definition of $t \cdot \alpha_s$, the third equality is the inductive hypothesis, and the last equality is Lemma \ref{lem:recursive_formula_for_twisted_vars}.
\end{proof}

\begin{lemma}\label{lem:abe_realization_dihedral_case}
    Let $W = W_m = \langle s,t \rangle$ be a finite dihedral group and let $W' \cong W_v$ be a rank-2 reflection subgroup of $W$.
    Let $\fr{h}$ be a pre-realization of $W$.
    Suppose that 
    \begin{equation}\label{eq:abe_realization_dihedral_case}
        \binom{v}{k}_{x(m/v)} (a_{s,t}, a_{t,s}) = \binom{v}{k}_{y(m/v)} (a_{s,t}, a_{t,s}) = 0
    \end{equation}
     for all $1 \leq k \leq v-1$. Then $\fr{h}$ is an Abe realization for $W'$.
\end{lemma}
\begin{proof}
    By Lemma \ref{lem:refln_subgroups_conjugation_and_Abe} and possibly relabelling $s$ and $t$, we may replace $W'$ by a conjugate subgroup with $S' = \{ s, (ts)^k t\}$ for some $k$.
    Note that $v = m/(k+1)$. Write $t' = (ts)^k t$ and define the simple root and coroot associated to $t'$ by 
    \[\alpha_{t'} = \begin{cases} ({}_t k) \cdot \alpha_t & k \text{ even}, \\ ({}_t k) \cdot \alpha_s & k \text{ odd}, \end{cases} \qquad\qquad \alpha_{t'}^{\vee} = \begin{cases} ({}_t k) \cdot \alpha_t^{\vee} & k \text{ even}, \\ ({}_t k) \cdot \alpha_s^{\vee} & k \text{ odd}. \end{cases}\] 
   By Lemma \ref{lem:twisted_vars_and_pairings}, we have that
    \[(x (k+1)) (a_{s,t}, a_{t,s}) = (-1)^k \langle \alpha_s^{\vee}, \alpha_{t'} \rangle \text{ and } (y (k+1)) (a_{s,t}, a_{t,s}) = (-1)^k \langle \alpha_{t'}^{\vee}, \alpha_s \rangle.\]
    Since $[n]_x \in x\Z[xy]$ and $[n]_y \in y\Z[xy]$ when $n$ is even and $[n]_x, [n]_y \in \Z[xy]$ when $n$ is odd, it is easy to conclude that
    \[ [n]_{x(k+1)} (a_{s,t}, a_{t,s}) = \pm [n]_x (\langle \alpha_s^{\vee}, \alpha_{t'} \rangle, \langle \alpha_{t'}^{\vee}, \alpha_s \rangle),\] 
    \[ [n]_{y(k+1)} (a_{s,t}, a_{t,s}) = \pm [n]_y (\langle \alpha_s^{\vee}, \alpha_{t'} \rangle, \langle \alpha_{t'}^{\vee}, \alpha_s \rangle), \]
    for all $n \in \Z_{\geq 0}$.
    Therefore, (\ref{eq:abe_realization_dihedral_case}) is equivalent to the defining condition for an Abe realization. 
\end{proof}

\begin{proposition}\label{prop:refl_abe_realizations}
    Let $(W,S)$ be a Coxeter group and $W'$ be a reflection subgroup of $W$ with canonical generators $S'$.
    Let $\fr{h}$ be a reflection-stable Abe realization. Then $\fr{h}$ is an Abe realization for $W'$.
\end{proposition}
\begin{proof}
    By Lemma \ref{lem:pre_realizations_and_reflection_subgroups}, $\fr{h}$ is a pre-realization for $W'$.
    Let $s',t' \in S'$ such that $s' t'$ has finite order (greater than 1). Write $W_{\{s',t'\}}'$ for the parabolic subgroup of $W'$ generated by $s', t'$.
    By Lemma \ref{lem:parabolic_subgroups_and_Abe_realizations}, it suffices to show that $\fr{h}$ is an Abe realization for $W_{\{s',t'\}}'$.
    Note that $W_{\{s',t'\}}'$ is a rank-2 reflection subgroup of $W$.
    By Lemma \ref{lem:parabolic_closures_of_refln_subgroups}, we can find some $w\in W$ and $s,t \in S$ such that $W_{\{s',t'\}}' \subseteq {}^w W_{\{s,t\}}$.
    By another application of Lemma  \ref{lem:parabolic_subgroups_and_Abe_realizations}, we have that $\fr{h}$ is a reflection-stable Abe realization for ${}^w W_{\{s,t\}}$.
    Then $\fr{h}$ is an Abe realization for $W_{\{s',t'\}}'$ by Lemma \ref{lem:abe_realization_dihedral_case}.
\end{proof}

\begin{corollary}\label{cor:refl_abe_are_endo}
    Let $(W,S)$ be a Coxeter system and $\fr{o}$ be a $W$-set.
    Any reflection-stable Abe realization of $W$ is also an endoscopic Abe realization.
\end{corollary}
\begin{proof}
    The claim follows from Proposition \ref{prop:refl_abe_realizations} after the observation that every endoscopic Coxeter group is a reflection subgroup (see Proposition \ref{prop:endo_gps_are_coxeter_gps}). 
\end{proof}

\begin{remark}
    Proposition \ref{prop:refl_abe_realizations} implies that $\fr{h}$ being reflection-stable is sometimes a necessary condition for $\fr{h}$ to be endoscopic.
    Indeed, by Example \ref{ex:endoscopic_groups} (5), every reflection subgroup can arise as an endoscopic Coxeter group. As a result, we can pick $\fr{o}$ such that every reflection subgroup of $W$ is of the form $W_{\scrL}^{\circ}$ for some $\scrL \in \fr{o}$.
    On the other extreme, we can always consider the non-monodromic case where $\fr{o} = 1$. Here, being reflection-stable is substantially stronger than being endoscopic.     
\end{remark}

\begin{proposition}\label{prop:mon_abe_and_endo_rk_2}
    Let $(W,S)$ be a Coxeter system and $\fr{o}$ be a $W$-set.
    Let $\fr{h}$ be a monodromic Abe realization of $W$.
    Fix some $s, t\in S$ and $\scrL \in \fr{o}$.
    Define $ W_{s,t}^{\scrL} \coloneq W_{\{s,t\}} \cap W_{\scrL}^{\circ}$.
    Then $\fr{h}$ is an Abe realization for $W_{s,t}^{\scrL}$.
\end{proposition}
\begin{proof}
    If $W_{s,t}^{\scrL}$ is of rank 0 or 1, then $\fr{h}$ is trivially an Abe realization for $W_{s,t}^{\scrL}$.
    If $W_{s,t}^{\scrL}$ is of rank 2, then we are done by Lemma \ref{lem:abe_realization_dihedral_case}.
\end{proof}

\begin{proposition}\label{prop:crystallographic_abe}
    If $\fr{h}$ arises from a Kac--Moody root datum whose Weyl group is $W$, then $\fr{h}$ is a reflection-stable Abe realization.
\end{proposition}
\begin{proof}
    The case of $v = m_{s,t}$ is covered by \cite[Proposition 3.7]{Abe21}.
    Consider the following table describing the possible cases which can be computed via the formula (\ref{eq:cor_of_twisted_vars_and_two_color}):
    \begin{center}
        \begin{tabular}{c | c | c | c | c | c | c }
            $m_{s,t}$ & $v$ & $k$ & $\binom{v}{k}_{x(m_{s,t}/v)}$ & $\binom{v}{k}_{y(m_{s,t}/v)}$ & $\langle \alpha_s^{\vee}, \alpha_t \rangle$ &  $\langle \alpha_t^{\vee}, \alpha_s \rangle$ \\
            \hline
            $4$ & $2$ & $1$ & $-xy+2$ & $-xy+2$ & $-1$ & $-2$ \\
            $6$ & $3$ & $1$ & $x^2 y^2 - 4xy + 3$ & $x^2 y^2 - 4xy + 3$ & $-1$ & $-3$ \\
            $6$ & $3$ & $2$ & $x^2 y^2 - 4xy + 3$ & $x^2 y^2 - 4xy + 3$ & $-1$ & $-3$ \\
            $6$ & $2$ & $1$ & $-x^2 y + 3x$ & $-xy^2 + 3y$ & $-1$ & $-3$
        \end{tabular}
    \end{center}
    The result then follows by evaluating the polynomials appearing in the above table.
\end{proof}

\subsubsection{Balanced Realizations}

Let $(W,S)$ be a Coxeter group.
We saw in the previous section that the restriction of a realization $\fr{h}$ of $W$ to a reflection subgroup $W' \subseteq W$ is not strictly well-defined.
In particular, it depends on the choice of a reflection datum.
In practice, keeping track of this additional data becomes quite tedious. As a result, we will usually restrict our attention to \emph{balanced} realizations. 
This ensures that the restriction to a reflection subgroup does not depend on the reflection datum.
However, it is not clear whether the restricted realization is also balanced. To compensate, we will impose a further condition that ensures all restricted realizations are also balanced.
 
The constraint that our realizations are balanced is not strictly necessary. For the theory of Soergel bimodules, we will not use it at all.
For the diagrammatic theory, using balanced realizations will greatly simplify our constructions. 

\begin{definition}
    Let $(W,S)$ be a Coxeter group with pre-realization $\fr{h}$.
    We say that $\fr{h}$ is \emph{balanced} if 
    \[[m_{s,t} - 1]_x (a_{s,t}, a_{t,s}) = 1 = [m_{s,t} -1]_y (a_{s,t}, a_{t,s})\]
    for all $s,t \in S$.

    We further say that $\fr{h}$ is \emph{reflection-balanced} if
     \[[m_{s,t}/d - 1]_x (a_{s,t}, a_{t,s}) = 1 = [m_{s,t}/d -1]_y (a_{s,t}, a_{t,s})\]
     for all non-trivial divisors $d$ of $m_{s,t}$.  
\end{definition}

\begin{remark}
    If $\fr{h}$ is balanced, then $[m_{s,t} - k]_x (a_{s,t}, a_{t,s}) = [k]_x (a_{s,t}, a_{t,s})$ for all $k$. 
    Using Proposition \ref{prop:twisted_vars_and_two_color}, we can then see that a sufficient but not necessary constraint for $\fr{h}$ to be reflection-balanced is if $[d]_x (a_{s,t}, a_{t,s}) \neq 0$ for all proper divisors $d$ of $m_{s,t}$. 
\end{remark}

\begin{lemma}[{\cite[Lemma 1.2]{P25}}]\label{lem:well_defined_pos_roots_balancedness}
    Let $(W,S)$ be a Coxeter group with a balanced pre-realization $\fr{h}$.
    Let $r$ be a reflection in $W$.
    Assume that $r = xsx^{-1} = yty^{-1}$ for some $x,y\in W$, $s,t \in S$ such that $xs > x$ and $yt > y$. Then $x \alpha_s = y \alpha_t$.
\end{lemma}

It immediately follows from Lemma \ref{lem:well_defined_pos_roots_balancedness} that if $\fr{h}$ is balanced, that the restriction of $\fr{h}$ to any reflection subgroup is independent of the choice of reflection datum.

\begin{proposition}\label{prop:balanced_realizations_and_reflections}
    Let $(W,S)$ be a Coxeter group with a reflection-balanced reflection-stable Abe realization $\fr{h}$.
    Let $(W', S')$ be a reflection subgroup of $W$. Then $\fr{h}$ is a balanced Abe realization for $(W', S')$.
\end{proposition}
\begin{proof}
    Let $s', t' \in S'$ such that $s't'$ has order $v < \infty$. By Lemma \ref{lem:parabolic_closures_of_refln_subgroups}, there exists some $s,t \in S$ and $x \in W$ such that 
    $\langle s',t' \rangle \subset {}^x W_{s,t}$.
    Note that $\fr{h}$ is a balanced realization for ${}^x W_{s,t}$.
    By conjugation, we may assume that $s' = xsx^{-1}$ and $t' = x (ts)^{d-1} t x^{-1}$ for some $d \geq 1$ with $m = dv$.
    By Lemma \ref{lem:well_defined_pos_roots_balancedness}, $\alpha_{s'} = x\alpha_s$ and $\alpha_{t'} = x ({}_{t} (d-1)) \alpha_u$ where $u = t$ if $d$ is odd and $u = s$ if $d$ is even.
    From this description, we have that $a_{s',t'} = \langle \alpha_{s'}^{\vee}, \alpha_{t'} \rangle = \langle \alpha_s^{\vee}, {}_{t} (d-1) \alpha_u \rangle$ and $a_{t', s'} = \langle \alpha_{t'}^{\vee}, \alpha_{s'} \rangle = \langle  {}_{t} (d-1) \alpha_u^{\vee}, \alpha_s\rangle$.
    We will now prove that
    \[[v-1]_x (a_{s',t'}, a_{t',s'}) = 1 = [v-1]_y (a_{s',t'}, a_{t',s'})\]
    using a case-by-case analysis on $v$ and $d$. The second equality follows from the same equality as the first.

    If $v$ is even, then $[v-1]_x \in \Z[xy]$. Since $[v-1]_x$ is concentrated in even degrees, by Lemma \ref{lem:twisted_vars_and_pairings}, we have equalities $[v-1]_x (a_{s',t'}, a_{t',s'}) = [v-1]_{x(d)} (a_{s,t}, a_{t,s}) = 1$.

    If $v$ is odd, then $d$ must also be odd divisor. By Lemma \ref{lem:twisted_vars_and_pairings}, we have equalities $(x(d)) (a_{s,t}, a_{t,s}) = a_{s',t'}$ and $(y(d)) (a_{s,t}, a_{t,s}) = a_{t',s'}$.
    Therefore, $[v-1]_x (a_{s',t'}, a_{t',s'}) = [v-1]_{x(d)} (a_{s,t}, a_{t,s}) = 1$. 
\end{proof}

\begin{remark}\label{rem:crystallographic_balancedness}
    Suppose $\fr{h}$ arises from a Kac--Moody root datum whose Weyl group is $W$. We will show that $\fr{h}$ is a reflection-balanced.
    This can be checked explicitly case-by-case using the table below.
    \begin{center}
        \begin{tabular}{c | c | c | c | c | c | c }
            $m_{s,t}$ & $v$ & $d$ & $[v-1]_{x(d)}$ & $[v-1]_{y(d)}$ & $\langle \alpha_s^{\vee}, \alpha_t \rangle$ &  $\langle \alpha_t^{\vee}, \alpha_s \rangle$ \\
            \hline
            $4$ & $2$ & $2$ & $1$ & $1$ & $-1$ & $-2$ \\
            $6$ & $2$ & $3$ & $1$ & $1$ & $-1$ & $-3$ \\
            $6$ & $3$ & $2$ & $xy-2$ & $xy-2$ & $-1$ & $-3$ 
        \end{tabular}
    \end{center}
\end{remark}

    \section{Monodromic Soergel Bimodules}\label{sec:abe}

    A foundational tool for working with the Elias--Williamson diagrammatic category is the Soergel functor
\begin{equation}\label{eq:comp_functor_v1}
    \EW^{\BS} (\fr{h}, W) \to \BSBim(\fr{h}, W)
\end{equation}
from the diagrammatic Hecke category to Bott--Samelson bimodules.
For finite Coxeter groups, sufficiently nice realizations over $\k$, and certain small primes being invertible in $\k$, this functor is an equivalence of categories.
Problems quickly arise when any of these three constraints are removed. For our later applications, it will be crucial to have a well-behaved version of (\ref{eq:comp_functor_v1}) which works for arbitrary $W$ and $\k$.
Fortunately, a solution to this problem already exists. In \cite{Abe19}, Abe had the ingenious idea to modify the category of Bott--Samelson bimodules by keeping track of certain localization data.
The resulting category $\scrA^{\BS} (\fr{h}, W)$ of Abe--Bott--Samelson bimodules gives a bimodule-theoretic incarnation of the Elias--Williamson category. In particular, by \cite{Abe19}, there is an equivalence of categories
\begin{equation}\label{eq:comp_functor_v2}
    \EW^{\BS} (\fr{h}, W) \stackrel{\sim}{\to} \scrA^{\BS} (\fr{h}, W).
\end{equation}

The goal of this section is to construct a monodromic version of Abe--Bott--Samelson bimodules.

\subsection{Abe--Soergel Theory}

We will closely follow the approach of \cite{Abe19}, \cite{Abe21}, and \cite{RV}. The main difference between our presentation and that of \emph{loc. cit.} will be a 2-categorical enrichment which keeps track of monodromy parameters.

Fix a Coxeter system $(W,S)$, a set of monodromy parameters $\fr{o}$, and a reflection-balanced reflection-stable Abe realization $\fr{h}$ of $W$.

We denote $R$ for the symmetric algebra of $\fr{h}^*$ over $\k$.
We view $R$ as a $\Z$-graded $\k$-algebra where $\fr{h}^*$ is concentrated in degree 2.
Consider the localized ring
\[Q \coloneq R \left[ \frac{1}{\alpha_w} \colon w \in \scrR (W) \right].\]
The roots $\alpha_w$ are well-defined by Lemma \ref{lem:well_defined_pos_roots_balancedness}.
The ring $Q$ is naturally $\Z$-graded. Moreover, the $W$-action on $\fr{h}^*$ induces actions on $R$ and on $Q$ by graded algebra automorphisms.

\subsubsection{The Category \texorpdfstring{$\Cmon{}{}'$}{C'}}

We can define a collection of 1-categories $\Cmon{\scrL}{\scrL'}' (\fr{h}, W, \fr{o})$ whose objects are triples
\[(M, (M_Q^w)_{w \in W}, \xi_M)\]
such that
\begin{enumerate}
    \item $M$ is a graded $R$-bimodule;
    \item each $M_Q^w$ is a graded $(R,Q)$-bimodule such that $m \cdot f = w (f) \cdot m$ for all $m \in M_Q^w$ and $f \in R$;
    \item $M_Q^w = 0$ if $w \notin \W{\scrL}{\scrL'}$ and is only nonzero for finitely many $w \in \W{\scrL}{\scrL'}$;
    \item $\xi_M : M_Q \coloneq M \otimes_R Q \to \bigoplus_{w \in W} M_Q^w$ is an isomorphism of graded $(R, Q)$-bimodules.
\end{enumerate}
A morphism in $\Cmon{\scrL}{\scrL'}' (\fr{h}, W, \fr{o})$ from $(M, (M_Q^w)_{w \in W}, \xi_M)$ to $(N, (N_Q^w)_{w \in W}, \xi_N)$ is a morphism of $R$-bimodules\footnote{We allow for morphisms of $R$-bimodules which are not just concentrated in degree 0.} $\varphi : M \to N$ such that
\[(\xi_N \circ (\varphi \otimes_R Q) \circ \xi_M^{-1}) (M_Q^w) \subset N_Q^w\]
for any $w \in W$. For simplicity, we often abuse notation and only write $M$ instead of the triple $(M, (M_Q^w)_{w \in W}, \xi_M)$. 

There are bifunctors,
\begin{equation}\label{eq:abe_tensor}
    (-) \otimes_R (-) : \Cmon{\scrL}{\scrL'}' (\fr{h}, W, \fr{o}) \times \Cmon{\scrL'}{\scrL''}' (\fr{h}, W, \fr{o}) \to \Cmon{\scrL}{\scrL''}' (\fr{h}, W, \fr{o})
\end{equation}
induced from the tensor product of underlying graded $R$-bimodules. These bifunctors are suitably associative and there is a unit object $R_e$ with underlying graded bimodule $R$.
This tensor product allows us to assemble a 2-category $\Cmon{}{}' (\fr{h}, W, \fr{o})$ as follows:
\begin{enumerate} 
    \item The object set of $\Cmon{}{}' (\fr{h}, W, \fr{o})$ is $\fr{o}$. 
    \item The morphism categories are defined by $\Hom_{\Cmon{}{}' (\fr{h}, W, \fr{o})} (\scrL, \scrL') = \Cmon{\scrL}{\scrL'}' (\fr{h}, W, \fr{o})$. The composition of morphism categories is by the tensor product (\ref{eq:abe_tensor})
\end{enumerate}
The triple $(\fr{h}, W, \fr{o})$ will sometimes be omitted from $\Cmon{\scrL}{\scrL'}' (\fr{h}, W, \fr{o})$ and $\Cmon{}{}' (\fr{h}, W, \fr{o})$.

The category $\Cmon{\scrL}{\scrL'}' (\fr{h}, W, \fr{o})$ inherits a grading shift automorphism, 
\[ (n) : \Cmon{\scrL}{\scrL'}' (\fr{h}, W, \fr{o}) \to \Cmon{\scrL}{\scrL'}' (\fr{h}, W, \fr{o}), \] 
for each $n \in \Z$ given by $M(n)^i = M^{n+i}$ for all $i \in \Z$. 

\subsubsection{The Category \texorpdfstring{$\Cmon{}{}$}{C}}

For each $\scrL, \scrL' \in \fr{o}$, we consider a full subcategory $\Cmon{\scrL}{\scrL'} (\fr{h}, W, \fr{o})$ of $\Cmon{\scrL}{\scrL'}' (\fr{h}, W, \fr{o})$ consisting of triples $(M, (M_Q^w)_{w \in W}, \xi_M)$
such that $M$ is finitely generated as an $R$-bimodule and flat as a right $R$-module. Note that the grading shift $(1)$ preserves $\Cmon{\scrL}{\scrL'} (\fr{h}, W, \fr{o})$. 

\begin{lemma}[{\cite[Lemma 2.6]{Abe19}}]\label{lem:fg_of_left_and_right_modules}
    If $(M, (M_Q^w)_{w \in W}, \xi_M) \in \Cmon{\scrL}{\scrL'} (\fr{h}, W, \fr{o})$, then $M$ is finitely generated as a left $R$-module and as a right $R$-module.
\end{lemma}

Let $M \in \Cmon{\scrL}{\scrL'} (\fr{h}, W, \fr{o})$. Assume that $M$ is graded free as a left $R$-module. By Lemma \ref{lem:fg_of_left_and_right_modules}, it must also be finitely generated.
As a result, we may define the \emph{graded rank} of $M$, denoted $\grk (M) \in \Z [v^{\pm}]$, as the polynomial such that $M \cong R^{\oplus \grk (M)}$.
Note, by definition $R^{\oplus v^n}$ should be interpreted as $R (n)$. 

Lemma \ref{lem:fg_of_left_and_right_modules} implies that if $M \in \Cmon{\scrL}{\scrL'} (\fr{h}, W, \fr{o})$ and $N \in \Cmon{\scrL}{\scrL'} (\fr{h}, W, \fr{o})$, then $M \otimes_R N \in \Cmon{\scrL}{\scrL'} (\fr{h}, W, \fr{o})$.
As a result, we can define a sub-2-category $\Cmon{}{} (\fr{h}, W, \fr{o})$ of $\Cmon{}{}' (\fr{h}, W, \fr{o})$ whose morphisms categories are given by $\Cmon{\scrL}{\scrL'} (\fr{h}, W, \fr{o})$.
The tuple $(\fr{h}, W, \fr{o})$ will sometimes be omitted from $\Cmon{\scrL}{\scrL'} (\fr{h}, W, \fr{o})$ and $\Cmon{}{} (\fr{h}, W, \fr{o})$.

\begin{lemma}\label{lem:for_ff_on_faithful_realizations}
    Suppose the realization $\fr{h}$ is a faithful $W$-representation. Then the forgetful functor
    \[\Cmon{\scrL}{\scrL'} (\fr{h}, W, \fr{o}) \to \bim{R}\]
    is fully faithful. 
\end{lemma}
\begin{proof}
    It is clear that the forgetful functor is always faithful. The content of the lemma is that it is also full.
    Let $M \in \Cmon{\scrL}{\scrL'} (\fr{h}, W, \fr{o})$.
    Define a sub $(R,Q)$-bimodule,
    \[\underline{M}_Q^w \coloneq \{ m \in M \otimes_R Q \mid f \cdot m = m \cdot w(f) \text{ for all } f \in R\}.\]
    We claim that the restriction of $\xi_M$ given by $\underline{M}_Q^w \to \bigoplus_{x \in W} M_Q^x \twoheadrightarrow M_Q^w$ is an isomorphism.
    It is clear that $\underline{M}_Q^w \to M_Q^w$ is surjective. Let $m \in \underline{M}_Q^w$ and write $\xi_M (m) = \sum_{x \in W} m_x$ with $m_x \in M_Q^x$.
    We claim that if $x \neq w$, then $m_x = 0$. For all $f \in R$, we must have that $w (f) \cdot m_x =  m_x \cdot f = x (f) \cdot m_x$.
    On the other hand, $M_Q^x$ is a free $R$-module, so this can only happen if either $w (f) = x (f)$ or $m_x = 0$. Since $\fr{h}$ is a faithful $W$-representation, we deduce that $m_x = 0$, and hence $\underline{M}_Q^w \to M_Q^w$ is an isomorphism.

    Let $M,N \in \Cmon{\scrL}{\scrL'} (\fr{h}, W, \fr{o})$ and let $\varphi : M \to N$ be a morphism of underlying $R$-bimodules.
    It is clear from the $R$-bimodules structure that $\varphi (\underline{M}_Q^w) \subseteq \underline{N}_Q^w$ for all $w \in W$.
    The previous paragraph then tells us that $\varphi$ is actually a morphism in $\Cmon{\scrL}{\scrL'} (\fr{h}, W, \fr{o})$.
\end{proof}

\subsubsection{Support of Bimodules}\label{subsec:support_of_bims}

Let $M \in \Cmon{\scrL}{\scrL'} (\fr{h}, W, \fr{o})$. 
Denote by $\zeta_M : M \to \bigoplus_{w \in W} M_Q^w$ the composition $M \to M \otimes_R Q \stackrel{\xi_M}{\to} \bigoplus_{w \in W} M_Q^w$.
Note that $\zeta_M$ is injective since $M$ is flat as a right $R$-module and since $R \hookrightarrow Q$ is injective.
For each $I \subset W$, define a map $\pi_M^I : M \to \bigoplus_{w \in  I} M_Q^w$ via the composition $M \stackrel{\zeta_M}{\to} \bigoplus_{w \in W} M_Q^w \twoheadrightarrow \bigoplus_{w \in I} M_Q^w.$
For each $m \in M$, we will write $m_I \coloneq \pi_M^I (m) \in \bigoplus_{w \in I} M_Q^w$.
If $I = \{w\}$, we will simplify the notation to $\pi_M^w = \pi_M^{\{w\}}$ and $m_w = m_{\{w\}}$.

We define the \emph{support} of $M$ by
\[ \supp_W (M) \coloneq \{ w \in W \mid M_Q^w \neq 0 \}. \]
Similarly, we define the support of $m \in M$ by
\[\supp_W (m) \coloneq \{w \in W \mid m_w \neq 0\}.\]
It is clear from the definitions that both $\supp_W (M)$ and $\supp_W (m)$ are subsets of $\W{\scrL}{\scrL'}$. 
The support of a bimodule is compatible with tensor products in the following sense:
\begin{equation}\label{eq:tensor_and_support}
    \supp_W (M \otimes_R N) = \{xy \mid x \in \supp_W (M), y \in \supp_W (N) \}.
\end{equation}

For $I \subset W$, let $M_I = \zeta_M^{-1} \left( \bigoplus_{w \in W} M_Q^w \right) \subseteq M$ and $M^I = \pi_M^I (M) \subseteq \bigoplus_{w \in I} M_Q^w$.
If $I = \{w\}$, we will simplify the notation, $M_w = M_{\{w\}}$ and $M^w = M^{\{w\}}$. Note that $M_I = \{m \in M \mid \supp_W (m) \subset I\}$.

We recall some standard facts about $M_I$ and $M^I$ that follow from the same arguments given in \cite[\S2.2]{Abe19}.

\begin{lemma}\label{lem:std_facts_about_inv_and_coinvs}
    \begin{enumerate}
        Let $I \subseteq W$ and $M, N \in \Cmon{\scrL}{\scrL'} (\fr{h}, W, \fr{o})$ such that $\supp_W (N) \subset I$.
        \item The bimodules $M_I$ and $M^I$ are both objects in $\Cmon{\scrL}{\scrL'}' $ such that  $(M_I)_Q = (M^I)_Q = \bigoplus_{w \in I} M_Q^w$.
        \item There are natural isomorphisms
        \[\Hom_{\Cmon{\scrL}{\scrL'}'} (M, N) \cong \Hom_{\Cmon{\scrL}{\scrL'}'} (M^I, N),\]
        \[\Hom_{\Cmon{\scrL}{\scrL'}'} (N, M) \cong \Hom_{\Cmon{\scrL}{\scrL'}'} (N, M_I).\]
    \end{enumerate}
\end{lemma}

\subsection{Bott--Samelson Bimodules}

Let $w \in \W{\scrL}{\scrL'}$. We first define the \emph{standard bimodule} $R_w \in \Cmon{\scrL}{\scrL w} (\fr{h}, W, \fr{o})$  as follows. As a left $R$-module, $R_w = R$ and the right $R$-module structure is given by $m\cdot f = w(f) \cdot m$ for $m \in R_w$ and $f \in R$.
The localization data for $R_w$ is given by
\[(R_w)_Q^x = \begin{cases} Q_w & x = w, \\ 0 & x \neq w,\end{cases}\]
where $Q_w = R_w \otimes_R Q$.

For $w \in W$, we write $R^w = \{ r \in R \mid w(r) = r \}$ for the subring of $w$-invariants.
Suppose that $s \in W_{\scrL}^\circ$ is an endosimple reflection. Fix $\delta_s \in \fr{h}^*$ such that $\langle \alpha_s^{\vee}, \delta_s \rangle = 1$. Note that $\delta_s$ exists by Demazure surjectivity on the restricted realization for $W_{\scrL}^{\circ}$.
We consider the $R$-bimodule $C_s = R \otimes_{R^s} R(1)$.
The $R$-bimodule $C_s$ can be upgraded to an object in $\Cmon{\scrL}{\scrL} (\fr{h}, W, \fr{o})$ by setting
\begin{equation}\label{eq:localization_of_C_s}
    (C_s)_Q^e = (\delta_s \otimes 1 - 1 \otimes s(\delta_s)) \cdot Q, \qquad\qquad (C_s)_Q^s = (\delta_s \otimes 1 - 1 \otimes \delta_s) \cdot Q,
\end{equation}
and $(C_s)_Q^w = 0$ for all $w \notin \{ e, s\}$. The fact that this localization data actually defines an object in $\Cmon{\scrL}{\scrL}$ is routine (cf., \cite[Claim 3.11]{EW}).

Let $s \in S$. We define an object $B_{s}^{\scrL} \in \Cmon{\scrL}{\scrL s}$ as follows:
\[B_s^{\scrL} = \begin{cases} C_s & \scrL s = \scrL, \\ R_s & \scrL s \neq \scrL.\end{cases}\]
For an expression $\uw = (s_1, \ldots, s_k)$ with $\scrL' = \scrL \uw$, we define the Bott--Samelson bimodule in $\Cmon{\scrL}{\scrL'}$ associated to $\uw$ by
\[ B_{\uw}^{\scrL} := B_{s_1}^{\scrL} \otimes_R B_{s_2}^{\scrL s_1} \otimes_R \ldots \otimes_R B_{s_k}^{\scrL s_1 \ldots s_{k-1}}.\]
Since $B_{\uw}^{\scrL} \in \Cmon{\scrL}{\scrL'} $, we have a decomposition $B_{\uw}^{\scrL} \otimes_R Q \cong \bigoplus_{x \in W}  (B_{\uw}^{\scrL})_Q^x$.
We set $B_{\uw, Q}^{x, \scrL} \coloneq (B_{\uw}^{\scrL})_Q^x$.
For a fixed $x \in W$, we also write $\pi_{\uw}^x \coloneq \pi_{B_{\uw}^{\scrL}}^x$. Similarly, for $I \subseteq W$, we write $B_{\uw, I}^{\scrL} \coloneq (B_{\uw}^{\scrL})_I$ and  $B_{\uw}^{I, \scrL} \coloneq (B_{\uw}^{\scrL})^I$

We can then define a full replete subcategory $\Amon{\scrL}{\scrL'}^{\BS} (\fr{h}, W, \fr{o})$ of $\Cmon{\scrL}{\scrL'} (\fr{h}, W, \fr{o})$ generated by objects of the form 
\[ B_{\uw}^{\scrL} := B_{s_1}^{\scrL} \otimes_R \ldots \otimes_R B_{s_k}^{\scrL s_1 \ldots s_{k-1}}\]
where $\uw = (s_1, \ldots, s_k) \in \Exp (W)$ such that $\scrL \uw = \scrL'$. 
The category $\Amon{\scrL}{\scrL'}^{\BS} (\fr{h}, W, \fr{o})$ is called the category of \emph{$(\scrL, \scrL')$-monodromic Bott--Samelson bimodules}.

We also will consider the full subcategory $\Amon{\scrL}{\scrL'}^{\oplus} (\fr{h}, W, \fr{o})$ of $\Cmon{\scrL}{\scrL'} (\fr{h}, W, \fr{o})$ which is the additive hull of $\Amon{\scrL}{\scrL'}^{\BS} (\fr{h}, W, \fr{o})$.
When $\k$ is a complete local ring, we will denote by $\Amon{\scrL}{\scrL'} (\fr{h}, W, \fr{o})$ the full subcategory of $\Cmon{\scrL}{\scrL'} (\fr{h}, W, \fr{o})$ given by taking the idempotent completion of  $\Amon{\scrL}{\scrL'}^{\oplus} (\fr{h}, W, \fr{o})$.
The category $\Amon{\scrL}{\scrL'} (\fr{h}, W, \fr{o})$ is called the category of \emph{$(\scrL, \scrL')$-monodromic Soergel bimodules}.
It is a standard fact about additive categories over complete local rings that $\Amon{\scrL}{\scrL'} (\fr{h}, W, \fr{o})$ is Krull--Schmidt (cf., \cite[Example A.8.6]{Ach1}).

We can assemble these 1-categories into wide sub-2-categories of $\Cmon{}{} (\fr{h}, W, \fr{o})$. These are defined as follows:
\begin{enumerate}
    \item $\Amon{}{}^{\BS} (\fr{h}, W, \fr{o})$ whose morphism categories are the $\Amon{\scrL}{\scrL'}^{\BS} (\fr{h}, W, \fr{o})$'s;
    \item $\Amon{}{}^{\oplus} (\fr{h}, W, \fr{o})$ whose morphism categories are the $\Amon{\scrL}{\scrL'}^{\oplus} (\fr{h}, W, \fr{o})$'s;
    \item if $\k$ is a complete local ring, $\Amon{}{} (\fr{h}, W, \fr{o})$ whose morphism categories are the $\Amon{\scrL}{\scrL'} (\fr{h}, W, \fr{o})$'s.
\end{enumerate}
As mentioned for $\Cmon{}{} (\fr{h}, W, \fr{o})$, we will sometimes drop the tuple $(\fr{h}, W, \fr{o})$ from the notation in all of these 2-categories and their morphism categories.

\begin{lemma}\label{lem:bott_samelson_bimodules_are_free}
    Let $\uw \in \Exp (W)$.
    The module $B_{\uw}^{\scrL}$ is graded free as a left (resp. right) $R$-module with graded rank $(v+v^{-1})^{\ell_{\scrL} (\uw)}$.
\end{lemma}
\begin{proof}
    Let $s \in S$. If $\scrL s = \scrL$, then $B_s^{\scrL} \cong R(-1) \oplus R(1)$ as left (resp. right) $R$-modules.
    If $\scrL s \neq \scrL$, then $B_s^{\scrL} \cong R$ as a left (resp. right) $R$-module.
    The result then follows from induction on the length of $\uw$ by taking tensor products.
\end{proof}

\begin{lemma}\label{lem:soergel_product_formulas}
    \begin{enumerate}
        \item Let $x,y \in W$. There is a unique isomorphism in $\Cmon{\scrL}{\scrL xy} (\fr{h}, W, \fr{o})$,
            \[R_x \otimes_R R_y \stackrel{\sim}{\to} R_{xy},\]
            taking $1 \otimes 1$ to $1$.
        \item  Let $w \in W$ and $s \in S$ such that $\scrL w s = \scrL w$ and $\ell_{\scrL} (w) = 0$.
        There are unique isomorphisms in $\Cmon{\scrL}{\scrL} (\fr{h}, W, \fr{o})$,
        \[R_w \otimes_R C_s \otimes_R R_{w^{-1}} \stackrel{\sim}{\to} C_{wsw^{-1}},\]
        taking $1 \otimes 1 \otimes 1 \otimes 1$ to $1 \otimes 1$.
    \end{enumerate}
\end{lemma}
\begin{proof}
    Statement (1) follows directly from definitions. 
    
    For statement (2), the conditions on $w$ and $s$ ensure that $wsw^{-1}$ is an endosimple reflection in $W_{\scrL}^\circ$.
    Consider the morphism of $\k$-bimodules,
    \begin{equation}
        \varphi : R_w \otimes R_{w^{-1}} \to R \otimes_{R^{wsw^{-1}}} R, \qquad\qquad a \otimes b \mapsto a \otimes w(b).
    \end{equation}
    We can first check that $\varphi$ is a morphism of $R$-bimodules. 
    The left $R$-linearity is obvious. For the right $R$-linearity, we can compute
    \[\varphi (a \otimes bw^{-1} (r)) = a \otimes w(bw^{-1} (r)) = a \otimes w (b)\]
    for $a \in R_w, b\in R_{w^{-1}}, r \in R$. Next, we will show that $\varphi$ is $R^s$-balanced.
    Let $f \in R^s$. Mote that there is a ring morphism $R^s \to R^{wsw^{-1}}$ given by $f \mapsto w(f)$, i.e., $w (f)$ is $wsw^{-1}$-invariant.
    Hence,
    \begin{equation*}
        \varphi ( a w(f) \otimes b) = aw(f) \otimes w(b) = a \otimes w(f b) = \varphi (a \otimes fb).
    \end{equation*}
    As a result, there is an induced morphism of $R$-bimodules
    \[\overline{\varphi} : R_w \otimes_{R^s} R_{w^{-1}} \to R \otimes_{R^{wsw^{-1}}} R.\]
    It is easy to check that as left $R$-modules, both $R_w \otimes_{R^s} R_{w^{-1}}$ and $ R \otimes_{R^{wsw^{-1}}} R$ are free $R$-modules of rank 2.
    Moreover, since $w \cdot (-) : R \to R$ takes $s$-(anti)invariants to $wsw^{-1}$-(anti)invariants, it is easy to check that $\overline{\varphi}$ is in fact an isomorphism of $R$-bimodules.
    We then have an isomorphism of $R$-bimodules
    \[R_w \otimes_R C_s \otimes_R R_{w^{-1}} \cong R_w \otimes_{R^s} R_{w^{-1}} (1) \stackrel{\overline{\varphi}}{\to}  R \otimes_{R^{wsw^{-1}}} R (1) = C_{wsw^{-1}}.\]  
    Clearly, this isomorphism takes $1 \otimes 1 \otimes 1 \otimes 1$ to $1 \otimes 1$. Since the degree 0 endomorphism algebra of bimodule morphisms is rank 1, the isomorphism is uniquely determined.  
\end{proof}

\begin{lemma}\label{lem:BS_tensor_adjunction}
    Let $s \in S$. The functor $(-) \otimes_R B_s^{\scrL'} : \Amon{\scrL}{\scrL'}^{\BS} (\fr{h}, W, \fr{o}) \to \Amon{\scrL}{\scrL' s}^{\BS} (\fr{h}, W, \fr{o})$ is biadjoint to $(-) \otimes_R B_s^{\scrL' s} : \Amon{\scrL}{\scrL' s}^{\BS} (\fr{h}, W, \fr{o}) \to \Amon{\scrL}{\scrL'}^{\BS} (\fr{h}, W, \fr{o})$.
\end{lemma}
\begin{proof}
    When $\scrL' s = \scrL'$, this is proved in \cite[Lemma 2.15]{Abe19}.
    If $\scrL' s \neq \scrL'$, then there is an isomorphism $B_s^{\scrL'} \otimes_R B_s^{\scrL' s} \cong R$.
    As a result $(-) \otimes_R B_s^{\scrL'}$, is an equivalence of categories with inverse $(-) \otimes_R B_s^{\scrL' s}$, in particular the functors are biadjoint.
\end{proof}

\subsubsection{Duality}

Let $M \in \Cmon{\scrL}{\scrL'} (\fr{h}, W, \fr{o})$. Define
\[ D(M) \coloneq \Hom_{\mod{R}} (M, R) \qquad\text{and}\qquad D(M)_Q^w \coloneq \Hom_{\mod{Q}} (M_Q^w, Q).\]
The $R$-bimodule structure on $D(M)$ is given by $(rfs) (m) = f(rms)$ for $r,s \in R$, $m \in M$, and $f \in D(M)$.
It is easy to check that $D(M)$ is well-defined in $\Cmon{\scrL}{\scrL'}' (\fr{h}, W, \fr{o})$ (cf., \cite[\S2.6]{Abe19}).
Moreover, $D(M)$ extends to a functor
\[D : \Cmon{\scrL}{\scrL'}^{\op} (\fr{h}, W, \fr{o}) \to \Cmon{\scrL}{\scrL'}' (\fr{h}, W, \fr{o}).\]

The following lemma is a combination of results in \cite[\S2.6]{Abe19}.
\begin{lemma}\label{lem:dual_and_invs} 
    Let $I \subseteq W$ and $w \in W$. Let $M \in \Cmon{\scrL}{\scrL'} (\fr{h}, W, \fr{o})$.
    \begin{enumerate}
        \item There is an isomorphism $D (M^I) \cong D(M)_I$.
        \item We have that $D(M)_w$ is graded free as a left $R$-module with graded rank 
        \[ \grk (D(M)_w) (v) = \grk (M^w) (v^{-1}). \]
        \item We have a natural isomorphism $D (D(M)_w) \cong M^w$.
        \item The functor $D$ preserves $\Amon{\scrL}{\scrL'}^{\BS} (\fr{h}, W, \fr{o})$. Moreover, there is a natural isomorphism of functors $D^2 \cong \id$ on $\Amon{\scrL}{\scrL'}^{\BS} (\fr{h}, W, \fr{o})$.
    \end{enumerate}
\end{lemma}
\begin{proof}
    Statement (1) is \cite[Lemma 2.18]{Abe19} and Statements (2) and (3) is \cite[Lemma 2.19]{Abe19}. Statement (4) is a slight generalization of \cite[Lemma 2.20]{Abe19}. A proof of this generalized variant can be found in \cite[Proposition 2.11]{Masaharu}.
\end{proof}

\subsection{Monodromic Hecke Algebroid}\label{subsec:mha}

\subsubsection{Definitions and First Properties}

We recall the monodromic Hecke algebroid introduced in \cite{Sandvik} based on the monodromic Hecke algebra developed by Lusztig \cite{Lus16, Lus19}.

\begin{definition}
    The \emph{monodromic Hecke algebroid} of $W$ with monodromy $\fr{o}$, denoted by $\scrH^{\mon} (W, \fr{o})$, is a $\Z$-linear category whose objects are elements of $\fr{o}$.
    We write ${}_{\scrL} \scrH^{\mon}_{\scrL'} (W, \fr{o})$ for the morphism space $\Hom_{\scrH^{\mon} (W, \fr{o})} (\scrL, \scrL')$. It is defined as the free $\Z [v,v^{-1}]$-module with basis $\{ H_w^{\scrL} \}_{w \in \W{\scrL}{\scrL'}}$.
    The composition in $\scrH^{\mon} (W, \fr{o})$ is defined by the generating relations
    \begin{equation}
      H_x^{\scrL} H_s^{\scrL x} \coloneq \begin{cases} H_{xs}^{\scrL} & xs > x, \\ (v^{-1} - v)H_{x}^{\scrL} + H_{xs}^{\scrL}  & xs < x, \scrL xs = \scrL x, \\ H_{xs}^{\scrL} & xs < x, \scrL xs \neq \scrL x.\end{cases}
    \end{equation}
\end{definition}

We define the \emph{standard form} as the sesquilinear pairing
\[\langle -, - \rangle : {}_{\scrL} \scrH^{\mon}_{\scrL'} (W, \fr{o}) \times {}_{\scrL} \scrH^{\mon}_{\scrL'} (W, \fr{o}) \to \Z [v,v^{-1}]\]
given by the relations $\langle H_x^{\scrL}, H_y^{\scrL} \rangle = \delta_{x,y}$ for all $x,y \in \W{\scrL}{\scrL'}$.
The standard form satisfies the following ``biadjointness'' relation: for all $x\in \W{\scrL}{\scrL'}$, $y \in \W{\scrL}{\scrL' s}$, and $s \in S$,
\begin{equation}\label{eq:mha_biadjoint}
    \langle H_x^{\scrL} H_s^{\scrL'}, H_y^{\scrL} \rangle = \langle H_x^{\scrL}, H_y^{\scrL} H_s^{\scrL' s} \rangle.
\end{equation} 

The \emph{bar involution} on the monodromic Hecke algebroid is the $\Z$-linear functor
\[\overline{(-)} : \scrH^{\mon} (W, \fr{o}) \to \scrH^{\mon} (W, \fr{o})\]
defined on generators by
\[\overline{H_s^{\scrL}} = \begin{cases} H_s^{\scrL} + (v - v^{-1}) H_e^{\scrL} & \text{if } \scrL s = \scrL, \\ H_s^{\scrL} & \text{if } \scrL s \neq \scrL, \end{cases}\]
and on Laurent polynomials by $\overline{v} = v^{-1}$.

\subsubsection{Preliminary Relationship with Bimodules}

For each $s \in S$, define 
\[ \underline{H}_s^{\scrL} \coloneq \begin{cases} H_s^{\scrL} + v H_e^{\scrL} & \scrL s = \scrL, \\ H_s^{\scrL} & \scrL s \neq \scrL. \end{cases}\]
More generally, if $\uw = (s_1, \ldots, s_k) \in \Exp (W)$. We define $\underline{H}_{\uw}^{\scrL} = H_{s_1}^{\scrL} \ldots H_{s_k}^{\scrL s_{1} \ldots s_{k-1}}$.
Furthermore, we define a collection of polynomials $p_{\uw}^{x, \scrL} = \langle \underline{H}_{\uw}^{\scrL},  H_x^{\scrL} \rangle  \in \Z [v^{\pm 1}]$.
It is easy to check that 
\[\underline{H}_{\uw}^{\scrL} = \sum_{x \in \W{\scrL}{\scrL'}} p_{\uw}^{x, \scrL} H_x^{\scrL}.\] 

\begin{lemma}\label{lem:pwx_calculation}
    Let $\beta \in \uW{\scrL}{\scrL'}$.
    Let $\uw \in \Exp (W)$ be an expression for $w \in \beta$. Then 
    \[ \sum_{x \in \beta } v^{\ell_{\scrL} (x)} p_{\uw}^{x, \scrL} (v^{-1}) = (v+v^{-1})^{\# K (\uw , \scrL)}.\]
\end{lemma}
\begin{proof}
    Let $\underline{\Z[v,v^{-1}]}$ denote the algebroid given by viewing $\Z[v,v^{-1}]$ as an algebroid with a single object.
    We define a functor $F : \scrH^{\mon} (W, \fr{o}) \to \underline{\Z[v,v^{-1}]}$ by $F(H_w^{\scrL}) = v^{-\ell_{\scrL} (w)}$.
    It can be easily checked from the defining relations for the monodromic Hecke algebroid that $F$ is well-defined.
    Note that for $s \in S$,
    \[F (\underline{H}_s^{\scrL}) = \begin{cases} v + v^{-1} & \scrL s = \scrL, \\ 1 & \scrL s \neq \scrL. \end{cases}\]
    In particular, $F (\underline{H}_{\uw}^{\scrL}) = (v+v^{-1})^{\# K(\uw, \scrL)}$.
    Note that $p_{\uw}^{x, \scrL} = 0$ if $x \notin \W{\scrL}{\scrL'}$.
    The result then follows from applying $F$ to $\underline{H}_{\uw}^{\scrL} = \sum_{x \in \W{\scrL}{\scrL'}} p_{\uw}^{x, \scrL} H_x^{\scrL}$ and replacing $v$ by $v^{-1}$.
\end{proof}

\begin{lemma}\label{lem:dim_of_Bwx_over_Q}
    Let $\uw \in \Exp (W)$ and $x \in W$. Then $\grk_Q (B_{\uw, Q}^{x, \scrL}) = p_{\uw}^{x, \scrL} (1)$.
\end{lemma}
\begin{proof}
    Recall the groupoid $M^{\fr{o}} (W)$ whose object set is $\fr{o}$ and whose morphism sets are given by $\W{\scrL}{\scrL'}$.
    Let $\Z [M^{\fr{o}} (W)]$ denote the free $\Z$-algebroid over $M^{\fr{o}} (W)$.
    We can define the specialization functor
    \[(-)_{v=1} : \scrH^{\mon} (W, \fr{o}) \to \Z [M^{\fr{o}} (W)]\]
    defined by $v \mapsto 1$ and $H_w^{\scrL} \mapsto w \in \W{\scrL}{\scrL w}$.

    We will now argue by induction on the length of $\uw$.
    If $\uw = \emptyset$, the claim is obvious.
    Suppose $\uw = \uw' s$ for some $s \in S$.
    By applying specialization to $\underline{H}_{\uw'}^{\scrL} = \sum_{x \in \W{\scrL}{\scrL' s}} p_{\uw'}^{x, \scrL} H_x^{\scrL}$ and $\underline{H}_{\uw}^{\scrL} = \sum_{x \in \W{\scrL}{\scrL'}} p_{\uw}^{x, \scrL} H_x^{\scrL}$ we get the relation
    \[p_{\uw}^{x, \scrL} (1) = \begin{cases} p_{\uw}^{x, \scrL} (1) + p_{\uw}^{xs, \scrL} (1) & s \in \scrL' s = \scrL', \\ p_{\uw}^{xs, \scrL} (1) & \scrL' s \neq \scrL'. \end{cases}\]
    On the other hand, for any $M \in \Cmon{\scrL}{\scrL' s}$,
    \[ (M \otimes_R B_s^{\scrL' s})_Q^x \cong \begin{cases} M_Q^x \oplus M_Q^{xs} & \scrL' s = \scrL', \\ M_Q^{xs} & \scrL' s \neq \scrL'. \end{cases}  \]
    As a result, $\grk_Q (B_{\uw, Q}^{x, \scrL}) = \grk_Q ((B_{\uw'}^{\scrL} \otimes_R B_{s}^{\scrL' s})_Q^x ) = p_{\uw}^{x, \scrL} (1)$.
\end{proof}

\subsection{Morphisms of Bimodules}\label{subsec:abe_mors}

We will define a collection of morphisms which we will later show generate the morphism spaces of Bott--Samelson bimodules.

\subsubsection{1-Color Morphisms}

Let $s \in S$ such that $\scrL s = \scrL$. By Demazure surjectivity, we can find some $\delta_s \in \fr{h}^*$ for which $\langle \alpha_s^{\vee}, \delta_s \rangle = 1$.
Let $\partial_s : R \to R^s$ be the $R^s$-module map defined by $\partial_s (f) = (f-s(f))/\alpha_s$. We use the identification $C_s \otimes_R C_s = R \otimes_{R^s} R \otimes_{R^s} R (2)$.
We then define morphisms of graded $R$-bimodules
\begin{align*}
    \nu^s &: C_s \to C_s \otimes_R C_s (-1), & f \otimes g &\mapsto f \otimes 1 \otimes g, \\
    \epsilon^s &: C_s \to R (1), & f \otimes g &\mapsto fg, \\
    \mu^s &: C_s \otimes_R C_s \to C_s (-1), & f\otimes g \otimes h &\mapsto \partial_s (g) f \otimes h, \\
    \eta^s &: R \to C_s (1), & 1 &\mapsto \delta_s \otimes 1 - 1 \otimes s(\delta_s), \\
    \cup^s &: R \to C_s \otimes_R C_s, & 1 &\mapsto \delta_s \otimes 1 \otimes 1 - 1 \otimes 1 \otimes s(\delta_s), \\
    \cap^s &: C_s \otimes_R C_s \to R, & f \otimes g \otimes h &\mapsto \partial_s (g) fh.
\end{align*}
Note that $\cup^s = \nu^s \circ \eta^s$ and $\cap^s = \epsilon^s \circ \mu^s$. It is straightforward (albeit tedious) to check that all the above maps are morphisms in $\Cmon{\scrL}{\scrL}$.

Let $s \in S$ such that $\scrL s \neq \scrL$. We then define morphisms of graded $R$-bimodules
\begin{align*}
    \cup^s &: R \to R_s \otimes_R R_s,  & 1 &\mapsto 1 \otimes 1, \\
    \cap^s &: R_s \otimes_R R_s \to R,  & f \otimes g &\mapsto f s(g).
\end{align*}
Likewise, these are morphisms in $\Cmon{\scrL}{\scrL s}$.

\subsubsection{Rex Moves}

For each $\uw = (s_1, \ldots, s_k) \in \Exp (W)$, we define an element $u_{\uw} = u_1 \otimes \ldots \otimes u_k \in B_{\uw}^{\scrL}$ by
\[u_i = \begin{cases} u_{s_i} = 1 \otimes 1 \in C_{s_i} & \text{if }i \in K(\uw, \scrL), \\ 1 \in R_{s_i} & \text{if }i \notin K(\uw, \scrL).\end{cases}\]
The element $u_{\uw}$ is called the \emph{1-tensor} of $B_{\uw}^{\scrL}$.

Let $s,t \in S$ be distinct simple reflections such that $m = m_{s,t} < \infty$. We wish to construct a morphism
\[\beta_{s,t} : B_{{}_s \underline{m}}^{\scrL} \to B_{{}_t \underline{m}}^{\scrL}\]
such that $\beta_{s,t} (u_{{}_s \underline{m}}) = u_{{}_t \underline{m}}$.
Unfortunately, such a map is known to only exist under constraints on the realization $\fr{h}$.

\begin{proposition}\label{prop:existence_of_rex_moves}
    Assume that $\fr{h}$ is a monodromic Abe realization. There exists a morphism $\beta_{s,t} : B_{{}_s\underline{m}}^{\scrL} \to B_{{}_t\underline{m}}^{\scrL}$ such that $\beta_{s,t} (u_{{}_s \underline{m}}) = u_{{}_t \underline{m}}$.
\end{proposition}
\begin{proof}
    Let $W_{s,t}^{\scrL} \coloneq W_{\{s,t\}} \cap W_{\scrL}^{\circ}$. Let $w_{s,t}$ (resp. $w_{s,t}^{\scrL}$) denote the longest element of $W_{\{s,t\}}$ (resp. $W_{s,t}^{\scrL}$).
    By Proposition \ref{prop:mon_abe_and_endo_rk_2}, $\fr{h}$ is an Abe realization for $W_{s,t}^{\scrL}$ 
    If $W_{s,t}^{\scrL} = 1$, then $B_{{}_s \underline{m}}^{\scrL} \cong R_{w_{s,t}} \cong B_{{}_t \underline{m}}^{\scrL}$.
    Similarly, if $W_{s,t}^{\scrL}$ has rank 1, then $B_{{}_s \underline{m}}^{\scrL} \cong B_{{}_t \underline{m}}^{\scrL}$. As a result, in these cases the desired $\beta_{s,t}$ map clearly exists.

    Now suppose $W_{s,t}^{\scrL}$ has rank 2 with simple generators $s' = (st)^k s$ and $t' = (ts)^{\ell} t$.
    Note $v \coloneq v_{s,t}^{\scrL} = m_{s',t'}$. By \cite[Theorem 3.9]{Abe21}, there then exists a morphism 
    \[ \beta_{s',t'} : \underbrace{C_{s'} \otimes_R C_{t'} \otimes_R \ldots }_{v \text{ factors}} \longrightarrow \underbrace{C_{t'} \otimes_R C_{s'} \otimes_R \ldots }_{v \text{ factors}} \]
    such that $\beta_{s',t'} (u_{{}_{s'} \underline{v}}) = u_{{}_{t'} \underline{v}}$. 
    
    We can then find $\beta \in \uW{\scrL}{\scrL w_{s,t}}$ such that $w_{s,t}^{\scrL} w^{\beta} = w_{s,t}$. By Lemma \ref{lem:soergel_product_formulas}, there are isomorphisms
    \[ B_{{}_s\underline{m}}^{\scrL} \cong \underbrace{C_{s'} \otimes_R C_{t'} \otimes_R \ldots }_{v \text{ factors}} \otimes_R R_{w^{\beta}} \qquad \text{and}\qquad B_{{}_t \underline{m}}^{\scrL} \cong \underbrace{C_{t'} \otimes_R C_{s'} \otimes_R \ldots }_{v \text{ factors}} \otimes_R R_{w^{\beta}}\]
    which take 1-tensors to 1-tensors.
    We can then define $\beta_{s,t} \coloneq \beta_{s',t'} \otimes \id_{R_{w^{\beta}}} : B_{{}_s \underline{m}}^{\scrL} \to B_{{}_t \underline{m}}^{\scrL}$ which satisfies $\beta_{s,t} (u_{{}_s \underline{m}}) = u_{{}_t \underline{m}}$.
\end{proof}

More generally, if $\ux$ and $\uy$ are reduced expressions of some $w\in W$, then each path $\Gamma$ in the rex graph for $w$ gives rise to a morphism $\beta_{\Gamma} : B_{\ux}^{\scrL} \to B_{\uy}^{\scrL}$
composed of the various $\beta_{s,t}$-maps.
It is clear from its construction that $\beta_{\Gamma} (u_{\ux}) = u_{\uy}$. Of course such a morphism depends on the choice of $\Gamma$.

\subsection{Algebraic Monodromic--Endoscopic Equivalence} 

In this section, we will provide an equivalence between the category of monodromic Bott--Samelson bimodules and the category of Bott--Samelson bimodules for the endoscopic Coxeter group.

\subsubsection{Block Decomposition of Soergel Bimodules}

\begin{definition}
    Let $\beta \in \uW{\scrL}{\scrL'}$. The \emph{$\beta$-block} of $\Amon{\scrL}{\scrL'}^{\BS} (\fr{h}, W, \fr{o})$ is the full replete subcategory $\Amon{\scrL}{\scrL'}^{\BS, \beta} (\fr{h}, W, \fr{o})$ generated by the objects $B_{\uw}^{\scrL}$ where $\uw \in \Exp (W)$ is an expression of $w \in \beta$.
    We also define $\Amon{\scrL}{\scrL'}^{\oplus, \beta} (\fr{h}, W, \fr{o})$ as the additive hull of $\Amon{\scrL}{\scrL'}^{\BS, \beta} (\fr{h}, W, \fr{o})$.
\end{definition}

It is clear from definitions that $\otimes_R$ preserves blocks, i.e., for $\beta \in \uW{\scrL}{\scrL'}$ and $\gamma \in \uW{\scrL'}{\scrL''}$, we have that 
\begin{equation}\label{eq:abe_conv_presrves_blocks}
    \Amon{\scrL}{\scrL'}^{\BS, \beta} (\fr{h}, W, \fr{o}) \otimes_R \Amon{\scrL'}{\scrL''}^{\BS, \gamma} (\fr{h}, W, \fr{o}) \subset \Amon{\scrL}{\scrL''}^{\BS, \beta \gamma} (\fr{h}, W, \fr{o}).
\end{equation}

\begin{lemma}\label{lem:localization_sharpening_along_blocks}
    Let $\beta \in \uW{\scrL}{\scrL'}$ and $M \in \Amon{\scrL}{\scrL'}^{\oplus, \beta} (\fr{h}, W, \fr{o})$, then $\supp_W (M) \subseteq \beta$.
\end{lemma}
\begin{proof}
    Let $\uw\in \Exp (W)$ be an expression for $w \in \beta$. 
    We argue by induction on the length of $\uw$ that $\supp_W (B_{\uw}^{\scrL}) \subseteq \beta$.
    If $\uw = \emptyset$, then the claim is obvious.
    Now suppose $\uw = \uy s$ with $\uy$ an expression for $ws \in \gamma$ for some block $\gamma \in \W{\scrL}{\scrL' s}$.
    
    If $\scrL' s = \scrL'$, then $\beta = \gamma$.
    It follows from the definitions that
    \begin{equation}\label{eq:localization_sharpening_along_blocks_1} B_{\uw, Q}^{x, \scrL} \cong B_{\uy, Q}^{x, \scrL} \oplus B_{\uy, Q}^{xs, \scrL}.\end{equation}
    Since $\supp_W \left( B_{\uy}^{\scrL}\right) \subseteq \gamma$ and $s \in W_{\scrL'}^{\circ}$, we can conclude that (\ref{eq:localization_sharpening_along_blocks_1}) is only nonzero when $x \in \beta$.

    If $\scrL' s \neq \scrL'$, then we have $B_{\uw, Q}^{x, \scrL} \cong B_{\uy, Q}^{xs, \scrL}$.
    Therefore, 
    \[ \supp_W \left( B_{\uw}^{\scrL}\right) = \supp_W (B_{\uy}^{\scrL}) \cdot s \subseteq \gamma \cdot s = \beta. \]
\end{proof}

\begin{lemma}\label{lem:disjoint_support_implies_Hom_vanishing}
    Let $M, N \in \Cmon{\scrL}{\scrL'} (\fr{h}, W, \fr{o})$ such that $\supp_W (M)$ and $\supp_W (N)$ are disjoint. Then $\Hom_{\Cmon{\scrL}{\scrL'}} (M,N) = 0$.
\end{lemma}
\begin{proof}
    Let $I = \supp_W (M)$.
    By Lemma \ref{lem:std_facts_about_inv_and_coinvs} (2), we have an isomorphism 
    \[\Hom_{\Cmon{\scrL}{\scrL'}} (M,N) \cong \Hom_{\Cmon{\scrL}{\scrL'}'} (M, N_I).\]
    However, since $\supp_W (N) \cap I = \emptyset$, we have that $N_I = 0$. 
\end{proof}

\begin{lemma}\label{lem:abe_block_decomp}
    We have direct sum decompositions of additive categories
    \[\Amon{\scrL}{\scrL'}^{\oplus} (\fr{h}, W, \fr{o}) = \bigoplus_{\beta \in \uW{\scrL}{\scrL'}} \Amon{\scrL}{\scrL'}^{\oplus, \beta} (\fr{h}, W, \fr{o}).\]
\end{lemma}
\begin{proof}
    It is clear from definitions that $\left\{ \Amon{\scrL}{\scrL'}^{\oplus, \beta} (\fr{h}, W, \fr{o}) \right\}_{\beta \in \uW{\scrL}{\scrL'}}$ generate $\Amon{\scrL}{\scrL'}^{\oplus} (\fr{h}, W, \fr{o})$ under direct sums.
    It remains to show that if $\ux, \uy \in \Exp (W)$ are expressions for $x \in \beta$ and $y \in \gamma$ where $\beta, \gamma \in \uW{\scrL}{\scrL'}$ are distinct blocks, then
    \begin{equation}\label{eq:abe_block_decomp_1}
        \Hom_{\Cmon{\scrL}{\scrL'}} (B_{\ux}^{\scrL}, B_{\uy}^{\scrL}) = 0.
    \end{equation}
    The vanishing of (\ref{eq:abe_block_decomp_1}) follows immediately from Lemmas \ref{lem:localization_sharpening_along_blocks} and Lemma \ref{lem:disjoint_support_implies_Hom_vanishing}.
\end{proof}

\subsubsection{Neutral Block Endoscopy}

\begin{proposition}\label{prop:abe_endo}
    Let $\scrL \in \fr{o}$. There is an equivalence of monoidal categories
    \[ \Psi_{\scrL}^{\alg} : \Amon{1}{1}^{\oplus} (W_{\scrL}^{\circ}, \fr{h}, 1) \stackrel{\sim}{\to} \Amon{\scrL}{\scrL}^{\oplus, \circ} (\fr{h}, W, \fr{o})\]
    such that $\Psi_{\scrL}^{\alg} (B_{\ux}) \cong B_{\ux'}^{\scrL}$ where $\ux'$ is an expression in $W$ obtained from $\ux$ by replacing endosimple reflections with reduced expressions in $W$.
\end{proposition}
\begin{proof}
    Let $Q' = R [\frac{1}{\alpha_w} \colon w \in \scrR (W_{\scrL}^{\circ})]$.
    Since $\scrR (W_{\scrL}^{\circ}) \subseteq \scrR (W)$, there is an inclusion of graded rings $Q' \to Q$.
    We can then define a functor
    \begin{equation}
        \Psi_{\scrL}^{\alg} : \Cmon{1}{1} (W_{\scrL}^{\circ}, \fr{h}, 1) \to \Cmon{\scrL}{\scrL} (\fr{h}, W, \fr{o})
    \end{equation}
    as follows.
    Given an object $(M, M_{Q'}^w, \xi_M)$, we define the underlying $R$-bimodule of $\Psi_{\scrL}^{\alg} (M)$ to be $M$.
    The localization data is given by $\Psi_{\scrL}^{\alg} (M)_Q^w \coloneq M_{Q'}^w \otimes_{Q'} Q$ and $\xi_{\Psi_{\scrL}^{\alg} (M)} = \xi_M \otimes_{Q'} Q$.
    It is easy to check that $\Psi_{\scrL}^{\alg}$ defines a fully faithful monoidal functor which restricts to a functor
    \begin{equation}
        \Psi_{\scrL}^{\alg} : \Amon{1}{1}^{\oplus} (W_{\scrL}^{\circ}, \fr{h}, 1) \to \Amon{\scrL}{\scrL}^{\oplus, \circ} (\fr{h}, W, \fr{o}).
    \end{equation}

    It remains to check that $\Psi_{\scrL}^{\alg}$ is essentially surjective.
    Let $\uw$ be an expression in $W$ of an element $w \in W_{\scrL}^{\circ}$.
    We argue by induction on $\ell_{\scrL} (\uw)$. If $\ell_{\scrL} (\uw) = 0$, then $B_{\uw}^{\scrL} \cong R_e$ which is clearly in the essential image of $\Psi_{\scrL}^{\alg}$.
    If $\ell_{\scrL} (\uw) = k > 0$, then we can write $\uw = \ux \uy$ such that $\ell_{\scrL} (\ux) = k-1$ and $\ell_{\scrL \ux} (\uy) = 1$. Let $x$ and $y$ be the evaluations of $\ux$ and $\uy$ respectively.
    Pick a block $\beta \in \uW{\scrL x}{\scrL}$ such that $x w^{\beta}$ is in the neutral block. 
    Since $w$ is in the neutral block, we must also have that $w^{\beta, -1} y$ is in the neutral block.
    Let $\underline{v}$ be a reduced expression of $w^{\beta}$.
    By induction, $B_{\ux \underline{v}}^{\scrL}$ is in the essential image of $\Psi_{\scrL}^{\alg}$.
    Since $\ell_{\scrL} (\uy) = 1$ and $w^{\beta, -1} y$ is in the neutral block, $w^{\beta, -1} y$ is an endosimple reflection. As a result, $B_{\overline{\underline{v}} \uy}^{\scrL} \cong C_{w^{\beta, -1} y}$ is also in the essential image.
    Since $\Psi_{\scrL}^{\alg}$ is monoidal, we then have that $B_{\ux \underline{v}}^{\scrL} \otimes_R B_{\overline{\underline{v}} \uy}^{\scrL} \cong B_{\uw}^{\scrL}$ is in the essential image.
    As a result, $\Psi_{\scrL}^{\alg}$ is essentially surjective, and hence an equivalence of categories.
\end{proof}

\subsubsection{General Endoscopy}

Let $\scrL, \scrL', \scrL'' \in \fr{o}$, $\beta \in \uW{\scrL}{\scrL'}$, and $\gamma \in \uW{\scrL'}{\scrL''}$. There is an equivalence of monoidal categories
\[{}^{\beta} (-) : \Amon{1}{1}^{\oplus} (\fr{h}, W_{\scrL'}^{\circ}, 1) \to \Amon{1}{1}^{\oplus} (\fr{h}, W_{\scrL}^{\circ}, 1) \]
induced by the isomorphism of Coxeter groups $W_{\scrL'}^{\circ} \to W_{\scrL}^{\circ}$ defined by $w \mapsto w^{\beta} w w^{\beta, -1}$.
It follows from definitions that ${}^{\beta} ( {}^{\gamma} (-)) \cong {}^{\beta \gamma} (-)$.

On the monodromic side, there is an equivalence of monoidal categories
\[{}^{\beta} (-) : \Amon{\scrL'}{\scrL'}^{\oplus, \circ} (\fr{h}, W, \fr{o}) \to \Amon{\scrL}{\scrL}^{\oplus, \circ} (\fr{h}, W, \fr{o}) \]
defined by $M \mapsto R_{w^{\beta}} \otimes_R M \otimes_R R_{w^{\beta, -1}}$. By Lemma \ref{lem:soergel_product_formulas} that ${}^{\beta} ( {}^{\gamma} (-)) \cong {}^{\beta \gamma} (-)$.

\begin{lemma}\label{lem:Phi_alg_and_block_conj}
    There are natural isomorphisms making the following diagram commute for each $\beta \in \uW{\scrL}{\scrL'}$.
    \begin{equation}\label{eq:Phi_alg_and_block_conj_1}
        \begin{tikzcd}
            {\Amon{1}{1}^{\oplus} (\fr{h}, W_{\scrL'}^{\circ}, 1)} \arrow[d, "{}^{\beta} (-)"'] \arrow[r, "\Psi_{\scrL'}^{\alg}"] & {\Amon{\scrL'}{\scrL'}^{\oplus, \circ} (\fr{h}, W, \fr{o})} \arrow[d, "{}^{\beta} (-)"] \\
            {\Amon{1}{1}^{\oplus} (\fr{h}, W_{\scrL}^{\circ}, 1)} \arrow[r, "\Psi_{\scrL}^{\alg}"]                            & {\Amon{\scrL}{\scrL}^{\oplus,\circ} (\fr{h}, W, \fr{o})}                           
            \end{tikzcd}
    \end{equation}
    Moreover, these natural transformations are compatible with respect to block composition.
\end{lemma}
\begin{proof} 
    After unpacking definitions and using monoidality of the maps involved in (\ref{eq:Phi_alg_and_block_conj_1}), it suffices to construct isomorphisms
    \[\xi_s^{\beta} : C_{w^{\beta} s w^{\beta, -1}} \stackrel{\sim}{\to} R_{w^{\beta}} \otimes_R C_s \otimes_R R_{w^{\beta, -1}}\]
    for each endosimple reflection $s \in S_{\scrL'}^{\circ}$. By Lemma \ref{lem:soergel_product_formulas}, we can take $\xi_s^{\beta}$ to be the unique isomorphism which preserves the 1-tensor.
    This then defines a natural transformation of functors $\xi^{\beta} : \Psi_{\scrL'}^{\alg} \circ {}^{\beta} (-) \implies {}^{\beta} (-) \circ \Psi_{\scrL}^{\alg} $.

    Let $\gamma \in \uW{\scrL'}{\scrL''}$.
    Since $\xi_s^{\beta}$ is the unique map which preserves the 1-tensor, we have that 
    $\xi_s^{\gamma} \circ \xi_{w^{\gamma} s w^{\gamma, -1}}^{\beta} = \xi_s^{\beta \gamma}$ for $s \in S_{\scrL''}^{\circ}$. As a result, the natural transformations $\xi^{\beta}$ are compatible with respect to block composition.
\end{proof}

We will construct a 2-category $\scrA^{\oplus} (\fr{h}, M^{\fr{o}} (W))$, called the \emph{endoscopic 2-category of Bott--Samelson bimodules} as follows. The object set of $\scrA^{\oplus} (\fr{h}, M^{\fr{o}} (W))$ is $\fr{o}$.
The morphism category from $\scrL$ to $\scrL'$ of $\scrA^{\oplus} (\fr{h}, M^{\fr{o}} (W))$ is given by $\bigoplus_{\beta \in \uW{\scrL}{\scrL'}} \scrA^{\oplus} (\fr{h}, \beta)$ where $\scrA^{\oplus} (\fr{h}, \beta) \coloneq \Amon{1}{1}^{\oplus} (\fr{h}, W_{\scrL}^{\circ}, 1)$.
The vertical composition is induced from the composition in $\scrA^{\oplus} (\fr{h}, \beta)$. Let $\beta \in \uW{\scrL}{\scrL'}$ and $\gamma \in \uW{\scrL'}{\scrL''}$. 
We define the horizontal composition, denoted $\otimes$ as the bifunctor
\begin{align*}
    \scrA^{\oplus} (\fr{h}, \beta) \times \scrA^{\oplus} (\fr{h}, \gamma) &= \Amon{1}{1}^{\oplus} (\fr{h}, W_{\scrL}^{\circ}, 1) \times \Amon{1}{1}^{\oplus} (\fr{h}, W_{\scrL'}^{\circ}, 1) \\
    &\stackrel{\id \times {}^{\beta} (-) }{\longrightarrow} \Amon{1}{1}^{\oplus} (\fr{h}, W_{\scrL}^{\circ}, 1) \times \Amon{1}{1}^{\oplus} (\fr{h}, W_{\scrL}^{\circ}, 1) \\
    &\stackrel{\otimes_R}{\to} \Amon{1}{1}^{\oplus} (\fr{h}, W_{\scrL}^{\circ}, 1) \\
    &= \scrA^{\oplus} (\fr{h}, \beta \gamma).
\end{align*}

\begin{theorem}[Algebraic Monodromic-Endoscopic Equivalence]\label{thm:alg_endoscopy}
    There is an equivalence of 2-categories
    \[ \Psi^{\alg} : \scrA^{\oplus} (\fr{h}, M^{\fr{o}} (W)) \stackrel{\sim}{\to} \Amon{}{}^{\oplus} (\fr{h}, W, \fr{o}) \]
    which restricts to the equivalence of monoidal categories 
    \[ \Psi_{\scrL}^{\alg} : \Amon{1}{1}^{\oplus} (\fr{h}, W_{\scrL}^{\circ}, 1) \stackrel{\sim}{\to} \Amon{\scrL}{\scrL}^{\oplus} (\fr{h}, W, \fr{o})\]
    for each $\scrL \in \fr{o}$ coming from Proposition \ref{prop:abe_endo}.
\end{theorem}
\begin{proof}
    Let $\scrL, \scrL' \in \fr{o}$ and $\beta \in \uW{\scrL}{\scrL'}$. On objects, $\Psi^{\alg} (\scrL) = \scrL$. On morphism categories, we define $\Psi^{\alg}$ as the composition
    \[\begin{tikzcd}
{\scrA^{\oplus} (\fr{h}, \beta) = \Amon{1}{1}^{\oplus} (\fr{h}, W_{\scrL}^{\circ}, 1)} \arrow[r, "\Psi_{\scrL}^{\alg}"] & {\Amon{\scrL}{\scrL}^{\oplus, \circ} (\fr{h}, W, \fr{o})} \arrow[rr, "(-) \otimes_R R_{w^{\beta}}"] &  & {\Amon{\scrL}{\scrL'}^{\oplus, \beta} (\fr{h}, W, \fr{o}).}
\end{tikzcd}\]
If we can prove that $\Psi^{\alg}$ is a 2-functor, then it is automatic from Proposition \ref{prop:abe_endo} that it will be an equivalence of 2-categories.
Let $M \in \scrA^{\oplus} (\fr{h}, \beta)$ and $N \in \scrA^{\oplus} (\fr{h}, \gamma)$. We will construct a natural isomorphism
\begin{equation}\label{eq:alg_endoscopy_0}
    \Psi^{\alg} (M) \otimes_R \Psi^{\alg} (N) \stackrel{\sim}{\to} \Psi^{\alg} (M \otimes N)
\end{equation}
given by the composition
\begin{align}
    \Psi^{\alg} (M) \otimes_R \Psi^{\alg} (N) &= \Psi_{\scrL}^{\alg} (M) \otimes_R R_{w^{\beta}} \otimes_R \Psi_{\scrL}^{\alg} (N) \otimes_R R_{w^{\gamma}} \notag \\
    &\cong \Psi_{\scrL}^{\alg} (M) \otimes_R R_{w^{\beta}} \otimes_R \Psi_{\scrL}^{\alg} (N) \otimes_R R_{w^{\beta, -1}} \otimes_R R_{w^{\beta \gamma}} \notag \\
    &\cong \Psi_{\scrL}^{\alg} (M) \otimes_R \Psi_{\scrL}^{\alg} ({}^{\beta} N) \otimes_R R_{w^{\beta \gamma}} \label{eq:alg_endoscopy_1} \\
    &\cong \Psi_{\scrL}^{\alg} (M \otimes_R {}^{\beta} N ) \otimes_R R_{w^{\beta \gamma}} \label{eq:alg_endoscopy_2} \\
    &=  \Psi^{\alg} (M \otimes N), \notag
\end{align}
where (\ref{eq:alg_endoscopy_1}) is Lemma \ref{lem:Phi_alg_and_block_conj} and (\ref{eq:alg_endoscopy_2}) comes from $\Psi_{\scrL}^{\alg}$ being monoidal.
It is easy to check from the compatibility of Lemma \ref{lem:Phi_alg_and_block_conj} with block composition that (\ref{eq:alg_endoscopy_0}) is associative. 
Therefore, (\ref{eq:alg_endoscopy_0}) makes $\Psi^{\alg}$ into a 2-functor.
\end{proof}

    \section{Algebraic Light Leaves}\label{sec:ll}

The goal of this section is to extend the construction of light leaves for Bott--Samelson bimodules (cf., \cite{Lib08, Abe19}) to monodromic Bott--Samelson bimodules. 
We closely follow the construction and proofs given in \cite{Abe19}.
We also provide three main applications of light leaves:
\begin{enumerate}
    \item a version of Soergel's Hom formula to compute the graded rank of homomorphisms in the category of Bott--Samelson bimodules;
    \item a basis of the Hom spaces between Bott--Samelson bimodules;
    \item a cellular structure on Bott--Samelson bimodules.
\end{enumerate} 

In order to use the morphism afforded by Proposition \ref{prop:existence_of_rex_moves}, we will assume throughout this section that $\fr{h}$ is a reflection-balanced monodromic Abe realization.

\subsection{Preliminaries on Monodromic Subexpressions}

\subsubsection{Subexpressions and Order}

Let $\uw = (s_1, \ldots, s_k)$ be an expression in $W$.
Recall the set,
\[K (\uw, \scrL) = \{ 1 \leq i \leq k \mid \scrL s_{1} \ldots s_i =  \scrL s_{1} \ldots s_{i-1} \}.\]
We will denote by $\Subexp (\uw)$ the set of subexpressions of $\uw$.
For a given $x \in W$, we can consider the subset $\Subexp (\uw, x)$ consisting of subexpressions $\ue$ of $\uw$ such that $\uw^{\ue} = x$.

A subexpression $\ue = (e_1, \ldots, e_k)$ of $\uw$ is said to be ($\scrL$-)\emph{monodromic} if $e_i = 1$ whenever $i \notin K (\uw, \scrL)$.
We will denote by $\Subexp^{\scrL} (\uw)$ the subset of $\Subexp (\uw)$ consisting of $\scrL$-monodromic subexpressions.
Similarly, for $x\in W$, we will write $\Subexp^{\scrL} (\uw, x) = \Subexp^{\scrL} (\uw) \cap \Subexp (\uw, x)$.

Let $\uw = (s_1, \ldots, s_k) \in \Exp (W)$ and $\ue \in \Subexp (\uw)$.
For each $i=1, \ldots, k$, we write $\uw_{\leq i}$ (resp. $\ue_{\leq i}$) for the first $i$ terms of $\uw$ (resp. $\ue$). 
Then $\ue_{\leq i}$ is a subexpression of $\uw_{\leq i}$.
Let $w_i$ be the element expressing $\uw_{\leq i}$.
We can define the $i$-th \emph{label}, $\lab_i (\uw, \ue) \in \{ U, D\}$ by 
\[\lab_i (\uw, \ue) \coloneq \begin{cases} U & \text{if } w_{i-1} s_i > w_{i-1}, \\ D & \text{if } w_{i-1} s_i < w_{i-1}.\end{cases}\]
The $i$-th \emph{decoration} of $\ue$, denoted $\dec_i (\uw, \ue) \in \{ U0, U1, D0, D1\}$, is defined by 
\[\dec_i (\uw, \ue) \coloneq \lab_i (\uw, \ue) e_i.\]
The \emph{defect} of $\ue$, denoted $d(\uw,\ue)$, is the integer defined by
\[d(\uw,\ue) \coloneq \#\{i \mid \dec_i(\uw,\ue) = U0\} - \# \{i \mid \dec_i (\uw,\ue) = D0\}.\]

For $x \in W$, we can define a relation $<$, called the \emph{label-lexicographic order}, on $\Subexp (\uw, x)$ by $\uf < \ue$ if and only if there exists some $1 \leq i = i (\ue, \uf)  \leq k$ such that
\begin{enumerate}
    \item $\lab_j (\uw, \ue) = \lab_j (\uw, \uf)$ for all $j < i$,
    \item $\lab_i (\uw, \ue) = D$ and $\lab_i (\uw, \uf) = U$.
\end{enumerate}
By \cite[Lemma 3.9]{Abe19}, the relation $<$ gives a total order on $\Subexp (\uw, x)$. 
Moreover, if $\ue$ and $\uf$ are $\scrL$-monodromic subexpressions, then it follows from the definitions that $i (\ue, \uf) \in K(\uw, \scrL)$.
As a result, the relation $<$ restricts to a total order on $\Subexp^{\scrL} (\uw, x)$.

We define another relation on subexpressions, although this time without a fixed evaluation.
Let $\ue, \uf$ be two subexpressions of $\uw \in \Exp (W)$. 
Let $x_0, x_1, \ldots, x_m$ and $y_0, y_1, \ldots, y_m$ denote the Bruhat strolls of $\ue$ and $\uf$, respectively.
We say that $\ue \prec \uf$ if $x_i \leq y_i$ for all $0 \leq i \leq m$. This defines a partial order $\prec$ on $\Subexp (\uw)$ called the \emph{path dominance order}.
The path dominance order restricts to a partial order on monodromic subexpressions.

\begin{lemma}[Monodromic Deodhar Defect Formula]\label{lem:defect_formula_for_pwx}
    Let $\uw \in \Exp (W)$. Then
    \begin{equation}\label{eq:defect_formula_1}
        \uH_{\uw}^{\scrL} = \sum_{\ue \in \Subexp^{\scrL} (\uw)} v^{d(\uw,\ue)} H_{\uw^{\ue}}^{\scrL}.
    \end{equation}
    As a result, for all $x\in W$, we have 
    \begin{equation}\label{eq:defect_formula_2}
        p_{\uw}^{x, \scrL} (v) = \sum_{\ue \in \Subexp^{\scrL} (\uw, x)} v^{d(\uw, \ue)}.
    \end{equation}
\end{lemma}
\begin{proof}
    Note (\ref{eq:defect_formula_2}) follows from (\ref{eq:defect_formula_1}) from the definition of the standard form.
    We argue by induction on the length of $\uw$. If $\uw = \emptyset$, then (\ref{eq:defect_formula_1}) clearly holds.
    Let $\uw = \uy s$.
    Suppose (\ref{eq:defect_formula_1}) holds for $\uy$.
    Note that for all $x \in \W{\scrL}{\scrL'}$, we have that
    \[H_x^{\scrL} \uH_{s}^{\scrL'} = \begin{cases} H_{xs}^{\scrL} + vH_x^{\scrL} & xs > x, \scrL' s = \scrL', \\ H_{xs}^{\scrL} + v^{-1} H_x^{\scrL} & xs < x, \scrL' s = \scrL', \\ H_{xs}^{\scrL} & \scrL' s \neq \scrL'.\end{cases} \]
    If $\scrL \uw \neq \scrL \uy$, then any monodromic subexpression of $\uw$ arises from a monodromic subexpression of $\uy$ by attaching a ``1'' to the end. In particular, the defect does not change.
    We can then compute
    \[ \uH_{\uw}^{\scrL} = \sum_{\ue \in \Subexp^{\scrL} (\uy)} v^{d(\uy,\ue)} H_{\uy^{\ue}}^{\scrL} \uH_{s}^{\scrL \uy} =  \sum_{\ue \in \Subexp^{\scrL} (\uw)} v^{d(\uw,\ue)} H_{\uw^{\ue}}^{\scrL}.\]
    Now assume $\scrL \uw = \scrL \uy$. Let $k = \ell (\uw)$.
    \begin{align*}
        \uH_{\uw}^{\scrL} &= \sum_{\ue \in \Subexp^{\scrL} (\uy)} v^{d(\uy,\ue)} H_{\uy^{\ue}}^{\scrL} \uH_{s}^{\scrL'} \\
        &= \sum_{\stackrel{\ue \in \Subexp^{\scrL} (\uy)}{\uy^{\ue} s > \uy^{\ue}}} v^{d(\uy,\ue)} H_{\uy^{\ue}}^{\scrL} \uH_{s}^{\scrL'} + \sum_{\stackrel{\ue \in \Subexp^{\scrL} (\uy)}{\uy^{\ue} s < \uy^{\ue}}} v^{d(\uy,\ue)} H_{\uy^{\ue}}^{\scrL} \uH_{s}^{\scrL'} \\
        &= \sum_{\stackrel{\ue \in \Subexp^{\scrL} (\uy)}{\uy^{\ue} s > \uy^{\ue}}} v^{d(\uy,\ue)} (H_{\uy^{\ue} s}^{\scrL} + vH_{\uy^{\ue}}) + \sum_{\stackrel{\ue \in \Subexp^{\scrL} (\uy)}{\uy^{\ue} s < \uy^{\ue}}} v^{d(\uy,\ue)} (H_{\uy^{\ue} s}^{\scrL} + v^{-1} H_{\uy^{\ue}}) \\
        &= \sum_{\stackrel{\ue \in \Subexp^{\scrL} (\uw)}{\dec_k (\uw, \ue) = U}} v^{d(\uw,\ue)} H_{\uw^{\ue}}^{\scrL}  + \sum_{\stackrel{\ue \in \Subexp^{\scrL} (\uw)}{\dec_k (\uw, \ue) = D}} v^{d(\uw,\ue)} H_{\uw^{\ue}}^{\scrL} \\
        &=  \sum_{\ue \in \Subexp^{\scrL} (\uw)} v^{d(\uw,\ue)} H_{\uw^{\ue}}^{\scrL}.
    \end{align*}
\end{proof}

\subsubsection{Rigidification Data}\label{subsec:rigidification_data}

As in the non-monodromic setting, the basis of light leaves we will construct is far from canonical.
We define an \emph{$\LL$-datum} as a triple ($\rex$, $\{\rex_s\}_{s \in S}$, $\{\Gamma_{\ux, \uy}\}_{\ux, \uy}$) consisting of the following:
\begin{enumerate}\label{ll_disiderata}
    \item a map $\rex : W \to \Exp (W)$ such that for each $x \in W$, $\rex (x)$ is a reduced expression for $x$;
    \item for each $s \in S$, a map $\rex_s : \mathscr{R}_s \to \Exp (W)$ where $\mathscr{R}_s = \{ w \in W \mid \ell (ws) < \ell (w) \}$ and $\rex_s (x)$ is a reduced expression for $x$ ending in $s$;
    \item for any two reduced expressions $\ux$ and $\uy$ for $x$, a path $\Gamma_{\ux, \uy}$ from $\ux$ to $\uy$ in the rex graph of $x$.
\end{enumerate}
For the remainder of the section, we will fix an $\LL$-datum. 

For reduced expressions $\ux$ and $\uy$ of $x$, we will write $\beta_{\ux,\uy} : B_{\ux}^{\scrL} \to B_{\uy}^{\scrL}$ for the morphism determined by $\Gamma_{\ux, \uy}$ as in \S\ref{subsec:abe_mors}.

\subsection{Definition and Algorithm}

Let $\ux = (s_1, \ldots, s_k)$ be an expression in $W$, and let $\ue$ be an $\scrL$-monodromic subexpression of $\ux$. For each $1 \leq i \leq k$, let $x_i$ be the element realizing $\ux_{\leq i}$.
Let $w_i = \uw_{\leq i}^{\ue_{\leq i}}$.
We write $\uw_i = \rex (w_i)$ for the reduced expression of $w_i$ fixed in \S \ref{subsec:rigidification_data}. 
The condition that $\ue$ be monodromic ensures that $\scrL x_i = \scrL w_i$.

The construction of the light leaves is almost identical to the non-monodromic variants.
The key difference is that we will only allow for monodromic subexpressions.
For the sake of completeness, we will detail the algorithm below.

We construct the light leaf morphism $\LL_{\ux, \ue}^{\scrL} : B_{\ux}^{\scrL} \to B_{\uw}^{\scrL}$ of degree $d(\ux,\ue)$ inductively on the length $k$ of $\ux$.
The base case is covered by defining $\LL_{\emptyset, \emptyset}^{\scrL} : R \to R$ as the identity map. 
Suppose that we have already constructed a map 
\[\LL_{i-1}^{\scrL} := \LL_{\ux_{\leq i-1}, \ue_{\leq i-1}}^{\scrL} : B_{\ux_{\leq i-1}}^{\scrL} \to B_{\uw_{i-1}}^{\scrL}.\]
In a moment, we will define a map 
\[\phi_i : B_{\ux_{\leq i-1} s_i}^{\scrL} \to B_{\uw_{i}}^{\scrL}.\] 
We then define 
\[\LL_i^{\scrL} = \phi_i \circ ( \LL_{i-1}^{\scrL} \otimes_R \id_{B_s^{\scrL w_{i-1}}}).\]

The definition of $\phi_i$ depends on a decoration $\dec_i (\ux, \ue)$ and whether $i \in K(\ux, \scrL)$.
We will now explain $\phi_i$ for each decoration.\\
($U0$, $i \in K(\ux, \scrL)$): 
\[
    \begin{tikzcd}
        \phi_i : B_{\ux_{\leq i-1}}^{\scrL} \otimes_R C_{s_i} \arrow[rr, "\id \otimes \epsilon^{s_i}"] &  &  B_{\ux_{\leq i-1}}^{\scrL} \arrow[rr, "{\beta_{\ux_{\leq i-1}, \uw_i}}"] &  &  B_{\uw_i}^{\scrL}.
    \end{tikzcd}
\]
($U1$, $i \in K(\ux, \scrL)$):
\[
    \begin{tikzcd}
        \phi_i : B_{\ux_{\leq i-1}}^{\scrL} \otimes_R C_{s_i} = B_{\ux_{\leq i}}^{\scrL} \arrow[rr, "{\beta_{\ux_{\leq i}, \uw_i}}"] &  & B_{\uw_i}^{\scrL}.
    \end{tikzcd}
\]
($U1$, $i \notin K(\ux, \scrL)$):
\[
    \begin{tikzcd}
        \phi_i : B_{\ux_{\leq i-1}}^{\scrL} \otimes_R R_{s_i} = B_{\ux_{\leq i}}^{\scrL} \arrow[rr, "{\beta_{\ux_{\leq i}, \uw_i}}"] &  & B_{\uw_i}^{\scrL}.
    \end{tikzcd}
\]
($D0$, $i \in K(\ux, \scrL)$): Let $\uw_{i-1, s_i} = \rex_{s_i} (w_i)$. We can write $\uw_{i-1, s_i} = \uw_{i-1, s_i}' s_i$ for some reduced expression $\uw_{i-1, s_i}'$.
\[
    \begin{tikzcd}
        \phi_i : B_{\ux_{\leq i-1}}^{\scrL} \otimes_R C_{s_i} \arrow[rrr, "{\beta_{\ux_{\leq i-1}, \uw_{i-1, s_i}} \otimes \id}"] & & & {B_{\uw_{i-1, s_i}'}^{\scrL} \otimes_R C_{s_i} \otimes_R C_{s_i}} \arrow[r, "\id \otimes \mu^{s_i}"]  & {B_{\uw_{i-1, s_i}}^{\scrL}} \arrow[r, "{\beta_{\uw_{i-1, s_i}, \uw_i}}"] &  B_{\uw_i}^{\scrL}.
    \end{tikzcd}
\]
($D1$, $i \in K(\ux, \scrL)$): Let $\uw_{i-1, s_i} = \rex_{s_i} (w_i)$. Write $\uw_{i-1, s_i} = \uw_{i-1, s_i}' s_i$.
\[ 
    \begin{tikzcd}
        \phi_i : B_{\ux_{\leq i-1}}^{\scrL} \otimes_R C_{s_i} \arrow[rrr, "{\beta_{\ux_{\leq i-1}, \uw_{i-1, s_i}} \otimes \id}"] &  & & {B_{\uw_{i-1, s_i}'}^{\scrL} \otimes_R C_{s_i} \otimes_R C_{s_i}} \arrow[r, "\id \otimes \cap^{s_i}"]  & {B_{\uw_{i-1, s_i}'}^{\scrL}} \arrow[r, "{\beta_{\uw_{i-1, s_i}', \uw_i}}"] &  B_{\uw_i}^{\scrL}.
    \end{tikzcd}
\]
($D1$, $i \notin K(\ux, \scrL)$): Let $\uw_{i-1, s_i} = \rex_{s_i} (w_i)$. Write $\uw_{i-1, s_i} = \uw_{i-1, s_i}' s_i$.
\[
    \begin{tikzcd}
        \phi_i : B_{\ux_{\leq i-1}}^{\scrL} \otimes_R R_{s_i} \arrow[rrr, "{\beta_{\ux_{\leq i-1}, \uw_{i-1, s_i}} \otimes \id}"] &  &  & {B_{\uw_{i-1, s_i}'}^{\scrL} \otimes_R R_{s_i} \otimes_R R_{s_i}} \arrow[r, "\id \otimes \cap^{s_i}"] & {B_{\uw_{i-1, s_i}'}^{\scrL}} \arrow[r, "{\beta_{\uw_{i-1, s_i}', \uw_i}}"] &  B_{\uw_i}^{\scrL}.
    \end{tikzcd}
\]

\subsection{Linear Independence}

For each $s \in S$, fix $\delta_s \in \fr{h}^*$ such that $\langle \alpha_s^{\vee}, \delta_s \rangle = 1$ as in \S\ref{subsec:abe_mors}.
Let $\ux = (s_1, \ldots, s_k) \in \Exp (W)$, $w\in W$, and $\ue \in \Subexp^{\scrL} (\ux, w)$.
Define $b_{\ux, \ue} \in B_{\ux}^{\scrL}$ by $b_{\ux, \ue} = b_1 \otimes \ldots \otimes b_k$ where $b_i \in B_{s_i}^{\scrL s_{1} \ldots s_{i-1}}$ is defined by
\[b_i = \begin{cases} 1 \otimes 1 \in C_{s_i} & \text{if } \lab_i (\ux, \ue) = U, i \in K(\ux, \scrL), \\ \delta_{s_i} \otimes 1 \in C_{s_i} & \text{if } \lab_i (\ux, \ue) = D, i \in K(\ux, \scrL), \\ 1 \in R_{s_i} & \text{if } i \notin K(\ux, \scrL).\end{cases}\]
Let $\LL_{\ux, w}^{\scrL} \coloneq \{ \LL_{\ux, \ue}^{\scrL} \colon \ux^{\ue} = w \}$.

\begin{lemma}\label{lem:ll_are_lin_indep}
    Let $w \in W$, $\ux \in \Exp (W)$, and $\ue, \uf \in \Subexp^{\scrL} (\ux, w)$. Write $\uw = \rex (w)$.
    Then
    \[\LL_{\ux, \ue}^{\scrL} (b_{\ux, \uf}) = \begin{cases} u_{\uw} & \text{if } \uf = \ue, \\ 0 & \text{if } \uf < \ue. \end{cases}\]
    In particular, $\LL_{\ux, w}^{\scrL}$ is linearly independent over $R$.
\end{lemma}
\begin{proof}
    We claim that if $\lab_i (\ux, \ue) = \lab_i (\ux,\uf)$ for all $i \leq j$, then $\LL_j^{\scrL} (b_{\ux_{\leq j}, \uf_{\leq j}}) = u_{\uw_j}$ by induction on $j$.

    Assume $\dec_j (\ux, \ue) = U0$ and $j \in K(\ux, \scrL)$. In this case, $b_{\ux_{\leq j}, \uf_{\leq j}} = b_{\ux_{\leq j-1}, \uf_{\leq j-1}} \otimes (1 \otimes 1)$.
    Then we have
    \begin{align*}
        \LL_j^{\scrL} (b_{\ux_{\leq j}, \uf_{\leq j}}) &= \beta_{\uw_{j-1}, \uw_j} \left(\LL_{j-1}^{\scrL} (b_{\ux_{\leq j-1}, \uf_{\leq j-1}}) \otimes \epsilon^{s_j} (1 \otimes 1) \right) \\
        &= \beta_{\uw_{j-1}, \uw_j} (u_{\uw_{j-1}}) \\
        &= u_{\uw_j}.
    \end{align*}

    Assume $\dec_j (\ux, \ue) = U1$ and $j \in K(\ux, \scrL)$. In this case, $b_{\ux_{\leq j}, \uf_{\leq j}} = b_{\ux_{\leq j-1}, \uf_{\leq j-1}} \otimes (1 \otimes 1)$.
    Then we have
    \begin{align*}
        \LL_j^{\scrL} (b_{\ux_{\leq j}, \uf_{\leq j}}) &= \beta_{\uw_{j-1}s_j, \uw_j} \left(\LL_{j-1}^{\scrL} (b_{\ux_{\leq j-1}, \uf_{\leq j-1}}) \otimes 1 \otimes 1 \right) \\
        &= \beta_{\uw_{j-1}s_j, \uw_j} (u_{\uw_{j-1}} \otimes 1 \otimes 1) \\
        &= u_{\uw_j}.
    \end{align*}

    Assume $\dec_j (\ux, \ue) = U1$ and $j \notin K(\ux, \scrL)$. In this case, $b_{\ux_{\leq j}, \uf_{\leq j}} = b_{\ux_{\leq j-1}, \uf_{\leq j-1}} \otimes 1$.
    Then we have
    \begin{align*}
        \LL_j^{\scrL} (b_{\ux_{\leq j}, \uf_{\leq j}}) &= \beta_{\uw_{j-1}s_j, \uw_j} \left(\LL_{j-1}^{\scrL} (b_{\ux_{\leq j-1}, \uf_{\leq j-1}}) \otimes 1 \right) \\
        &= \beta_{\uw_{j-1}s_j, \uw_j} (u_{\uw_{j-1}} \otimes 1 ) \\
        &= u_{\uw_j}.
    \end{align*}

    Assume $\dec_j (\ux, \ue) = D0$ and $j \in K(\ux, \scrL)$. In this case, $b_{\ux_{\leq j}, \uf_{\leq j}} = b_{\ux_{\leq j-1}, \uf_{\leq j-1}} \otimes (\delta_{s_j} \otimes 1)$.
    Then we have
    \begin{align*}
        \LL_j^{\scrL} (b_{\ux_{\leq j}, \uf_{\leq j}}) &= \beta_{\uw_{j-1,s_j}, \uw_j} \left( (\id \otimes m^{s_j}) \left( \beta_{\uw_{j-1}, \uw_{j-1,s_j}} \left( \LL_{j-1}^{\scrL} (b_{\ux_{\leq j-1}, \uf_{\leq j-1}}) \right) \otimes (\delta_{s_j} \otimes 1) \right) \right) \\
        &= \beta_{\uw_{j-1,s_j}, \uw_j} \left( (\id \otimes m^{s_j}) \left( u_{\uw_{j-1,s_j}} \otimes \delta_{s_j} \otimes 1 \right) \right) \\
        &= \beta_{\uw_{j-1,s_j}, \uw_j} \left( u_{\uw_{j-1,s_j}}\right) \\
        &= u_{\uw_j}.
    \end{align*}

    Assume $\dec_j (\ux, \ue) = D1$ and $j \in K(\ux, \scrL)$. In this case, $b_{\ux_{\leq j}, \uf_{\leq j}} = b_{\ux_{\leq j-1}, \uf_{\leq j-1}} \otimes (\delta_{s_j} \otimes 1)$.
    Then we have
    \begin{align*}
        \LL_j^{\scrL} (b_{\ux_{\leq j}, \uf_{\leq j}}) &= \beta_{\uw_{j-1,s_j}, \uw_j} \left( (\id \otimes \cap^{s_j}) \left( \beta_{\uw_{j-1}, \uw_{j-1,s_j}} \left( \LL_{j-1}^{\scrL} (b_{\ux_{\leq j-1}, \uf_{\leq j-1}}) \right) \otimes (\delta_{s_j} \otimes 1) \right) \right) \\
        &= \beta_{\uw_{j-1,s_j}', \uw_j} \left( (\id \otimes \cap^{s_j}) \left( u_{\uw_{j-1,s_j}} \otimes \delta_{s_j} \otimes 1 \right) \right) \\
        &= \beta_{\uw_{j-1,s_j}', \uw_j} \left( u_{\uw_{j-1,s_j}'}\right) \\
        &= u_{\uw_j}.
    \end{align*}

    Assume $\dec_j (\ux, \ue) = D1$ and $j \notin K(\ux, \scrL)$. In this case, $b_{\ux_{\leq j}, \uf_{\leq j}} = b_{\ux_{\leq j-1}, \uf_{\leq j-1}} \otimes 1$.
    Then we have
    \begin{align*}
        \LL_j^{\scrL} (b_{\ux_{\leq j}, \uf_{\leq j}}) &= \beta_{\uw_{j-1,s_j}, \uw_j} \left( (\id \otimes \cap^{s_j}) \left( \beta_{\uw_{j-1}, \uw_{j-1,s_j}} \left( \LL_{j-1}^{\scrL} (b_{\ux_{\leq j-1}, \uf_{\leq j-1}}) \right) \otimes 1 \right) \right) \\
        &= \beta_{\uw_{j-1,s_j}', \uw_j} \left( (\id \otimes \cap^{s_j}) \left( u_{\uw_{j-1,s_j}} \otimes 1 \right) \right) \\
        &= \beta_{\uw_{j-1,s_j}', \uw_j} \left( u_{\uw_{j-1,s_j}'}\right) \\
        &= u_{\uw_j}.
    \end{align*}
    By induction, we can then conclude that $\LL_{\ux, \ue}^{\scrL} (b_{\ux, \ue}) = u_{\uw}$.

    Assume that $\uf < \ue$. Take $j$ such that for all $i < j$, $\lab_i (\ux, \ue) = \lab_i (\ux, \uf)$ and $\lab_j (\ux, \ue) = D$ and $\lab_j (\ux, \uf) = U$.
    Since $\ue$ and $\uf$ are $\scrL$-monodromic, we must have that $j \in K(\ux, \scrL)$. As a result, $b_{\uw_{\leq j}, \uf_{\leq j}} = b_{\uw_{\leq j-1}, \uf_{\leq j -1 }} \otimes (1 \otimes 1)$.

    If $e_j = 0$, we have
    \begin{align*}
        \LL_j^{\scrL} (b_{\ux_{\leq j}, \uf_{\leq j}}) &= \beta_{\uw_{j-1,s_j}, \uw_j} \left( (\id \otimes m^{s_j}) \left( \beta_{\uw_{j-1}, \uw_{j-1,s_j}} \left( \LL_{j-1}^{\scrL} (b_{\ux_{\leq j-1}, \uf_{\leq j-1}}) \right) \otimes (1 \otimes 1) \right) \right) \\
        &= \beta_{\uw_{j-1,s_j}, \uw_j} \left( (\id \otimes m^{s_j}) \left( u_{\uw_{j-1,s_j}} \otimes 1 \otimes 1 \right) \right) \\
        &= \beta_{\uw_{j-1,s_j}, \uw_j} \left( 0\right) \\
        &= 0.
    \end{align*}

    If $e_j = 1$, we have
    \begin{align*}
        \LL_j^{\scrL} (b_{\ux_{\leq j}, \uf_{\leq j}}) &= \beta_{\uw_{j-1,s_j}, \uw_j} \left( (\id \otimes \cap^{s_j}) \left( \beta_{\uw_{j-1}, \uw_{j-1,s_j}} \left( \LL_{j-1}^{\scrL} (b_{\ux_{\leq j-1}, \uf_{\leq j-1}}) \right) \otimes (1 \otimes 1) \right) \right) \\
        &= \beta_{\uw_{j-1,s_j}', \uw_j} \left( (\id \otimes \cap^{s_j}) \left( u_{\uw_{j-1,s_j}} \otimes 1 \otimes 1 \right) \right) \\
        &= \beta_{\uw_{j-1,s_j}', \uw_j} \left( 0\right) \\
        &= 0.
    \end{align*}

    Therefore, $\LL_{\ux, \ue}^{\scrL} (b_{\ux, \uf}) = 0$ if $\uf < \ue$.
\end{proof}

\subsection{Support Theory}\label{subsec:support_of_abe}

Let $\ux \in \Exp (W)$, $w \in W$, $\ue \in \Subexp^{\scrL} (\ux, w)$, and $\uw = \rex (w)$. 
Recall the map $\pi_{\ux}^w : B_{\ux}^{\scrL} \to B_{\ux}^{w, \scrL} \subseteq B_{\ux, Q}^{w, \scrL}$ defined in \S\ref{subsec:support_of_bims}.
Define $b_{\ux, \ue}^w \coloneq \pi_{\ux}^w (b_{\ux, \ue}) \in  B_{\ux}^{w, \scrL}$. 

Write $\uw = (s_1, \ldots, s_k)$. We define a degree $\ell_{\scrL} (w)$ morphism $\varphi_{\uw} : B_{\uw}^{\scrL} \to R_w$ in $\Cmon{\scrL}{\scrL w} (\fr{h}, W, \fr{o})$ inductively on $k$.
Let $w_i \in W$ be the evaluation of $\uw_{\leq i}$ for each $1 \leq i \leq k$.
First, we define $\varphi_{\emptyset} : R \to R$ as the identity map. 
Assume that $\varphi_{\uw_{\leq i - 1}} : B_{\uw_{\leq i-1}}^{\scrL} \to R_{w_{i-1}}$ is already defined.
If $\scrL w_{i-1} s_i = \scrL w_{i-1}$, we define $\varphi_{\uw_{\leq i}}$ by the composition 
\[\begin{tikzcd}
    \varphi_{\uw_{\leq i}} : B_{\uw_{\leq i-1}}^{\scrL} \otimes_R C_{s_i} \arrow[rr, "\varphi_{\uw_{\leq i-1}} \otimes a_{s_i}"] &  &  R_{w_{i-1}} \otimes_R R_{s_i} \arrow[r, "\sim"] &  R_{w_i} ,
    \end{tikzcd}\]
where $a_{s_i} : C_{s_i} = R \otimes_{R^{s_i}} R \to R_{s_i}$ is defined by $a_{s_i} (f \otimes g) = f s(g)$. 
If $\scrL w_{i-1} s_i \neq \scrL w_{i-1}$, we define  $\varphi_{\uw_{\leq i}}$ by the composition 
\[\begin{tikzcd}
    \varphi_{\uw_{\leq i}} : B_{\uw_{\leq i-1}}^{\scrL} \otimes_R R_{s_i} \arrow[rr, "\varphi_{\uw_{\leq i-1}} \otimes \id"] &  &  R_{w_{i-1}} \otimes_R R_{s_i} \arrow[r, "\sim"] &  R_{w_i} .
    \end{tikzcd}\]

\begin{proposition}\label{prop:basis_for_S_uw_x}
    Let $\ux$, $\uw$, and $w$ be as above.
    \begin{enumerate}
        \item The left $R$-module $B_{\ux}^{w, \scrL}$ has a basis $\{ b_{\ux, \ue}^w\}_{\ue \in \Subexp^{\scrL} (\ux, w)}$.
        \item The left $R$-module $B_{\ux}^{w, \scrL}$ is graded free with graded rank 
        \[\grk_R (B_{\ux}^{w, \scrL}) = \sum_{\ue \in \Subexp^{\scrL} (\ux, w)} v^{d(\ux, \ue) + \ell_{\scrL} (w)} = v^{\ell_{\scrL} (w) } p_{\ux}^{w, \scrL} (v). \]
        \item The set $\{ \varphi_{\uw} \circ \LL_{\ux,\ue}^{\scrL} \}_{\ue \in \Subexp^{\scrL} (\ux, w)}$ is a left $R$-basis of $\Hom_{\Cmon{\scrL}{\scrL w}} (B_{\ux}^{\scrL}, R_w)$.
    \end{enumerate}
\end{proposition}
\begin{proof}
    By Lemma \ref{lem:ll_are_lin_indep}, we have
    \[\varphi_{\uw} \left( \LL_{\ux, \ue}^{\scrL} (b_{\ux, \uf}) \right) = \begin{cases} 1 & \text{if } \uf = \ue, \\ 0 & \text{if } \uf < \ue. \end{cases}\]
    By Lemma \ref{lem:std_facts_about_inv_and_coinvs}, we get a homomorphism of $R$-bimodules $\psi_{\ue} : B_{\ux}^{w, \scrL} \to R_x$ corresponding to $\varphi_{\ux} \circ \LL_{\uw, \ue}^{\scrL}$.
    We then have
    \[\psi_{\ue} (b_{\uw, \uf}^x) = \begin{cases} 1 & \text{if } \uf = \ue, \\ 0 & \text{if } \uf < \ue. \end{cases}\]

    Via backwards induction on $\ue$, we construct a map $\psi_{\ue}' : B_{\ux}^{w, \scrL} \to R_x$ such that $\psi_{\ue}' \in \psi_{\ue} + \sum_{\ue' > \ue} R\psi_{\ue'}$ and $\psi_{\ue}' (b_{\ux,\uf}^x) = \delta_{\ue, \uf}$ for all $\uf$.
    If $\ue$ is the maximum subexpression, then $\psi_{\ue}' = \psi_{\ue}$.
    Assume that $\psi_{\ue'}'$ has been defined for all $\ue' > \ue$. We define 
    \[\psi_{\ue}' = \psi_{\ue} - \sum_{\ue' > \ue} \psi_{\ue} (b_{\ux, \ue'}^x) \psi_{\ue'}'.\]
    It can be easily checked that $\psi_{\ue}'$ satisfies the desired constraints.
    
    We can then construct a map $\psi : B_{\ux}^{w, \scrL} \to \bigoplus_{\ux^{\ue} = w} R b_{\ux, \ue}^w$ by $\psi (m) = \sum_{\ux^{\ue} = w} \psi_{\ue}' (m) b_{\ux,\ue}^w$.
    By construction, $\psi$ is a section of the inclusion map $\bigoplus_{\ux^{\ue} = w} R b_{\ux, \ue}^w \hookrightarrow B_{\ux}^{w, \scrL}$.
    Therefore, we can find some $N \in \Cmon{\scrL}{\scrL'}$ such that $B_{\ux}^{w, \scrL} = \bigoplus_{\ux^{\ue} = w} R b_{\ux, \ue}^w \oplus N$.
    We then have an isomorphism $B_{\ux}^{w, \scrL} \otimes_R Q = \bigoplus_{\ux^{\ue} = w} Q b_{\ux, \ue}^w \oplus (N \otimes_R Q)$.
    On the one hand, by Lemma \ref{lem:dim_of_Bwx_over_Q}, $\dim_Q (B_{\ux}^{w, \scrL} \otimes_R Q) = p_{\ux}^{w, \scrL} (1)$.
    On the other hand, by Lemma \ref{lem:defect_formula_for_pwx}, $\dim_Q \left( \bigoplus_{\ux^{\ue} = w} Q b_{\ux, \ue}^w \right) = p_{\ux}^{w, \scrL} (1)$.
    As a result, $N \otimes_R Q = 0$. However, since $N \subset B_{\ux}^{w, \scrL} \subset B_{\ux, Q}^{w, \scrL}$, we know that $N$ has no $Q$-torsion. Therefore, $N = 0$.

    Statement (2) follows immediately from Lemma \ref{lem:defect_formula_for_pwx}. 

    Finally, we will prove (3). By construction $\{\psi_{\ue}'\}$ is a basis of $\Hom_{\mod{R}} (B_{\ux}^{w, \scrL}, R)$ which is dual to $\{b_{\uw, \ue}\}$.
    As a result, $\{\psi_{\ue}\}$ is also a basis. Under the isomorphism from Lemma \ref{lem:std_facts_about_inv_and_coinvs} (2),
    \[\Hom_{\mod{R}} (B_{\ux}^{w, \scrL} , R) \cong \Hom_{\Cmon{\scrL}{\scrL w}} (B_{\ux}^{w, \scrL} , R_w) \cong \Hom_{\Cmon{\scrL}{\scrL w}} (B_{\ux}^{\scrL}, R_w ),\]
    the map $\psi_{\ue}$ corresponds to $\varphi_{\uw} \circ \LL_{\ux, \ue}^{\scrL}$.  
\end{proof}

\begin{corollary}[{\cite[Corollary 3.13]{Abe19}}]\label{cor:bott_samelson_graded_free}
    Let $\ux \in \Exp (W)$ and $w \in W$.
    The left $R$-module $B_{\ux, w}^{\scrL}$ is graded free and $\grk (B_{\ux, w}^{\scrL}) = v^{-\ell_{\scrL} (w)} p_{\ux}^{w, \scrL} (v^{-1})$.
\end{corollary}

We can regard $W$ as a topological space with the order topology for the Bruhat order. 
In particular, a subset $I \subseteq W$ is \emph{closed} if for all $x \in I$, $y \in W$, then $y \leq x$ implies that $y \in I$.

\begin{lemma}[{\cite[Lemma 3.16]{Abe19}}]
    Let $I \subseteq W$ be a closed finite subset and $w \in I$ a maximal element under the Bruhat order.
    Then there exists an enumeration $w_1, w_2, \ldots, w_{\# I}$ of the elements in $I$ such that $\{ w_1, \ldots, w_i\}$ is closed for all $i$ and $w_{\# I} = w$. 
\end{lemma}

Let $\ux \in \Exp (W)$, $w \in W$, and $\ue \in \Subexp^{\scrL} (\ux, w)$. Since $D (B_{\ux}^{\scrL}) = B_{\ux}^{\scrL}$ (Lemma \ref{lem:dual_and_invs}), we can define 
$\overline{\LL}_{\ux, \ue}^{\scrL} \coloneq D (\LL_{\ux,\ue}^{\scrL} ) : B_{\uw}^{\scrL} \to B_{\ux}^{\scrL}$ which has degree $d(\ux, \ue)$.

Let $I \subseteq W$ and $w \in I$.
Note that $B_{\ux, I \setminus \{w\}}^{\scrL}$ is kernel of the composition $B_{\ux, I}^{\scrL} \to B_{\ux}^{\scrL} \to B_{\ux}^{w, \scrL}$.
As a result, $\pi_{\ux}^w$ factors through a map $B_{\ux, I}^{\scrL} \to B_{\ux, I}^{\scrL}/B_{\ux, I\setminus\{w\}}^{\scrL}$ which we will also denote by $\pi_{\ux}^w$.

\begin{proposition}\label{prop:intermediate_bases}
    Let $I$ be a closed subset and $w$ a maximal element in $I$.
    Let $\uw = \rex (w)$ and $\ux \in \Exp (W)$.
    Then $\{\pi_{\ux}^w (\overline{\LL}_{\ux, \ue}^{\scrL}) (u_{\uw})\}_{\ue \in \Subexp^{\scrL} (\ux, w)}$ is a left $R$-module basis of $B_{\ux, I}^{\scrL}/B_{\ux,I\setminus\{w\}}^{\scrL}$.
\end{proposition}
\begin{proof}
    The proposition follows from the same argument give for \cite[Theorem 3.17]{Abe19} after the results on the polynomials $p_{\ux}^{w}$ are replaced by their monodromic counterparts (see Lemma \ref{lem:bott_samelson_bimodules_are_free} and Lemma \ref{lem:pwx_calculation}).
\end{proof}

\begin{corollary}[{\cite[Corollary 3.18]{Abe19}}]\label{cor:when_other_w_invs_is_surj}
    Let $M \in \Amon{\scrL}{\scrL'}^{\oplus} (\fr{h}, W, \fr{o})$.
    \begin{enumerate}
        \item Let $I \subset W$ be a closed subset with maximal element $w \in I$. Then there is an isomorphism $M_{\leq w}/M_{<w} \cong M_{I}/M_{I \setminus \{w\}}$.
        \item Let $\uw = \rex (w)$ for $w \in W$. Then the map $\Hom (B_{\uw}^{\scrL}, M) \to \Hom (B_{\uw}^{w, \scrL}, M_{\leq w} / M_{<w})$ is surjective.
    \end{enumerate}
\end{corollary}

\subsection{Soergel Hom Formula}

\begin{proposition}\label{prop:abe_soergel_hom_formula}
    Let $\ux, \uy \in \Exp (W)$ such that $\scrL \ux = \scrL' = \scrL \uy$.
    Then the graded left $R$-module $\Hom_{\Cmon{\scrL}{\scrL'}} (B_{\ux}^{\scrL}, B_{\uy}^{\scrL})$ is graded free of finite rank. Moreover, its graded rank can be computed by
    \[\grk_R \Hom (B_{\ux}^{\scrL}, B_{\uy}^{\scrL}) = \langle H_{\uw}^{\scrL}, H_{\uy}^{\scrL} \rangle.\]
\end{proposition}
\begin{proof}
    By Lemma \ref{lem:BS_tensor_adjunction} and (\ref{eq:mha_biadjoint}), we may assume that $\ux = \emptyset$ and $\scrL \uy = \scrL$.
    There is an isomorphism of left graded $R$-modules,
    \[\Hom_{\Cmon{\scrL}{\scrL}} (R_e, B_{\uy}^{\scrL}) \cong B_{\uy, e}^{\scrL}.\]
    This left $R$-module is graded free and of graded rank $p_{\uy}^{e, \scrL} (v^{-1})$ by Corollary \ref{cor:bott_samelson_graded_free}. 
    On the other hand, a straightforward computation (cf., the proof of \cite[Theorem 4.6]{Abe19}) using $\overline{\uH_{\uy}^{\scrL}} = \uH_{\uy}^{\scrL}$ shows that 
    $\langle \uH_{e}^{\scrL}, \uH_{\uy}^{\scrL} \rangle = p_{\uy}^{e, \scrL} (v^{-1})$.
\end{proof}

\begin{corollary}\label{cor:abe_soergel_hom_application}
    Let $\ux, \uy \in \Exp (W)$.
    Then 
    \[\grk_R \Hom (B_{\ux}^{\scrL}, B_{\uy}^{\scrL}) = \sum_{w \in W} p_{\ux}^{w, \scrL} (v^{-1}) p_{\uy}^{w, \scrL} (v^{-1}).\]
\end{corollary}

\begin{corollary}\label{cor:soergel_hom_defects}
    Let $\ux, \uy \in \Exp (W)$.
    Then 
    \[\grk_R \Hom (B_{\ux}^{\scrL}, B_{\uy}^{\scrL}) = \sum_{\ux^{\ue} = \uy^{\uf}} v^{d (\ux, \ue) + d (\uy, \uf)}.\]
\end{corollary}

\subsection{Double Leaves}\label{subsec:alg_dll}

\begin{definition}
    Let $\ux, \uy \in \Exp (W)$ such that $\scrL \ux = \scrL' = \scrL \uy$.
    Let $w \in \W{\scrL}{\scrL'}$, $\ue \in \Subexp^{\scrL} (\ux, w)$, and $\uf \in \Subexp^{\scrL} (\uy, w)$.
    A \emph{double leaf} is a degree $d (\ux, \ue) + d (\uy, \uf)$ map given by a composition of the form 
    \[ \dLL_{\ue,\uf}^{w, \scrL} \coloneq \overline{\LL}_{\uy, \uf}^{\scrL} \circ \LL_{\ux, \ue}^{\scrL} : B_{\ux}^{\scrL} \to B_{\uy}^{\scrL}.\]
\end{definition}

\begin{theorem}\label{thm:abe_dl_are_basis}
    Let $\ux, \uy \in \Exp (W)$ such that $\scrL \ux = \scrL' = \scrL \uy$.
    Let $\dLL_{\ux, \uy}^{\scrL}$ denote the set containing one double leaf $\dLL_{\ue,\uf}^{w, \scrL} : B_{\ux}^{\scrL} \to B_{\uy}^{\scrL}$ for each $w \in \W{\scrL}{\scrL'}$ and each pair of monodromic subsequences $(\ux, \ue)$ and $(\uy, \uf)$ expressing $w$.
    Then $\dLL_{\ux, \uy}^{\scrL}$ is a left graded $R$-module basis for $\Hom_{\Amon{\scrL}{\scrL'}^{\BS}} (B_{\ux}^{\scrL}, B_{\uy}^{\scrL})$.
\end{theorem}
\begin{proof}
    Let $w \in W$, $\ue \in \Subexp^{\scrL} (\ux, w)$, $\uf \in \Subexp^{\scrL} (\uy, w)$.
    The degree of $\dLL_{\ue,\uf}^{w, \scrL}$ is $d(\ux, \ue) + d(\uy, \uf)$.
    We can then compute using Corollary \ref{cor:abe_soergel_hom_application},
    \begin{align*}
        \sum_{\stackrel{w \in W}{\ux^{\ue} = w = \uy^{\uf}}} v^{-\deg (\dLL_{\ue, \uf}^{w, \scrL})} &=  \sum_{\stackrel{w \in W}{\ux^{\ue} = w}} v^{-d(\ux, \ue)}  \sum_{\stackrel{w \in W}{\uy^{\uf} = w}} v^{-d(\uy, \uf)} \\
        &= \sum_{w \in W} p_{\ux}^{w, \scrL} (v^{-1}) p_{\uy}^{w, \scrL} (v^{-1}) \\
        &= \grk_R \Hom_{\Amon{\scrL}{\scrL'}^{\BS}} (B_{\ux}^{\scrL}, B_{\uy}^{\scrL}).
    \end{align*}
    As a result, it suffices to prove that the $\dLL_{\ux, \uy}^{\scrL}$ is linearly independent over $R$.

    Suppose that $\sum c_{\ue, \uf}^w \dLL_{\ue, \uf}^{w, \scrL} = 0$.
    Define $I = \{ w \in W \mid \ux^{\ue} = \uy^{\uf}, c_{\ue, \uf}^w \neq 0\}$. We will argue by contradiction that $I$ must be empty.
    Let $x \in I$ be an element with maximal length in $W$.
    We can then consider the closure $\overline{I}$ of $I$.
    Note that $x$ also has maximal length in $\overline{I}$. 
    For all $\ue$, $\uf$, and $w \in W$ such that $\ux^{\ue} = w = \uy^{\uf}$, the image of $c_{\ue, \uf}^w \dLL_{\ue,\uf}^{w, \scrL}$ is supported on $\overline{I}$. 
    As a result, we have that $c_{\ue, \uf}^w \dLL_{\ue,\uf}^{w, \scrL}$ factors through $B_{\uy, \overline{I}}^{\scrL} \hookrightarrow B_{\uy}^{\scrL}$.
    Moreover, since $I \subset \overline{I}$, we have that $\pi_{\uy}^x \circ \left( c_{\ue, \uf}^w \dLL_{\ue,\uf}^{w, \scrL} \right) = 0$ unless $w=x$.

    Define $E = \{ \ue \mid c_{\ue, \uf}^x \neq 0 \}$. Let $\ue' \in E$ denote the minimal element under the label-lexicographic order.
    By Lemma \ref{lem:ll_are_lin_indep}, for all $\ue \in E$, we have that $\LL_{\ux, \ue}^{\scrL} (b_{\ux, \ue'}) = 0$ unless $\ue = \ue'$.
    Therefore, by Lemma \ref{lem:ll_are_lin_indep} again, we have 
    \[ \sum_{\uy^{\uf} = x} c_{\ue', \uf}^x \pi_{\uy}^x (\overline{\LL}_{\uy, \uf}^{\scrL} (u_{\uw}) ) = \sum_{\uy^{\uf} = x} c_{\ue', \uf}^x \pi_{\uy}^x (\dLL_{\ue', \uf}^{x, \scrL} (b_{\ux, \ue'}) )  =0.\]
    By Proposition \ref{prop:basis_for_S_uw_x}, we then must have that $c_{\ue', \uf}^x = 0$ for all $\uf$ such that $\uy^{\uf} = x$. 
    This contradicts our assumption that $I$ was not empty. Therefore, $\dLL_{\ux, \uy}^{\scrL}$ is linearly independent.
\end{proof}

Let $\uw = (s_1, \ldots, s_k) \in \Exp (W)$ and take $\beta \in \uW{\scrL}{\scrL \uw}$ to be the block containing the evaluation of $\uw$.  
We will define a morphism of degree $\ell_{\scrL} (\uw)$
\[\epsilon^{\uw} : B_{\uw}^{\scrL} \to R_{w^{\beta}} \]
inductively on $\ell (\uw)$. When $\uw = \emptyset$, we set $\epsilon^{\emptyset} : R \to R$ to be the identity map.
Suppose $\epsilon^{\uw_{\leq k-1}}$ is defined. If $\scrL \uw = \scrL \uw_{\leq k-1}$, we define $\epsilon^{\uw}$ as the composition
\[B_{\uw}^{\scrL} = B_{\uw_{\leq k-1}}^{\scrL} \otimes_R C_{s_k} \stackrel{\epsilon^{\uw_{\leq k-1}} \otimes \epsilon^{s_k}}{\longrightarrow} R_{w^{\beta}} \otimes_R R \cong R_{w^{\beta}}. \]
If $\scrL \uw \neq \scrL \uw_{\leq k-1}$, we define $\epsilon^{\uw}$ as the composition
\[B_{\uw}^{\scrL} = B_{\uw_{\leq k-1}}^{\scrL} \otimes_R R_{s_k} \stackrel{\epsilon^{\uw_{\leq k-1}} \otimes \id}{\longrightarrow} R_{w^{\beta} s_k} \otimes_R R_{s_k} \cong R_{w^{\beta}}. \]
Note that from the definition of $\epsilon^{\uw}$ we have that $\epsilon^{\uw} (u_{\uw}) = 1$. 

We can use Theorem \ref{thm:abe_dl_are_basis} to get a few uniqueness criterions for $\beta_{s,t}$.

\begin{lemma}\label{lem:uniqueness_of_beta}
    Let $s,t \in S$ be distinct simple reflections with $m = m_{s,t} < \infty$. The map $\beta_{s,t} : B_{{}_s \underline{m}}^{\scrL} \to B_{{}_t \underline{m}}^{\scrL}$ is the unique map satisfying either of the following criterion:
    \begin{enumerate}
        \item $\beta_{s,t} (u_{{}_s \underline{m}}) = u_{{}_t \underline{m}}$;
        \item $\k \otimes_{R} (\epsilon^{{}_t \underline{m}} \circ \beta_{s,t}) = \k \otimes_{R} \epsilon^{{}_s \underline{m}} : \k \otimes_{R} B_{{}_s \underline{m}}^{\scrL} \to \k$.
    \end{enumerate}
\end{lemma}
\begin{proof}
    (1): It is clear from its defining condition that $\beta_{s,t}$ takes the 1-tensor to the 1-tensor. Moreover, we can compute from Theorem \ref{thm:abe_dl_are_basis} that the space of degree 0 morphisms $B_{{}_s \underline{m}}^{\scrL} \to B_{{}_t \underline{m}}^{\scrL}$ is a rank 1 $\k$-module generated by $\beta_{s,t}$.
    This then uniquely determines $\beta_{s,t}$.

    (2): Using the same argument as in (1), the uniqueness claim will be clear after proving that $\k \otimes_{R} (\epsilon^{{}_t \underline{m}} \circ \beta_{s,t}) = \k \otimes_{R} \epsilon^{{}_s \underline{m}}$.
    By Theorem \ref{thm:abe_dl_are_basis}, we can write
    \[\epsilon^{{}_t \underline{m}} \circ \beta_{s,t} = \sum_{\ue \in \Exp^{\scrL} ({}_s \underline{m}), e} c_{\ue} \LL_{{}_s \underline{m}, \ue}^{\scrL}\]
    for some $c_{\ue} \in R$. Let $\uf = (f_1, \ldots, f_m)$ denote the maximally 0 monodromic subexpression, i.e., $f_i = 0$ for $i \in K ({}_s \underline{m}, \scrL)$. 
    Now the degree of $\LL_{{}_s \underline{m}, \ue}^{\scrL}$ is $d ({}_s \underline{m}, \ue)$ and the degree of $\epsilon^{{}_t \underline{m}} \circ \beta_{s,t}$ is $m$.
    In particular, we have that $\deg \left( \LL_{{}_s \underline{m}, \ue}^{\scrL} \right) \leq m$ with equality if and only if $\ue = \uf$.
    As a consequence the constant $c_{\ue}$ has no component in degree $0$ unless $\ue = \uf$. Moreover, $\LL_{{}_s \underline{m}, \uf}^{\scrL} = \epsilon^{{}_s \underline{m}}$.
    These facts imply that 
    \[ \k \otimes_{R} (\epsilon^{{}_t \underline{m}} \circ \beta_{s,t}) = \k \otimes_{R} c_{\uf} \epsilon^{{}_s \underline{m}}.\]
    It remains to check that $c_{\uf} = 1$. To check this, we evaluate both sides at the 1-tensor,
    \[(\epsilon^{{}_t \underline{m}} \circ \beta_{s,t}) (u_{{}_s \underline{m}}) = \epsilon^{{}_t \underline{m}} (u_{{}_t \underline{m}}) = 1 =  \epsilon^{{}_s \underline{m}} (u_{{}_s \underline{m}}). \]
    Therefore, $c_{\uf} = 1$.
\end{proof}

\begin{proposition}\label{prop:duality_and_str_mors}
     Let $s \in S$. Recall that $D(M) = M$ for all $M \in \Amon{\scrL}{\scrL'}^{\BS} (\fr{h}, W, \fr{o})$.
     \begin{enumerate}
         \item If $\scrL s = \scrL$, then we have
         \[D (\nu^s) = \mu^s, \qquad D(\epsilon^s) = \eta^s, \qquad D(\cup^s) = \cap^s,\]
         \[D (\mu^s) = \nu^s, \qquad D(\eta^s) = \epsilon^s, \qquad D(\cap^s) = \cup^s.\]
         \item If $\scrL s \neq \scrL$, then we have
         \[D(\cup^s) = \cap^s \qquad \text{and} \qquad D(\cap^s) = \cup^s.\]
         \item If $t \in S$, then we have
         \[D (\beta_{s,t}) = \beta_{t,s}.\]
     \end{enumerate}
\end{proposition}
\begin{proof}
    One can check (1) from definitions using the proof of \cite[Lemma 2.20]{Abe19}. Likewise, (2) can be similarly computed from definitions.
    We will now prove (3). Let $m = m_{s,t}$. By Lemma \ref{lem:uniqueness_of_beta} that $\beta_{t,s} : B_{{}_t \underline{m}} \to B_{{}_s \underline{m}}$ is the unique degree 0 morphism taking the 1-tensor to the 1-tensor.
    One can then check from definitions that $D (\beta_{s,t})$ also takes the 1-tensor to the 1-tensor. As a result, we may conclude that $\beta_{t,s} = D(\beta_{s,t})$.
\end{proof}
\begin{remark}\label{rem:upside_down_LL}
    From Proposition \ref{prop:duality_and_str_mors}, we can see that every upside down light leaf $\overline{\LL}_{\ux, \ue}^{\scrL}$ is generated under tensors and compositions by $\id_{B_s^{\scrL}}, \eta^s, \nu^s, \cap^s$, or the rex moves $\beta_{s,t}$.
\end{remark}
\subsection{Localization and Cellular Structure}

In this section, we will discuss some techniques for studying Bott--Samelson bimodules via localization.
We will then use localization to further study light leaves and double leaves. Our main result will be a generalization of Theorem \ref{thm:abe_dl_are_basis} for the ``ideal of lower terms'' which gives a cellular structure on the category of Bott--Samelson bimodules.

\subsubsection{The Localization Functor}

For each $\scrL, \scrL' \in \fr{o}$, we define a 1-category $\CmonQ{\scrL}{\scrL'} (\fr{h}, W, \fr{o})$ whose objects consist of tuples $(P^w)_{w \in W}$ such that
\begin{enumerate}
    \item $P^w$ is a graded $Q$-bimodule;
    \item $m \cdot q = w (q) \cdot m$ for any $m \in P^w$ and $q \in Q$;
    \item $P^w = 0$ if $w \notin \W{\scrL}{\scrL'}$ and is only nonzero for finitely many $w \in \W{\scrL}{\scrL'}$.
\end{enumerate}
A morphism in $\CmonQ{\scrL}{\scrL'} (\fr{h}, W, \fr{o})$ from $(P_1^w)_{w \in W}$ to $(P_2^w)_{w \in W}$ consists of a tuple $(\varphi_w)_{w \in W}$ where each $\varphi_w : P_1^w \to P_2^w$ is a $Q$-bimodule homomorphism.

There are bifunctors
\[(-) \otimes_Q (-) : \CmonQ{\scrL}{\scrL'} (\fr{h}, W, \fr{o}) \times \CmonQ{\scrL'}{\scrL''} (\fr{h}, W, \fr{o}) \to \CmonQ{\scrL}{\scrL''} (\fr{h}, W, \fr{o})\]
induced from the tensor product of graded $Q$-bimodules. More precisely, $(P_1 \otimes_Q P_2)^w = \bigoplus_{xy = w} P_1^x \otimes_Q P_2^y$.
These bifunctors are suitably associative and there is a unit object $Q_e$ given by $(Q_e)^e = Q$ and $(Q_e)^w = 0$ if $w \neq e$.
The tensor product allows us to assemble a 2-category $\CmonQ{}{} (\fr{h}, W, \fr{o})$ with object set $\fr{o}$ and whose morphism categories are given by $\CmonQ{\scrL}{\scrL'} (\fr{h}, W, \fr{o})$.
We sometimes omit the tuple $(\fr{h}, W, \fr{o})$ from the notation for $\CmonQ{\scrL}{\scrL'} (\fr{h}, W, \fr{o})$ and $\CmonQ{}{} (\fr{h}, W, \fr{o})$.

The assignment $(M, (M_Q^w)_{w \in W}, \xi_M) \mapsto (M_Q^w)_{w \in W}$ defines a 2-functor 
\[(-)_Q : \Cmon{}{} (\fr{h}, W, \fr{o}) \to \CmonQ{}{} (\fr{h}, W, \fr{o})\]
called the \emph{localization functor}.
Since $M \to M \otimes_R Q$ is injective ($M$ is free over $R$), the localization functor is faithful on morphism categories.

\subsubsection{Localizing Bott--Samelson Bimodules}

For $w \in W$, we define $Q_w \in \CmonQ{\scrL}{\scrL w} (\fr{h}, W, \fr{o})$ by 
\begin{itemize}
    \item $(Q_w)^w = Q$ as a left $Q$-module and the right action of $q \in Q$ is given by $m \cdot q = w(q) \cdot m$;
    \item $(Q_w)^x = 0$ if $x \neq w$.
\end{itemize}
Any object in $\CmonQ{\scrL}{\scrL'} (\fr{h}, W, \fr{o})$ is isomorphic to a direct sum of $Q_w (n)$'s where $w \in \W{\scrL}{\scrL'}$ and $n \in \Z$. Moreover, there is a canonical isomorphism $Q_x \otimes_Q Q_y \cong Q_{xy}$ via $f \otimes g \mapsto f x(g)$ characterized as the unique $Q$-bimodule morphism taking $1 \otimes 1$ to $1$.
 
By (\ref{eq:localization_of_C_s}), we have isomorphisms of graded $Q$-bimodules
\begin{equation}\label{eq:localization_of_Bs}
    B_{s, Q}^{\scrL} \cong \begin{cases} Q_e \oplus Q_s (1) & \scrL s = \scrL, \\ Q_s & \scrL s \neq \scrL.\end{cases}
\end{equation}
Let $\uw = (s_1, \ldots, s_k) \in \Exp (W)$ and $\ue = (e_1, \ldots, e_k) \in \Subexp^{\scrL} (\uw)$. Define the \emph{relative length} of $\uw$ with respect to $\ue$ by
\[\ell_{\scrL} (\uw, \ue) = \# \{i \mid e_i = 1\} + \ell_{\scrL} (\uw) - \ell (\uw).\] 
We can define the \emph{$\ue$-component} of $B_{\uw, Q}^{\scrL}$, denoted $B_{\uw, Q}^{\ue, \scrL}$, as the direct summand corresponding to the bimodule 
$Q_{s_1^{e_1}} \otimes_Q Q_{s_2^{e_2}} \otimes_Q \ldots \otimes_Q Q_{s_k^{e_k}} (\ell_{\scrL} (\uw, \ue))\cong Q_{\uw^{\ue}} (\ell_{\scrL} (\uw, \ue))$ under the isomorphisms afforded by (\ref{eq:localization_of_Bs}).
The following lemma then follows via induction on the length of an expression along with (\ref{eq:localization_of_Bs}).

\begin{lemma}\label{lem:localization_of_BS_bimods}
    Let $\uw \in \Exp (W)$. There is a decomposition of graded $Q$-bimodules,
    \[B_{\uw, Q}^{\scrL} \cong \bigoplus_{\ue \in \Subexp^{\scrL} (\uw)} B_{\uw, Q}^{\ue, \scrL} \cong \bigoplus_{\ue \in \Subexp^{\scrL} (\uw)} Q_{\uw^{\ue}} (\ell_{\scrL} (\uw, \ue)).\]
\end{lemma}

Let $\varphi : B_{\ux}^{\scrL} \to B_{\uy}^{\scrL}$ be a morphism of Bott--Samelson bimodules.
For $\scrL$-monodromic subexpressions $\ue$ of $\ux$ and $\uf$ of $\uy$, we can consider the map 
\[(\varphi_Q)_{\ue}^{\uf} : B_{\ux, Q}^{\ue, \scrL} \to B_{\uy, Q}^{\uf, \scrL}\]
given by the composition
\[B_{\ux, Q}^{\ue, \scrL} \hookrightarrow B_{\ux, Q}^{\scrL} \stackrel{\varphi_Q}{\to} B_{\uy, Q}^{\scrL} \twoheadrightarrow B_{\uy, Q}^{\uf, \scrL}.\]
Since both the source and target of this map are 1-dimensional $Q$-vector spaces, it is determined by a scalar $c_{\ue}^{\uf} (\varphi) \in Q$.
Note that if $\ux^{\ue} \neq \uy^{\uf}$, then $c_{\ue}^{\uf} (\varphi) = 0$. We call these scalars the \emph{localization coefficients} of $\varphi$.
Since $(-)_Q$ is faithful, it follows that $\varphi$ is uniquely determined by its matrix of localization coefficients $\left( c_{\ue}^{\uf} (\varphi)\right)_{\ue, \uf}$.

\begin{example}\label{ex:localization_of_one_color_morphisms}
    We write the matrix of localization coefficients for various morphisms of Bott--Samelson bimodules (cf., \cite{EW17}). 
    The top left corner indicates the morphism, the column headers indicate the subexpressions of the domain, and the row headers indicate the subexpressions of the target.

    First, consider the case where $\scrL s = \scrL$.

    \begin{center}
    \begin{tabular}[t]{c | c  c}
        $\epsilon^s$ & 0 & 1 \\
        \hline
        $\emptyset$ & $\alpha_s$ & 0 
    \end{tabular}\qquad\qquad
    \begin{tabular}[t]{c | c c c c}
        $\mu^s$ & 00 & 01 & 10 & 11 \\
        \hline
        0 & 1 & 0 & 0 & 1 \\
        1 & 0 & 1 & 1 & 0
    \end{tabular}\qquad\qquad
    \begin{tabular}[t]{c | c c c c}
        $\cap^s$ & 00 & 01 & 10 & 11 \\
        \hline
        $\emptyset$ & $\alpha_s$ & 0 & 0 & $\alpha_s$ 
    \end{tabular}
\end{center}

\begin{center}
    \begin{tabular}[t]{c | c }
        $\eta^s$ & $\emptyset$ \\
        \hline
        0 & 1 \\
        1 & 0 
    \end{tabular}\qquad\qquad
    \begin{tabular}[t]{c | c c }
        $\nu^s$ & 0 & 1 \\
        \hline
        00 & $1/\alpha_s$ & 0 \\
        01 & 0 & $1/\alpha_s$ \\
        10 & 0 & $-1/\alpha_s$ \\
        11 & $-1/\alpha_s$ & 0 
    \end{tabular}\qquad\qquad
    \begin{tabular}[t]{c | c }
        $\cup^s$ & $\emptyset$ \\
        \hline
        00 & $1/\alpha_s$  \\
        01 & 0  \\
        10 & 0  \\
        11 & $-1/\alpha_s$ 
    \end{tabular}
\end{center}

    Next, consider the case where $\scrL s \neq \scrL$.

    \begin{center}
    \begin{tabular}{c | c }
        $\cap^s$ & 11 \\
        \hline
        $\emptyset$ & 1
    \end{tabular}\qquad\qquad
    \begin{tabular}{c | c }
        $\cup^s$ & $\emptyset$ \\
        \hline
        11 & 1
    \end{tabular}
\end{center}
\end{example}

\begin{remark}\label{rem:loc_coeffs_for_beta}
    The localization coefficients for $\beta_{s,t} : B_{{}_s \underline{m}}^{\scrL} \to B_{{}_t \underline{m}}^{\scrL}$ are known for being rather difficult to compute.
    In the non-monodromic case a formula is given in \cite{Abe21} and \cite{EW17}.
    By applying the monodromic-endoscopic equivalence, one can deduce a similar formula in the monodromic case.
    We only need the following corollary of this formula \cite[(2.31)]{EW17}: if $\ue$ and $\uf$ are the all 1's subexpressions, then $c_{\ue}^{\uf} (\beta_{s,t}) = 1$.
\end{remark}

\subsubsection{Localizing Light Leaves}

We will now study the localization coefficients for light leaves and double leaves.
Let $\LL_{\ux, \ue}^{\scrL} : B_{\ux}^{\scrL} \to B_{\uw}^{\scrL}$ be a light leaf. 
We write $c_{\ue, \ue'}$ for the localization coefficient $c_{\ue'}^{\uf'} (\LL_{\ux, \ue}^{\scrL})$ where $\uf'$ is the all 1's subexpression of $\uw$.
The following is a monodromic variation of what is called ``path dominance upper triangularity'' in \cite{EW}.

\begin{proposition}\label{prop:path_dominance_upper_triangularity}
    Let $\ux \in \Exp (W)$ and $\ue, \ue' \in \Subexp^{\scrL} (\ux)$.
    If $c_{\ue'}^{\ue} \neq 0$, then $\ue' \prec \ue$. Moreover, $c_{\ue}^{\ue}$ is nonzero. 
\end{proposition}

The argument for Proposition \ref{prop:path_dominance_upper_triangularity} will follow the proof of \cite[Proposition 6.6]{EW} closely.
Before we can prove the proposition, we need the following lemma.
\begin{lemma}\label{lem:loc_of_shrinking_expression_moves}
    Let $\phi : B_{\ux}^{\scrL} \to B_{\uy}^{\scrL}$ be a morphism of Bott--Samelson bimodules consisting solely of compositions and tensors of morphisms of the form $\id_{B_s^{\scrL}}$, $\cap^s$, and $\beta_{s,t}$.
    Let $\ue$ and $\uf$ denote the all 1's subexpressions of $\ux$ and $\uy$ respectively.
    Then $c_{\ue}^{\uf} (\phi) \neq 0$. 
\end{lemma}
\begin{proof}
    Write $\phi = \phi_1 \circ \phi_{2} \circ \ldots \circ \phi_n$ such that $\phi_1$ is an identity morphism and each $\phi_i$ for $i > 1$ is either a morphism of the form $\cap^s$ or $\beta_{s,t}$ with identity morphisms tensored on either side.
    We will argue by induction on $n$. The claim is obvious when $n = 1$.
    Write $\phi' = \phi_1 \circ \ldots \circ \phi_{n-1} : B_{\uw}^{\scrL} \to B_{\uy}^{\scrL}$ and $\phi_n : B_{\ux}^{\scrL} \to B_{\uw}^{\scrL}$.
    Let $\ue'$ be the all 1's subexpression for $\uw$. 
    
    We claim that the following diagram commutes:
    \begin{equation}\label{eq:loc_of_shrinking_moves_1}
        \begin{tikzcd}
            {B_{\ux, Q}^{\ue, \scrL}} \arrow[r, hook] \arrow[d, hook] & {B_{\ux, Q}^{\scrL}} \arrow[r, "(\phi_n)_Q"]  & {B_{\uw, Q}^{\scrL}}                       \\
            {B_{\ux, Q}^{\scrL}} \arrow[r, "(\phi_n)_Q"]                  & {B_{\uw, Q}^{\scrL}} \arrow[r, two heads] & {B_{\uw, Q}^{\ue', \scrL}} \arrow[u, hook]
            \end{tikzcd}
    \end{equation}
    We will check (\ref{eq:loc_of_shrinking_moves_1}) case by case.
    The claim is equivalent to showing that $(\phi_n)_Q \left( B_{\ux, Q}^{\ue, \scrL} \right) \subseteq B_{\uw, Q}^{\ue', \scrL}$. By 2-functorality of localization, we can check this when $\phi_n$ is of the form $\id_{B_s^{\scrL}}$, $\cap^s$, and $\beta_{s,t}$.
    This is obvious when $\phi_n = \id_{B_s^{\scrL}}$ or $\phi_n = \cap^s$ (with either $\scrL s = \scrL$ or $\scrL s \neq \scrL$). 
    When $\phi_n = \beta_{s,t}$, the only subexpression of ${}_t \underline{m}$ whose evaluation is $w_{s,t}$ is the all 1's subexpression (since ${}_t \underline{m}$ is reduced).
    As a result, there are no nonzero maps between $B_{{}_s \underline{m}, Q}^{\ue, \scrL}$ and $B_{{}_t \underline{m}, Q}^{\uf', \scrL}$ unless $\uf'$ is the all 1's subexpression.
    Therefore, (\ref{eq:loc_of_shrinking_moves_1}) commutes. 

    By induction, it then suffices to show that $c_{\ue}^{\ue'} (\phi_n) \neq 0$. This can be checked case by case using the tables in Example \ref{ex:localization_of_one_color_morphisms} and the comments in Remark \ref{rem:loc_coeffs_for_beta}.
\end{proof}

\begin{midsecproof}{Proposition \ref{prop:path_dominance_upper_triangularity}}
    We will first show upper-triangularity.
    Write $w_k$ for the element expressing $(\uw_{\leq k}, \ue_{\leq k})$ and $v_k$ for the element expressing $(\uw_{\leq k}, \ue_{\leq k}')$.
    Write $\LL_k$ for the light leaf constructed after $k$-steps.
    Since the target of $\LL_k$ is a reduced expression $\uw_k$ for $w_k$, we have that $c_{\ue_{\leq k}}^{\ue_{\leq k}'} (\LL_k) = 0$ unless $v_k \leq w_k$.
    Therefore, $c_{\ue'}^{\ue} = 0$ unless $v_k \leq w_k$ for every $k$. This is the precise condition for $\ue' \prec \ue$.

    It remains to show that $c_{\ue}^{\ue} \neq 0$. We first claim that $c_{\ue}^{\ue} = q \cdot c_{\uf'}^{\uf} (\phi)$ where $q \in Q^{\times}$ and $\phi$ is some morphism of Bott--Samelson bimodules consisting solely of compositions and tensors of morphisms of the form $\id_{B_s^{\scrL}}$, $\cap^s$, and $\beta_{s,t}$, and $\uf, \uf'$ are all 1's subexpressions.
    Let $\ux = (s_1, \ldots, s_k)$. We will argue inductively on $k$.
    Suppose $e_k$ is decorated by $U0$, then by Example \ref{ex:localization_of_one_color_morphisms}, we have that
    \[c_{\ue}^{\ue} = \alpha_{s_k} c_{\ue_{\leq k -1}}^{\ue_{\leq k-1}}.\]
    If $e_k$ is decorated by $D0$, then one can check using Example \ref{ex:localization_of_one_color_morphisms} and from definitions that
     \[c_{\ue}^{\ue} = c_{\ue_{\leq k}}^{(1, \ldots, 1)} \left( \beta_{\uw_{k-1, s_k}, \uw_k} \circ \beta_{\ux_{\leq k-1}, \uw_{\leq k-1, s_k}} \circ \LL_{k-1}\right).\]
    Note that the right-hand side of the above equation is still $c_{\ue_{\leq k -1}}^{\ue_{\leq k-1}}$ but for a different choice of $\LL$-datum in defining the light leaf.
    As a result, by induction we can reduce to checking that $c_{\ue}^{\ue} \neq 0$ when $\ue$ is the all 1's subexpression.
    The proposition then follows from Lemma \ref{lem:loc_of_shrinking_expression_moves}.
\end{midsecproof}

Let $\dLL_{\ue, \uf}^{\scrL} : B_{\ux}^{\scrL} \to B_{\uy}^{\scrL}$ be a double leaf. We write $c_{\ue, \ue'}^{\uf, \uf'}$ for the localization coefficient $c_{\ue'}^{\uf'} (\dLL_{\ue, \uf}^{\scrL})$.

\begin{corollary}\label{cor:Bruhat_upper_triangularity}
    Let $\ux, \uy \in \Exp (W)$, $\ue, \ue' \in \Subexp^{\scrL} (\ux)$, and $\uf, \uf' \in \Subexp^{\scrL} (\uy)$.
    \begin{enumerate}
        \item $c_{\ue, \ue'}^{\uf, \uf'}= 0$ unless both $\ue' \prec \ue$ and $\uf' \prec \uf$. In particular, if $c_{\ue, \ue'}^{\uf, \uf'} \neq 0$, then $\ux^{\ue'} \leq \ux^{\ue}$. 
        \item $c_{\ue, \ue'}^{\uf, \uf'} \neq 0$. 
    \end{enumerate}
\end{corollary}

Let $I \subseteq W$ be a closed subset, and $M, N \in \Amon{\scrL}{\scrL'}^{\oplus} (\fr{h}, W, \fr{o})$.
We define a subset $\Hom_I (M,N) \subseteq \Hom_{\Amon{\scrL}{\scrL'}^{\oplus}} (M,N)$ consisting of maps $f : M \to N$ such that $f (M_Q^w) = 0$ whenever $w \notin I$.

The following lemma readily follows from Lemma \ref{lem:localization_of_BS_bimods}.
\begin{lemma}\label{lem:referaming_localization_support}
    Let $I \subseteq W$ be a closed subset.
    A map $f : B_{\ux}^{\scrL} \to B_{\uy}^{\scrL} (d)$ is in $\Hom_I ( B_{\ux}^{\scrL}, B_{\uy}^{\scrL})$ if and only if $c_{\ue}^{\uf} (f) = 0$ for all $\ue \in \Subexp^{\scrL} (\ux)$, $\uf \in \Subexp^{\scrL} (\uy)$ such that $\ux^{\ue}, \uy^{\uf} \notin I$.
\end{lemma}

Regard $\W{\scrL}{\scrL'}$ as a topological space under the order topology for the restricted Bruhat order.

\begin{proposition}\label{prop:ideal_Hom_spaces}
    Let $I \subseteq \W{\scrL}{\scrL'}$ be a closed subset.
    Let $\dLL_{\ux, \uy}^{I, \scrL}$ denote the set containing one double leaf $\dLL_{\ue, \uf}^{w, \scrL}$ for each $w \in I$ and each pair of monodromic subsequences $(\ux, \ue)$ and $(\uy, \uf)$ expressing $w$.
    Then $\dLL_{\ux, \uy}^{I, \scrL}$ is a left graded $R$-module basis for $\Hom_I (B_{\ux}^{\scrL}, B_{\uy}^{\scrL})$.
\end{proposition}
\begin{proof}
    The fact that $\dLL_{\ue, \uf}^{w, \scrL} \in \Hom_I (B_{\ux}^{\scrL}, B_{\uy}^{\scrL})$ follows from Corollary \ref{cor:Bruhat_upper_triangularity}.
    Moreover, $\dLL_{\ux, \uy}^{I, \scrL}$ is linearly independent by Theorem \ref{thm:abe_dl_are_basis}.
    It remains to show that any morphism $f \in \Hom_I (B_{\ux}^{\scrL}, B_{\uy}^{\scrL})$ is spanned by double leaves in $\dLL_{\ux, \uy}^{I, \scrL}$.

    Write $f$ as a linear combination of double leaves. Unless $f$ is zero, some $\dLL_{\ue, \uf}^{w, \scrL}$ has a nonzero coefficient.
    Choose $w, \ue,$ and $\uf$ successively such that each is maximal in the Bruhat/path dominance order relative to the constraint that, with the previous choices, there is a nonzero coefficient $\gamma$ for $\dLL_{\ue, \uf}^{w, \scrL}$.
    By Corollary \ref{cor:Bruhat_upper_triangularity}, $c_{\ue, \uf} (f) = \gamma c_{\ue, \ue}^{\ue, \ue} \neq 0$.
    We can then conclude by Lemma \ref{lem:referaming_localization_support} that $w \in I$.
    Since this is true for each maximal choice of $w$, we must have that $f$ is in the $R$-span of $\dLL_{\ux, \uy}^{I, \scrL}$.
\end{proof}

\begin{remark}
    Proposition \ref{prop:ideal_Hom_spaces} essentially states that $\Amon{\scrL}{\scrL'}^{\BS} (\fr{h}, W, \fr{o})$ is an object-adapted cellular category (see Proposition \ref{prop:dmhc_is_cellular}).
\end{remark}

    \section{Monodromic Hecke Category}\label{sec:diag}
    Let $(W,S)$ be a Coxeter system and $\fr{o}$ a set of monodromy parameters for $W$.
We will fix a reflection-balanced reflection-stable Abe realization $\fr{h}$ of $W$.
By Corollary \ref{cor:refl_abe_are_endo} and Proposition \ref{prop:balanced_realizations_and_reflections}, we have that $\fr{h}$ is a balanced Abe realization of $W_{\scrL}^{\circ}$ for all $\scrL \in \fr{o}$ as well.

We will define numerous 2-categories in this section. 
To keep notation consistent, unless otherwise stated, $\circ$ will refer to the vertical composition and $\star$ will refer to the horizontal composition in a 2-category. 

\subsection{Face Colorings of Elias--Williamson Graphs}

\begin{definition}
    An \emph{Elias--Williamson graph} is a planar graph $\Gamma$ embedded on the planar strip $\R \times [0,1]$ with the following data:
    \begin{enumerate}
        \item the edges of $\Gamma$ are colored by elements of $S$;
        \item the faces of $\Gamma$ are decorated by boxes labelled by homogenous $f \in R$;
        \item the vertices of $\Gamma$ are of 3 types:
        \begin{enumerate}
            \item univalent vertices,
            \item trivalent vertices, where all three adjoining edges have the same color,
            \item $2m_{s,t}$-valent vertices, where the adjoining edges alternate in color between two elements $s \neq t \in S$.
        \end{enumerate}
    \end{enumerate}
\end{definition}

Given an Elias--Williamson graph, we define its \emph{degree} as the sum of the degrees of each vertex, and the degrees of each polynomial. The degrees on vertices and boxes are given as follows:
\begin{equation}
    
  % \tikzsetnextfilename{#1}
  \tikzstyle{every picture}=[tikzfig]
  \input{./figures/deg_of_gens.tikz}

\end{equation}

\begin{definition}
    A \emph{monodromic Elias--Williamson graph} is an Elias--Williamson graph with an additional coloring of the faces by elements of $\fr{o}$ such that if two faces colored $\scrL \in \fr{o}$ and $\scrL' \in \fr{o}$ are adjacent along an edge colored by $s \in S$, then $\scrL' = \scrL s$.
\end{definition}

We can define the \emph{degree} of a monodromic Elias--Williamson graph as the degree of its underlying Elias--Williamson graph.

\begin{example}
    Let $W = S_3 \times S_2$ with simple reflections $S = \{ \textcolor{red}{s}, \textcolor{blue}{t}, \textcolor{green}{u} \}$ where $(\textcolor{red}{s} \textcolor{blue}{t})^3 = 1$.
    Define a right $W$-set $\fr{o} = \{\textcolor{orange}{1},\textcolor{violet}{2},\textcolor{yellow}{3}\}$ by 
    \[\textcolor{orange}{1} \cdot \textcolor{red}{s} = \textcolor{orange}{1}, \qquad \textcolor{violet}{2} \cdot \textcolor{red}{s} = \textcolor{yellow}{3}, \qquad \textcolor{orange}{1} \cdot \textcolor{blue}{t} = \textcolor{violet}{2}, \qquad \textcolor{yellow}{3} \cdot \textcolor{blue}{t} = \textcolor{yellow}{3}, \qquad \textcolor{orange}{1} \cdot \textcolor{green}{u} = \textcolor{orange}{1}, \qquad \textcolor{violet}{2} \cdot \textcolor{green}{u} = \textcolor{violet}{2}, \qquad \textcolor{yellow}{3} \cdot \textcolor{green}{u} = \textcolor{yellow}{3}.\] 
    An example of a monodromic Elias--Williamson graph for $(W, \fr{o})$ is given below.
    \[
  % \tikzsetnextfilename{#1}
  \tikzstyle{every picture}=[tikzfig]
  \input{./figures/ew_graph_ex.tikz}
\]
\end{example}

\begin{convention}
    The face coloring of a monodromic Elias--Williamson graph is uniquely determined by the coloring of any face in the underlying Elias--Williamson graph.
    As a result, we will not always color every face of a monodromic Elias--Williamson graph. If the left most color is clear from context, we often will only draw the underlying Elias--Williamson graph.
    At times, we will color only key parts of a monodromic Elias--Williamson graph to highlight important behavior. 
\end{convention}

We denote the set of isotopy classes of Elias--Williamson graphs by $\SGraph$ and the set of isotopy classes of monodromic Elias--Williamson graphs by $\VSGraph$.
If $\scrL, \scrL' \in \fr{o}$, we will write ${}_{\scrL} \VSGraph_{\scrL'}$ for the subset of $\VSGraph$ consisting of graphs whose most left face is colored by $\scrL$ and whose right most face is colored by $\scrL'$.

Let $\ux, \uy \in \Exp (W)$.
We can consider the subset $\SGraph (\ux, \uy)$ of $\SGraph$ consisting of Elias--Williamson graphs with bottom boundary $\ux$ and top boundary $\uy$.
We similarly define a subset $\VSGraph (\ux, \uy)$ of $\VSGraph$ consisting of monodromic Elias--Williamson graphs whose underlying Elias--Williamson graph is in $\SGraph (\ux, \uy)$.
Likewise, we denote ${}_{\scrL} \VSGraph_{\scrL'} (\ux, \uy) \coloneq \VSGraph (\ux, \uy) \cap {}_{\scrL} \VSGraph_{\scrL'}$.

\subsubsection{Endosimple Expansion}

Let $I \subseteq S$. Let $\ux, \uy \in \Exp (W)$ such that the simple reflections constituting $\ux$ and $\uy$ are in $I$. 
We will write $\SGraph^{I,1} (\ux, \uy)$ for the subset of \emph{$(I,1)$-colored} Elias--Williamson graphs of $\SGraph (\ux, \uy)$ which consists of Elias--Williamson graphs with no $2m_{s,t}$-valent vertices and all strands are colored by elements of $I$.
We also define $\VSGraph^{I,1} (\ux, \uy)$ consisting of monodromic Elias--Williamson graphs whose underlying Elias--Williamson graph is \emph{($I$,1)-colored}. Likewise, we define ${}_{\scrL} \VSGraph_{\scrL'}^{I,1} (\ux, \uy) = \VSGraph^{I,1} (\ux, \uy) \cap {}_{\scrL} \VSGraph_{\scrL'}$.
If $S=I$, we will usually drop the $I$ from the notation in all the previously defined collections of graphs.

Fix $\scrL \in \fr{o}$.
Let $s' \in S_{\scrL}^{\circ}$ and $\us' \in \Exp (W, s')$ be a reduced expression. Since $s'$ is a reflection, $\us'$ is of the form $(s_1, \ldots, s_k, s, s_k, \ldots, s_1)$ where $\scrL s_1 \ldots s_k s = \scrL s_1 \ldots s_k$.
For each $n \in \Z_{\geq 0}$, we define $s^{\star n}$ as the expression $(s,s,\ldots, s)$ of length $n$.
Define a map, called \emph{endosimple expansion}, 
\[\textnormal{endo}_{\us} : \SGraph^{\{s\}, 1} (s^{\star n}, s^{\star m}) \to {}_{\scrL} \VSGraph_{\scrL}^1 ((\us')^{\star n}, (\us')^{\star m}),\]
which on generating vertices is defined by
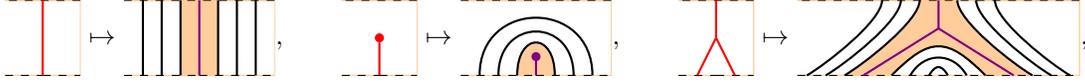
\[
  % \tikzsetnextfilename{#1}
  \tikzstyle{every picture}=[tikzfig]
  \input{./figures/id1.tikz}
 \mapsto 
  % \tikzsetnextfilename{#1}
  \tikzstyle{every picture}=[tikzfig]
  \input{./figures/id2.tikz}
, \qquad 
  % \tikzsetnextfilename{#1}
  \tikzstyle{every picture}=[tikzfig]
  \input{./figures/startdot3.tikz}
 \mapsto 
  % \tikzsetnextfilename{#1}
  \tikzstyle{every picture}=[tikzfig]
  \input{./figures/startdot4.tikz}
, \qquad 
  % \tikzsetnextfilename{#1}
  \tikzstyle{every picture}=[tikzfig]
  \input{./figures/mult3.tikz}
 \mapsto 
  % \tikzsetnextfilename{#1}
  \tikzstyle{every picture}=[tikzfig]
  \input{./figures/mult4.tikz}
,\]
and is the identity on polynomial boxes.

We can generalize the endosimple expansion to multiple simple reflections as follows. 
Let $I \subseteq S_{\scrL}^{\circ}$. Write $W_{\scrL, I}^{\circ}$ for the standard parabolic subgroup of $W_{\scrL}^{\circ}$ generated by $I$.
Let $\iota : I \to \Exp (W)$ be an endosimple expansion datum for $W_{\scrL, I}^{\circ}$.
Let $\ux, \uy \in \Exp (W_{\scrL, I}^{\circ})$.  
By taking concatenations of $\textnormal{endo}_{\us}$ maps, we obtain a function
\[\textnormal{endo}_{\iota} : \SGraph^{I,1} (\ux, \uy) \to {}_{\scrL} \VSGraph_{\scrL}^1 (\iota (\ux), \iota (\uy)).\]

\subsection{One-Color Calculus}

\subsubsection{Diagrammatics for \texorpdfstring{$B_s^{\scrL}$}{BsL}}

\begin{definition}\label{defn:one_color_cat}
    Let $(W,S)$ be a Coxeter system of type $A_1$ with $S = \{s\}$. Let $\fr{o}$ be a $W$-set and $\fr{h}$ be a pre-realization for $W$.
    The \emph{one-color diagrammatic monodromic Hecke category}, denoted $\EWmon{}{}^{\BS} (\fr{h}, s, \fr{o})$, is the $\k$-linear strict 2-category defined as follows.
    The objects are given by elements of $\fr{o}$. The 1-morphisms are tensor generated by 1-morphisms $B_s^{\scrL} : \scrL \stackrel{s}{\to} \scrL s$ for all $\scrL \in \fr{o}$.
    We depict these diagrammatically by strings
    \begin{equation}
        
  % \tikzsetnextfilename{#1}
  \tikzstyle{every picture}=[tikzfig]
  \input{./figures/defn_0.tikz}

    \end{equation}
    As such for each $n \in \Z_{\geq 0}$, we may define the 1-morphism $B_{s^{\star n}}^{\scrL} : \scrL \to \scrL s^{\star n}$ given by the $n$-fold tensor of generating 1-morphisms.

    Let $n,m \in \Z_{\geq 0}$ such that $\scrL s^{\star n} = \scrL' = \scrL s^{\star m}$.
    The space of 2-morphisms $\TwoHom_{\EWmon{}{}^{\BS}} (B_{s^{\star n}}^{\scrL}, B_{s^{\star m}}^{\scrL})$ is the free $\k$-module with basis ${}_{\scrL} \VSGraph_{\scrL'} (s^{\star n}, s^{\star m})$ modulo the following relations.  
    \begin{equation}\label{eq:one_color_cat} 
        
  % \tikzsetnextfilename{#1}
  \tikzstyle{every picture}=[tikzfig]
  \input{./figures/one_color_relns.tikz}

    \end{equation}
\end{definition}

We have used the same notation for the objects in the one-color diagrammatic category as we did for Bott--Samelson bimodules. 
We should think that the one-color diagrammatic category encodes the behavior of $B_s^{\scrL}$ in the category of Bott--Samelson bimodules.
This relationship will be made more precise in the following sections.
We will also use the one-color category as a crucial building block of more general diagrammatic monodromic Hecke categories.

\begin{proposition}\label{prop:one_color_soergel_functor}
    Let $(W,S)$ be a Coxeter system of type $A_1$ with $S = \{s\}$. Let $\fr{o}$ be a $W$-set and $\fr{h}$ be a pre-realization for $W$.
    Then there is a 2-functor
    \[\Upsilon_s : \EWmon{}{}^{\BS} (\fr{h}, s, \fr{o}) \to \Amon{}{}^{\BS} (\fr{h}, W, \fr{o}) \]
    defined on 1-morphisms by $B_s^{\scrL} \mapsto B_s^{\scrL}$ and on 2-morphisms by
    \begin{align*}
        \Upsilon_s \left( 
  % \tikzsetnextfilename{#1}
  \tikzstyle{every picture}=[tikzfig]
  \input{./figures/poly2.tikz}
 \right) &\coloneq  f & \Upsilon_s \left(
  % \tikzsetnextfilename{#1}
  \tikzstyle{every picture}=[tikzfig]
  \input{./figures/enddot2.tikz}
 \right) &\coloneq \eta^s , & \Upsilon_s \left(
  % \tikzsetnextfilename{#1}
  \tikzstyle{every picture}=[tikzfig]
  \input{./figures/startdot2.tikz}
 \right) &\coloneq \epsilon^s,  & \Upsilon_s \left(
  % \tikzsetnextfilename{#1}
  \tikzstyle{every picture}=[tikzfig]
  \input{./figures/mult2.tikz}
 \right) &\coloneq  \mu^s, 
    \end{align*}
    \begin{align*}
        \Upsilon_s \left(
  % \tikzsetnextfilename{#1}
  \tikzstyle{every picture}=[tikzfig]
  \input{./figures/comult2.tikz}
 \right) &\coloneq \nu^s,  & \Upsilon_s \left(
  % \tikzsetnextfilename{#1}
  \tikzstyle{every picture}=[tikzfig]
  \input{./figures/cap2.tikz}
 \right) &\coloneq \cap^s , & \Upsilon_s \left(
  % \tikzsetnextfilename{#1}
  \tikzstyle{every picture}=[tikzfig]
  \input{./figures/cup2.tikz}
 \right) &\coloneq  \cup^s.  
    \end{align*}
\end{proposition}
\begin{proof}
    Almost all the relations can be derived from their non-monodromic counterparts (cf., \cite[Claim 5.13]{EW}).
    It remains to check that $\Upsilon_s$ is invariant under the polynomial forcing relation when $\scrL s \neq \scrL$ and the block minimality relations.
    The polynomial forcing relation is the defining property of the $R$-bimodule structure on $R_s = B_s^{\scrL}$.
    Similarly, block minimality follows from $\cup^s$ and $\cap^s$ being inverse, i.e., they define the canonical isomorphism $R_s \otimes_R R_s \cong R$.
\end{proof}

It turns out that $\Upsilon_s$ is actually fully faithful on morphism categories. However, we will need to develop much more technology before proving this.

\subsubsection{The Universal Calculus}\label{subsubsec:universal_category}

In the non-monodromic setting, one usually goes straight from the 1-color diagrammatics to 2-colors diagrammatics when defining Hecke categories.
However, in the monodromic setting, it is useful to introduce an intermediary where we allow for arbitrary 1-morphisms, but only consider 2-morphisms corresponding to rank 1 parabolic subgroups.
This has the advantage of allowing for monodromy parameters to play a more central role-- giving way to glimpses of endoscopy.

\begin{definition}\label{defn:univ_one_color_cat}
    Let $(W,S)$ be a Coxeter system. Let $\fr{o}$ be a $W$-set and $\fr{h}$ be a pre-realization for $W$.
    The \emph{universal diagrammatic monodromic Hecke category}, denoted $\EWmon{}{}^{\BS} (\fr{h}, S, \fr{o})$, is the $\k$-linear strict 2-category defined as follows.
    The object set is $\fr{o}$. The 1-morphisms are tensor generated by 1-morphisms $B_s^{\scrL} : \scrL \stackrel{s}{\to} \scrL s$ for all $\scrL \in \fr{o}$ and $s \in S$.
    In particular, a general 1-morphism is of the form $B_{\uw}^{\scrL} : \scrL \stackrel{\uw}{\to} \scrL \uw$ where $\uw \in \Exp (W)$.

    Let $\ux, \uy \in \Exp (W)$ such that $\scrL \ux = \scrL' = \scrL \uy$.
    The 2-morphisms $\TwoHom_{\EWmon{}{}^{\BS}} (B_{\ux}^{\scrL}, B_{\uy}^{\scrL})$ is the free $\k$-module with basis ${}_{\scrL} \VSGraph_{\scrL'}^1 (\ux, \uy)$ modulo the 1-color relations of (\ref{eq:one_color_cat}).

    Given $\scrL, \scrL' \in \fr{o}$, we define 1-categories $\EWmon{\scrL}{\scrL'}^{\BS} (\fr{h}, S, \fr{o}) \coloneq \Hom_{\EWmon{}{}^{\BS} (\fr{h}, S, \fr{o})} (\scrL, \scrL')$. 
\end{definition}

\begin{remark}
The 2-category $\EWmon{}{}^{\BS} (\fr{h}, S, \fr{o})$ is universal in the sense that it is the diagrammatic monodromic Hecke category for the free Coxeter group on $S$ generators.
In particular, we will later see that for a general Coxeter group $W$, its diagrammatic monodromic Hecke category is a quotient of the universal category.
\end{remark}

The inclusion of $\{s\} \hookrightarrow S$ gives rise to a locally faithful 2-functor
\[I_{s} : \EWmon{}{}^{\BS} (\fr{h}, s, \fr{o}) \to \EWmon{}{}^{\BS} (\fr{h}, S, \fr{o}).\]
Similarly, the same construction and argument from Proposition \ref{prop:one_color_soergel_functor} defines a 2-functor
\[\Upsilon_S : \EWmon{}{}^{\BS} (\fr{h}, S, \fr{o}) \to \Amon{}{}^{\BS} (\fr{h}, W, \fr{o})\]
such that $\Upsilon_S \circ I_{s} = \Upsilon_s$. Unlike the one-color case, $\Upsilon_S$ is not locally full. 
In particular, it only depends on the set $S$ and its action on $\fr{o}$, but does not include any Coxeter theoretic information.
In other words, the universal one-color category does not include information on rank 2 and rank 3 parabolic subgroups of $W$.

The main advantage of introducing the universal category is that it contains many interesting Frobenius algebras that were not present in the one-color category.
Let $\uw = (s_1, \ldots, s_k, s, s_k, \ldots, s_1)$ be an endo-reduced expression for some endosimple reflection $s' \in S_{\scrL}^{\circ}$.
Write $\ux = (s_1, \ldots, s_k)$ and let $\scrL' = \scrL \ux$.
We can then give $B_{\uw}^{\scrL}$ the structure of a Frobenius algebra as follows.
\begin{equation}\label{eq:frob_algebra_str_diagrammatically}
    
  % \tikzsetnextfilename{#1}
  \tikzstyle{every picture}=[tikzfig]
  \input{./figures/endo_reduced_frob.tikz}

\end{equation}
It is easy to check from the Frobenius algebra structure on $B_{s}^{\scrL'}$ and block minimality that $B_{\uw}^{\scrL}$ along with the maps in (\ref{eq:frob_algebra_str_diagrammatically}) define a Frobenius algebra.
Moreover, it follows that under $\Upsilon_S$, this Frobenius algebra is sent to $B_{\uw}^{\scrL} \cong C_{t}$ with its usual Frobenius algebra structure.

\begin{proposition}[One-Color Endoscopy]\label{prop:one_color_endoscopy}
    Let $\scrL \in \fr{o}$ and $\iota$ be an endosimple expansion datum for $\scrL$.
    There is a monoidal functor
    \[ \Psi_{\scrL}^{\iota, \textnormal{diag}} : \EWmon{1}{1}^{\BS} (\fr{h}, S_{\scrL}^{\circ}, 1) \to \EWmon{\scrL}{\scrL}^{\BS} (\fr{h}, S, \fr{o})\]
    defined by $\Psi_{\scrL}^{\iota, \textnormal{diag}} (B_{s'}^{1}) = B_{\iota (s')}^{\scrL}$ for $s' \in S_{\scrL}^{\circ}$ and taking the Frobenius structure maps on $B_{s'}^1$ to the Frobenius structure maps on $B_{\iota (s')}^{\scrL}$.
    Moreover, this functor makes the following square commute up to natural isomorphism
    \begin{equation}\label{eq:one_color_endoscopy_1}
    \begin{tikzcd}
        {\EWmon{1}{1}^{\BS} (\fr{h}, S_{\scrL}^{\circ}, 1)} \arrow[r, "{\Psi_{\scrL}^{\iota, \textnormal{diag}}}"] \arrow[d, "\Upsilon_{S_{\scrL}^{\circ}}"] & {\EWmon{\scrL}{\scrL}^{\BS} (\fr{h}, S, \fr{o})} \arrow[d, "\Upsilon_S"] \\
        {\Amon{1}{1}^{\BS} (\fr{h}, W_{\scrL}^{\circ}, 1)} \arrow[r, "{\Psi_{\scrL}^{\textnormal{alg}}}"]                                         & {\Amon{\scrL}{\scrL}^{\BS} (\fr{h}, W, \fr{o})}.                     
        \end{tikzcd}
    \end{equation}
\end{proposition}
\begin{proof}
    The fact that $\Psi_{\scrL}^{\iota, \diag}$ is well-defined follows from the discussion on Frobenius algebra structures preceding the proposition.
    
    It remains to check that (\ref{eq:one_color_endoscopy_1}) commutes. 
    After unpacking definitions, it suffices to define an isomorphism $\xi : B_{\iota (s')}^{\scrL} \stackrel{\sim}{\to} C_{s'}$ of Frobenius algebras in $\Amon{\scrL}{\scrL}^{\BS} (\fr{h}, W, \fr{o})$ for each $s' \in S_{\scrL}^{\circ}$.
    Define $\xi$ as the isomorphism taking the 1-tensor $u_{\iota (s')}$ to $1 \otimes 1 \in R \otimes_{R^{s'}} R (1) = C_{s'}$. 
    This property ensures that $\xi$ is an isomorphism of coalgebras. Indeed, the coalgebra structure maps for $C_{s'}$ are the unique $R$-bimodule maps taking 1-tensors to 1-tensors.
    By Proposition \ref{prop:duality_and_str_mors}, $D(\xi^{-1}) :  B_{\iota (s')}^{\scrL} \stackrel{\sim}{\to} C_{s'}$ is an isomorphism of algebras. It is then easy to check that $D(\xi^{-1})$ also takes the $u_{\iota (s')}$ to $1 \otimes 1$.
    As a result, by uniqueness, we must have that $D(\xi^{-1}) = \xi$. Therefore, $\xi$ is an isomorphism of Frobenius algebras.
\end{proof} 

\subsection{The Monodromic Dihedral Cathedral} 

The goal of this section is to define the diagrammatic monodromic Hecke category for Coxeter groups of type $I_2 (m)$ for $m < \infty$.
In terms of generators, this amounts to adding $2m$-valent vertices to Elias--Williamson graphs.
The relations amongst these generators will combine the one-color relations along with two new relations: two-color associativity and the Elias--Jones--Wenzl relation.
Two-color associativity is a straightforward generalization of the corresponding non-monodromic concept.
The Elias--Jones--Wenzl relation is much more subtle. 
In the non-monodromic setting, the Elias--Jones--Wenzl relation is defined using Jones--Wenzl projectors for the two-colored Temperley--Lieb algebra.
Unfortunately, we do not have a good notion of a monodromic Temperley--Lieb algebra. 
Instead, we will consider an endoscopic Temperley--Lieb algebra and then apply endosimple expansion. 
This has the benefit of allowing us to use known results on Temperley--Lieb algebras.

Before continuing to the construction of the Jones--Wenzl projectors. We will set up some notation which will be used throughout this section.
Let $W_m$ denote the dihedral group with simple reflections $s,t$ where $m = m_{s,t} < \infty$.
Let $\fr{o}$ be a $W_m$-set. We will write $W_{s,t}^{\scrL}$ for the endoscopic subgroup of $W_m$ corresponding to $\scrL \in \fr{o}$.
Note that $W_{s,t}^{\scrL}$ has rank 0, 1, or 2. 
When $W_{s,t}^{\scrL}$ is rank 2, we will write $s', t'$ for the endosimple reflections in $W_{s,t}^{\scrL}$.
Moreover, we will write $v = v_{s,t}^{\scrL}$ for the order of $s't'$ which divides $m$. We then have that $W_{s,t}^{\scrL}$ is isomorphic as a Coxeter group to $W_v$.

\subsubsection{Temperley--Lieb Category}

Let $A$ be a $\Z[x,y]$-algebra.
We will review the construction of the two-colored Temperley--Lieb 2-category $\TLcat_A$ and discuss some properties.
There is no new mathematical content here and a more detailed treatise can be found elsewhere (cf., \cite{EW, EW17, Hazi}).

\begin{definition}
    The \emph{two-colored Temperley--Lieb 2-category} $\TLcat_A$ is the 2-category defined as follows. 
    There are two objects of $\TLcat_A$ which are denoted by $s$ and $t$.
    The 1-morphisms of $\TLcat_A$ consist of finite alternating sequences $(\ldots, s, t, s, t, \ldots)$.
    For an integer $n \geq 0$, we write ${}_{s} n$ for the sequence of length $n+1$ which begins with $s$, and ${}_{t} n$ for the sequence of length $n+1$ which begins with $t$.
    We visualize the 1-morphisms as a sequence of $n$ dots on a line whose regions are alternately colored by $s$ and $t$. E.g.,
    \[{}_s 5 = 
  % \tikzsetnextfilename{#1}
  \tikzstyle{every picture}=[tikzfig]
  \input{./figures/tl_1_mors.tikz}
\]
    There are no 2-morphisms between ${}_{s} n$ and ${}_{t} n'$ in either direction.
    Let $u \in \{s,t\}$. The $A$-module of 2-morphisms $\TwoHom_{\TLcat} ({}_u n, {}_u n')$ is given by the free $A$-module whose basis $\CM ({}_u n, {}_u n')$ consists of two colored $(n,n')$-crossingless matchings.
    Vertical composition ($\circ$) of 2-morphisms is defined by concatenation of crossingless matchings along with the following ``bubble evaluation'' rule,
    \begin{equation}
        
  % \tikzsetnextfilename{#1}
  \tikzstyle{every picture}=[tikzfig]
  \input{./figures/tl_bubble.tikz}
.
    \end{equation}
    Horizontal composition ($\star$) of 2-morphisms is defined by horizontal concatenation of crossingless matchings. 
\end{definition}

The $A$-algebra $\TL_A ({}_s n) \coloneq \TwoEnd_{\TLcat_A} ({}_s n)$ is called the \emph{two-colored Temperley--Lieb algebra}.
For each $1 \leq i \leq n-1$, we define a crossingless matching $e_i \in \TL_A ({}_s n)$ as follows:
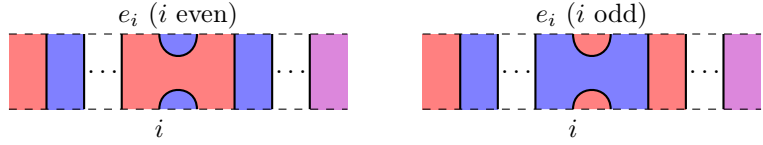
\[
  % \tikzsetnextfilename{#1}
  \tikzstyle{every picture}=[tikzfig]
  \input{./figures/TL_ei.tikz}
\]
where the purple region is red if $n$ is even and blue if $n$ is odd.
It can be checked that $\TL_A ({}_s n)$ is the $A$-algebra with generators $e_i$ for $1 \leq i \leq n-1$, subject to the relations
\begin{align*}
    e_i^2 &= -x & & \text{if }i \text{ odd}, \\
    e_i^2 &= -y & & \text{if }i \text{ even},\\
    e_i e_j&= e_j e_i & & \text{if }\abs{i-j} > 1, \\
    e_i e_{i\pm 1} e_i &= e_i. & &
\end{align*}
The algebra $\TL_A ({}_t n) \coloneq \TwoEnd_{\TLcat_A} ({}_t n)$ admits a similar presentation except for the parity conditions on $e_i^2$ being switched.

\begin{definition}
    Let $u \in \{s,t\}$.
    An element $\JW_A ({}_u n) \in \TL_A ({}_u n)$ is called a \emph{two-colored Jones--Wenzl projector} if
    \begin{enumerate}
        \item $\JW_A ({}_u n)$ is an idempotent,
        \item $e_i \JW_A ({}_u n) = 0$ for all $1 \leq i \leq n-1$,
        \item The coefficient of 1 in $\JW_A ({}_u n)$ is 1.
    \end{enumerate}
\end{definition}

It follows from general theory that if two-colored Jones--Wenzl projectors exist, then they are unique.

\begin{definition}
    Suppose $\JW_A ({}_s n)$ exists.
    We say that $\JW_A ({}_s n)$ is \emph{rotatable} if 
    \[
  % \tikzsetnextfilename{#1}
  \tikzstyle{every picture}=[tikzfig]
  \input{./figures/rot_of_JW.tikz}
\]
    for some invertible scalar $\lambda_{s,t}$.
    Likewise, $\JW_A ({}_t n)$ is said to be rotatable if the analogous equations hold after switching colors.
\end{definition}

It is clear from definitions that if $\JW_A ({}_s n)$ is rotatable, then $\JW_A ({}_t n)$ will be as well.
In this case, it is not hard to show that $\lambda_{s,t} = \lambda_{t,s}^{-1}$.
By \cite[Lemma 6.25]{EW17}, one can calculate that $\lambda_{s,t} = [n]_y$.

Historically, it has been a very troublesome question to determine when Jones--Wenzl projectors exist and when they are rotatable.
Fortunately, a complete solution to this question has been provided by Hazi.

\begin{theorem}[{\cite[Theorem B]{Hazi}}]\label{thm:JW_exist_and_rotatable}
    The two-colored Jones--Wenzl projectors, $\JW_A ({}_s n)$ and $\JW_A ({}_t n)$, exist and are rotatable if and only if $\binom{n+1}{k}_x = \binom{n+1}{k}_y = 0$ in $A$ for all $1 \leq k \leq n$.
\end{theorem}

Let $\fr{h}$ be an Abe realization for a dihedral group $W_m$ of order $m$ with simple reflections $s,t$.
Let $A = \k [x,y] / (x - a_{s,t}, y - a_{t,s})$. Theorem \ref{thm:JW_exist_and_rotatable} implies that the two-colored Jones--Wenzl projectors exist for this $A$.
Moreover, the condition that $\fr{h}$ be balanced is equivalent to $\lambda_{s,t} = 1$.

\subsubsection{From Temperley--Lieb to Monodromic Elias--Williamson graphs}

Throughout this section, we will assume that $W_{s,t}^{\scrL}$ is rank 2.

Define a 1-category
\[\EWmon{\scrL}{-}^{\BS} (\fr{h}, \{s,t\}, \fr{o}) \coloneq \bigsqcup_{\scrL' \in \scrL \cdot W_m } \EWmon{\scrL}{\scrL'}^{\BS} (\fr{h}, \{s,t\}, \fr{o}).\]
We fix $A = \k[x,y]/(x-a_{s',t'}, y-a_{t',s'})$.
The goal of this section is to define a 1-functor\footnote{To make sense of this, we regard $\TLcat_A$ as a 1-category with respect to horizontal concatenation. Upgrading this to a 2-functor would require a theory of singular monodromic diagrammatics.}
\[\Sigma_{\scrL}^{\tau} : \TLcat_A \to \EWmon{\scrL}{-}^{\BS} (\fr{h}, \{s,t\}, \fr{o}).\]
This functor should take the Jones--Wenzl projector $\JW_A ({}_{s'} (v-1))$ to a morphism $B_{{}_s \underline{m}}^{\scrL} \to B_{{}_s \underline{m}}^{\scrL}$.

It follows from the non-monodromic setting, that the deformation retract of crossingless matchings induces a functor
\[ \Sigma : \TLcat_A \to \EWmon{1}{1}^{\BS} (\fr{h}, \{s',t'\},  1).\]
Moreover, by composing with the endosimple expansion of Proposition \ref{prop:one_color_endoscopy}, we get a functor
\[\Sigma_{\scrL} :  \TLcat_A \to \EWmon{\scrL}{\scrL}^{\BS} (\fr{h}, \{s,t\}, \fr{o}).\]
This functor is close to the desired functor; however, it doesn't take the Jones--Wenzl projector to the correct target.
In particular, every object in the image of $\Sigma_{\scrL}$ the same left and right monodromy, whereas it may be the case that $\scrL {}_s \underline{m} \neq \scrL$.

\begin{definition}
    Let $f : B_{{}_u \underline{N}}^{\scrL} \to B_{{}_u \underline{N}'}^{\scrL}$ be a morphism in $\EWmon{\scrL}{-}^{\BS} (\fr{h}, \{s,t\}, \fr{o})$ and let $a,b \in \Z_{\geq 0}$.
    \begin{enumerate}
        \item We say that $f$ is $(a,b)$-truncatable if $f = f' \star \id$ for some morphism $f' : B_{{}_u \underline{a}}^{\scrL} \to B_{{}_u \underline{b}}^{\scrL}$.
        \item We say that $f$ is $(a,b)$-extendable if $f \star \id = f'$ for some morphism $f' : B_{{}_u \underline{a}}^{\scrL} \to B_{{}_u \underline{b}}^{\scrL}$.
        \item We say that $f$ is $(a,b)$-adjustable if it is either $(a,b)$-truncatable or $(a,b)$-extendable.
    \end{enumerate}
    If $f$ is $(a,b)$-adjustable, we will write $\tau_a^b f$ for the morphism $f' : B_{{}_u \underline{a}}^{\scrL} \to B_{{}_u \underline{b}}^{\scrL}$ produced above. 
\end{definition}

Note that if $f : B_{{}_u \underline{N}}^{\scrL} \to B_{{}_u \underline{N}'}^{\scrL}$ is $(a,b)$-truncatable, then $N \geq a$ and $N' \geq b$.
Similarly, if $f$ is $(a,b)$-extendable, then $N \leq a$ and $N' \leq b$. In particular, $f$ is both $(a,b)$-truncatable and $(a,b)$-extendable if and only if $N = a$ and $N' = b$. 
We will write $\Hom_{\EWmon{\scrL}{-}^{\BS}}^{(a,b)} (B_{{}_u \underline{N}}^{\scrL}, B_{{}_u \underline{N}'}^{\scrL})$ for the sub-$R$-module of $(a,b)$-adjustable morphisms.
\begin{lemma}\label{lem:comp_of_ess_morphisms}
    Let $f : B_{{}_u \underline{N}}^{\scrL} \to B_{{}_u \underline{N}'}^{\scrL}$ be an $(a,b)$-adjustable morphism and let $g : B_{{}_u \underline{N'}}^{\scrL} \to B_{{}_u \underline{N''}}^{\scrL}$ be a $(b,c)$-adjustable morphism.
    Then $g\circ f$ is $(a,c)$-adjustable. Moreover, $\tau_a^c (g \circ f) =\tau_b^c (g) \circ \tau_a^b (f)$.
\end{lemma}
\begin{proof}
    Suppose $f$ is truncated and $g$ is truncated. It follows from definitions that $g \circ f  = (\tau_b^c g \circ \tau_a^b f) \star \id$, and so $g \circ f$ is $(a,c)$-truncated with $\tau_a^c (g \circ f) = \tau_b^c g \circ \tau_a^b f$.
    
    Suppose $f$ is truncated and $g$ is extendable.
    Since $f$ is truncated, we have that $N' - b \geq 0$. On the other hand, since $g$ is extendable, we have that $N' - b \leq 0$.
    Therefore, we are in the cases where $N' = b$ which implies that $g$ is truncated. We can then proceed as in the first case.
    
    The remaining two cases follow from similar arguments and are omitted.
\end{proof}

Since $W_{s,t}^{\scrL}$ has rank 2, we must have that $\ell (s') < m$ and $\ell (t') < m$.
As a result, $s'$ and $t'$ have unique reduced expressions. This gives a canonical choice of an endosimple expansion datum $\iota : \Exp (S_{\scrL}^{\circ}) \to \Exp (W)$.

We define a (strict) subcategory $\EWmon{\scrL}{\scrL}^{\BS, \tau} (\fr{h}, \{s,t\}, \fr{o})$ of $\EWmon{\scrL}{\scrL}^{\BS} (\fr{h}, \{s,t\}, \fr{o})$ whose objects are of the form $B_{\iota ({}_{u'} \underline{n+1})}^{\scrL}$ for $n \in \Z_{\geq 0}$ and $u' \in \{s',t'\}$.
The morphisms in $\EWmon{\scrL}{\scrL}^{\BS, \tau} (\fr{h}, \{s,t\}, \fr{o})$ are given by
\[\Hom_{\EWmon{\scrL}{\scrL}^{\BS, \tau}} (B_{\iota ({}_{u'} \underline{n+1})}^{\scrL}, B_{\iota ({}_{u'} \underline{n'+1})}^{\scrL} ) = \Hom_{\EWmon{\scrL}{\scrL}^{\BS}}^{(d(n+1), d(n'+1))} (B_{\iota ({}_{u'} \underline{n+1})}^{\scrL}, B_{\iota ({}_{u'} \underline{n'+1})}^{\scrL})\]
where $n,n' \in \Z_{\geq 0}$ and $u' \in \{s',t'\}$. For $u', v' \in \{s',t'\}$ with $u' \neq v'$ and $n,n' \in \Z_{\geq 0}$, we define
\[\Hom_{\EWmon{\scrL}{\scrL}^{\BS, \tau}} (B_{\iota ({}_{u'} \underline{n+1})}^{\scrL}, B_{\iota ({}_{v'} \underline{n'+1})}^{\scrL} ) = 0. \]
By Lemma \ref{lem:comp_of_ess_morphisms}, $\EWmon{\scrL}{\scrL}^{\BS, \tau} (\fr{h}, \{s,t\}, \fr{o})$ is indeed a well-defined subcategory.

We can define a functor
\[\tau : \EWmon{\scrL}{\scrL}^{\BS, \tau} (\fr{h}, \{s,t\}, \fr{o}) \to \EWmon{\scrL}{-}^{\BS} (\fr{h}, \{s,t\}, \fr{o})\]
by $\tau (f) \coloneq \tau_{d(n+1)}^{d(n'+1)} (f)$ for all morphisms $f$ in $\EWmon{\scrL}{\scrL}^{\BS, \tau} (\fr{h}, \{s,t\}, \fr{o})$. 
Lemma \ref{lem:comp_of_ess_morphisms} ensures that $\tau$ is well-defined.

\begin{lemma}\label{lem:truncations_and_extensions}
    Let $u' \in \{s',t'\}$ and $\varphi : {}_{u'} n \to {}_{u'} n'$ be a 2-morphism in $\TLcat_A$.
    Then $\Sigma_{\scrL} (\varphi)$ is $(d(n+1), d(n'+1))$-adjustable.
\end{lemma}
\begin{proof}
    Without loss of generality, we may take $u' = s'$.
    Let $N$ (resp. $N'$) be the integer such that $\iota ({}_{s'} \underline{n+1}) = {}_s \underline{N}$ (resp. $\iota ({}_{s'} \underline{n'+1}) = {}_s \underline{N}'$).
    If $\varphi = 0$, there is nothing to prove.

    Note that $d = \frac{1}{2} (\ell (s') + \ell (t'))$.
    If $n$ is odd, then $\varphi$ being nonzero implies that $n'$ is also odd.
    In this case, $N = d(n+1)$ and $N' = d(n'+1)$ since $s'$ and $t'$ appear the same number of times in ${}_s \underline{n+1}$ (resp. ${}_s \underline{n'+1}$).
    As a result, $\Sigma_{\scrL} (\varphi)$ is trivially $(d(n+1), d(n'+1))$-truncatable (and extendable).

    Suppose $n$ is even. Again, $\varphi \neq 0$ implies that $n'$ is also even.
    In this case, $N = dn + \ell (s')$ and $N' = dn' + \ell (s')$.
    Write $r = \frac{1}{2} (\ell (t') - \ell (s'))$.
    First, assume that $\ell (s') \leq \ell (t')$, so that $r \geq 0$.
    We then have that $N + r = d (n+1)$ and $N' + r = d(n'+1)$.
    As a result, $\Sigma_{\scrL} (\varphi) \star \id_{B_{{}_t \underline{r}}^{\scrL}} : B_{{}_s \underline{d(n+1)}}^{\scrL} \to  B_{{}_s \underline{d(n'+1)}}^{\scrL}$, and thus $\varphi$ is $(d(n+1), d(n'+1))$-extendable. 
    Now assume that $\ell (s') > \ell (t')$, so that $r < 0$.
    From the definition of endosimple expansion given in Proposition \ref{prop:one_color_endoscopy}, we can see that $\Sigma_{\scrL} (\varphi)$ is of the form $f' \star \id_{B_{{}_t \underline{c}}^{\scrL}}$ where $c = \frac{1}{2} (\ell (s') - 1)$ (since these $c$-terms are minimal in their corresponding block).
    Then $c \geq r$, so we have that $\Sigma_{\scrL} (\varphi)$ is $(d(n+1), d(n'+1))$-truncatable.
\end{proof}

By Lemma \ref{lem:truncations_and_extensions},  $\Sigma_{\scrL}$ factors through $\EWmon{\scrL}{\scrL}^{\BS, \tau} (\fr{h}, \{s,t\}, \fr{o})$.
We can then consider define $\Sigma_{\scrL}^{\tau}$ as the composition
\[\Sigma_{\scrL}^{\tau} : \TLcat_A \stackrel{\Sigma_{\scrL}}{\to} \EWmon{\scrL}{\scrL}^{\BS, \tau} (\fr{h}, \{s,t\}, \fr{o}) \stackrel{\tau}{\to} \EWmon{\scrL}{-}^{\BS} (\fr{h}, \{s,t\}, \fr{o}).\]
Let $u' \in \{s',t'\}$ and let $u$ be the first simple reflection in $\iota (u')$.
Since $\TwoHom_{\TLcat_A} ({}_{u'} n, {}_{u'} n')$ has an $A$-basis by crossingless matchings, we obtain a linearly independent subset of $\Hom (B_{{}_u \underline{d(n+1)}}^{\scrL}, B_{{}_u \underline{d(n+1)}}^{\scrL})$
obtained by applying $\Sigma_{\scrL}^{\tau}$ to this basis.

Note that $\Sigma_{\scrL}^{\tau}$ is a fairly simple modification of $\Sigma_{\scrL}$ obtained by either concatenating or trimming identity morphisms from the image of $\Sigma_{\scrL}$.
If $n, n'$ are odd, then $\Sigma_{\scrL}^{\tau}$ agrees with $\Sigma_{\scrL}$ on the induced morphism of Hom spaces,
\[\TwoHom_{\TLcat_A} ({}_{u'} n, {}_{u'} n') \to \Hom_{\EWmon{\scrL}{-}^{\BS}} (B_{{}_u \underline{d(n+1)}}^{\scrL}, B_{{}_u \underline{d(n+1)}}^{\scrL}).\]
Here we are using the observation from the proof of Lemma \ref{lem:truncations_and_extensions} that for $\varphi : {}_{u'} n \to {}_{u'} n'$ the morphism $\Sigma_{\scrL} (\varphi)$ is always $(d(n+1), d(n'+1))$-adjustable when $n,n'$ are odd.

\begin{example}
    Consider $m = 6$ with $s' = s$ and $t' = tst$.
    In this case, $v = 3$ and $d=2$.
    \[
  % \tikzsetnextfilename{#1}
  \tikzstyle{every picture}=[tikzfig]
  \input{./figures/trunc_def_TL.tikz}
\]
\end{example}

\subsubsection{Monodromic Jones--Wenzl Morphisms}

We define the \emph{$\scrL$-monodromic Jones--Wenzl morphisms} $\JW_{s,t}^{\scrL} : B_{{}_s \underline{m}}^{\scrL} \to B_{{}_s \underline{m}}^{\scrL}$ and $\JW_{t,s}^{\scrL} : B_{{}_t \underline{m}}^{\scrL} \to B_{{}_t \underline{m}}^{\scrL}$ in $\EWmon{\scrL}{-}^{\BS} (\fr{h}, \{s,t\}, \fr{o})$ by 
\[ \JW_{s,t}^{\scrL} = \begin{cases} \id_{B_{{}_s \underline{m}}^{\scrL} } & \text{if } \rk (W_{s,t}^{\scrL}) < 2, \\ \Sigma_{\scrL}^{\tau} (\JW_A ({}_{s'} (v-1))) & \text{if } \rk (W_{s,t}^{\scrL}) = 2, \end{cases}\]
\[\JW_{t,s}^{\scrL} = \begin{cases} \id_{ B_{{}_t \underline{m}}^{\scrL}} & \text{if } \rk (W_{s,t}^{\scrL}) < 2, \\ \Sigma_{\scrL}^{\tau} (\JW_A ({}_{t'} (v-1))) & \text{if } \rk (W_{s,t}^{\scrL}) = 2. \end{cases} \]
Note the monodromic Jones--Wenzl morphisms are well-defined since Jones--Wenzl morphisms exist by Theorem \ref{thm:JW_exist_and_rotatable} and Proposition \ref{prop:refl_abe_realizations}.

A \emph{monodromic pitchfork} is a monodromic Elias--Williamson graph of the form
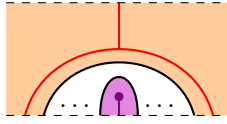
\[
  % \tikzsetnextfilename{#1}
  \tikzstyle{every picture}=[tikzfig]
  \input{./figures/mon_pitchfork.tikz}
\]
where the bottom expression $\uw$ satisfies $\ell_{\scrL} (\uw) = 3$, i.e., all the black strands are minimal in their blocks.

\begin{lemma}\label{lem:mon_JW_properties}
    The monodromic Jones--Wenzl morphisms satisfy the following properties:
    \begin{enumerate}
        \item $\JW_{s,t}^{\scrL}$ is idempotent;
        \item the coefficient of $\id_{B_{{}_s \underline{m}}^{\scrL}}$ in $\JW_{s,t}^{\scrL}$ is 1;
        \item $\JW_{s,t}^{\scrL}$ composed with any monodromic pitchfork is 0.
    \end{enumerate}
\end{lemma}
\begin{proof}
    The lemma is obvious if $\rk (W_{s,t}^{\scrL}) < 2$.
    If $\rk (W_{s,t}^{\scrL}) = 2$, then the lemma follows from the construction of $\JW_{s,t}^{\scrL}$ via the functor $\Sigma_{\scrL}^{\tau}$ and the defining properties of two-colored Jones--Wenzl projectors in $\TLcat_A$.  
\end{proof}

Lemma \ref{lem:mon_JW_properties} (3) is called the \emph{death by monodromic pitchfork relation}.

\subsubsection{Monodromic Hecke Category for Dihedral Groups}

\begin{definition}\label{def:two_color_cat}
Let $(W,S)$ be a Coxeter system of type $I_2 (m)$ with $S = \{s, t\}$, $\fr{o}$ be a $W$-set, and $\fr{h}$ be a reflection-balanced reflection-stable  Abe realization for $W$.
The \emph{two-color diagrammatic monodromic Hecke category}, denoted $\EWmon{}{}^{\BS} (\fr{h}, W, \fr{o})$, is the $\k$-linear strict 2-category defined as follows.
The object set is $\fr{o}$. The 1-morphisms are tensor generated by the 1-morphisms $B_s^{\scrL} : \scrL \stackrel{s}{\to} \scrL s$ and $B_t^{\scrL} : \scrL \stackrel{t}{\to} \scrL t$ for all $\scrL \in \fr{o}$.
A general 1-morphism is then of the form $B_{\uw}^{\scrL} : \scrL \stackrel{\uw}{\to} \scrL \uw$ for $\uw \in \Exp (W)$. 
The space of 2-morphisms $\TwoHom_{\EWmon{}{}^{\BS}} (B_{\uw}^{\scrL}, B_{\uy}^{\scrL})$ is the free $\k$-module with basis ${}_{\scrL} \VSGraph_{\scrL'} (\ux, \uy)$ modulo the following relations:

    \emph{(One-Color Relations)} The relations of (\ref{eq:one_color_cat}).
   
    \emph{(Two-color Associativity)} We give one example of each relation for each parity of $m = m_{s,t} < \infty$. The general form can be inferred from the non-monodromic setting.
    \ctikzfig{2assoc_5}
    \ctikzfig{2assoc_6}

    \emph{(Elias--Jones--Wenzl)}
        \[
  % \tikzsetnextfilename{#1}
  \tikzstyle{every picture}=[tikzfig]
  \input{./figures/ejw.tikz}
\]
\end{definition}

Note that when $(W,S)$ is an infinite dihedral group, $\EWmon{}{}^{\BS} (\fr{h}, W, \fr{o})$ coincides with $\EWmon{}{}^{\BS} (\fr{h}, S, \fr{o})$.

The following is a monodromic version of the Jones--Wenzl relation (also called two-color dot contraction).
\begin{lemma}\label{lem:jones_wenzl}
    Assume that $\scrL s = \scrL$. The following relation holds in $\EWmon{}{}^{\BS} (\fr{h}, W, \fr{o})$.
    \begin{equation}
        
  % \tikzsetnextfilename{#1}
  \tikzstyle{every picture}=[tikzfig]
  \input{./figures/jones_wenzl.tikz}

    \end{equation}
\end{lemma}
\begin{proof}
    We will just prove the case when $m$ is odd. The case when $m$ is even can be proved with the same argument.
    We can compute
    \[
  % \tikzsetnextfilename{#1}
  \tikzstyle{every picture}=[tikzfig]
  \input{./figures/jw_proof.tikz}
\]
    where the second equality is 2-color associativity and the third equality is the Elias--Jones--Wenzl relation.
\end{proof}

For two monodromy parameters $\scrL, \scrL' \in \fr{o}$, we define a 1-category 
\[ \EWmon{\scrL}{\scrL'}^{\BS} (\fr{h}, W, \fr{o}) \coloneq \Hom_{\EWmon{}{}^{\BS} (\fr{h}, W, \fr{o})} (\scrL, \scrL'). \]
Our first observation is that $\EWmon{\scrL}{\scrL'}^{\BS} (\fr{h}, W, \fr{o})$ admits a block decomposition analogous to Lemma \ref{lem:abe_block_decomp}.
Let $\beta \in \uW{\scrL}{\scrL'}$ be a block. We can then define a full subcategory $\EWmon{\scrL}{\scrL'}^{\BS, \beta} (\fr{h}, W, \fr{o})$ of $\EWmon{\scrL}{\scrL'}^{\BS} (\fr{h}, W, \fr{o})$, called the \emph{$\beta$-block},  consisting of objects $B_{\uw}^{\scrL}$ where $\uw \in \Exp (W)$ is an expression of an element in $\beta$.

\begin{lemma}\label{lem:dihedral_facts_about_blocks}
    Let $\scrL, \scrL', \scrL'' \in \fr{o}$.
    \begin{enumerate}
        \item There is a decomposition of categories
    \[\EWmon{\scrL}{\scrL'}^{\BS} (\fr{h}, W, \fr{o}) = \bigsqcup_{\beta \in \uW{\scrL}{\scrL'}} \EWmon{\scrL}{\scrL'}^{\BS, \beta} (\fr{h}, W, \fr{o}).\]
        \item Let $\beta \in \uW{\scrL}{\scrL'}$ and $\gamma \in \uW{\scrL'}{\scrL''}$. Horizontal composition is compatible with block multiplication in the following sense
        \[\EWmon{\scrL}{\scrL'}^{\BS, \beta} (\fr{h}, W, \fr{o}) \star \EWmon{\scrL'}{\scrL''}^{\BS, \beta \gamma} (\fr{h}, W, \fr{o}) \subseteq \EWmon{\scrL'}{\scrL''}^{\BS, \beta \gamma} (\fr{h}, W, \fr{o}).\]
        \item Let $\beta \in \uW{\scrL}{\scrL'}$. Let $\uw^{\beta}$ be an endo-reduced expression for $w^{\beta}$. There is an equivalence of categories
        \[ (-) \star \id_{\uw^{\beta}} : \EWmon{\scrL}{\scrL}^{\BS, \circ} (\fr{h}, W, \fr{o})  \stackrel{\sim}{\to} \EWmon{\scrL}{\scrL'}^{\BS, \beta} (\fr{h}, W, \fr{o}).\]
    \end{enumerate}
\end{lemma}
\begin{proof}
    \emph{(1): } Clearly, every morphism is contained in some block of $\EWmon{\scrL}{\scrL'}^{\BS, \beta} (\fr{h}, W, \fr{o})$.
    As a result, it suffices to show that every morphism $f : B_{\uw}^{\scrL} \to B_{\uy}^{\scrL}$ with $b_{\scrL} (\uw) \neq b_{\scrL} (\uy)$ is zero.
    We can show something stronger, namely, that no underlying monodromic Elias--Williamson graph can even exist.
    This can be easily checked on the generating vertices in an Elias--Williamson graph: trivalent and univalent vertices are always in the neutral block and $2m_{s,t}$-valent vertices do not change the block by Lemma \ref{lem:block_prelims}.

    \emph{(2): } This is obvious from the definition of blocks (Lemma \ref{lem:block_prelims}).

    \emph{(3): } Let $u \in \{s,t\}$ such that $\scrL u \neq \scrL$.
    Then the block minimality relations ensure that $B_{uu}^{\scrL} \cong B_{\emptyset}^{\scrL}$ via the cup and cap maps.
    As a result, $(-) \star \id_{B_{u}^{\scrL}}$ is an equivalence of categories with inverse $(-) \star \id_{B_{u}^{\scrL u}}$.
    More generally, if $\beta \in \uW{\scrL}{\scrL'}$ and $\uw^{\beta}$ is an endo-reduced expression for $w^{\beta}$, then we can apply the above equivalences inductively to show that
    $(-) \star \id_{B_{\uw^{\beta}}^{\scrL}}$ is an equivalence of categories with inverse $(-) \star \id_{B_{\overline{\uw}^{\beta}}^{\scrL'}}$.
\end{proof}

\subsubsection{Two-Color Endoscopy}

It is useful to give some extra notation to the $2m_{s,t}$-valent vertices.
Namely, suppose $2m = a+b$. We will then write
\[\beta ( {}_s \underline{a}, {}_t \underline{b} ) : B_{{}_s \underline{a}}^{\scrL} \to B_{{}_t \underline{b}}^{\scrL}\]
for the $2m_{s,t}$-valent vertex with strands ${}_s \underline{a}$ on the bottom and ${}_t \underline{b}$ strands on the top.
We reserve the notation $\beta_{s,t}$ for the morphism $\beta ({}_s \underline{m}, {}_t \underline{m})$.
Similarly, we will denote variants of the above where $s$ and $t$ are switched everywhere.

Let $v = v_{s,t}^{\scrL}$. Without loss of generality, we may assume that $\iota (s')$ (resp. $\iota (t')$) starts with $s$ (resp. $t$).
We wish to explain the behavior of endosimple expansion on $2v$-valent vertices.
Let $2v = a' + b'$. Note that $\iota ({}_{s'} \underline{2v}) = {}_s \underline{2m}$.
Let $a = \ell (\iota ({}_s \underline{a'}))$ and $b = \ell (\iota ({}_t \underline{b'}))$.
It then follows that $2m = a + b$.
We can then define the endosimple expansion of a $2v$-valent vertex as follows:
\begin{equation}\label{eq:Psi_on_2v_valent_vertices}
    \Psi_{\scrL}^{\iota, \textnormal{diag}} (\beta ({}_{s'} \underline{a}', {}_{t'} \underline{b}') ) \coloneq \beta ({}_s \underline{a}, {}_t \underline{b}).
\end{equation}

\allowdisplaybreaks

\begin{proposition}[Two-color Endoscopy]\label{prop:two_color_endoscopy}
    The functor 
    \[\Psi_{\scrL}^{\iota, \textnormal{diag}} : \EWmon{1}{1}^{\BS} (\fr{h}, \{s',t'\}, 1) \to \EWmon{\scrL}{\scrL}^{\BS} (\fr{h}, \{s,t\}, \fr{o})\] 
    extends to a functor
    \[\Psi_{\scrL}^{\iota, \textnormal{diag}} : \EWmon{1}{1}^{\BS} (\fr{h}, W_{s,t}^{\scrL},  1) \to \EWmon{\scrL}{\scrL}^{\BS} (\fr{h}, W, \fr{o})\]
    by defining $\Psi_{\scrL}^{\iota, \textnormal{diag}}$ on $2v$-valent vertices by (\ref{eq:Psi_on_2v_valent_vertices}).
\end{proposition}
\begin{proof}
    It suffices to check that $\Psi_{\scrL}^{\iota, \textnormal{diag}}$ is invariant under 2-color associativity and the Elias--Jones--Wenzl relation.

    \emph{(Two-color Associativity): } By rotation and possibly swapping colors, it suffices to show that when $v$ is odd that
    \[\Psi_{\scrL}^{\iota, \textnormal{diag}} \left( 
  % \tikzsetnextfilename{#1}
  \tikzstyle{every picture}=[tikzfig]
  \input{./figures/tce_1.tikz}
\right) = \Psi_{\scrL}^{\iota, \textnormal{diag}} \left( 
  % \tikzsetnextfilename{#1}
  \tikzstyle{every picture}=[tikzfig]
  \input{./figures/tce_2.tikz}
\right),\]
    and for $v$ is even that
     \[ \Psi_{\scrL}^{\iota, \textnormal{diag}} \left( 
  % \tikzsetnextfilename{#1}
  \tikzstyle{every picture}=[tikzfig]
  \input{./figures/tce_1.tikz}
\right) = \Psi_{\scrL}^{\iota, \textnormal{diag}} \left( 
  % \tikzsetnextfilename{#1}
  \tikzstyle{every picture}=[tikzfig]
  \input{./figures/tce_3.tikz}
\right),\]
    We will just show the case when $v$ is even. The case of $v$ being odd is similar.

    Let $k = \frac{1}{2} (\ell (t') -1)$. Assume that $k$ is even, so that $t' = ({}_t k) t ({}_s k)$.
    We can then compute directly. We omit the face labels on the monodromic side, since they can be inferred from the left side being labelled by $\scrL$.
    \begin{align*}
        \Psi_{\scrL}^{\iota, \textnormal{diag}} \left( 
  % \tikzsetnextfilename{#1}
  \tikzstyle{every picture}=[tikzfig]
  \input{./figures/tce_1.tikz}
\right) &= 
  % \tikzsetnextfilename{#1}
  \tikzstyle{every picture}=[tikzfig]
  \input{./figures/tce_4.tikz}
 \\
        &= 
  % \tikzsetnextfilename{#1}
  \tikzstyle{every picture}=[tikzfig]
  \input{./figures/tce_5.tikz}
 \tag{2-color associativity}\\
        &= 
  % \tikzsetnextfilename{#1}
  \tikzstyle{every picture}=[tikzfig]
  \input{./figures/tce_6.tikz}
 \tag{block minimality}\\
        &= \Psi_{\scrL}^{\iota, \textnormal{diag}} \left( 
  % \tikzsetnextfilename{#1}
  \tikzstyle{every picture}=[tikzfig]
  \input{./figures/tce_3.tikz}
\right).
    \end{align*}
    The case of $k$ odd is similar and omitted.

    \emph{(Elias--Jones--Wenzl): } We will cover both variants of the Elias--Jones--Wenzl relations, so without loss of generality, we may assume that $\ell (s') > \ell (t')$.
    We will assume that $v$ is odd and that $m$ is odd. The cases of $v$ being odd and $m$ being even is almost identical and the case of $vs$ being even is similar (and in fact easier).
    Let $r = \frac{1}{2} (\ell (t') - \ell (s')) > 0$.
    We can then compute directly, again omitting the face labels on the monodromic side.
    \begin{align*}
        \Psi_{\scrL}^{\iota, \textnormal{diag}} \left( 
  % \tikzsetnextfilename{#1}
  \tikzstyle{every picture}=[tikzfig]
  \input{./figures/tce_7.tikz}
\right) &= 
  % \tikzsetnextfilename{#1}
  \tikzstyle{every picture}=[tikzfig]
  \input{./figures/tce_8.tikz}
 \\
        &= 
  % \tikzsetnextfilename{#1}
  \tikzstyle{every picture}=[tikzfig]
  \input{./figures/tce_9.tikz}
 \tag{block minimality}\\
        &= 
  % \tikzsetnextfilename{#1}
  \tikzstyle{every picture}=[tikzfig]
  \input{./figures/tce_10.tikz}
 \tag{Elias--Jones--Wenzl}\\
        &= \tau \left(  \Psi_{\scrL}^{\iota, \textnormal{diag}} (\JW_{s',t'}^1) \right) \star \id_{B_{{}_t \underline{r}}^{\scrL}}  \\
        &= \Psi_{\scrL}^{\iota, \textnormal{diag}} (\JW_{s',t'}^1). 
    \end{align*}
    The last equality follows from the definition of $\tau$. Note that $\Psi_{\scrL}^{\iota, \textnormal{diag}} (\JW_{s',t'}^1)$ is truncatable in this case.

    We will now compute with the other variant of the Elias--Jones--Wenzl relation; albeit, slightly more indirectly.
    \begin{align*}
        \Psi_{\scrL}^{\iota, \textnormal{diag}} \left( 
  % \tikzsetnextfilename{#1}
  \tikzstyle{every picture}=[tikzfig]
  \input{./figures/tce_11.tikz}
\right) \star \id_{B_{{}_s \underline{r}}^{\scrL}} &= 
  % \tikzsetnextfilename{#1}
  \tikzstyle{every picture}=[tikzfig]
  \input{./figures/tce_12.tikz}
 \\
        &= 
  % \tikzsetnextfilename{#1}
  \tikzstyle{every picture}=[tikzfig]
  \input{./figures/tce_13.tikz}
 \\
        &= 
  % \tikzsetnextfilename{#1}
  \tikzstyle{every picture}=[tikzfig]
  \input{./figures/tce_14.tikz}
 \\
        &= \tau \left(  \Psi_{\scrL}^{\iota, \textnormal{diag}} (\JW_{s',t'}^1) \right) \\
        &= \Psi_{\scrL}^{\iota, \textnormal{diag}} (\JW_{t',s'}^1) \star \id_{B_{{}_s \underline{r}}^{\scrL}}. 
    \end{align*}
    The last line follows from the definition of $\tau$. Note that $\Psi_{\scrL}^{\iota, \textnormal{diag}} (\JW_{s',t'}^1)$ is extendable in this case.
    By Lemma \ref{lem:dihedral_facts_about_blocks}, we have that $(-) \star \id_{B_{{}_s \underline{r}}^{\scrL}}$ is an equivalence of categories.
    Therefore, we conclude that $\Psi_{\scrL}^{\iota, \textnormal{diag}} (\beta_{s',t'} \circ \beta_{t',s'}) = \Psi_{\scrL}^{\iota, \textnormal{diag}} (\JW_{t',s'}^1)$.
\end{proof}

\subsubsection{Functor to Bimodules}

The goal of this section is to construct a 2-functor
\[ \Upsilon_{\Abe} : \EWmon{}{}^{\BS} (\fr{h}, W, \fr{o}) \to \Amon{}{}^{\BS} (\fr{h}, W, \fr{o})\]
which extends the 2-functor $\Upsilon_{\{s,t\}} : \EWmon{}{}^{\BS} (\fr{h}, \{s,t\}, \fr{o}) \to \Amon{}{}^{\BS} (\fr{h}, W, \fr{o})$ constructed in \S \ref{subsubsec:universal_category}.
To define $\Upsilon_{\Abe}$, we only need to specify the image of a $2m$-valent vertex,
\[\Upsilon_{\Abe} \left( 
  % \tikzsetnextfilename{#1}
  \tikzstyle{every picture}=[tikzfig]
  \input{./figures/mst_valent4.tikz}
 \right) \coloneq \beta_{s,t} : B_{{}_s \underline{m}}^{\scrL} \to B_{{}_t \underline{m}}^{\scrL}.\]

\begin{claim}\label{claim:dihedral_Phi_functor}
    $\Upsilon_{\Abe} : \EWmon{}{}^{\BS} (\fr{h}, W, \fr{o}) \to \Amon{}{}^{\BS} (\fr{h}, W, \fr{o})$ is well-defined.
\end{claim}

The proof of Claim \ref{claim:dihedral_Phi_functor} will take the rest of the section.
Note that there is a special case when we know $\Upsilon_{\Abe}$ is already defined. Namely, it is shown in \cite{Abe19} that when $\fr{o} = 1$, then $\Upsilon_{\Abe}$ is well-defined.

\noindent \emph{Step 0: Compatibility with Block Translation}

We will write
\[\Upsilon_{\scrL} : \EWmon{\scrL}{-}^{\BS} (\fr{h}, W, \fr{o}) \to \Amon{\scrL}{-}^{\BS} (\fr{h}, W, \fr{o})\]
for the restriction of $\Upsilon_{\Abe}$ such that the first monodromy parameter is fixed. Let $\beta \in \uW{\scrL}{\scrL'}$. We can also write
\[\Upsilon_{\scrL}^{\beta} : \EWmon{\scrL}{\scrL'}^{\BS, \beta} (\fr{h}, W, \fr{o}) \to \Amon{\scrL}{\scrL'}^{\BS, \beta} (\fr{h}, W, \fr{o})\]
for the restriction of $\Upsilon_{\Abe}$ to the $\beta$-block. 
It is clear from its definition that $\Upsilon_{\Abe}$ indeed takes objects in the $\beta$-block on the diagrammatic side to objects in the $\beta$-block on the algebraic side.

The following lemma follows readily from definitions.

\begin{lemma}\label{lem:Phi_and_block_translation}
    Let $\beta \in \uW{\scrL}{\scrL'}$. Let $\uw^{\beta}$ be an endo-reduced expression for $w^{\beta}$.  
    The following diagram (strictly) commutes
    \[
    \begin{tikzcd}
        {\EWmon{\scrL}{\scrL}^{\BS, \circ} (\fr{h}, W, \fr{o})} \arrow[d, "(-) \star B_{\uw^{\beta}}^{\scrL}"] \arrow[r, "\Upsilon_{\scrL}^{\circ}"] & {\Amon{\scrL}{\scrL}^{\BS, \circ} (\fr{h}, W, \fr{o})} \arrow[d, "(-) \star B_{\uw^{\beta}}^{\scrL}"] \\
        {\EWmon{\scrL}{\scrL'}^{\BS, \beta} (\fr{h}, W, \fr{o})} \arrow[r, "\Upsilon_{\scrL}^{\beta}"]      & {\Amon{\scrL}{\scrL'}^{\BS, \beta} (\fr{h}, W, \fr{o}).}                 
        \end{tikzcd}\]
\end{lemma}

\noindent\emph{Step 1: Rotation invariance}

Note that $\beta_{s,t} : B_{{}_s \underline{m}}^{\scrL} \to B_{{}_t \underline{m}}^{\scrL}$ is uniquely determined by the condition that $\beta_{s,t} (u_{{}_s \underline{m}}) = u_{{}_t \underline{m}}$.
As a result,  it is easy to check that $\beta_{s,t}$ is invariant under rotations. Therefore, $\Upsilon_{\Abe}$ is well-defined on isotopy classes of monodromic Elias--Williamson graphs.
 
\noindent\emph{Step 2: Two-color Associativity}

Take $n = 2m-1$ and $n' = 2v-1$.
By rotation invariance, it suffices to prove the following
\begin{equation}\label{eq:Phi_and_2ca}
    \Upsilon_{\scrL}^{\circ} \left( 
  % \tikzsetnextfilename{#1}
  \tikzstyle{every picture}=[tikzfig]
  \input{./figures/Phi_2c_1.tikz}
 \right) = \Upsilon_{\scrL}^{\circ} \left( 
  % \tikzsetnextfilename{#1}
  \tikzstyle{every picture}=[tikzfig]
  \input{./figures/Phi_2c_2.tikz}
 \right)
\end{equation}
The basic idea is that we can check the relation holds after applying endoscopy. Endoscopy will translate (\ref{eq:Phi_and_2ca}) into the 2-color associativity relation for the endoscopic group, which holds from work in \cite{Elib}.
Write $s' = s$ and $t' = {}_t k$ for some $k$. We use green for $s'$ and purple for $t'$ in the calculations below.
\begin{align*}
    \Upsilon_{\scrL}^{\circ} \left( 
  % \tikzsetnextfilename{#1}
  \tikzstyle{every picture}=[tikzfig]
  \input{./figures/Phi_2c_1.tikz}
 \right) &=  \Psi_{\scrL}^{\circ, \alg} \Upsilon_1^{\circ} (\Upsilon_1^{\circ})^{-1}  \left(\Psi_{\scrL}^{\circ, \alg}\right)^{-1} \Upsilon_{\scrL}^{\circ} \left( 
  % \tikzsetnextfilename{#1}
  \tikzstyle{every picture}=[tikzfig]
  \input{./figures/Phi_2c_1.tikz}
 \right) \\
    &= \Psi_{\scrL}^{\circ, \alg} \Upsilon_1^{\circ} \left( 
  % \tikzsetnextfilename{#1}
  \tikzstyle{every picture}=[tikzfig]
  \input{./figures/Phi_2c_3.tikz}
\right) \\
    &= \Psi_{\scrL}^{\circ, \alg} \Upsilon_1^{\circ} \left( 
  % \tikzsetnextfilename{#1}
  \tikzstyle{every picture}=[tikzfig]
  \input{./figures/Phi_2c_4.tikz}
\right) \\
    &= \Upsilon_{\scrL}^{\circ} \left( 
  % \tikzsetnextfilename{#1}
  \tikzstyle{every picture}=[tikzfig]
  \input{./figures/Phi_2c_2.tikz}
 \right).
\end{align*}

\noindent\emph{Step 3: Elias--Jones--Wenzl}

First consider the case where $\rk (W_{s,t}^{\scrL}) < 2$.
In this case, $\JW_{s,t}^{\scrL} = \id_{B_{{}_s m}^{\scrL}}$.
On the other hand, in $\Amon{}{}^{\BS} (\fr{h}, W, \fr{o})$, we have that space of degree 0 endomorphisms of $B_{{}_s m}^{\scrL}$ is a rank-1 $\k$-module (Theorem \ref{thm:abe_dl_are_basis}).
Moreover, $\beta_{t,s} \circ \beta_{s,t}$ and $\id$ both take the 1-tensor to itself. Therefore, we must have that Jones--Wenzl relation is invariant under $\Upsilon_{\Abe}$. 

We can then assume that $\rk (W_{s,t}^{\scrL}) = 2$. 
Let $\beta \in \uW{\scrL}{\scrL'}$ be the block containing $\JW_{s,t}^{\scrL}$.
Let $\uw^{\beta}$ be an endo-reduced expression for $w^{\beta}$. We set 
\[ T_{\beta}^{\diag} \coloneq (-) \star B_{\uw^{\beta}}^{\scrL} : \EWmon{\scrL}{\scrL}^{\BS, \circ} \to \EWmon{\scrL}{\scrL'}^{\BS, \beta} \qquad\text{and}\qquad T_{\beta}^{\alg} \coloneq (-) \star B_{\uw^{\beta}}^{\scrL} : \Amon{\scrL}{\scrL}^{\BS, \circ} \to \Amon{\scrL}{\scrL'}^{\BS, \beta}.\]
By block minimality, we can give a slight modification to the definition of the Jones--Wenzl morphism,
\[\JW_{s,t}^{\scrL} = T_{\beta}^{\diag} \Sigma_{\scrL} (\JW_A ({}_{s'} (v-1))).\]
We can then compute
\begin{align*}
    \Upsilon_{\scrL} (\JW_{s,t}^{\scrL}) &= \Upsilon_{\scrL} T_{\beta}^{\diag} \Sigma_{\scrL} (\JW_A ({}_{s'} (v-1))) \\
    &= T_{\beta}^{\alg} \Upsilon_{\scrL} \Sigma_{\scrL} (\JW_A ({}_{s'} (v-1))) \\
    &= T_{\beta}^{\alg} \Upsilon_{\scrL} \Psi_{\scrL}^{\diag} \Sigma (\JW_A ({}_{s'} (v-1))) \\
    &= T_{\beta}^{\alg} \Psi_{\scrL}^{\alg} \Upsilon_{S_{\scrL}^{\circ}} \Sigma (\JW_A ({}_{s'} (v-1))) \\ 
    &= T_{\beta}^{\alg} \Psi_{\scrL}^{\alg} \Upsilon_{S_{\scrL}^{\circ}} \JW_{s',t'}^1 \\
    &= T_{\beta}^{\alg} \Psi_{\scrL}^{\alg} (\beta_{t',s'} \circ \beta_{s', t'}) \\
    &= T_{\beta}^{\alg} (\beta (\iota ({}_{t'} \underline{v}), \iota ({}_{s'} \underline{v})) \circ \beta (\iota ({}_{s'} \underline{v}), \iota ({}_{t'} \underline{v}))) \\
    &= \beta_{t,s} \circ \beta_{s,t}.
\end{align*}
The second equality is from Lemma \ref{lem:Phi_and_block_translation}. The fourth equality is from Proposition \ref{prop:one_color_endoscopy}. The sixth equality follows from the Jones--Wenzl relation for non-monodromic Bott--Samelson bimodules (a consequence of $\Upsilon_1^{\circ}$ being well-defined).
Finally, the last equality is an application of block minimality. This completes the proof of Claim \ref{claim:dihedral_Phi_functor}.

\subsection{Presentation of the Monodromic Hecke Category}

We will now return to the general case where $(W,S)$ is an arbitrary Coxeter system and $\fr{o}$ is a $W$-set.

We are almost ready to give the definition of the monodromic Hecke category in full generality.
Before doing so, we first need to impose a technical constraint on our Coxeter systems to ensure that the three-color relations are computable.
The constraint is essentially identical to the usual constraint that $H_3$ does not appear as a parabolic subgroup for the non-monodromic Hecke category.

\begin{assumption}\label{ass:no_H3_copy}
    For all type $H_3$ standard parabolic subgroups $W_I$ of $W$ and all $\scrL \in \fr{o}$, we have that $W_I \cap W_{\scrL}^{\circ}$ is a proper subgroup of $W_I$.
\end{assumption}

\subsubsection{Definition}

\begin{definition}\label{def:diag_hecke_cat}
    Let $(W,S)$ be a Coxeter system and $\fr{o}$ be a $W$-set such that Assumption \ref{ass:no_H3_copy} is satisfied.
    Let $\fr{h}$ be a reflection-balanced reflection-stable Abe realization for $W$.
    The \emph{diagrammatic monodromic Hecke category}, denoted $\EWmon{}{}^{\BS} (\fr{h}, W, \fr{o})$, is the $\k$-linear strict 2-category defined as follows.
    The object set is $\fr{o}$. The 1-morphisms are tensor generated by 1-morphisms $B_s^{\scrL} : \scrL \stackrel{s}{\to} \scrL s$ for all $\scrL \in \fr{o}$, $s \in S$.
    A general 1-morphism is then of the form $B_{\ux}^{\scrL} : \scrL \stackrel{\ux}{\to} \scrL \ux$ for $\ux \in \Exp (W)$. 
    The 2-morphism space $\TwoHom_{\EWmon{}{}^{\BS}} (B_{\ux}^{\scrL}, B_{\uy}^{\scrL})$ is the free $\k$-module with basis ${}_{\scrL} \VSGraph_{\scrL'} (\ux, \uy)$ modulo the following relations:
    
        \emph{(One-Color Relations)} The relations of the one-color category (\ref{defn:one_color_cat}).
       
        \emph{(Two-color Associativity)} The two-color associativity and Elias--Jones--Wenzl relations of (\ref{def:two_color_cat}).

        \emph{(Zamolodchikov Equations)} There is a 3-color relation for each finitary rank-3 subset $I = \{s,t,u\}$ of $S$. 
        The specific form depends on the type of the parabolic subgroup $W_I$ along with its action on $\fr{o}$.
        By the classification of finite Coxeter groups the finite rank three parabolic can only be one of the following types: $A_1 \times I_2 (m)$, $A_3$, $B_3$, or $H_3$.
        \begin{equation}
            
  % \tikzsetnextfilename{#1}
  \tikzstyle{every picture}=[tikzfig]
  \input{./figures/zam_relns.tikz}
 
        \end{equation}
    \end{definition}

    \begin{remark}\label{rem:three_color_relns} 
        \begin{enumerate}
            \item   Assumption \ref{ass:no_H3_copy} is needed since the non-monodromic Hecke category still lacks a computation of the lower order terms needed for the corresponding three-color relation to hold. Further discussion on this topic can be found in \cite[\S 3.6]{EW17}.
                    Finding the $H_3$ Zamolodchikov equation in the non-monodromic Hecke category would also complete the description in the monodromic Hecke category.
            \item   It is somewhat surprising that lower order terms appear in the Zamolodchikov equations for type $B_3$ with an endoscopy subgroup of type $A_3$.
                    We can somewhat motivate this by the monodromic-endoscopic equivalence. 
                    After endoscopy, the monodromic three-color relations becomes $(\leq 3)$-color relations on the endoscopic side.
                    The lower order terms can only appear when the endoscopic subgroup also has rank 3. This happens quite rarely. 
                    \begin{enumerate}
                        \item In type $A_1 \times I_2(m)$, all endoscopic subgroups are of the following types: $A_1 \times I_2 (v)$ and $I_2 (v)$ where $0 \leq v \leq m$ divides (by definition, we take $I_2 (0) = 1$ and $I_2 (1) = A_1$).
                        \item In type $A_3$, all nontrivial endoscopic subgroups are of the following types: $A_3$, $A_2$, $A_1 \times A_1$, and $A_1$.
                        \item In type $B_3$, all nontrivial endoscopic subgroups are of the following types: $B_3$, $A_3$, $A_1 \times A_1 \times A_1$, $A_1 \times A_2$, $A_1 \times A_1$, $A_2$,  $B_2$, and $A_1$. 
                        \item In type $H_3$, all nontrivial endoscopic subgroups are of the following types: $H_3$, $A_1 \times A_1 \times A_1$, $I_2 (5)$, $A_2$, $A_1 \times A_1$, and $A_1$.
                    \end{enumerate}
                    One way to interpret the lack of lower terms is that in almost every case the canonical orientation of the rex graph on the monodromic side aligns with the canonical orientation of the rex graph on the endoscopic side.
                    There is exactly two places where this fails: (i) in type $H_3$ when no canonical orientation exists, and (ii) in type $B_3$ when the induced orientation on the rex graph of $A_3$ is not the canonical one. 
        \end{enumerate}
    \end{remark}

    \begin{remark}
        If $\fr{o} = 1$, then $\EWmon{}{}^{\BS} (\fr{h}, W, 1)$ is the non-monodromic diagrammatic Hecke category of \cite{EW} viewed as a 2-category with one object.
    \end{remark}

    \begin{remark}
        It should be possible to define $\EWmon{}{}^{\BS} (\fr{h}, W, \fr{o})$ even when $\fr{h}$ is not reflection-balanced.
        In particular, since $\fr{h}$ is reflection-stable, by \cite[Theorem B]{Hazi} one can still make sense of the Jones--Wenzl projectors necessary to define the Jones--Wenzl morphisms.
        When $\fr{h}$ is not reflection-balanced, these Jones--Wenzl morphisms fail to be rotation invariant, but instead are rotation invariant up to an invertible scalar.
        To correct this, one must introduce two different $2m$-valent vertices, and modify the relations accordingly. Such a correction should be very close to the modifications done in \cite[\S5.6]{EW}.
    \end{remark}

\subsubsection{Associated 1-Categories}
    For two monodromy parameters $\scrL, \scrL' \in \fr{o}$, we can consider the category 
    \[\EWmon{\scrL}{\scrL'}^{\BS} (\fr{h}, W, \fr{o}) \coloneq \Hom_{\EWmon{}{}^{\BS}} (\scrL, \scrL').\]
    Let $\beta \in \uW{\scrL}{\scrL'}$ be a block and define a full subcategory $\EWmon{\scrL}{\scrL'}^{\BS, \beta} (\fr{h}, W, \fr{o})$ of $\EWmon{\scrL}{\scrL'}^{\BS} (\fr{h}, W, \fr{o})$, called the \emph{$\beta$-block}, consisting of objects $B_{\uw}^{\scrL}$ where $\uw \in \Exp (W)$ is an expression for an element in $\beta$.

    The following is a generalization of Lemma \ref{lem:dihedral_facts_about_blocks} for arbitrary Coxeter systems. The proof is identical, and as such omitted.

    \begin{lemma}\label{lem:diag_block_props}
        Let $\scrL, \scrL', \scrL'' \in \fr{o}$.
    \begin{enumerate}
        \item There is a decomposition of categories
    \[\EWmon{\scrL}{\scrL'}^{\BS} (\fr{h}, W, \fr{o}) = \bigsqcup_{\beta \in \uW{\scrL}{\scrL'}} \EWmon{\scrL}{\scrL'}^{\BS, \beta} (\fr{h}, W, \fr{o}).\]
        \item Let $\beta \in \uW{\scrL}{\scrL'}$ and $\gamma \in \uW{\scrL'}{\scrL''}$. Horizontal composition is compatible with block multiplication, i.e.,
        \[\EWmon{\scrL}{\scrL'}^{\BS, \beta} (\fr{h}, W, \fr{o}) \star \EWmon{\scrL'}{\scrL''}^{\BS, \beta \gamma} (\fr{h}, W, \fr{o}) \subseteq \EWmon{\scrL'}{\scrL''}^{\BS, \beta \gamma} (\fr{h}, W, \fr{o}).\]
        \item Let $\beta \in \uW{\scrL}{\scrL'}$ and let $\uw^{\beta}$ be an endo-reduced expression for $w^{\beta}$. There is an equivalence of categories
        \[ (-) \star \id_{B_{\uw^{\beta}}^{\scrL}} : \EWmon{\scrL}{\scrL}^{\BS, \circ} (\fr{h}, W, \fr{o})  \stackrel{\sim}{\to} \EWmon{\scrL}{\scrL'}^{\BS, \beta} (\fr{h}, W, \fr{o}).\]
    \end{enumerate}
    \end{lemma}

\subsubsection{From Diagrammatics to Bimodules}

The goal of this section is to define a 2-functor 
\[\Upsilon_{\Abe} : \EWmon{}{}^{\BS} (\fr{h}, W, \fr{o}) \to \Amon{}{}^{\BS} (\fr{h}, W, \fr{o}).\]
On objects, $\Upsilon (\scrL) = \scrL$, and on 1-morphisms $\Upsilon (B_{\uw}^{\scrL}) = B_{\uw}^{\scrL}$ for $\uw \in \Exp (W)$.
We define $\Upsilon$ on 2-morphisms by specifying where the generating 2-morphisms are sent,
\begin{align*}
    \Upsilon \left( 
  % \tikzsetnextfilename{#1}
  \tikzstyle{every picture}=[tikzfig]
  \input{./figures/poly2.tikz}
 \right) &\coloneq  f & \Upsilon \left(
  % \tikzsetnextfilename{#1}
  \tikzstyle{every picture}=[tikzfig]
  \input{./figures/enddot2.tikz}
 \right) &\coloneq \eta^s , & \Upsilon \left(
  % \tikzsetnextfilename{#1}
  \tikzstyle{every picture}=[tikzfig]
  \input{./figures/startdot2.tikz}
 \right) &\coloneq \epsilon^s,  & \Upsilon \left(
  % \tikzsetnextfilename{#1}
  \tikzstyle{every picture}=[tikzfig]
  \input{./figures/mult2.tikz}
 \right) &\coloneq  \mu^s, \\
    \Upsilon \left(
  % \tikzsetnextfilename{#1}
  \tikzstyle{every picture}=[tikzfig]
  \input{./figures/comult2.tikz}
 \right) &\coloneq \nu^s  & \Upsilon \left(
  % \tikzsetnextfilename{#1}
  \tikzstyle{every picture}=[tikzfig]
  \input{./figures/cap2.tikz}
 \right) &\coloneq \cap^s , & \Upsilon \left(
  % \tikzsetnextfilename{#1}
  \tikzstyle{every picture}=[tikzfig]
  \input{./figures/cup2.tikz}
 \right) &\coloneq  \cup^s, & \Upsilon \left( 
  % \tikzsetnextfilename{#1}
  \tikzstyle{every picture}=[tikzfig]
  \input{./figures/mst_valent3.tikz}
 \right) &\coloneq \beta_{s,t}.   
\end{align*}

\begin{claim}\label{claim:Phi_well_defined}
    $\Upsilon$ is a well-defined 2-functor.
\end{claim}
\begin{proof}
    We have already shown that $\Upsilon$ is invariant under the 1-color relations in Proposition \ref{prop:one_color_soergel_functor} and under the 2-color relations in Claim \ref{claim:dihedral_Phi_functor}.
    As a result, it only remains to check that $\Upsilon$ is invariant under the 3-color relations.
    
    We compute these case-by-case for types $A_3, B_3$, and $H_3$. There are 30 possible monodromy labeling in type $A_3$, 98 possible monodromy labeling in type $B_3$, and 164 possible monodromy labeling in type $H_3$. 
    Most of these cases, one can check by constraining Hom spaces or from existing literature that  $\Upsilon$ is invariant under the 3-color relations.
    The one exception is when $B_3$ has an endoscopic subgroup of type $A_3$. This will be checked by hand using algebraic endoscopy.
    
    In type $A_1 \times I_2 (m)$, the possible endoscopic subgroups are of the form $A_1 \times I_2 (v)$ and $I_2 (v)$ where $0 \leq v \leq m$ divides $m$.
    When the coloring of the 3-color relation arises from an endoscopic subgroup of rank $\leq 2$, the invariance under $\Upsilon$ follows from constraints on Hom spaces in low rank.
    When the endoscopic subgroup is rank 3, then the 3-color relation for $A_1 \times I_2 (m)$ maps under $\Upsilon$ to the 3-color relation for $A_1 \times I_2 (v)$.

    The complete proof is covered in Appendix \ref{apdx:three_color}. 
    Therein, we also show how the possible monodromy parameters can be computed, some important examples, and how $\Upsilon$-invariance can be computed. 
\end{proof}

\subsection{Double Leaves}

We will provide a diagrammatic counterpart to the double leaves basis constructed for Bott--Samelson bimodules. In the non-monodromic setting these were first introduced in \cite{EW}.
Throughout we will fix an $\LL$-datum as in \S \ref{subsec:rigidification_data}.

\subsubsection{Definition and the Algorithm}\label{subsubsec:diag_ll_section}

Let $\ux = (s_1, \ldots, s_k)$ be an expression in $W$, and let $\ue$ be an $\scrL$-monodromic subexpression of $\ux$. For each $1 \leq i \leq k$, let $x_i$ be the element realizing $\ux_{\leq i}$.
Let $w_i = \ux_{\leq i}^{\ue_{\leq i}}$.
We write $\uw_i = \rex (w_i)$ for the reduced expression of $x_i$ fixed by the $\LL$-datum. 
The condition that $\ue$ be monodromic ensures that $\scrL x_i = \scrL w_i$.

The construction of the light leaves is identical to algebraic version after obvious replacements are made. 
Nonetheless, for the sake of exposition, we will detail the algorithm.

We construct the light leaf $\LL_{\ux, \ue}^{\scrL} : B_{\ux}^{\scrL} \to B_{\uw}^{\scrL}$ inductively on the length $k$ of $\ux$.
The base case is covered by $\LL_{\emptyset, \emptyset}^{\scrL} : B_{\emptyset}^{\scrL} \to B_{\emptyset}^{\scrL}$ being the identity map. 
Suppose that we have already constructed a map 
\[\LL_{i-1}^{\scrL} := \LL_{\ux_{\leq i-1}, \ue_{\leq i-1}}^{\scrL} : B_{\ux_{\leq i-1}}^{\scrL} \to B_{\uw_{i-1}}^{\scrL}.\]
In a moment, we will define a map 
\[\phi_i : B_{\ux_{\leq i-1} s_i}^{\scrL} \to B_{\uw_{i}}^{\scrL}.\] 
We then define 
\[\LL_i^{\scrL} = \phi_i \circ ( \LL_{i-1}^{\scrL} \star \id_{B_s^{\scrL w_{i-1}}}).\]

The definition of $\phi_i$ depends on a decoration $\dec_i (\uw, \ue)$.
We will now explain $\phi_i$ for each decoration. While the definition does not actually depend on whether $i \in K (\uw, \scrL)$, it is still useful to show the face colorings for both cases. 
The figure below describes $\phi_i$ in all such cases. Here red is used for the simple reflection $s_i$. The orange corresponds to $\scrL$, the green corresponds to $\scrL w_{i-1}$, and the purple corresponds to $\scrL w_i$. 
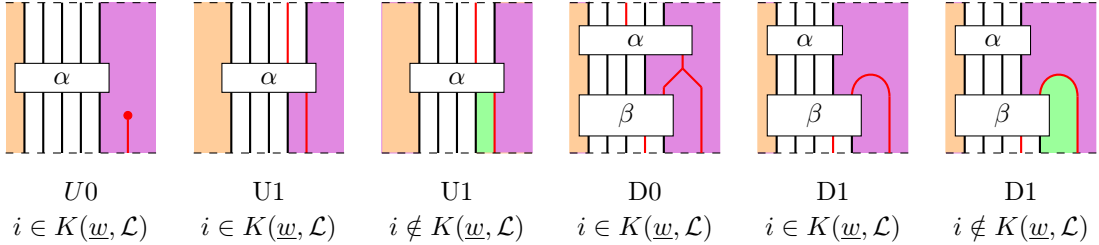
\begin{equation*}
    
  % \tikzsetnextfilename{#1}
  \tikzstyle{every picture}=[tikzfig]
  \input{./figures/light_leaves.tikz}

\end{equation*}

The $\alpha$ indicates the rex move to $\uw_i$ and $\beta$ indicates the rex move from $\uw_{i-1}$ to $\rex_s (w_{i-1})$. Both of these rex moves are uniquely determined after choosing the corresponding path in the rex graph as in \S \ref{subsec:rigidification_data}.

\subsubsection{The Basis Theorem}

\begin{definition}
    Let $\ux, \uy \in \Exp (W)$ such that $\scrL \ux = \scrL' = \scrL \uy$. Let $\ue \subset \ux$ and $\uf \subset \uy$ be monodromic subexpressions such that $\uw^{\ue} = w = \uy^{\uf}$.
    A \emph{double leaf} is a composition of the form
    \[ \dLL_{\ue, \uf}^{w, \scrL} := \overline{\LL}_{\uy, \uf}^{\scrL} \circ \LL_{\ux, \ue}^{\scrL} : B_{\ux}^{\scrL} \to B_{\uy}^{\scrL}\]
    where $\overline{\LL}_{\uy, \uf}^{\scrL} : B_{\uw}^{\scrL} \to B_{\uy}^{\scrL}$ denotes the horizontal reflection of $\LL_{\uy, \uf}^{\scrL}$.
\end{definition}

\begin{theorem}\label{thm:dll_is_a_basis}
    Let $\ux, \uy \in \Exp (W)$ such that $\scrL \ux = \scrL' = \scrL \uy$.
    Let $\dLL_{\ux, \uy}^{\scrL}$ denote the set of all $\scrL$-monodromic double leaves from $B_{\ux}^{\scrL}$ to $B_{\uy}^{\scrL}$ with respect to a fixed $\LL$-datum.
    Then $\dLL_{\ux, \uy}^{\scrL}$ is a left graded $R$-module basis for $\Hom_{\EWmon{\scrL}{\scrL'}^{\BS}} (B_{\ux}^{\scrL}, B_{\uy}^{\scrL})$.
\end{theorem}

The proof of Theorem \ref{thm:dll_is_a_basis} is rather involved. Showing linear independence is easy and will be covered in the following lemma. The difficult part is showing that the double leaves span the Hom space.
We have deferred the proof of this component until the next section.

\begin{lemma}\label{lem:dl_are_lin_indep_for_EW}
    Let $\ux, \uy \in \Exp (W)$ such that $\scrL \ux = \scrL' = \scrL \uy$.
    Then $\dLL_{\ux, \uy}^{\scrL}$ is linearly independent over $R$ as a subset of $\Hom_{\EWmon{\scrL}{\scrL'}^{\BS}} (B_{\ux}^{\scrL}, B_{\uy}^{\scrL})$.
\end{lemma}
\begin{proof}
    By construction, the functor $\Upsilon : \EWmon{}{}^{\BS} (\fr{h}, W, \fr{o}) \to \Amon{}{}^{\BS} (\fr{h}, W, \fr{o})$ takes double leaves to double leaves (see Proposition \ref{prop:duality_and_str_mors} and Remark \ref{rem:upside_down_LL}).
    The result then follows from Theorem \ref{thm:abe_dl_are_basis}.
\end{proof}

\begin{remark}
    Let $\fr{o} = W$ where $W$ acts on $\fr{o}$ by right multiplication.
    For all $\scrL, \scrL' \in \fr{o}$, there is an equivalence of categories
    \[\textnormal{perm}_{\scrL}^{\scrL'} : \EWmon{\scrL}{-}^{\BS} (\fr{h}, W, \fr{o}) \to \EWmon{\scrL'}{-}^{\BS} (\fr{h}, W, \fr{o})\]
    obtained by the permutation of the monodromy labelings $\fr{o} \to \fr{o}$ given by multiplication by the unique element $w$ such that $\scrL' = \scrL w$.
    We can then equip $\EWmon{e}{-}^{\BS} (\fr{h}, W, \fr{o})$ with the structure of a monoidal category via the composition
    \[\EWmon{e}{\scrL}^{\BS} (\fr{h}, W, \fr{o}) \times \EWmon{e}{-}^{\BS} (\fr{h}, W, \fr{o}) \stackrel{\textnormal{perm}_{\scrL}^{e}}{\cong} \EWmon{e}{\scrL}^{\BS} (\fr{h}, W, \fr{o}) \times \EWmon{\scrL}{-}^{\BS} (\fr{h}, W, \fr{o}) \to \EWmon{e}{-}^{\BS} (\fr{h}, W, \fr{o}).\]
    It is easy to check from definitions, that $\EWmon{e}{-}^{\BS} (\fr{h}, W, \fr{o})$ is equivalent as a monoidal category to the diagrammatic category of standard bimodules defined in \cite[\S 4.1]{EW}.

    Theorem \ref{thm:dll_is_a_basis} implies the following Hom space calculation,
    \[\Hom_{\EWmon{e}{-}^{\BS}} (B_{\ux}^{\scrL}, B_{\uy}^{\scrL}) \cong \begin{cases} R & \ev (\ux) = \ev (\uy), \\ 0 & \textnormal{otherwise}. \end{cases}\]
    This essentially gives a diagrammatic proof of \cite[Theorem 6.6]{EW0}.
\end{remark}

\subsubsection{Monodromic Abe Theorem}

As a consequence of the double leaves basis for the category of Bott--Samelson bimodules and the diagrammatic category, we deduce that $\Upsilon$ is an equivalence of categories.
This result generalizes \cite[Theorem 5.6]{Abe19} and \cite[Theorem 3.15]{Abe21}. 

\begin{theorem}\label{thm:mon_Abe}
    The 2-functor 
    \[ \Upsilon : \EWmon{}{}^{\BS} (\fr{h}, W, \fr{o}) \to \Amon{}{}^{\BS} (\fr{h}, W, \fr{o})\] 
    is an equivalence of 2-categories.
\end{theorem}
\begin{proof}
    It is clear that $\Upsilon$ is essentially surjective.
    Moreover, $\Upsilon$ takes double leaves to double leaves.
    The result then follows from Theorem \ref{thm:abe_dl_are_basis} and Theorem \ref{thm:dll_is_a_basis}.
\end{proof}

\subsection{Diagrammatic Endoscopy}

In this section, we will establish a monodromic-endoscopic equivalence for the diagrammatic Hecke categories which serves as a diagrammatic counterpart to Theorem \ref{thm:diag_endoscopy}.

Let $\scrL, \scrL', \scrL'' \in \fr{o}$, $\beta \in \uW{\scrL}{\scrL'}$, and $\gamma \in \uW{\scrL'}{\scrL''}$. There is an equivalence of monoidal categories
\[{}^{\beta} (-) : \EWmon{1}{1}^{\oplus} (\fr{h}, W_{\scrL'}^{\circ}, 1) \to \EWmon{1}{1}^{\oplus} (\fr{h}, W_{\scrL}^{\circ}, 1) \]
induced by the isomorphism of Coxeter groups $W_{\scrL'}^{\circ} \to W_{\scrL}^{\circ}$ defined by $w \mapsto w^{\beta} w w^{\beta, -1}$.
It follows from definitions that ${}^{\beta} ( {}^{\gamma} (-)) \cong {}^{\beta \gamma} (-)$.

Let $\ux$ be an endo-reduced expression for $w^{\beta}$. 
There is an equivalence of monoidal categories
\[{}^{\beta} (-) : \EWmon{\scrL'}{\scrL'}^{\oplus, \circ} (\fr{h}, W, \fr{o}) \to \EWmon{\scrL}{\scrL}^{\oplus, \circ} (\fr{h}, W, \fr{o}) \]
defined by $M \mapsto B_{\ux}^{\scrL} \star M \star B_{\overline{\ux}}^{\scrL'}$. Given another endo-reduced expression $\uy$ for $w^{\beta}$, there is a unique isomorphism $B_{\ux}^{\scrL} \to B_{\uy}^{\scrL}$ composed solely of cups, caps, and $2m$-valent vertices (with no coefficient).
This is a consequence of the 3-color relations and the Jones--Wenzl morphisms being trivial. In particular, ${}^{\beta} (-)$ does not depend on the choice of endo-reduced expression up to natural isomorphism.
It then follows from this uniqueness that ${}^{\beta} ( {}^{\gamma} (-)) \cong {}^{\beta \gamma} (-)$.

The following lemma can easily be checked from definitions.
\begin{lemma}\label{lem:Phi_and_block_conj}
    There are natural isomorphisms making the following diagram commute for each $\beta \in \uW{\scrL}{\scrL'}$.
    \[
        \begin{tikzcd}
            {\EWmon{1}{1}^{\oplus} (\fr{h}, W_{\scrL'}^{\circ}, 1)} \arrow[d, "{}^{\beta} (-)"'] \arrow[r, "\Upsilon_{\Abe}"] & {\Amon{1}{1}^{\oplus} (\fr{h}, W_{\scrL'}^{\circ}, 1)} \arrow[d, "{}^{\beta} (-)"'] \\
            {\EWmon{1}{1}^{\oplus} (\fr{h}, W_{\scrL}^{\circ}, 1)} \arrow[r, "\Upsilon_{\Abe}"]                             & {\Amon{1}{1}^{\oplus} (\fr{h}, W_{\scrL}^{\circ}, 1)}                            
            \end{tikzcd}
    \]
    Moreover, these natural transformations are compatible with respect to block composition.
\end{lemma}

We will construct a 2-category $\EW^{\oplus} (\fr{h}, M^{\fr{o}} (W))$, called the \emph{endoscopic Elias--Williamson diagrammatic 2-category} as follows. 
The object set of $\EW^{\oplus} (\fr{h}, M^{\fr{o}} (W))$ is $\fr{o}$.
The morphism category from $\scrL$ to $\scrL'$ of $\EW^{\oplus} (\fr{h}, M^{\fr{o}} (W))$ is given by $\bigoplus_{\beta \in \uW{\scrL}{\scrL'}} \EW^{\oplus} (\fr{h}, \beta)$ where $\EW^{\oplus} (\fr{h}, \beta) \coloneq \EWmon{1}{1}^{\oplus} (\fr{h}, W_{\scrL}^{\circ}, 1)$.
The vertical composition is induced from the composition in $\EW^{\oplus} (\fr{h}, \beta)$. Let $\beta \in \uW{\scrL}{\scrL'}$ and $\gamma \in \uW{\scrL'}{\scrL''}$. 
We define the horizontal composition, denoted $\star$ as the bifunctor
\begin{align*}
    \EW^{\oplus} (\fr{h}, \beta) \times \EW^{\oplus} (\fr{h}, \gamma) &= \EWmon{1}{1}^{\oplus} (\fr{h}, W_{\scrL}^{\circ}, 1) \times \EWmon{1}{1}^{\oplus} (\fr{h}, W_{\scrL'}^{\circ}, 1) \\
    &\stackrel{\id \times {}^{\beta} (-) }{\longrightarrow} \EWmon{1}{1}^{\oplus} (\fr{h}, W_{\scrL}^{\circ}, 1) \times \EWmon{1}{1}^{\oplus} (\fr{h}, W_{\scrL}^{\circ}, 1) \\
    &\stackrel{\star}{\to} \EWmon{1}{1}^{\oplus} (\fr{h}, W_{\scrL}^{\circ}, 1) \\
    &= \EW^{\oplus} (\fr{h}, \beta \gamma).
\end{align*}

\begin{theorem}[Diagrammatic Monodromic-Endoscopic Equivalence]\label{thm:diag_endoscopy}
    There is an equivalence of 2-categories
    \[ \Psi^{\diag} : \EW^{\oplus} (\fr{h}, M^{\fr{o}} (W)) \stackrel{\sim}{\to} \EWmon{}{}^{\oplus} (\fr{h}, W, \fr{o}).\]
\end{theorem}
\begin{proof}
    The same argument used in Theorem \ref{thm:alg_endoscopy} using Lemma \ref{lem:Phi_and_block_conj} shows that $\Upsilon_{\Abe}$ gives rise to an equivalence of 2-categories
    \[ \EW^{\oplus} (\fr{h}, M^{\fr{o}} (W)) \stackrel{\sim}{\to} \scrA^{\oplus} (\fr{h}, M^{\fr{o}} (W)).\]
    The theorem then follows from the composition of 2-functors,
   \[
   \EW^{\oplus} (\fr{h}, M^{\fr{o}} (W)) \stackrel{\sim}{\to} \scrA^{\oplus} (\fr{h}, M^{\fr{o}} (W)) \stackrel{\sim}{\to} \Amon{}{}^{\oplus} (\fr{h}, W, \fr{o}) \stackrel{\sim}{\to} \EWmon{}{}^{\oplus} (\fr{h}, W, \fr{o}),
   \]
   where the second morphism is Theorem \ref{thm:alg_endoscopy} and the third morphism is Theorem \ref{thm:mon_Abe}. 
\end{proof}

\begin{remark}
    There is no convenient diagrammatic description of the 2-functor in Theorem \ref{thm:diag_endoscopy}. 
    We have constructed such a 2-functor in the dihedral case in Proposition \ref{prop:two_color_endoscopy}, but for a general Coxeter group, there is no obvious construction.
\end{remark}

\subsection{Mixed Derived Categories}\label{subsec:mixed_diagrammatic_cat}

We will show using Theorem \ref{thm:dll_is_a_basis} that $\EWmon{\scrL}{\scrL'}^{\BS} (\fr{h}, W, \fr{o})$ is an object-adapted cellular category (OACC).
We can then use Appendix \ref{apdx:mdc} to construct a mixed derived category for the diagrammatic monodromic Hecke category.
The study of these mixed derived categories will be a crucial component in the later proof of the monodromic Riche--Williamson theorem (see Theorem \ref{thm:mon_RW_thm}).

\subsubsection{Object-Adapted Cellular Structure}

In order to apply for formalism of Appendix \ref{apdx:mdc}, we first must show that the diagrammatic monodromic Hecke category is an OACC.

Let $I \subseteq W$ and $\scrL, \scrL' \in \fr{o}$. We set
\[ {}_{\scrL} I_{\scrL'} \coloneq \{ w \in I \mid \scrL w = \scrL'\}.\]
The Bruhat order on $W$ restricts to a partial order on $\W{\scrL}{\scrL'}$.
This induces an order topology on $\W{\scrL}{\scrL'}$ from $W$. In particular, a subset $K \subset \W{\scrL}{\scrL'}$ is closed (resp. open) in $\W{\scrL}{\scrL'}$ if and only if there exists a closed (resp. open) subset $I \subset W$ such that $K = {}_{\scrL} I_{\scrL'}$.

Given $w \in W$, we write $\{ \leq w \}$ and $\{ < w\}$ for the closed subsets of $W$ consisting of all elements $x\in W$ such that $x \leq w$ and $x < w$ respectively.

\begin{proposition}\label{prop:dmhc_is_cellular}
    The diagrammatic monodromic Hecke category $\EWmon{\scrL}{\scrL'}^{\BS} (\fr{h}, W, \fr{o})$ is an OACC with the following data:
    \begin{itemize}
        \item We fix a reduced expression $\uw = \rex (w)$ for each $w \in \W{\scrL}{\scrL'}$. As a result, we can identify $\W{\scrL}{\scrL'}$ with a subset $\Lambda$ of objects in $\EWmon{\scrL}{\scrL'}^{\BS} (\fr{h}, W, \fr{o})$ by $w \mapsto B_{\uw}^{\scrL}$.
                The set $\Lambda$ is a poset under the restricted Bruhat order in $\W{\scrL}{\scrL'}$.
        \item For any $\ux \in \Exp (W)$ and $w \in \W{\scrL}{\scrL'}$, the sets $E(\ux, w)$ and $\overline{E} (w, \ux)$ are both equal to the set of $\scrL$-monodromic subexpressions of $\ux$ expressing $w$
        \item For any $w \in \Lambda$ and $\ue \in E(\ux, w) = \overline{E} (w,\ux)$, the maps $\LL_{\ux, \ue}^{\scrL}$ and $\overline{\LL}_{\ux, \ue}^{\scrL}$ are the light leaves which are taken so that $\uw$ is always the target (resp. domain).
            We pick the $\LL$-datum such that $\LL_{\uw, \ue}^{\scrL}$ is the identity map when $\ue$ is the all 1's subexpression.
        \item The contravariant autoequivalence $\DD (-)$ is the functor which takes an Elias--Williamson graph to its horizontal reflection. 
    \end{itemize} 
\end{proposition}
\begin{proof}
    The axioms (CC2) and (CC3) follow from definitions. Axiom (CC1) is a restatement of Theorem \ref{thm:dll_is_a_basis}. It remains to show that
    \[ \chi_{\leq w} \coloneq \Hom_{\leq w} (B_{\ux}^{\scrL}, B_{\uy}^{\scrL})\]
    is a two-sided ideal.  
    By Proposition \ref{prop:ideal_Hom_spaces} and Theorem \ref{thm:mon_Abe}, we have that $\chi_{\leq w}$ consists of maps $f : B_{\ux}^{\scrL} \to B_{\uy}^{\scrL}$ such that $\left( \Upsilon (f) \right) (B_{\ux, Q}^{v, \scrL}) = 0$ whenever $v \not\leq w$.
    It is easy to check from the definition of morphisms in the algebraic category that these maps form a 2-sided ideal. 
\end{proof}

Proposition \ref{prop:dmhc_is_cellular} allows us to use the constructions of Appendix \ref{apdx:mdc}.
We denote the mixed derived category for $\EWmon{\scrL}{\scrL'}^{\BS} (\fr{h}, W, \fr{o})$ by
\[\DDBE{\scrL}{\scrL'} (\fr{h}, W, \fr{o}) \coloneq D^m \left( \EWmon{\scrL}{\scrL'}^{\BS} (\fr{h}, W, \fr{o})\right).\]
These categories can be assembled into a 2-category over $\fr{o}$, denoted $\DDBE{}{} (\fr{h}, W, \fr{o})$.

For $I \subseteq W$ locally closed, we can also define
\[\DDBEI{I}{\scrL}{\scrL'} (\fr{h}, W, \fr{o})\coloneq D^m_{{}_{\scrL} I_{\scrL'}} \left( \EWmon{\scrL}{\scrL'}^{\BS} (\fr{h}, W, \fr{o}) \right).\]

\subsubsection{Generation Assumption}

We will work towards showing that Assumption \ref{ass:generation_assumption} holds for the mixed derived category of $\EWmon{\scrL}{\scrL'}^{\BS} (\fr{h}, W, \fr{o})$.

\begin{lemma}[{\cite[Lemma 6.1]{ARV}}]\label{lem:cone_of_rex}
    Let $\ux$ and $\uy$ be reduced expressions for an element $w \in W$.
    Let $f : B_{\ux}^{\scrL} \to B_{\uy}^{\scrL}$ be a rex move.
    Then the cone of $f$ belongs to $\DDBEI{\{ < w\} }{\scrL}{\scrL'} (\fr{h}, W, \fr{o})$.
\end{lemma}
\begin{proof}
    The proof can be copied almost verbatim from \cite{ARV}. We remark that the usage of \cite[Lemma 7.4]{EW} and \cite[Lemma 7.5]{EW} in \emph{loc. cit.} should be replaced by Lemma \ref{lem:neg_pos_decomps}.
\end{proof}

Recall the \emph{Hecke product} on $W$ (see \cite{BM}).\footnote{This is also sometimes called in literature the Demazure product, the 0-Hecke product, or the greedy product.}
Let $x \in W$, $s \in S$, define
\[x \star s \coloneq \begin{cases}
    x s & \text{if } xs > x, \\
    x & \text{if } xs < x. 
\end{cases}\]
More generally, let $x,y \in W$ and $y = s_1 \ldots s_k$ be any reduced expression for $y$.
Define the Hecke product of $x$ and $y$ is given by
\[x \star y = x \star s_1 \star s_2 \star \ldots \star s_k.\]
This does not depend on the choice of reduced expression for $y$ and endows $W$ with the structure of a monoid.
For $\uw = (s_1, \ldots, s_k)$ an expression, we write
\[\star \uw \coloneq s_1 \star \ldots \star s_k \in W.\]

\begin{lemma}[{\cite[Lemma 6.2]{ARV}}]\label{lem:generation_assumption}
    For any expression $\uw \in \Exp (W)$, the object $B_{\uw}^{\scrL}$ belongs to $\DDBEI{\{ \leq \star \uw \}}{\scrL}{\scrL \uw} (\fr{h}, W, \fr{o})$.
    Therefore, Assumption \ref{ass:generation_assumption} holds for $\EWmon{\scrL}{\scrL'}^{\BS} (\fr{h}, W, \fr{o})$.
\end{lemma}
\begin{proof}
    We argue by induction on $\ell (\uw)$. The claim is obvious when $\uw$ is reduced.
    Write $\uw = \ux s$ for some $s \in S$, and assume by induction the claim is known for all expressions of length $< \ell (\uw)$. In particular, the claim holds for $\ux$.
    We then must show that if $\uy$ is a reduced expression for some $y \in \W{\scrL}{\scrL \uw s}$ and $y \leq \star \ux$, then $B_{\uy s}^{\scrL} \in  \DDBEI{\{ \leq \star \uw \}}{\scrL}{\scrL \uw}$.

    If $\ell (\uy) < \ell (\ux)$, then $\ell (\uy s) < \ell (\uw)$, so by induction $B_{\uy s}^{\scrL} \in  \DDBEI{ \{ \leq \star (\uy s) \}}{\scrL}{\scrL \uw}$.
    By \cite[Proposition 3.1]{BM}, we have $\star (\uy s) = y \star s \leq (\star \ux ) \star s = \star \uw$. The claim then follows.

    Assume now that $\ell (\uy) = \ell (\ux)$. Then $\ux$ is a reduced expression, and hence, $y = \star \ux$.
    If $y s > y$, then $\star \uw = (\star \ux) s$ and $\uy s$ is a reduced expression for $\star \uw$. The claim then follows from definitions.
    If $ys < y$, choose a reduced expression $\uy'$ for $y$ ending with $s$ and a rex move $\beta : B_{\uy}^{\scrL} \to B_{\uy'}^{\scrL}$.
    We can then consider the distinguished triangle
    \[B_{\uy}^{\scrL} \stackrel{\beta}{\to} B_{\uy'}^{\scrL} \to C \to,\]
    where $C = \textnormal{cone} (f)$ is in $\DDBEI{\{ < y \}}{\scrL}{\scrL y} $ by Lemma \ref{lem:cone_of_rex}.
    We can then apply $(-) \star B_s^{\scrL}$ to produce a distinguished triangle
    \begin{equation}\label{eq:generation_assumption_1}
        B_{\uy s}^{\scrL} \stackrel{\beta \star B_s^{\scrL}}{\to} B_{\uy' s}^{\scrL} \to C \star B_s^{\scrL} \to.
    \end{equation}
    By induction, we have that $C \star B_s^{\scrL} \in \DDBEI{\{ \leq \star w \}}{\scrL}{\scrL \uw } $.
    If $\scrL ys = \scrL y$, then there is an isomorphism $B_{\uy' s}^{\scrL} \cong B_{\uy'}^{\scrL} (1) \oplus B_{\uy'}^{\scrL} (-1)$ (cf., \cite[(5.14)]{EW}).
    Therefore, $B_{\uy' s}^{\scrL} \in \DDBEI{\{ \leq y \}}{\scrL}{\scrL \uw } $.
    If $\scrL ys \neq \scrL y$, then by block minimality, there is an isomorphism $B_{\uy' s}^{\scrL} \cong B_{\uy''}^{\scrL}$ where $\uy''$ is the expression such that $\uy' = \uy'' s$.
    Note $\star \uy'' = (\star \uy') = ys < y$.
    As a result, $B_{\uy' s}^{\scrL} \in \DDBEI{\{ \leq y \}}{\scrL}{\scrL \uw }$.
    Therefore, (\ref{eq:generation_assumption_1}) implies that $B_{\uy s}^{\scrL} \in \DDBEI{ \{ \leq \star w \}}{\scrL}{\scrL \uw }$.
\end{proof}

\subsubsection{Convolution of Standards and Costandards}

As discussed in \S\ref{subsub:stds_and_costds}, there is a theory of standard and costandard objects in $\DDBE{\scrL}{\scrL'} (\fr{h}, W, \fr{o})$.
In particular, for all $w \in \W{\scrL}{\scrL'}$, there is an associated standard object and costandard object
\[ \Delta_w^{\scrL} \coloneq j_{{}_{\scrL }\{\leq w \}_{\scrL'} !}^{\W{\scrL}{\scrL'}} j_{w!}^{{}_{\scrL} \{\leq w \}_{\scrL'} } b_w^{\scrL} \qquad\text{and}\qquad \nabla_w^{\scrL} \coloneq j_{{}_{\scrL} \{\leq w \}_{\scrL'} !}^{\W{\scrL}{\scrL'}} j_{w*}^{{}_{\scrL} \{\leq w \}_{\scrL'}} b_w^{\scrL},\]
where $b_w^{\scrL}$ is the canonical object in $\DDBEI{\{ w \}}{\scrL}{\scrL w} (\fr{h}, W, \fr{o})$ corresponding to $B_{\uw}^{\scrL}$.

\begin{example}\label{ex:stds_and_costds_for_simples}
    Let $s \in S$.
    \begin{enumerate}
        \item Note that $\{e\}$ is closed in $\W{\scrL}{\scrL}$. As a result, we have isomorphisms $\Delta_e^{\scrL} \cong B_{\emptyset}^{\scrL} \cong \nabla_e^{\scrL}$.
        \item Assume $\scrL s = \scrL$. In light of Example \ref{ex:min_rec_dts} and (\ref{eq:verdier_duality_and_pushforwards}), the complex $\Delta_s^{\scrL}$ coincides with the complex
            \[\ldots \to 0 \to B_s^{\scrL} \stackrel{\epsilon^s}{\to} B_{\emptyset}^{\scrL} (1) \to 0 \to \ldots\]
            where the nonzero terms are in degrees $0$ and $1$ respectively. Similarly, the complex $\nabla_s^{\scrL}$ is given by
            \[\ldots \to 0 \to B_{\emptyset}^{\scrL} (-1) \stackrel{\eta^s}{\to} B_{s}^{\scrL} \to 0 \to \ldots\]
            where the nonzero terms are in degrees $-1$ and $0$ respectively.
            We can then conclude that there are distinguished triangles
            \[B_{\emptyset}^{\scrL} \langle -1 \rangle \to \Delta_s^{\scrL} \to B_s^{\scrL} \to\qquad\text{and}\qquad  B_s^{\scrL} \to \nabla_s^{\scrL} \to B_{\emptyset}^{\scrL} \langle 1 \rangle  \to.\]
        \item Assume $\scrL s \neq \scrL$. In this case, $\{s\}$ is a closed subset of $\W{\scrL}{\scrL s}$.
            As a result, we have isomorphisms $\Delta_s^{\scrL} \cong B_s^{\scrL} \cong \nabla_s^{\scrL}$.
    \end{enumerate}
\end{example}

\begin{lemma}[{\cite[Lemma 6.10]{ARV}}]\label{lem:mixed_push_pull_dts}
    Let $w \in W$ and $s \in S$ be such that $\scrL ws = \scrL w$ and $ws > w$. Then there exist distinguished triangles
    \[\Delta_w^{\scrL} \langle -1\rangle \to \Delta_{ws}^{\scrL} \to \Delta_{w}^{\scrL} \star B_{s}^{\scrL} \to \qquad\text{and}\qquad \nabla_w^{\scrL} \star B_s^{\scrL} \to \nabla_{ws}^{\scrL} \to \nabla_w^{\scrL} \langle 1 \rangle \to \]
    in $\DDBE{\scrL}{\scrL w} (\fr{h}, W, \fr{o})$, in which the third arrows are generators of the free rank-1 $\k$-modules
    \[\Hom_{\DDBE{\scrL}{\scrL w}} (\Delta_{w}^{\scrL} \star B_{s}^{\scrL}, \Delta_w^{\scrL} \langle -1 \rangle [1]) \qquad\text{and}\qquad \Hom_{\DDBE{\scrL}{\scrL w}} (\nabla_w^{\scrL} \langle 1 \rangle, \nabla_w^{\scrL} \star B_s^{\scrL} [1])\]
    respectively.
\end{lemma}

\begin{proposition}[{\cite[Proposition 6.11]{ARV}}]\label{prop:mixed_conv_rules}
    Let $w \in W$.
    \begin{enumerate}
        \item If $\uw = (s_1, \ldots, s_k)$ is a reduced expression for $w$, then we have isomorphisms
        \[\Delta_w^{\scrL} \cong \Delta_{s_1}^{\scrL} \star \Delta_{s_2}^{\scrL s_1} \star \ldots \star \Delta_{s_k}^{\scrL s_1 \ldots s_{k-1}} \qquad\text{and}\qquad \nabla_w^{\scrL} \cong \nabla_{s_1}^{\scrL} \star \nabla_{s_2}^{\scrL s_1} \star \ldots \star \nabla_{s_k}^{\scrL s_1 \ldots s_{k-1}}.\]
        \item We have isomorphisms
            \[\Delta_w^{\scrL} \star \nabla_{w^{-1}}^{\scrL w} \cong B_{\emptyset}^{\scrL} \cong \nabla_{w^{-1}}^{\scrL} \star \Delta_w^{\scrL w^{-1}}.\]
    \end{enumerate}
\end{proposition}
\begin{proof}
    We will prove the claims by induction on $\ell (w)$. 
    We can first prove (2) when $w = s \in S$. If $\scrL s = \scrL$, then this follows from \cite[Lemma 4.2.4]{AMRW1}.
    If $\scrL s \neq \scrL$, then by Example \ref{ex:stds_and_costds_for_simples} (3), we have that $\Delta_s^{\scrL} \star \nabla_s^{\scrL s} \cong B_s^{\scrL} \star B_s^{\scrL s}$.
    The claim then follows by the block minimality relation.
    By induction, we then have that (2) follows from (1).

    If $w$ is minimal in $\W{\scrL}{\scrL'}$, then (1) follows from block minimality.
    Write $\W{\scrL}{\scrL'}^+$ for the subset of $W$ containing elements $w$ such that $\ell_{\scrL} (w) > 0$.
    Assume the claim holds for all elements $x \in \W{\scrL}{\scrL'}^+$ whose length is strictly smaller than that of $w$.
    We will prove the first isomorphism in (1) for $w$. The second isomorphism follows from duality.
    Let $(s_1, \ldots, s_k)$ be a reduced expression for $w$. Set $x = s_1\ldots s_{k-1}$ and $s = s_k$.
    By induction, we have an isomorphism
    \[\Delta_y^{\scrL} \cong \Delta_{s_1}^{\scrL} \star \Delta_{s_2}^{\scrL s_1} \star \ldots \star \Delta_{s_{k-1}}^{\scrL s_1 \ldots s_{k-2}},\]
    hence it suffices to prove that $\Delta_w^{\scrL} \cong \Delta_y^{\scrL} \star \Delta_s^{\scrL y}$. If $\scrL ys \neq \scrL y$, then this isomorphism follows from block minimality.
    We can then assume that $\scrL ys = \scrL y$. By (2) for $y$, the functor
    \[\Delta_y^{\scrL} \star (-) : \DDBE{\scrL y}{\scrL w} (\fr{h}, W, \fr{o}) \to \DDBE{\scrL}{\scrL w} (\fr{h}, W, \fr{o}) \]
    is an equivalence of triangulated categories with quasi-inverse $\nabla_{y^{-1}}^{\scrL y}$.
    By applying $\Delta_y^{\scrL} \star (-)$ and using Example \ref{ex:stds_and_costds_for_simples} (2), we obtain a distinguished triangle
    \[\Delta_y^{\scrL} \to \Delta_y^{\scrL} \star \Delta_s^{\scrL y} \to \Delta_y^{\scrL} \star B_s^{\scrL y} \to,\]
    where the third arrow is a generator of the free rank-1 $\k$-module, $\Hom_{\DDBE{\scrL}{\scrL w}} (\Delta_s^{\scrL y}, \Delta_y^{\scrL} \star B_s^{\scrL y})$.
    We can then apply Lemma \ref{lem:mixed_push_pull_dts} to deduce the isomorphism $\Delta_w^{\scrL} \cong \Delta_y^{\scrL} \star \Delta_s^{\scrL y}$ as desired.
\end{proof}

\subsubsection{Right Equivariant Category}\label{subsec:right_equiv_diagrammatic}

Let $\scrL, \scrL' \in \fr{o}$. By Theorem \ref{thm:dll_is_a_basis}, we can define a category
\[\EWmonME{\scrL}{\scrL'}^{\BS} (\fr{h}, W, \fr{o})\]
obtained from $\EWmon{\scrL}{\scrL'}^{\BS} (\fr{h}, W, \fr{o})$ by applying $\k \otimes_{R} (-)$ to $\Hom_{\EWmon{\scrL}{\scrL'}^{\BS}} (B_{\ux}^{\scrL}, B_{\uy}^{\scrL})$.
By Proposition \ref{prop:dmhc_is_cellular} one can deduce that $\EWmonME{\scrL}{\scrL'}^{\BS} (\fr{h}, W, \fr{o})$ is also an OACC with poset $\W{\scrL'}{\scrL}$ under the usual Bruhat order.
Note that the Hom spaces in $\EWmonME{\scrL}{\scrL'}^{\BS} (\fr{h}, W, \fr{o})$ are finitely generated free $\k$-modules.
As a result, we can consider its mixed derived category
\[\DDME{\scrL}{\scrL'} (\fr{h}, W, \fr{o}) \coloneq D^m \left( \EWmonME{\scrL}{\scrL'}^{\BS} (\fr{h}, W, \fr{o}) \right)\]

There is an ``action'' of $\EWmon{\scrL'}{\scrL''}^{\BS} (\fr{h}, W, \fr{o})$ on $\EWmonME{\scrL}{\scrL'}^{\BS}$. Namely, there are bifunctors
\[(-) \star (-) : \EWmonME{\scrL}{\scrL'}^{\BS} (\fr{h}, W, \fr{o}) \times \EWmon{\scrL'}{\scrL''}^{\BS} (\fr{h}, W, \fr{o}) \to \EWmonME{\scrL}{\scrL''}^{\BS}  (\fr{h}, W, \fr{o})\]
which is suitably associative with respect to the horizontal composition in $\EWmon{}{}^{\BS} (\fr{h}, W, \fr{o})$.
This induces a bifunctor
\[(-) \star (-) : \DDME{\scrL}{\scrL'} (\fr{h}, W, \fr{o}) \times \DDBE{\scrL'}{\scrL''} (\fr{h}, W, \fr{o}) \to \DDME{\scrL}{\scrL''} (\fr{h}, W, \fr{o})\]
with similar properties.

There is an obvious functor
\[\ForME{\scrL} : \DDBE{\scrL}{\scrL'} \to \DDME{\scrL}{\scrL'} (\fr{h}, W, \fr{o})\]
which commutes with the recollement structure on both sides.
We abuse notation and view any object in $\DDBE{\scrL}{\scrL'} (\fr{h}, W, \fr{o})$ as an object in $\DDME{\scrL}{\scrL'} (\fr{h}, W, \fr{o})$ via the functor $\ForME{\scrL}$.
It is easy to check that there is a natural isomorphism 
\[\ForME{\scrL} (X \star Y) \cong \ForME{\scrL} (X) \star Y\]
for all $X \in \DDBE{\scrL}{\scrL'} (\fr{h}, W, \fr{o})$ and $Y \in \DDBE{\scrL'}{\scrL''} (\fr{h}, W, \fr{o})$.

\subsubsection{Categorification}\label{subsubsec:cat_of_diag_hecke}

One application of the mixed derived category formalism is that it allows one to prove that the monodromic Hecke category categorifies the monodromic Hecke algebra without any constraint on $\k$.
Without the mixed derived category, standard techniques (cf., \cite{EL, EW}) can be used to show this categorification result, but only when $\k$ is a complete local ring or a field.

Define a 2-category $\grEWmon{}{}^{\BS} (\fr{h}, W, \fr{o})$ whose object set is $\fr{o}$.
The 1-morphisms in $\grEWmon{}{}^{\BS} (\fr{h}, W, \fr{o})$ are given by symbols $B_{\uw}^{\scrL} (n)$ where $\uw \in \Exp (W)$ and $n \in \Z$.
The 2-morphisms in $\grEWmon{}{}^{\BS} (\fr{h}, W, \fr{o})$ from $B_{\ux}^{\scrL} (m)$ to $B_{\uy}^{\scrL} (n)$ are the 2-morphisms $B_{\ux}^{\scrL} \to B_{\uy}^{\scrL}$ of degree $n-m$ in $\EWmon{}{}^{\BS} (\fr{h}, W, \fr{o})$.
The horizontal and vertical compositions are those inherited from $\EWmon{}{}^{\BS} (\fr{h}, W, \fr{o})$. Note that on morphism categories, using the notation of Appendix \ref{apdx:mdc}, $\grEWmon{\scrL}{\scrL'}^{\BS} (\fr{h}, W, \fr{o}) = \left( \EWmon{\scrL}{\scrL'}^{\BS} (\fr{h}, W, \fr{o})\right)^{\circ}$.
We will also consider the 2-category $\grEWmon{}{}^{\oplus} (\fr{h}, W, \fr{o})$ whose object set is $\fr{o}$ and whose morphism categories $\grEWmon{\scrL}{\scrL'}^{\oplus} (\fr{h}, W, \fr{o})$ are the additive hulls of $\grEWmon{\scrL}{\scrL'}^{\BS} (\fr{h}, W, \fr{o})$.

Let $[\DDBE{}{} (\fr{h}, W, \fr{o}) ]_{\Delta}$ denote the Grothendieck algebroid of $\DDBE{}{} (\fr{h}, W, \fr{o})$. In other words, it is the category whose object set is $\fr{o}$ and whose morphism spaces are given by taking the Grothendieck groups of $\DDBE{\scrL}{\scrL'} (\fr{h}, W, \fr{o})$.
Similarly, let $[\grEWmon{}{}^{\oplus} (\fr{h}, W, \fr{o})]$ denote the split Grothendieck algebroid of $\grEWmon{}{}^{\oplus} (\fr{h}, W, \fr{o})$. Explicitly, $[\grEWmon{}{}^{\oplus} (\fr{h}, W, \fr{o})]_{\oplus}$ is the category whose object set is $\fr{o}$ and whose morphism spaces are given by taking the split Grothendieck groups $[\grEWmon{\scrL}{\scrL'}^{\oplus} (\fr{h}, W, \fr{o})]_{\oplus}$.

\begin{theorem}\label{thm:soergel_categorification}
    There exists a unique isomorphism of algebroids
    \[\scrH^{\mon} (W, \fr{o}) \stackrel{\sim}{\to} [\grEWmon{}{}^{\oplus} (\fr{h}, W, \fr{o}) ]_{\oplus}\]
    sending $v H_e^{\scrL}$ to $[B_{\emptyset}^{\scrL} (1)]$ and $\uH_s^{\scrL}$ to $[B_s^{\scrL}]$ for any $s \in S$ and $\scrL \in \fr{o}$.
\end{theorem}
\begin{proof}
    By Lemma \ref{lem:red_to_tri_groth_gps}, it suffices to prove that there exists a unique isomorphism
    \[\scrH^{\mon} (W, \fr{o}) \stackrel{\sim}{\to} [\DDBE{}{} (\fr{h}, W, \fr{o}) ]_{\Delta}\]
    sending $vH_e^{\scrL}$ to $[B_{\emptyset}^{\scrL} (1)]$ and $\uH_s^{\scrL}$ to $[B_s^{\scrL}]$. Uniqueness is clear since $\scrH^{\mon} (W, \fr{o})$ is generated as a category by the morphisms $vH_e^{\scrL}$ and $\uH_s^{\scrL}$.

    We will now show that the assignment 
    \[vH_e^{\scrL} \mapsto [B_{\emptyset}^{\scrL} (1)], \qquad\qquad H_w^{\scrL} \mapsto [\Delta_w^{\scrL}]\]
    induces a functor $\scrH^{\mon} (W, \fr{o}) \to [\DDBE{}{} (\fr{h}, W, \fr{o}) ]_{\Delta}$.
    By the defining relations for the monodromic Hecke algebroid it suffices to show the follow relations hold:
    \begin{align}\label{eq:soergel_categorification_1}
        [\Delta_s^{\scrL}][\Delta_s^{\scrL}] &= [\Delta_e^{\scrL}] + [\Delta_s^{\scrL} (-1)] - [\Delta_s^{\scrL} (1)] & &s \in S \text{ and } \scrL s = \scrL, \\ \label{eq:soergel_categorification_1.5}
        [\Delta_s^{\scrL}][\Delta_s^{\scrL s}] &= \Delta_e^{\scrL} & &s \in S \text{ and }  \scrL s \neq \scrL,\\  \label{eq:soergel_categorification_2}
        [\Delta_{xy}^{\scrL}] &= [\Delta_x^{\scrL} \star \Delta_y^{\scrL}] & &x,y \in W\text{ such that }\ell (xy) = \ell (x) + \ell (y).
    \end{align}
    Here (\ref{eq:soergel_categorification_1}) and (\ref{eq:soergel_categorification_1.5}) follow from Example \ref{ex:stds_and_costds_for_simples} along with the isomorphisms 
    \[ B_s^{\scrL} \star B_s^{\scrL s} \cong \begin{cases} B_s^{\scrL} (-1) \oplus B_s^{\scrL} (1) & \scrL s = \scrL, \\ B_{\emptyset}^{\scrL} & \scrL s \neq \scrL.\end{cases}\]
    On the other hand, (\ref{eq:soergel_categorification_2}) follows from Proposition \ref{prop:mixed_conv_rules}.
    It is clear that this functor then takes $vH_e^{\scrL}$ to $[B_{\emptyset}^{\scrL} (1)]$ and $\uH_s^{\scrL}$ to $[B_s^{\scrL}]$.

    It remains to check that the functor $\scrH^{\mon} (W, \fr{o}) \to [\DDBE{}{} (\fr{h}, W, \fr{o}) ]_{\Delta}$ is an isomorphism. 
    By Proposition \ref{prop:OACC_and_categorification}, the standard objects $[\Delta_w^{\scrL} (m)]$ for $w \in \W{\scrL}{\scrL'}$ and $m \in \Z$ form a $\Z$-basis for $[\DDBE{\scrL}{\scrL'} (\fr{h}, W, \fr{o}) ]_{\Delta}$.
    As a result, our functor takes a $\Z$-basis for ${}_{\scrL} \scrH^{\mon}_{\scrL'} (W, \fr{o})$ to a $\Z$-basis for $[\DDBE{\scrL}{\scrL'} (\fr{h}, W, \fr{o}) ]_{\Delta}$. 
\end{proof}

    \section{Diagrammatic Double Leaves Span}\label{sec:spanning}

    In order to complete the proof of Theorem \ref{thm:dll_is_a_basis}, we must show that double leaves span.
The proof closely follows \cite[\S 7]{EW}.
Since the proof is quite lengthy and involved, we will just outline the key statements and ideas when the arguments are virtually identical to \cite{EW}.

For the remainder of the section, we will work with embedded graphs rather than isotopy classes thereof.
We abuse terminology and use the term \emph{monodromic Elias--Williamson graphs} to refer to a graph embedded without horizontal tangent lines so that no two vertices share the same $y$-coordinate.
It is convenient to assume that all monodromic Elias--Williamson graphs are embedded in the plane $\R \times [0,1]$ such that $\R \times \{0\}$ is the bottom of the graph and that $\R \times \{1\}$ is the top of the graph.

\subsection{Negative-Positive Decompositions}

Let $f$ be a monodromic Elias--Williamson graph. We define the \emph{height} of $f$, denoted $\textnormal{ht} (f)$ as the number of generators needed to produce $f$ (ignoring polynomials).
We also define the \emph{width} of $f$ at a $y$-coordinate $a \in [0,1]$ (with no vertices on the line $y=a$), denoted $\text{width}_a (f)$, as the number of strands in $f$ the line $y=a$ passes through.
The \emph{maxwidth} of $f$, denoted $\text{maxwidth} (f)$, is the maximal width attained by $f$.

Given a generator, we assign an integer called the \emph{signed height} as follows.
\begin{equation}\label{eq:sgn_of_gens}
    
  % \tikzsetnextfilename{#1}
  \tikzstyle{every picture}=[tikzfig]
  \input{./figures/sign_of_generators.tikz}

\end{equation}
More generally, the \emph{negative height} of a graph is the number of negative generators used with multiplicities given by (\ref{eq:sgn_of_gens}).
Similarly, the \emph{positive height} is the number of positive generators used with multiplicities given by (\ref{eq:sgn_of_gens}).
Note that the negative height and positive height of a map only depends on the underlying Elias--Williamson graph and not on the choice of face labelling by monodromy parameters.
Moreover, light leaves have positive height 0.

We call a graph $f : B_{\uw}^{\scrL} \to B_{\uy}^{\scrL}$ \emph{(non-)negative} (resp. \emph{(non-)positive}) if $f$ consists of solely (non-)negative (resp. (non-)positive) generators.
We say that $f$ is \emph{negative-positive} if it factors as $f = f_2 \circ f_1$ with $f_2$ non-negative and $f_1$ non-positive.
Furthermore, we call $f$ \emph{strictly} negative-positive if $\text{width}_a (f) < \ell (\uw)$ for some $a \in [0,1]$. 

Given a morphism $f$ in $\EWmon{\scrL}{\scrL'}^{\BS} (\fr{h}, W, \fr{o})$, we say that $f$ admits a \emph{negative-positive decomposition} if it can be expressed as a $\k$-linear combination of negative-positive graphs. 
Similarly, $f$ admits a \emph{strictly} negative-positive decomposition if $f$ can be expressed as a linear combination of strictly negative-positive graphs.

The following lemma is a compilation of results from \cite[\S 7.1]{EW} which are generalized to the monodromic setting.

\begin{lemma}\label{lem:neg_pos_decomps}
    \begin{enumerate}
        \item Let $s, t\in S$ with $m = m_{s,t} < \infty$. The monodromic Jones--Wenzl morphism admits a decomposition $\JW_{s,t}^{\scrL} = \id_{B_{{}_s \underline{m}}^{\scrL} } + f$ where $f$ has a strictly negative-positive whose terms have $\textnormal{width}_a \leq m-2$ for $a \in (0,1)$.
        \item Let $f : B_{\uw}^{\scrL} \to B_{\uy}^{\scrL}$ be a negative monodromic Elias--Williamson graph with $\uw$ reduced. Then $f$ is in the span of maps with a bottom boundary dot.
        \item If $\phi = f g h$ for $f$ non-negative, $h$ non-positive, and $g$ having a (strictly) negative-positive decomposition, then $\phi$ has a (strictly) negative-positive decomposition.
        \item Let $\uw$ and $\uw'$ be rexes for $w \in W$. Let $\beta, \beta' : B_{\uw}^{\scrL} \to B_{\uw'}^{\scrL}$ be two rex moves. Then $\beta - \beta'$ has a strictly negative-positive decomposition.
        \item Let $\ux \in \Exp (W)$. Then $\id_{\ux}$ has a negative-positive decomposition which factors through some reduced expression $\uw$ for some $w$ such that there exists a subsequence $\ue$ of $\ux$ with $\ux^{\ue} = w$.
    \end{enumerate}
\end{lemma}
\begin{proof}
    \emph{(1): } Note that endosimple expansion takes negative (resp. positive) graphs to negative (resp. positive)  graphs.
    Similarly, $\tau : \EWmon{\scrL}{\scrL}^{\BS, \tau} (\fr{h}, \{s,t\}, \fr{o}) \to \EWmon{\scrL}{-}^{\BS} (\fr{h}, \{s,t\}, \fr{o})$ takes negative (resp. positive) graphs to negative (resp. positive) graphs.
    Let $s', t'$ denote the endosimple reflections in $W_{s,t}^{\scrL}$.
    If $g$ is an Elias--Williamson graph viewed as a morphism in $\EWmon{1}{1}^{\BS} (\fr{h}, \{s',t'\}, 1)$, then for all $a \in (0,1)$,
    \begin{align*}
        2 &\leq \text{maxwidth} (g) - \text{width}_a (g) \\
        &\leq \text{maxwidth} (\iota (g)) - \text{width}_a (\iota (g)) \\
        &= \text{maxwidth} (\tau (\iota (g))) - \text{width}_a (\tau (\iota (g))).
    \end{align*}
    As a result, it follows from \cite[Claim 7.1]{EW} that $\JW_{s,t}^{\scrL} = (\tau \circ \iota) (\JW_{s',t'}^1)$ has the desired properties.

    \emph{(2): } The key idea is to apply the Jones--Wenzl relation which we proved in Lemma \ref{lem:jones_wenzl}. The argument then precedes as in \cite[Claim 7.2]{EW}.

    \emph{(3): } This is a restatement of \cite[Claim 7.3]{EW}. It is obvious from definitions.

    \emph{(4): } The argument is essentially the same as \cite[Lemma 7.4]{EW}. The only substantial difference is that in the monodromic setting, one has to check that in type $B_3$ with endoscopic type $A_3$, the additional terms appearing in the 3-color relations are in the lower terms ideal. This is easy to verify by hand.
    
    \emph{(5): } One additional case is required other than those covered in \cite[Lemma 7.5]{EW}. When $\ux = (s_1, \ldots, s_k)$ with $s_i = s_{i+1}$. We must consider both $i \in K (\ux, \scrL)$ and $i \notin K (\ux, \scrL)$.
    In the former case, the argument given in \emph{loc. cit.} transfers. When $i \notin K( \ux, \scrL)$, the claim follows from the block minimality relation.
\end{proof}

\subsection{Modulo Lower Terms}

Let $\uw$ be a reduced expression for some $w \in W$. 
We define an ideal under vertical composition 
\[ I_{\uw}^{\scrL} \subseteq \bigoplus_{\ux \in \Exp (W)} \Hom_{\EWmon{\scrL}{\scrL \ux}^{\BS}} (B_{\ux}^{\scrL}, B_{\uw}^{\scrL})\] 
consisting of $f \in \Hom_{\EWmon{\scrL}{\scrL \ux}^{\BS}} (B_{\ux}^{\scrL}, B_{\uw}^{\scrL})$ which has a decomposition $f = f_1 \circ f_2$ where $f_1$ is positive.
Lemma \ref{lem:neg_pos_decomps} implies that $I_{\uw}^{\scrL}$ is generated by all the top boundary dots. 
From this description, it is obvious that when $w = w^\beta$ for some block $\beta \in \uW{\scrL}{\scrL'}$, then $I_{\uw}^{\scrL} = 0$.

The following proposition is the monodromic analogue of \cite[Proposition 7.6]{EW}. 

\begin{proposition}\label{prop:the_evil_EW_prop}
    Let $\ux \in \Exp (W)$. Fix a reduced expression $\uw$ for some element $w \in W$. 
    Choose a map $\LL_{\ux, \ue}^{\scrL}$ for each monodromic subexpression $\ue$ expressing $w$.
    These maps form a left $R$-basis for $\Hom_{\EWmon{\scrL}{\scrL w}^{\BS}} (B_{\ux}^{\scrL}, B_{\uw}^{\scrL}) / I_{\uw}^{\scrL}$.
\end{proposition}

The proof of Proposition \ref{prop:the_evil_EW_prop} will take substantial work. 
As a corollary, we will complete the proof of Theorem \ref{thm:dll_is_a_basis}. The proof of Theorem \ref{thm:dll_is_a_basis} from the above proposition follows from the argument given in \cite[\S 7.3]{EW}.
We have included it only for the sake of completeness. 

\begin{midsecproof}{Theorem \ref{thm:dll_is_a_basis} assuming Proposition \ref{prop:the_evil_EW_prop}}
    Fix $\ux, \uy \in \Exp (W)$ such that $\scrL \ux = \scrL' = \scrL \uy$, and let $\varphi : B_{\ux}^{\scrL} \to B_{\uy}^{\scrL}$ be a morphism.
    By Lemma \ref{lem:dl_are_lin_indep_for_EW}, it suffices to show that $\varphi$ is an $R$-linear combination of double leaves.

    By Lemma \ref{lem:neg_pos_decomps} (5), we may assume, without loss of generality, that $\varphi$ factors through some $B_{\uw}^{\scrL}$ where $\uw$ is a reduced expression.
    We can then write $\varphi = fg$ with $g : B_{\ux}^{\scrL} \to B_{\uw}^{\scrL}$ and $f : B_{\uw}^{\scrL} \to B_{\uy}^{\scrL}$.
    We will prove the statement by induction on $\ell (\uw)$. If $\ell (\uw) = 0$ or more generally if $\ell_{\scrL} (\uw) = 0$, then the claim follows from Proposition \ref{prop:the_evil_EW_prop} using the observation that $I_{\uw}^{\scrL} = 0$.

    By Proposition \ref{prop:the_evil_EW_prop}, we can write $g = g_w + g_l$ and $f = f_w + f_l$ where $g_w$ (resp. $f_w$) is an $R$-linear combination of $\LL_{\ux, \ue}^{\scrL}$ (resp. $\overline{\LL}_{\uy, \ue}^{\scrL}$) and $g_l \in I_{\uw}^{\scrL}$ (resp. $f_l \in I_{\uw}^{\scrL}$).
    It then suffices to prove that $\psi g_l$ is in the span of the double leaves for any $\psi$. The argument with $f_l$ is the same but horizontally reflected.
    
    Now $g_l \in I_{\uw}^{\scrL}$ so it is generated by top boundary dots, and we can further factor $g_l$ into terms generated by a single top boundary dot. In particular, $\psi g_l$ is of the form
    \[
  % \tikzsetnextfilename{#1}
  \tikzstyle{every picture}=[tikzfig]
  \input{./figures/dll_arg_1.tikz}
\]
    Let $\underline{z}$ denote the expression obtained from $\uw$ by deleting the index with the dot above.
    We then see that $\psi g_l$ factors through $B_{\underline{z}}^{\scrL}$. By Lemma \ref{lem:neg_pos_decomps} (5), we can further factor $\psi g_l$ through reduced expressions for elements $v$ found in subexpressions of $\underline{z}$.
    Observe that $v < w$, so we can now apply induction on $\psi g_l$ to deduce that $\psi g_l$ is in the span of double leaves.
\end{midsecproof}

\subsection{The Inductive Setup}

Up until now, we have fixed an $\LL$-datum required to construct a set of light leaves.
For the proof of Proposition \ref{prop:the_evil_EW_prop}, it will be essential to let this data vary.
We will denote the set of $\LL$-data by $\Delta$.
To manage the notation, we will write $\delta \in \Delta$ and denote the corresponding triple by $\delta = (\rex^{\delta} , \{ \rex_s^{\delta} \}_{s \in S}, \{\Gamma_{\ux, \uy}^{\delta} \}_{\ux, \uy})$.
Similarly, for $\ux \in \Exp (W)$ and $\ue \in \Subexp^{\scrL} (\ux)$, we will write ${}^{\delta} \LL_{\ux, \ue}^{\scrL}$ for the light leaf constructed from the choice of $\delta$.

Let $\hat{X}_{\ux}^{\scrL} \coloneq \{ {}^{\delta} \LL_{\ux, \ue}^{\scrL} \mid \ue \in \Subexp^{\scrL} (\ux), \delta \in \Delta \}$ viewed as a subset of ${}_{\scrL} \VSGraph_{\scrL \ux}$.
There is a map $\For_{\textnormal{poly}} : {}_{\scrL} \VSGraph_{\scrL \ux} \to {}_{\scrL} \VSGraph_{\scrL \ux}$ given by forgetting the polynomial labels of an Elias--Williamson graph.
We denote $X_{\ux}^{\scrL} = \For_{\textnormal{poly}}^{-1} \hat{X}_{\ux}^{\scrL}$, i.e., the set of all possible light leaf constructions with domain $B_{\ux}^{\scrL}$ along with arbitrary polynomial decorations.

For $M \geq 0$, we consider the following two statements.
\begin{enumerate}
    \item[$(L_M)$] Let $\delta \in \Delta$. Every map $B_{\ux}^{\scrL} \to B_{\uw}^{\scrL}$ of maxwidth $\leq M$ is in the left $R$-span of $ \{ {}^{\delta} \LL_{\ux, \ue} \}_{\ux^{\ue} = \uw}$ modulo $I_{\uw}^{\scrL}$.
    \item[$(X_M)$] Every map $B_{\ux}^{\scrL} \to B_{\uw}^{\scrL}$ of maxwidth $\leq M$ is in the left $\k$-span of $X_{\ux}^{\scrL}$ modulo $I_{\uw}^{\scrL}$.
\end{enumerate}

It is clear that $(L_M)$ implies $(X_M)$. The inductive argument to prove Proposition \ref{prop:the_evil_EW_prop} then comes in two main steps: (1) prove that $(X_M)$ and $(L_{M-1})$ implies $(L_M)$, and (2) prove that $(L_M)$ implies $(X_{M+1})$.

If $M = 0$, then a map has maxwidth 0 only when the underlying Elias--Williamson graph consists only of polynomial boxes. 
As a result, $(L_0)$ and $(X_0)$ trivially hold.

We will now assume that $M > 0$.
Assume that $\ell (\uw) = M$. If $\ell (\ux) < M$, then we can apply induction. As a result, we can also assume that $\ell (\ux) = M$.
It can be argued as in \cite[\S 7.4]{EW}, that one can reduce to the case of a map $f : B_{\ux}^{\scrL} \to B_{\uw}^{\scrL}$ consisting solely of $2m$-valent vertices and polynomials.
By polynomial forcing and the Elias--Jones--Wenzl relations, we can move all polynomials to the left at the cost of potentially introducing lower terms.
It then suffices to replace $f$ by a map where the polynomials are removed.
Note that for a single fixed light leaf $\LL = {}^{\delta} \LL_{\ux, \ue}^{\scrL}$ (where $\ue$ is the all 1's subexpression), $\LL$ consists solely of $2m$-valent vertices.
By Lemma \ref{lem:neg_pos_decomps} (4), the difference of $f$ and $\LL$ is contained in the lower terms ideal.
In particular, $f$ is in the span of $\LL$, modulo $I_{\uw}^{\scrL}$. Therefore, $(L_M)$ and likewise $(X_M)$ holds for $\uw$.
As a result, we may assume that $\ell (\uw) < M$.

\subsection{\texorpdfstring{$(X_M)$ and $(L_{M-1}) \implies (L_M)$}{XM and LM-1 implies LM}}

Before proceeding with the proof, we introduce the notion of light leaves with ``errors''.
Let $\ux \in \Exp (W)$, $\ue \in \Subexp^{\scrL} (\ux)$, and $\delta \in \Delta$. 
Given a map $E : B_{\uw_j}^{\scrL} \to B_{\uw_j}^{\scrL}$ which is strictly negative-positive, we can define a map $\LL_{\ux, \ue}^{E, \scrL} : B_{\ux}^{\scrL} \to B_{\uw}^{\scrL}$,
called the \emph{light leaf with error} $E$, via the same algorithm as light leaves in \S\ref{subsubsec:diag_ll_section}
except we replace $\phi_i$ by a map $\phi_i'$ which is defined by $\phi_i' = \phi_i$ if $i \neq j$ and $\phi_j' = E \circ \phi_j$.
We denote the set of all light leaves with errors whose domain is $B_{\ux}^{\scrL}$ by $\hat{X}_{\ux}^{\textnormal{er}, \scrL}$.

\begin{lemma}\label{lem:sliding_and_ll_with_errors}
    Fix $\delta \in \Delta$.
    \begin{enumerate}
        \item Let $\LL \in \For_{\textnormal{poly}}^{-1} ({}^{\delta} \LL_{\ux, \ue}^{\scrL})$. Then $\LL - {}^{\delta} \LL_{\ux, \ue}^{\scrL}$ is in the span of $\hat{X}_{\ux}^{\textnormal{er}, \scrL}$ modulo $I_{\uw}^{\scrL}$
        \item Let $\delta' \in \Delta$. Then ${}^{\delta'} \LL_{\ux, \ue}^{\scrL} - {}^{\delta} \LL_{\ux, \ue}^{\scrL}$ is in the span of $\hat{X}_{\ux}^{\textnormal{er}, \scrL}$ modulo $I_{\uw}^{\scrL}$.
    \end{enumerate}
\end{lemma}
\begin{proof}
    For (1), polynomial forcing implies that the cost of sliding a polynomial across a rex move is adding at most an error term.
    We can then force all polynomials to the left, modulo lower terms to show that $\LL - {}^{\delta} \LL_{\ux, \ue}^{\scrL}$ is in the lower terms ideal.

    For (2), recall from Lemma \ref{lem:neg_pos_decomps} (4) that the difference between two rex moves is an error.
    As a result, if ${}^{\delta'} \LL_{\ux, \ue}^{\scrL}$ and ${}^{\delta} \LL_{\ux, \ue}^{\scrL}$ differ only in the choice of rex moves, then ${}^{\delta'} \LL_{\ux, \ue}^{\scrL} - {}^{\delta} \LL_{\ux, \ue}^{\scrL}$ is spanned by lower terms.
    If  ${}^{\delta'} \LL_{\ux, \ue}^{\scrL}$ and ${}^{\delta} \LL_{\ux, \ue}^{\scrL}$ instead differs from the choice of $\rex_s (w_k)$ for some $k$, we can account for the difference using rex moves as well.
    Namely, we can replace the identity of $B_{\rex_s^{\delta} (w_k)}^{\scrL}$ appearing in a light leaf with a rex move $gf$ where $g : B_{\rex_s^{\delta'} (w_k)}^{\scrL} \to B_{\rex_s^{\delta} (w_k)}^{\scrL}$ and $f : B_{\rex_s^{\delta} (w_k)}^{\scrL} \to B_{\rex_s^{\delta'} (w_k)}^{\scrL}$ are rex moves.
    Then $f$ can be absorbed into the rex move in $\phi_k$, and $g$ can be absorbed into the rex move in $\phi_{k+1}$.
    Combining both cases, we may conclude that ${}^{\delta'} \LL_{\ux, \ue}^{\scrL} - {}^{\delta} \LL_{\ux, \ue}^{\scrL}$ is a linear combination of lower terms.
\end{proof}

Fix some $\delta_0 \in \Delta$.
Let $\LL \in X_{\ux}^{\scrL}$, so $\For_{\textnormal{poly}} (\LL) = {}^{\delta} \LL_{\ux, \ue}^{\scrL}$ for some $\delta \in \Delta$.
By Lemma \ref{lem:sliding_and_ll_with_errors}, we have that the difference between $\LL$ and ${}^{\delta_0} \LL_{\ux, \ue}^{\scrL}$ is in the span of light leaves with errors.
As a result, it suffices to prove that ${}^{\delta} \LL_{\ux, \ue}^{E, \scrL}$ is in the span of $\{ {}^{\delta_0} \LL_{\ux, \uf}^{\scrL} \}_{\uf}$ modulo lower terms.

Let $\LL_{M-1}$ denote the partially constructed light leaf with error map ${}^{\delta} \LL_{\ux, \ue}^{E, \scrL}$ after $M-1$ steps.
Now $\LL_{M-1}$ has maxwidth $M-1$, so we may apply induction and replace $\LL_{M-1}$ with a $\LL$ map of our choosing or with a term having dots on top.

If the error occurs at the last step, then by Lemma \ref{lem:neg_pos_decomps} (4), some dot will pull to the top, yielding a map in $I_{\uw}^{\scrL}$.
As a result, the light leaf with error map ${}^{\delta} \LL_{\ux, \ue}^{E, \scrL}$ can then be decomposed as follows:
\[{}^{\delta} \LL_{\ux, \ue}^{E, \scrL} = 
  % \tikzsetnextfilename{#1}
  \tikzstyle{every picture}=[tikzfig]
  \input{./figures/ll_with_errors_spanning_1.tikz}
\]
where $\phi_k$ has no error. There are four possibilities for applying $\phi_k$ to $\LL_{k-1}$ which are described below:
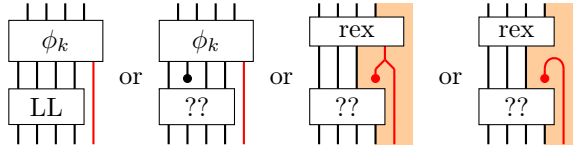
\[
  % \tikzsetnextfilename{#1}
  \tikzstyle{every picture}=[tikzfig]
  \input{./figures/ll_with_errors_spanning_2.tikz}
\]
By Lemma \ref{lem:neg_pos_decomps} (4), we can replace rex moves so that the first diagram is actually a light leaf with data $\delta_0$.
By Lemma \ref{lem:neg_pos_decomps}, the second diagram is in $I_{\uw}^{\scrL}$.
For the third diagram, we can replace the question marks box with a light leaf with data $\delta_0$ such that the resulting map is the light leaf with data $\delta_0$ with $e_k$ decorated by $U1$ (after modifying rex moves).
The fourth diagram is similar to the third, except the resulting light leaf has $e_k$ decorated by $U0$.

\subsection{\texorpdfstring{$(L_{M-1}) \implies (X_M)$}{LM-1 implies XM}}

Fix a reduced expression $\uw$ with $\ell (\uw) < M$. We will prove that $(L_{M-1})$ implies $(X_M)$ for $\uw$.
Let $f : B_{\ux}^{\scrL} \to B_{\uw}^{\scrL}$ be a diagram of maxwidth $M$ (the case of maxwidth $< M$ is covered by induction).

\begin{convention}
    We adopt some simplifications to make writing light leaves simpler. Samples of these conventions are given below,
    \[
  % \tikzsetnextfilename{#1}
  \tikzstyle{every picture}=[tikzfig]
  \input{./figures/convention_1.tikz}
\]
    The first diagram should be inferred as some light leaf without specifying the expression $\ux$, subexpression $\ue$, or monodromy parameter $\scrL$.
    If we wish to emphasize the subexpression $\ue$ used to define a light leaf, we will write $e_i$ over the $i$-th strand as in the second diagram (here $e_i = 1$). 
    If we want to further emphasize the decoration of $e_i$, we will similarly write $\dec_i (\ux, \ue)$ over the $i$-th strand as in the third diagram (here $\dec_i (\ux, \ue) = U1$).  

    We also adopt some simplifications for rex moves as given below.
    \[
  % \tikzsetnextfilename{#1}
  \tikzstyle{every picture}=[tikzfig]
  \input{./figures/convention_2.tikz}
\]
    The first diagram should be interpreted as a $2m$-valent vertex $\beta_{s,t}$ where $s,t$ can be inferred from context.
    The second diagram refers to some rex move without specifying the domain, codomain, or the chosen path in the rex graph.
\end{convention}

\subsubsection{Attaching Negative Generators}

\begin{lemma}\label{lem:add_neg_trivalent}
    Suppose a negative (merging) trivalent vertex is added below a light leaf. The result is a light leaf. 
\end{lemma}
\begin{proof}
    Suppose that a new trivalent vertex is attached to the $k$-th strand. Then we must have that $\scrL \ux_{\leq k} = \scrL \ux_{\leq k-1}$.
    The result then follows from the same argument given in \cite[Claim 7.9]{EW}.   
\end{proof}

\begin{lemma}\label{lem:add_neg_dot}
    Suppose a negative (bottom) dot is added below a light leaf yielding a map from $B_{\ux'}^{\scrL}$ for some $\ux' \in \Exp (W)$. The result is in $X_{\ux'}^{\scrL}$.
\end{lemma}
\begin{proof}
    Suppose that a new boundary dot colored by some $s \in S$ is added after the $k$-th strand. 
    Then we must have that $\scrL \ux_{\leq k} s = \scrL \ux_{\leq k}$.
    The result then follows from the same argument given in \cite[Claim 7.10]{EW}.   
\end{proof}

\begin{lemma}\label{lem:add_neg_cap}
    Suppose a negative cap is added below a light leaf yielding a map from $B_{\ux'}^{\scrL}$ for some $\ux' \in \Exp (W)$. The result is in $X_{\ux'}^{\scrL}$.
\end{lemma}
\begin{proof}
    Suppose that the cap is colored by some $s \in S$ and is added after the $k$-th strand.
    If $\scrL \ux_{\leq k} s = \scrL \ux_{\leq k}$, then the result follows from Lemmas \ref{lem:add_neg_trivalent} and \ref{lem:add_neg_dot}.
    We may then assume that $\scrL \ux_{\leq k} s \neq \scrL \ux_{\leq k}$.
    If $s$ is not in the right descent set of $w_k$, then the result is a light leaf, where the new strands are labelled by $U1$ and $D1$.
    If $s$ is in the right descent set, by choosing our light leaf appropriately, we can ensure that $s$ occurs on the right of $\uw_k$.
    By block minimality, we have an equality
    \[
  % \tikzsetnextfilename{#1}
  \tikzstyle{every picture}=[tikzfig]
  \input{./figures/add_neg_cap_0.tikz}
\]
    Note that the RHS is itself a light leaf where the new strand is decorated by $D1$.
\end{proof}

\subsubsection{Attaching Positive Generators}\label{subsec:add_pos_generators}

\begin{lemma}\label{lem:add_pos_trivalent}
    Assume $(L_{M-1})$. Let $f : B_{\ux}^{\scrL} \to B_{\uw}^{\scrL}$ be in $X_{\ux}^{\scrL}$ such that $\ell (\ux) \leq M$.
    If we add a positive (splitting) trivalent vertex below $f$ to obtain a map from $B_{\ux'}^{\scrL}$ for some $\ux' \in \Exp (W)$, then the result is in the span of $X_{\ux'}^{\scrL}$ modulo $I_{\uw}^{\scrL}$.
\end{lemma}
\begin{proof}
    Suppose that a new trivalent vertex is attached to the $k$-th and $(k+1)$-st strands. Then we must have that $\scrL \ux_{\leq k} = \scrL \ux_{\leq k-1}$.
    The result then follows from the same argument given in \cite[Claim 7.15]{EW}.   
\end{proof}

\begin{lemma}\label{lem:add_pos_dot}
    Assume $(L_{M-1})$. Let $f : B_{\ux}^{\scrL} \to B_{\uw}^{\scrL}$  be in $X_{\ux}^{\scrL}$ such that $\ell (\ux) \leq M$.
    If we add a positive (top) dot below $f$ to obtain a map from $B_{\ux'}^{\scrL}$ for some $\ux' \in \Exp (W)$, then the result is in the span of $X_{\ux'}^{\scrL}$ modulo $I_{\uw}^{\scrL}$.
\end{lemma}
\begin{proof}
    Suppose that a new top dot is added below the $k$-th strand. 
    Then we must have that $\scrL \ux_{\leq k} = \scrL \ux_{\leq k-1}$.
    The result then follows from the same argument given in \cite[Claim 7.17]{EW}.   
\end{proof}

\begin{lemma}\label{lem:add_pos_cap}
    Assume $(L_{M-1})$. Let $f : B_{\ux}^{\scrL} \to B_{\uw}^{\scrL}$ be in $X_{\ux}^{\scrL}$ such that $\ell (\ux) \leq M$.
    If we add a positive cup below $f$ to obtain a map from $B_{\ux'}^{\scrL}$ for some $\ux' \in \Exp (W)$, then the result is in the span of $X_{\ux'}^{\scrL}$ modulo $I_{\uw}^{\scrL}$.
\end{lemma}
\begin{proof}
    Suppose that the cap is colored by some $s \in S$ and is attached to the $k$-th and $(k+1)$-st strands.
    If $\scrL \ux_{\leq k} s = \scrL \ux_{\leq k}$, then the lemma follows from Lemmas \ref{lem:add_pos_trivalent} and \ref{lem:add_pos_dot}.
    We may then assume that $\scrL \ux_{\leq k} s \neq \scrL \ux_{\leq k}$.
    The only possible decorations for $e_k$ and $e_{k+1}$ are either ($U1$, $D1$) or ($D1$, $U1$).
    In the former case, we can replace our light leaf so that it has no rex move above $e_k$. By block minimality, we have an equality
    \[
  % \tikzsetnextfilename{#1}
  \tikzstyle{every picture}=[tikzfig]
  \input{./figures/add_pos_cup_0.tikz}
\]
    which is itself a light leaf.
    In the latter case, by choosing our light leaf appropriately, we can ensure that $s$ occurs on the right of $\uw_{k}$.
    By isotopy, we have an equality
    \[
  % \tikzsetnextfilename{#1}
  \tikzstyle{every picture}=[tikzfig]
  \input{./figures/add_pos_cup_1.tikz}
\]
    which, again, is itself a light leaf.
\end{proof}

\subsubsection{Attaching a \texorpdfstring{$2m$}{2m}-valent Vertex}\label{subsec:add_neg_generators}

The last major component we need to complete the proof of Proposition \ref{prop:the_evil_EW_prop} is the claim that adding a $2m$-valent vertex below a light leaf is in the span of light leaves modulo lower terms.

\begin{lemma}\label{lem:add_2m_vertex}
    Assume $(L_{M-1})$. Let $f : B_{\ux}^{\scrL} \to B_{\uw}^{\scrL}$ be in $X_{\ux}^{\scrL}$ such that $\ell (\ux) \leq M$.
    If we add a $2m$-valent vertex below $f$ to obtain a map from $B_{\ux'}^{\scrL}$ for some $\ux' \in \Exp (W)$, then the result is in the span of $X_{\ux'}^{\scrL}$ modulo $I_{\uw}^{\scrL}$.
\end{lemma}

The non-monodromic variant of Lemma \ref{lem:add_2m_vertex} is \cite[Claim 7.13]{EW}.
Our proof is quite different from \cite{EW}. In fact, there is a subtle error in the proof provided by Elias--Williamson.
The general strategy is to reduce to the case of $W$ being a dihedral group. 
From here, Elias--Williamson appeal to the calculations of morphisms spaces done in \cite{Elib} using Soergel's categorification theorem.
Unfortunately, \cite{Elib} crucially imposes additional constraints on realizations, namely that of ``lesser invertibility'' and ``local non-degeneracy''.
These constraints are much stronger than the rather relaxed constraints on realizations imposed in \cite{EW} or in this work. 
Our proof constitutes a corrected proof of \cite[Claim 7.13]{EW} by taking $\fr{o} = 1$.

\begin{lemma}[{\cite[Claim 7.12]{EW}}]\label{lem:thinning_reduction}
    Assume $(L_{M-1})$ and $\ell (\ux) = M$. 
    Let $f : B_{\ux}^{\scrL} \to B_{\uw}^{\scrL}$ such $f$ never returns to width $M$ after leaving it.
    Suppose further that $f$ can be written as follows for $0 < k < M$.
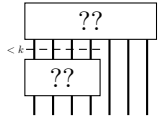
    \[
  % \tikzsetnextfilename{#1}
  \tikzstyle{every picture}=[tikzfig]
  \input{./figures/thinning_reduction_0.tikz}
\]
    Then $f$ is in the span of $X_{\ux}^{\scrL}$ modulo $I_{\uw}^{\scrL}$.  
\end{lemma}

\begin{claim}\label{claim:dihedral_add_2m_vertex}
    Assume $(L_{M-1})$. Let $W$ be a finite dihedral group of type $I_2 (m)$.
    Let $l \geq 0$ such that $l + m \leq M$ and take $\ux \in \Exp (W)$ such that $\ell (\ux) = l + m$.
    Take a subexpression $\ue$ of $\ux$ such that $e_i$ are decorated by $U1$ for all $1 \leq i \leq l$. I.e., $(\ux, \ue)_{\leq l}$ is a reduced expression and $e_i = 1$ for all $1 \leq i \leq l$.
    Then the following diagram
    \begin{equation}\label{thm:1}
            
  % \tikzsetnextfilename{#1}
  \tikzstyle{every picture}=[tikzfig]
  \input{./figures/add_2m_valent/thm_1.tikz}

    \end{equation}
    is in the span of $X_{\ux'}^{\scrL}$ modulo $I_{\ux'}^{\scrL}$ (where $\ux'$ is the expression obtained by applying a braid relation at the end of $\ux$).
\end{claim}

\begin{midsecproof}{Lemma \ref{lem:add_2m_vertex} assuming Claim \ref{claim:dihedral_add_2m_vertex}}
    We provide a sketch of the argument. For more details, see \cite[Claim 7.13]{EW}.
    Suppose that we add the $2m$-valent vertex of colors $s,t$ to strands $\ux_{l + 1}$ through $x_{l + m}$ of a light leaf $\LL_{\ux, \ue}^{\scrL}$.
    If $l + m < \ell (\ux)$, we can use Lemma \ref{lem:thinning_reduction} and be done.
    As a result, we may assume that $l+m = \ell (\ux) = M$.
    Another application of Lemma \ref{lem:thinning_reduction} allows us to further reduce to the following special case: $\ux_{\leq l}$ is a reduced expression and $e_{\leq l}$ is decorated solely by $U1$'s.

    By altering rex moves, we can choose any desired reduced expression $\underline{w}_l$ for $w_l$.
    Write $w_l = zv$ for $v \in W_{s,t}$ and $z$ a minimal right coset representative of $W_{s,t}$.
    We can then choose a reduced expression for $w_{l}$ of the form $\underline{z} \underline{v}$ where $\underline{z}$ and $\underline{v}$ are reduced expressions for $v$ and $u$ respectively.
    By further modifying rex moves, we may then assume that $\LL_{\ux, \ue}^{\scrL}$ is a concatenation of a light leaf with domain $B_{\underline{z}}^{\scrL}$ and a light leaf with domain $B_{\underline{v} {}_u \underline{m}}^{\scrL \underline{z}}$ for some $u \in \{s,t\}$ where $\{s,t\}$ are the colors on $\beta$.
    Without loss of generality, we may then assume that $w_l = v$ and that the diagram is an Elias--Williamson graph for $W_{s,t}$.
    We are then done by Claim \ref{claim:dihedral_add_2m_vertex}.  
\end{midsecproof}

The proof of Claim \ref{claim:dihedral_add_2m_vertex} will be by induction on the length of $\ux$. We first need to set up some preliminary lemmas which will allow the induction to be performed.

\begin{lemma}\label{lem:red_to_alternating}
    Assume that Claim \ref{claim:dihedral_add_2m_vertex} holds for all $\ux \in \Exp (W)$ of the form ${}_s \underline{n}$ and ${}_t \underline{n}$.
    Then Claim \ref{claim:dihedral_add_2m_vertex} holds for all $\ux \in \Exp (W)$.
\end{lemma}
\begin{proof}
    By definition, $\ux$ is either of the form ${}_s \underline{l} {}_u \underline{m}$ or  ${}_t \underline{l} {}_u \underline{m}$ for some $u \in \{s,t\}$.
    The only case not covered by our assumption is when the last term of ${}_s \underline{l}$ is $u$. For simplicity, we set $u = s$. 
    At the cost of lower terms, we may modify the $l$-th step of the light leaf construction so that it has no rex moves.
    
    Assume that $e_{l+1}$ is decorated by $D1$.
    We can then compute
    \begin{equation}
        
  % \tikzsetnextfilename{#1}
  \tikzstyle{every picture}=[tikzfig]
  \input{./figures/add_2m_valent/red_to_alt_1.tikz}
 =  
  % \tikzsetnextfilename{#1}
  \tikzstyle{every picture}=[tikzfig]
  \input{./figures/add_2m_valent/red_to_alt_2.tikz}
 = 
  % \tikzsetnextfilename{#1}
  \tikzstyle{every picture}=[tikzfig]
  \input{./figures/add_2m_valent/red_to_alt_3.tikz}
 = 
  % \tikzsetnextfilename{#1}
  \tikzstyle{every picture}=[tikzfig]
  \input{./figures/add_2m_valent/red_to_alt_4.tikz}
 \textnormal{ mod } I_{\ux}^{\scrL} 
    \end{equation}
    where the last equality follows from Lemma \ref{lem:add_neg_cap}. 
    The case can then be finished by Lemma \ref{lem:thinning_reduction}.

    Now assume that $e_{l+1}$ is decorated by $D0$. We compute
    \begin{equation}
        
  % \tikzsetnextfilename{#1}
  \tikzstyle{every picture}=[tikzfig]
  \input{./figures/add_2m_valent/red_to_alt_5.tikz}
 =  
  % \tikzsetnextfilename{#1}
  \tikzstyle{every picture}=[tikzfig]
  \input{./figures/add_2m_valent/red_to_alt_6.tikz}
 = 
  % \tikzsetnextfilename{#1}
  \tikzstyle{every picture}=[tikzfig]
  \input{./figures/add_2m_valent/red_to_alt_7.tikz}
 = 
  % \tikzsetnextfilename{#1}
  \tikzstyle{every picture}=[tikzfig]
  \input{./figures/add_2m_valent/red_to_alt_4.tikz}
 \textnormal{ mod } I_{\ux}^{\scrL},
    \end{equation}
    where the second equality is 2-color associativity and the last equality is a combination of Claim \ref{claim:dihedral_add_2m_vertex} for $\ux$ alternating and Lemma \ref{lem:add_neg_trivalent}.
    As in the previous case, we are done by Lemma \ref{lem:thinning_reduction}.
\end{proof} 

\begin{lemma}\label{lem:vertex_replacement}
    Assume that Claim \ref{claim:dihedral_add_2m_vertex} holds for a fixed $l \geq 0$. Let $\ell (\ux) = l + m$.
    Then every light leaf $\LL_{\ux, \ue}^{\scrL}$ is in the span of morphisms of the form 
    \[ 
  % \tikzsetnextfilename{#1}
  \tikzstyle{every picture}=[tikzfig]
  \input{./figures/add_2m_valent/thm_1f.tikz}
\]
    modulo $I_{\ux}^{\scrL}$.
\end{lemma}
\begin{proof}
    By Claim \ref{claim:dihedral_add_2m_vertex}, the following map
    \begin{equation*}
            
  % \tikzsetnextfilename{#1}
  \tikzstyle{every picture}=[tikzfig]
  \input{./figures/add_2m_valent/thm_1.tikz}

    \end{equation*}
    is in the span of $X_{\ux'}^{\scrL}$ modulo $I_{\ux'}^{\scrL}$ (where $\ux'$ is the expression obtained by applying a braid relation at the end of $\ux$).
    We can then attach another $2m$-valent vertex under the right most $m$-strands and apply the Elias--Jones--Wenzl relation to complete the proof.
\end{proof}

\begin{midsecproof}{Claim \ref{claim:dihedral_add_2m_vertex}}
    By Lemma \ref{lem:red_to_alternating}, we can reduce to the case where $\ux$ is of the form ${}_u \underline{n}$ and $n = l + m$.
    We will then argue by induction on $l \geq 0$. To keep notation and pictures concise, we will always assume that $\ux = {}_u \underline{n}$ ends in $t$.
    Before proceeding with the induction, we can make a useful observation. Without loss of generality, we may assume that $\ue$ has no decorations of the form $U0$.
    Indeed, if $\ue$ has a decoration of $U0$ over the $2m$-valent vertex, then the Jones--Wenzl relations reduce the computation to checking whether adding a non-neutral generator below a light leaf is in the span of light leaves modulo lower terms.
    These have already been shown in \S\ref{subsec:add_pos_generators} and \S\ref{subsec:add_neg_generators}.

    If $l = 0$, then $\ux = \underline{m}_t$. Since we only consider $\ue$ with no $U0$ decorations, it easily follows that $\ue$ must be decorated solely by $U1$'s.
    In this case, $\LL_{\ux, \ue}^{\scrL}$ composed with a $2m$-valent vertex is simply a series of compositions of $2m$-valent vertices.
    The base case then follows from successive applications of the Elias--Jones--Wenzl relation.

    Now assume that $l > 0$. Write $\ue = \ue_{\leq n-1} e_n$.
    By Lemma \ref{lem:vertex_replacement} and induction, we may replace $\LL_{\ux_{\leq n-1}, \ue_{\leq n-1}}^{\scrL}$ by 
    \begin{equation}\label{eq:replacement_via_valent_symmetry}
        
  % \tikzsetnextfilename{#1}
  \tikzstyle{every picture}=[tikzfig]
  \input{./figures/add_2m_valent/thm_2.tikz}

    \end{equation}
    where $\uf$ is a subexpression of $\uy$ such that $\uf$ is decorated by $U1$ for all $1 \leq i \leq l-1$.
    There are 5 possible cases for the pair $(\dec_n (\ue), \dec_{n-1} (\uf))$ which avoids a $U0$ decoration on $\ue$: ($D1$, $U1$), ($D1$, $D0$), ($D0$, $U1$), ($D0$, $D0$), and ($U1$, $U0$).
    Throughout our case-by-case analysis, we will often abuse notation and write $\LL_k$ to refer to a light leaf of maxwidth $k$.

    \emph{Case 1: $\dec_n (\ue) = D1$ and $\dec_{n-1} (\uf) = U1$.}
    We can compute
    \[ 
  % \tikzsetnextfilename{#1}
  \tikzstyle{every picture}=[tikzfig]
  \input{./figures/add_2m_valent/case_1a_0.tikz}
 = 
  % \tikzsetnextfilename{#1}
  \tikzstyle{every picture}=[tikzfig]
  \input{./figures/add_2m_valent/case_1a_1.tikz}
 = 
  % \tikzsetnextfilename{#1}
  \tikzstyle{every picture}=[tikzfig]
  \input{./figures/add_2m_valent/case_1a_2.tikz}
 = 
  % \tikzsetnextfilename{#1}
  \tikzstyle{every picture}=[tikzfig]
  \input{./figures/add_2m_valent/case_1a_3.tikz}
 \text{ mod } I_{\ux'}^{\scrL}.\]
    Since the domain of the light leaf on the LHS of the first
    equation ends in $t$ and the codomain of the rex move ends in $t$, modulo lower terms,
    we pull the cap through the rex move and the light leaf. 
    The Jones--Wenzl morphism can be absorbed into the light leaf above it using the results from \S\ref{subsec:add_pos_generators} and \S\ref{subsec:add_neg_generators}.

    \emph{Case 2: $\dec_n (\ue) = D1$ and $\dec_{n-1} (\uf) = D0$.}
    We can compute
    \[ 
  % \tikzsetnextfilename{#1}
  \tikzstyle{every picture}=[tikzfig]
  \input{./figures/add_2m_valent/case_1a_0.tikz}
 = 
  % \tikzsetnextfilename{#1}
  \tikzstyle{every picture}=[tikzfig]
  \input{./figures/add_2m_valent/case_1b_1.tikz}
 = 
  % \tikzsetnextfilename{#1}
  \tikzstyle{every picture}=[tikzfig]
  \input{./figures/add_2m_valent/case_1b_2.tikz}
 = 
  % \tikzsetnextfilename{#1}
  \tikzstyle{every picture}=[tikzfig]
  \input{./figures/add_2m_valent/case_1b_3.tikz}
 = 
  % \tikzsetnextfilename{#1}
  \tikzstyle{every picture}=[tikzfig]
  \input{./figures/add_2m_valent/case_1b_4.tikz}
 \text{ mod } I_{\ux'}^{\scrL}.\]
    The third equality follows from an application of 2-color associativity. The result
    then follows from induction and Lemma \ref{lem:add_neg_trivalent}.

    \emph{Case 3: $\dec_n (\ue) = D0$ and $\dec_{n-1} (\uf) = U1$.}
    We can compute
    \[ 
  % \tikzsetnextfilename{#1}
  \tikzstyle{every picture}=[tikzfig]
  \input{./figures/add_2m_valent/case_2_0.tikz}
 = 
  % \tikzsetnextfilename{#1}
  \tikzstyle{every picture}=[tikzfig]
  \input{./figures/add_2m_valent/case_2a_1.tikz}
 = 
  % \tikzsetnextfilename{#1}
  \tikzstyle{every picture}=[tikzfig]
  \input{./figures/add_2m_valent/case_2a_2.tikz}
 = 
  % \tikzsetnextfilename{#1}
  \tikzstyle{every picture}=[tikzfig]
  \input{./figures/add_2m_valent/case_2a_3.tikz}
 = 
  % \tikzsetnextfilename{#1}
  \tikzstyle{every picture}=[tikzfig]
  \input{./figures/add_2m_valent/case_2a_4.tikz}
 \text{ mod } I_{\ux'}^{\scrL}.\]
    Here we are using that since the domain of the light leaf on the LHS of the first
    equation ends in $t$ and the codomain of the rex move ends in $t$, we can pull the
    trivalent vertex through the rex move and the light leaf at the cost of some lower
    terms. The third equality follows from 2-color associativity. The result then
    follows from induction and Lemma \ref{lem:add_neg_trivalent}.

    \emph{Case 4: $\dec_n (\ue) = D0$ and $\dec_{n-1} (\uf) = D0$.}
    We can compute
    \[ 
  % \tikzsetnextfilename{#1}
  \tikzstyle{every picture}=[tikzfig]
  \input{./figures/add_2m_valent/case_2_0.tikz}
 = 
  % \tikzsetnextfilename{#1}
  \tikzstyle{every picture}=[tikzfig]
  \input{./figures/add_2m_valent/case_2b_1.tikz}
 = 
  % \tikzsetnextfilename{#1}
  \tikzstyle{every picture}=[tikzfig]
  \input{./figures/add_2m_valent/case_2b_2.tikz}
 = 
  % \tikzsetnextfilename{#1}
  \tikzstyle{every picture}=[tikzfig]
  \input{./figures/add_2m_valent/case_2b_3.tikz}
 = 
  % \tikzsetnextfilename{#1}
  \tikzstyle{every picture}=[tikzfig]
  \input{./figures/add_2m_valent/case_2b_4.tikz}
 \text{ mod } I_{\ux'}^{\scrL}.\]
    The third equality follows from 2-color associativity. The result then follows from
    induction and Lemma \ref{lem:add_neg_trivalent}.

    \emph{Case 5: $\dec_n (\ue) = U1$ and $\dec_{n-1} (\uf) = U0$.} We are in the following case.
    \[
  % \tikzsetnextfilename{#1}
  \tikzstyle{every picture}=[tikzfig]
  \input{./figures/add_2m_valent/case_3_0.tikz}
 = 
  % \tikzsetnextfilename{#1}
  \tikzstyle{every picture}=[tikzfig]
  \input{./figures/add_2m_valent/case_3_1.tikz}
 = 
  % \tikzsetnextfilename{#1}
  \tikzstyle{every picture}=[tikzfig]
  \input{./figures/add_2m_valent/case_3_2.tikz}
.\]
    
    We must handle two subcases. First, suppose $\scrL \ux \neq \scrL \ux_{\leq n-1}$. We can then compute
    \[
  % \tikzsetnextfilename{#1}
  \tikzstyle{every picture}=[tikzfig]
  \input{./figures/add_2m_valent/case_3_13.tikz}
 = 
  % \tikzsetnextfilename{#1}
  \tikzstyle{every picture}=[tikzfig]
  \input{./figures/add_2m_valent/case_3_11.tikz}
 = 
  % \tikzsetnextfilename{#1}
  \tikzstyle{every picture}=[tikzfig]
  \input{./figures/add_2m_valent/case_3_12.tikz}
.\]
    The first equality follows from block minimality.
    The Jones--Wenzl morphism can be absorbed into the light leaf above it using the results from \S\ref{subsec:add_pos_generators} and \S\ref{subsec:add_neg_generators}.

    Now, suppose $\scrL \ux = \scrL \ux_{\leq n-1}$ and write $\scrL' = \scrL \ux$. Recall that using Demazure surjectivity, we have fixed $\delta = \delta_t \in R$ such that $\partial_t (\delta) = 1$. 
    We can then compute
    \[ 
  % \tikzsetnextfilename{#1}
  \tikzstyle{every picture}=[tikzfig]
  \input{./figures/add_2m_valent/case_3_14.tikz}
 = 
  % \tikzsetnextfilename{#1}
  \tikzstyle{every picture}=[tikzfig]
  \input{./figures/add_2m_valent/case_3_3.tikz}
 - 
  % \tikzsetnextfilename{#1}
  \tikzstyle{every picture}=[tikzfig]
  \input{./figures/add_2m_valent/case_3_4.tikz}
 \text{ mod } I_{\ux'}^{\scrL} \]
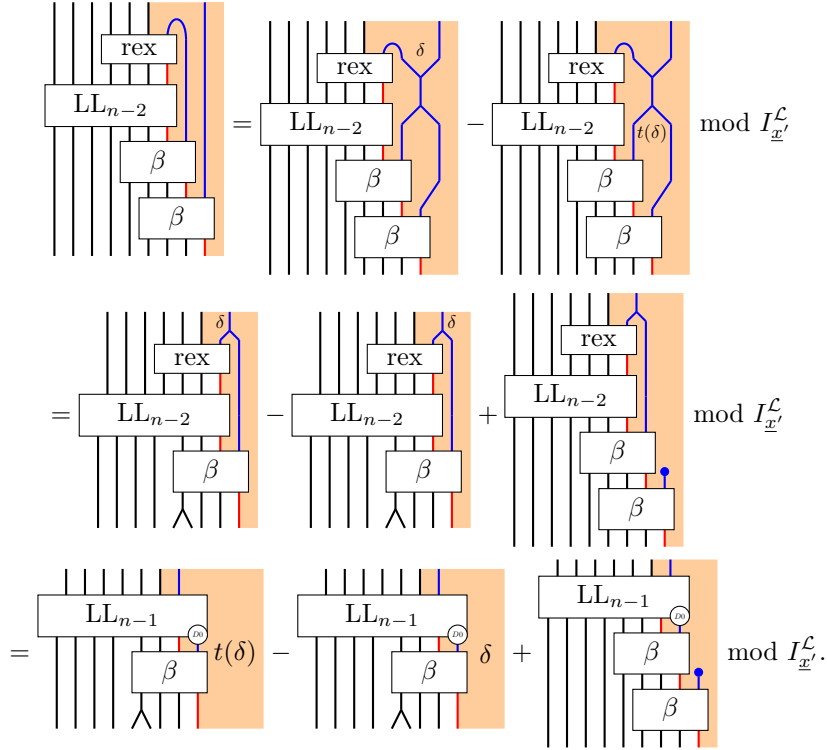
    \[ = 
  % \tikzsetnextfilename{#1}
  \tikzstyle{every picture}=[tikzfig]
  \input{./figures/add_2m_valent/case_3_5.tikz}
 - 
  % \tikzsetnextfilename{#1}
  \tikzstyle{every picture}=[tikzfig]
  \input{./figures/add_2m_valent/case_3_6.tikz}
 + 
  % \tikzsetnextfilename{#1}
  \tikzstyle{every picture}=[tikzfig]
  \input{./figures/add_2m_valent/case_3_7.tikz}
 \text{ mod } I_{\ux'}^{\scrL}\] 
    \[ = 
  % \tikzsetnextfilename{#1}
  \tikzstyle{every picture}=[tikzfig]
  \input{./figures/add_2m_valent/case_3_8.tikz}
 - 
  % \tikzsetnextfilename{#1}
  \tikzstyle{every picture}=[tikzfig]
  \input{./figures/add_2m_valent/case_3_9.tikz}
 + 
  % \tikzsetnextfilename{#1}
  \tikzstyle{every picture}=[tikzfig]
  \input{./figures/add_2m_valent/case_3_10.tikz}
 \text{ mod } I_{\ux'}^{\scrL}.\]
    The first equality follows from the splitting of $B_t^{\scrL'} \star B_t^{\scrL'}$ (cf., \cite[5.15]{EW}).
    The second and third equalities both use polynomial forcing. The first two summands of the RHS are in the span of light leaves by induction and Lemma \ref{lem:add_neg_trivalent}.
    The third summand can be resolved via the Jones--Wenzl relation as in Case 1.
\end{midsecproof}

\subsubsection{Completing the Proof}

\begin{midsecproof}{Proposition \ref{prop:the_evil_EW_prop}}
Let $f : B_{\ux}^{\scrL} \to B_{\uw}^{\scrL}$ be a diagram of maxwidth $M$.
Since $\ell (\uw) < M$, we can write $f = f' \circ g_n \circ \ldots \circ g_1$ such that $f'$ has maxwidth $< M$ and each $g_i$ has height 1 for $1 \leq i \leq n$.
By induction, we can replace $f'$ with a light leaf.
Each $g_i$ is a generator tensored with identity maps on either side. We can then apply Lemmas \ref{lem:add_neg_trivalent}, \ref{lem:add_neg_dot}, \ref{lem:add_neg_cap}, \ref{lem:add_pos_trivalent}, \ref{lem:add_pos_dot}, \ref{lem:add_pos_cap}, and \ref{lem:add_2m_vertex} inductively  
to replace $f' \circ g_n \circ \ldots \circ g_i$ with a light leaf.
As a result, $f$ is in the span of $X_{\ux}^{\scrL}$ modulo $I_{\uw}^{\scrL}$. 
\end{midsecproof}

    \section{Geometric Monodromic Hecke Category}\label{sec:geom}
    
    The celebrated Riche--Williamson theorem \cite[Theorem 10.5]{RW} gives a monoidal equivalence of categories
\[\EW^{\BS} (\fr{h}, W) \stackrel{\sim}{\to} \textnormal{Par}^{\BS} (B \backslash G / B, \k)\]
from the graded Elias--Williamson diagrammatic category to the category of Bott--Samelson parity sheaves on $B \backslash G / B$ where $G$ is a complex reductive group and $\k$ is an integral domain.\footnote{The Riche--Williamson theorem is in fact more general and holds when $G$ is a Kac--Moody group.}
In this section, we will establish a monodromic version of the Riche--Williamson theorem using the monodromic diagrammatic category.

Due to issues with choosing representatives of minimal IC sheaves in the (geometric) monodromic Hecke category, we will need to impose some constraints on our coefficient ring (see \S\ref{subsec:min_IC}).
Fix a finite extension $\O$ of $\Z_{\ell}$. Denote by $\K$ its field of fractions and by $\F$ its residue field.
The triple $(\K, \O, \F)$ is called an \emph{$\ell$-modular triple}. For this section, we will require $\k \in \{ \K, \O, \F\}$.
Note that $\k$ is either a complete local ring or a field.

\subsection{Preliminaries}

\subsubsection{Flag Varieties}

Let $G$ be a connected complex reductive group. Fix a Borel $B$ with Levi decomposition $B = T \ltimes U$.
We denote by $\bfX$ (resp. $\bfY$) the character (resp. cocharacter) lattice of $T$.
For any integral domain $\k$, one can consider the Kac--Moody realization $\fr{h}_{\k}^* = \k \otimes_{\Z} \bfX$ and $\fr{h}_{\k} = \k \otimes_{\Z} \bfY$.
Let $(W,S)$ denote the Weyl group of $G$ with respect to $T$. 
For each $s \in S$, there is a simple root $\alpha_s \in \bfX$ and a simple coroot $\alpha^{\vee}_s \in \bfY$.
We often view the simple roots (resp. coroots) as elements of $\fr{h}_{\k}^*$ (resp. $\fr{h}_{\k}$). 

Define a ring $\Z'$ as follows.
If all the maps $\alpha_s : \bfY \to \Z$ and $\alpha_s^{\vee} : \bfX \to \Z$ are surjective, we set $\Z' = \Z$, and otherwise, we set $\Z' = \Z [\frac{1}{2}]$.
For the entirety of this section, we assume that there exists a ring homomorphism $\Z' \to \k$.
The triple $(\fr{h}_{\k}, \{ \alpha_s^{\vee}\}_{s \in S}, \{\alpha_s\}_{s \in S})$ is a reflection-balanced reflection-stable Abe realization (Proposition \ref{prop:crystallographic_abe} and Remark \ref{rem:crystallographic_balancedness}).
Demazure surjectivity for $\fr{h}_{\k}$ is guaranteed from the ring homomorphism $\Z' \to \k$.
Since multiple coefficient rings will appear throughout this section, we will write $R_{\k} = \Sym (\fr{h}_{\k}^*)$ instead of $R$ to emphasize the dependence on $\k$.

Let $\eFl \coloneq U \backslash G$ denote the enhanced flag variety of $G$. Recall that the Bruhat decomposition
\[\eFl = \bigsqcup_{w \in W} \eFl_w\]
gives a stratification of $\eFl$
where each $\eFl_w$ is a $B$-orbit (non-canonically) isomorphic to $\A^{\ell (w)} \times T$.
We write $j_w : \eFl_w \to \eFl$ for the inclusion map.

\subsubsection{Hecke Categories}

Let $\Ch (T, \k)$ denote the set of rank 1 multiplicative local systems on $T$.
A local system $\scrL \in \Ch (T, \k)$ is said to be \emph{torsion} if there exists some $n \in \Z_{\geq 0}$ which is invertible in $\k$ such that $\scrL^{\otimes n} \cong \uk_T$.
We denote the subset of $\Ch (T, \k)$ consisting of torsion local systems by $\Ch^{\mu} (T, \k)$.
We also fix a $W$-orbit $\fr{o}$ in $\Ch^{\mu} (T, \k)$.

Let $\scrL, \scrL' \in \fr{o}$. We can view $\scrL'$ as a multiplicative local system on $B$ via pullback along $B \to T$. 
We can then define the derived category
\[\DEE{\scrL}{\scrL'} (G, \k) \coloneq D_{\cons} (T \backslash_{\scrL} \eFl /_{\scrL'} B, \k)\]
of $(T \times B, \scrL \boxtimes (\scrL')^{-1})$-equivariant sheaves on $\eFl$. For a precise construction of this category see \cite{Sandvik}.
As we vary $\scrL$ and $\scrL'$, we can assemble $\DEE{\scrL}{\scrL'} (G, \k)$ into a 2-category, denoted $\DEE{}{} (G, \k, \fr{o})$, whose object set is $\fr{o}$, and whose morphism categories are given by $\DEE{\scrL}{\scrL'} (G, \k)$ for $\scrL, \scrL' \in \fr{o}$.
The horizontal composition in $\DEE{}{} (G, \k, \fr{o})$ is by the convolution product $\star$ (see \cite[\S3.3]{Sandvik}).

For each $w \in W$ and lift $\dot{w} \in N_G (T)$ of $w$, by \cite[Proposition 3.2.2]{Sandvik}, there is a perverse $t$-exact equivalence of categories
\[\DEE{\scrL}{\scrL w} (w, \k) \coloneq D_{\cons} (T \backslash_{\scrL} \eFl_w /_{\scrL w} B, \k) \cong D_{\cons} (T \backslash \pt, \k)\]
given by taking the stalk at $\dot{w}$.
We can then define an object $\scrK_{\dot{w}}^{\scrL}$ in $\DEE{\scrL}{\scrL w} (w, \k)$ as the sheaf corresponding to the constant sheaf in $D_{\cons} (T \backslash \pt, \k)$ under the above equivalence.\footnote{The notation for objects is opposite to that of \cite{Sandvik} to match the presentation given for the diagrammatic and algebraic Hecke categories. In particular, in \emph{loc. cit.}, $\scrK_{\dot{w}}^{\scrL}$ refers to an object with right equivariant monodromy $\scrL$ whereas here it refers to an object with left equivariant monodromy $\scrL$.}
For two different lifts $\dot{w}, \dot{w}'$ of $w$, there is a non-unique isomorphism $\scrK_{\dot{w}}^{\scrL} \stackrel{\sim}{\to} \scrK_{\dot{w}'}^{\scrL}$. 
There are standard, costandard, and simple perverse sheaves defined by
\[\Delta_{\dot{w}}^{\scrL} \coloneq j_{w!} \scrK_{\dot{w}}^{\scrL} [\ell(w)], \qquad \nabla_{\dot{w}}^{\scrL} \coloneq j_{w*} \scrK_{\dot{w}}^{\scrL} [\ell(w)], \qquad \IC_{\dot{w}}^{\scrL} = \textnormal{Im} (\Delta_{\dot{w}}^{\scrL} \to \nabla_{\dot{w}}^{\scrL}),\]
respectively. These are all perverse sheaves in $\DEE{\scrL}{\scrL w} (G, \k)$. We will frequently replace the $\dot{w}$ with $w$ in the notation for these sheaves to refer to any object in their isomorphism class rather than a fixed representative.

\begin{remark}
  \begin{enumerate}
    \item There is a canonical representative of $\IC_e^{\scrL}$ given by the canonical lift of $e$ in $N_G (T)$.
  
    \item Let $s \in S$ such that $\scrL s = \scrL$.
    Consider the rank 1 Levi subgroup $L_s$ of $G$ corresponding to $s$.
    There is a unique extension of $\scrL$ to $L_s$, denoted $\scrL^s$ (\cite[Lemma 3.2.4]{Sandvik}).
    As a result, we have a canonical representative for $\IC_s^{\scrL}$ given by $\scrL^s [1]$ under the canonical isomorphism $\eFl_{\leq s} \cong U_s \backslash L_s$ where $U_s = U \cap L_s$.
    In particular, when we write $\IC_s^{\scrL}$, we will always refer to the canonical representative of its isomorphism class given here.
  \end{enumerate}
\end{remark}

Let $\scrL, \scrL' \in \fr{o}$. There is a forgetful functor
\[\ForME{\scrL} : \DEE{\scrL}{\scrL'} (G, \k) \to D_{\cons} (\eFl /_{\scrL'} B, \k)\]
given by forgetting the left $(T, \scrL)$-equivariance.
We define a full triangulated subcategory
\[\DME{\scrL}{\scrL'} (G, \k) \subset D_{\cons} (\eFl /_{\scrL'} B, \k)\]
generated by the essential image of $\ForME{\scrL}$. 
We often abuse notation and regard any object in $\DEE{\scrL}{\scrL'} (G, \k)$ as an object in $\DME{\scrL}{\scrL'} (G, \k)$ via $\ForME{\scrL}$. 
In particular, the standard, costandard, and IC sheaves can be viewed as objects therein.

\subsection{Minimal IC Sheaves}\label{subsec:min_IC}

For each block $\beta \in \uW{\scrL}{\scrL'}$, there is a minimal element $w^{\beta}$ under the partial order $\leq_{\scrL}$.
We can then consider the perverse sheaf $\IC_{w^{\beta}}^{\scrL} \in \DEE{\scrL}{\scrL'} (G, \k)$. 
A priori, these sheaves are only defined up to (non-unique) isomorphism, and there is no canonical representative of the isomorphism class.
Our first step will be to choose representatives in a coherent manner such that the resulting class is compatible with convolution.
The work in picking these representatives is covered by \cite[\S 5.8]{Sandvik}.

Let $\Xi^{\fr{o}} (W)$ denote the groupoid of blocks. In more detail, $\Xi^{\fr{o}} (W)$ is a category whose object set is $\fr{o}$ and whose morphism set from $\scrL$ to $\scrL'$ is given by $\uW{\scrL}{\scrL'}$.
In the following proposition, we view $\Xi^{\fr{o}} (W)$ as a 2-category by setting all the 2-morphisms to be identities.

\begin{proposition}[{\cite[Corollary 5.8.11]{Sandvik}}]\label{prop:min_IC_functor}
  There exists a 2-functor 
  \[\IC_{\min} : \Xi^{\fr{o}} (W) \to \DEE{}{} (G, \k, \fr{o})\]
   taking each $\beta \in \uW{\scrL}{\scrL'}$ to a sheaf in the isomorphism class $\IC_{w^{\beta}}^{\scrL}$.
\end{proposition}

We will fix a 2-functor $\IC_{\min} : \Xi^{\fr{o}} (W) \to \DEE{}{} (G, \k, \fr{o})$ arising from Proposition \ref{prop:min_IC_functor} for the remainder of the paper. 
It is useful to unpack what this 2-functor entails.
First, for each block $\beta \in \uW{\scrL}{\scrL'}$, there is an object
\[\theta_{\beta}^{\scrL} \coloneq \IC_{\min} (\beta) \in \DEE{\scrL}{\scrL'} (G, \k).\]
Second, for each pair of blocks $\beta \in \uW{\scrL}{\scrL'}$ and $\gamma \in \uW{\scrL'}{\scrL''}$, there is an isomorphism
\[b_{\beta, \gamma} : \theta_{\beta}^{\scrL} \star \theta_{\gamma}^{\scrL'} \stackrel{\sim}{\to} \theta_{\beta \gamma}^{\scrL}.\]
These isomorphisms are associative in the sense that if $\delta \in \uW{\scrL''}{\scrL'''}$, there is an equality
\[b_{\beta \gamma, \delta} (b_{\beta, \gamma} \times \id) = b_{\beta, \gamma \delta} (\id \times b_{\gamma, \delta}).\]

\begin{remark}
  Proposition \ref{prop:min_IC_functor} is where we crucially use the assumption that $\k \in \{\K, \O, \F\}$.
  The reason for this assumption is that the construction of $\IC_{\min}$ requires changing the sheaf-theoretic setting to étale sheaves to make use of a Whittaker model.
  The author conjectures that such a 2-functor exists under the same assumptions on $\k$ used throughout the paper (namely, noetherian domain of finite global dimension).
  If such a 2-functor can be constructed, the rest of our results, namely Theorem \ref{thm:mon_RW_thm}, can be easily derived using similar arguments for these more general coefficient rings. 
\end{remark}

\subsection{Parity Sheaves}\label{subsec:ew_par}

Let $s \in S$ and $\scrL \in \fr{o}$. Let $\beta \in \uW{\scrL}{\scrL s}$ be the block containing $s$.
Define an object in $\DEE{\scrL}{\scrL s} (G, \k)$ by
\[\scrE_s^{\scrL} \coloneq \begin{cases} \IC_s^{\scrL} & \scrL s = \scrL, \\ \theta_{\beta}^{\scrL} & \scrL s \neq \scrL. \end{cases}\]
Note that $\scrE_s^{\scrL}$ comes equipped with a morphism 
\[ \epsilon_s : \scrE_s^{\scrL} \to \theta_{\beta}^{\scrL} [\ell_{\scrL} (s)]\] 
defined as the adjunction morphism $\id \to j_{e*} j_e^*$ when $\scrL s = \scrL$ and $\epsilon_s = \id$ when $\scrL s \neq \scrL$.
More generally, if $\uw = (s_1, \ldots, s_k) \in \Exp (W)$, we can define the \emph{Bott--Samelson parity complex} associated to $\uw$ by
\[\scrE_{\uw}^{\scrL} \coloneq \scrE_{s_1}^{\scrL} \star \scrE_{s_2}^{\scrL s_1} \star \ldots \star \scrE_{s_k}^{\scrL s_1 \ldots s_{k-1}}.\]
Let $\beta \in \uW{\scrL}{\scrL \uw}$ be the block containing the evaluation of $\uw$, and for each $i =1, \ldots, k$ define a block $\beta_i \in \uW{\scrL \uw_{\leq i-1}}{\scrL \uw_{\leq i}}$ which contains $s_i$. 
Note that $\beta = \beta_1 \ldots \beta_k$.
We can define a morphism
\[\epsilon_{\uw} : \scrE_{\uw}^{\scrL} \to \theta_{\beta}^{\scrL} [\ell_{\scrL} (\uw)]\]
by the composition
\[\begin{tikzcd}
\scrE_{\uw}^{\scrL} \arrow[rr, "\epsilon_{s_1} \star \ldots \star \epsilon_{s_k}"] &  & { \theta^{\scrL}_{\beta_1} \star \ldots \star \theta^{\scrL \uw_{\leq k-1}}_{\beta_k} [\ell_{\scrL} (\uw)]} \arrow[r, "\sim"] & {\theta_{\beta}^{\scrL} [\ell_{\scrL} (\uw)]},
\end{tikzcd}\]   
where the second map is the canonical isomorphism of minimal IC sheaves arising from $\IC_{\min}$ (see Proposition \ref{prop:min_IC_functor} and the discussion following it).

Define a full subcategory $\PEE{\scrL}{\scrL'}^{\BS} (G, \k)$ of $\DEE{\scrL}{\scrL'} (G, \k)$ consisting of objects of the form $\scrE_{\uw}^{\scrL} [n]$ for $\uw \in \Exp (W)$ with $\scrL \uw = \scrL'$ and $n \in \Z$.
We call $\PEE{\scrL}{\scrL'}^{\BS} (G, \k)$ the category of \emph{Bott--Samelson parity complexes}. It is clear that convolution preserves Bott--Samelson parity complexes.
As a result, we can define a sub-2-category $\PEE{}{}^{\BS} (G, \k, \fr{o})$ of $\DEE{}{} (G, \k, \fr{o})$ whose morphism categories are given by  $\PEE{\scrL}{\scrL'}^{\BS} (G, \k)$.
We can further consider the idempotent completion of the additive hull of the Bott--Samelson parity complexes, denote by $\PEE{\scrL}{\scrL'} (G, \k)$.
The objects of $\PEE{\scrL}{\scrL'} (G, \k)$ are called \emph{parity sheaves}. 
These categories of parity sheaves can similarly be assembled into a 2-category denoted by $\PEE{}{} (G, \k, \fr{o})$.

We will also need the following right-equivariant version of the category of Bott--Samelson parity complexes. Consider the forgetful functor
\[\ForME{\scrL} : \DEE{\scrL}{\scrL'} (G, \k) \to \DME{\scrL}{\scrL'} (G, \k).\]
Define a full subcategory $\PME{\scrL}{\scrL'}^{\BS} (G, \k)$ of $\DME{\scrL}{\scrL'} (G, \k)$ given by the image of $\ForME{\scrL}$ restricted to $\PEE{\scrL}{\scrL'}^{\BS} (G, \k)$.
We abuse notation and write $\scrE_{\uw}^{\scrL}$ to refer to the object in $\PME{\scrL}{\scrL \uw}^{\BS} (G, \k)$ defined by $\ForME{\scrL} (\scrE_{\uw}^{\scrL})$.
We will also consider the idempotent completion of the additive hull of $\PME{\scrL}{\scrL'}^{\BS} (G, \k)$ which will be denoted by $\PME{\scrL}{\scrL'} (G, \k)$.

\subsection{Maximal IC Sheaves}

For this section, we require that $\k = \K$ so that $\k$ is a field of characteristic 0. 
There are variations of the results of this section that hold whenever $\k = \F$ provided certain small primes are invertible.
However, these variations require new proofs that are not provided in \cite{Sandvik}, and will not be need for the present work.
Note that when $\k = \K$, the decomposition theorem implies that $\PEE{\scrL}{\scrL'} (G, \K)$ is generated under direct sums and shifts by the simple perverse sheaves $\IC_w^{\scrL}$ for $w \in \W{\scrL}{\scrL'}$.

Let $\beta \in \uW{\scrL}{\scrL'}$ be a block. Since $W$ is finite, the block $\beta$ admits a maximal element $w_{\beta}$ under the partial order $\leq_{\beta}$.
We can then consider the perverse sheaf $\IC_{w_{\beta}}^{\scrL} \in \DEE{\scrL}{\scrL'} (G, \K)$. As with the minimal IC sheaves, these sheaves are only defined up to (non-unique) isomorphism.
The goal of the section is two-fold. First, we will use our choices of minimal IC sheaves to fix representatives of the maximal IC sheaves. 
Second, we will construct a canonical ``coalgebra'' structure on the maximal IC sheaves.

\subsubsection{Properties}

The following two propositions are mild generalizations of \cite[Proposition 4.4.3]{Sandvik} and \cite[Proposition 4.4.5]{Sandvik}. 
Related statements can be found in \cite[Proposition 6.2]{LY} and \cite[Proposition 6.4]{LY}.
We omit the proofs as they identical to the referenced propositions.

\begin{proposition}\label{prop:conv_of_max_ICs}
    Let $\beta \in \uW{\scrL}{\scrL'}$ and $\gamma \in \uW{\scrL'}{\scrL''}$.
    \begin{enumerate}
        \item The convolution $\IC_{w_{\beta}}^{\scrL} \star \IC_{w_{\gamma}}^{\scrL'}$ is isomorphic to a direct sum of shifts of $\IC_{w_{\beta \gamma}}^{\scrL}$.
        \item The perverse cohomology ${}^p H^i \left( \IC_{w_{\beta}}^{\scrL} \star \IC_{w_{\gamma}}^{\scrL'} \right) = 0$ unless $-\ell_{\scrL'} (w_{\gamma}) \leq i \leq \ell_{\scrL'} (w_{\gamma})$.
        \item There is an isomorphism
        \[{}^p H^{\pm \ell_{\scrL'} (w_{\gamma})} \left( \IC_{w_{\beta}}^{\scrL} \star \IC_{w_{\gamma}}^{\scrL'} \right) \cong \IC_{w_{\beta \gamma}}^{\scrL}.\]
    \end{enumerate}
\end{proposition}

\begin{proposition}\label{prop:stalks_of_max_ICs}
    Let $\beta \in \uW{\scrL}{\scrL'}$ and $w \in \beta$. We have isomorphisms
    \[ j_w^* \IC_{w_{\beta}}^{\scrL} \cong \scrK_w^{\scrL} [\ell_{\scrL} (w_{\beta}) + \ell (w) - \ell_{\scrL} (w)],\]
    \[ j_w^! \IC_{w_{\beta}}^{\scrL} \cong \scrK_w^{\scrL} [-\ell_{\scrL} (w_{\beta}) + \ell (w) + \ell_{\scrL} (w)].\]
\end{proposition}

\begin{definition}
    Let $\beta \in \uW{\scrL}{\scrL'}$.
    A \emph{rigidified maximal IC sheaf} for the block $\beta \in \uW{\scrL}{\scrL'}$ is a pair $(\Theta, \epsilon)$ where $\Theta \in \DEE{\scrL}{\scrL'} (G, \K)$ is such that $\Theta [\ell_{\scrL} (w_{\beta})] \cong \IC_{w_{\beta}^{\scrL}}$ and $\epsilon : \Theta \to \theta_{\beta}^{\scrL}$ is a nonzero map.
\end{definition}

By Proposition \ref{prop:stalks_of_max_ICs}, rigidified maximal IC sheaves exist. 
Moreover, there is a unique isomorphism between any two choices of rigidified maximal IC sheaves (cf., \cite[6.5]{LY}).
We can then fix, once-and-for-all, a rigidified maximal IC sheaf $(\Theta^{\scrL}_{\beta}, \epsilon_{\beta})$ for each $\beta \in \uW{\scrL}{\scrL'}$. 
When $\beta \in \uW{\scrL}{\scrL}$ is the neutral block, the rigidified maximal IC sheaf will be denoted by $(\Theta_{\circ}^{\scrL}, \epsilon_{\scrL})$.

\subsubsection{Coalgebra Structure on Maximal IC Sheaves}

Let $\Seq^{\scrL} (W, \fr{o})$ denote the set of sequences of composable blocks with left monodromy $\scrL$. 
Explicitly, the $\Seq^{\scrL} (W, \fr{o})$ consists of sequences $\beta_{\bullet} = (\beta_1, \ldots, \beta_n)$ where $\beta_i \in \uW{\scrL_i}{\scrL_{i+1}}$ for some $\scrL_i \in \fr{o}$ such that $\scrL_1 = \scrL$.
For each $\beta_{\bullet} \in \Seq^{\scrL} (W, \fr{o})$, we write $\pi (\beta_{\bullet})$ for the block $\beta_1 \beta_2 \ldots \beta_n$. 
Given $\beta_{\bullet} \in \Seq^{\scrL} (W, \fr{o})$, we define
\[\Theta^{\scrL}_{\star \beta_{\bullet}} = \Theta^{\scrL_1}_{\beta_1} \star \ldots \star \Theta^{\scrL_n}_{\beta_n} \qquad\text{and}\qquad \theta^{\scrL}_{\star \beta_{\bullet}} = \theta^{\scrL_1}_{\beta_1} \star \ldots \star \theta^{\scrL_n}_{\beta_n}. \]
We also define $\epsilon_{\beta_{\bullet}} : \Theta^{\scrL}_{\star \beta_{\bullet}} \to \theta^{\scrL}_{\star \beta_{\bullet}}$ as the convolution of the various $\epsilon_{\beta_i}$'s.

\begin{lemma}\label{lem:existence_of_comult}
    Let $\beta_{\bullet} \in \Seq^{\scrL} (W, \fr{o})$.
    There is a unique map
    \[\nu_{\beta_{\bullet}} : \Theta^{\scrL}_{\pi (\beta_{\bullet})} \to \Theta^{\scrL}_{\star \beta_{\bullet}}\]
    making the following diagram commute,
    \begin{equation}\label{eq:existence_of_comult_1}
        \begin{tikzcd}
            \Theta^{\scrL}_{\pi (\beta_{\bullet})} \arrow[r, "\nu_{\beta_{\bullet}}"] \arrow[d, "\epsilon_{\pi (\beta_{\bullet})}"] & \Theta^{\scrL}_{\star \beta_{\bullet}} \arrow[d, "\epsilon_{\beta_{\bullet}}"] \\
            \theta^{\scrL}_{\pi (\beta_{\bullet})} \arrow[r, "\sim"]                                                              & \theta^{\scrL}_{\star \beta_{\bullet}},                                        
            \end{tikzcd}
    \end{equation}
    where the bottom arrow is the canonical isomorphism for rigidified minimal IC sheaves.
\end{lemma}
\begin{proof}
    By iterating Proposition \ref{prop:conv_of_max_ICs}, we have that ${}^p H^i (\Theta^{\scrL}_{\star \beta_{\bullet}}) = 0$ for $i < 0$ 
    and ${}^p H^{0} (\Theta^{\scrL}_{\star \beta_{\bullet}}) \cong \Theta^{\scrL}_{\pi (\beta_{\bullet})}$.
    Therefore, there is a nonzero map $\nu_{\beta_{\bullet}} : \Theta^{\scrL}_{\pi (\beta_{\bullet})} \to \Theta^{\scrL}_{\star \beta_{\bullet}}$.
    Note that $\Hom (\Theta^{\scrL}_{\pi (\beta_{\bullet})}, \theta^{\scrL}_{\pi (\beta_{\bullet})}) \cong \K$. 
    As a consequence, if the composition
    \begin{equation}\label{eq:existence_of_comult_2}
        \Theta^{\scrL}_{\pi (\beta_{\bullet})} \stackrel{\nu_{\beta_{\bullet}}}{\to} \Theta^{\scrL}_{\star \beta_{\bullet}} \stackrel{\epsilon_{\beta_{\bullet}}}{\to} \theta^{\scrL}_{\star \beta_{\bullet}} \cong \theta^{\scrL}_{\pi (\beta_{\bullet})}
    \end{equation}
    is nonzero, then we can rescale $\nu_{\beta_{\bullet}}$ to ensure that (\ref{eq:existence_of_comult_1}) commutes.

    By induction, it suffices to prove (\ref{eq:existence_of_comult_2}) is nonzero for $\beta_{\bullet} = (\beta, \gamma)$ where $\beta \in \uW{\scrL}{\scrL'}$ and $\gamma \in \uW{\scrL'}{\scrL''}$.
    Since $\epsilon_{(\beta, \gamma)} = (\epsilon_{\beta} \star \theta_{\gamma}^{\scrL'}) \circ (\id \star \epsilon_{\gamma})$ and from Proposition \ref{prop:conv_of_max_ICs}, we obtain that ${}^p H^0 (\epsilon_{(\beta, \gamma)})$ is surjective.
    Therefore, ${}^p H^0 (\nu_{(\beta, \gamma)})$ is an isomorphism. In particular, ${}^p H^0 (\nu_{(\beta, \gamma)} \circ \epsilon_{(\beta, \gamma)} ) \neq 0$ which proves that (\ref{eq:existence_of_comult_2}) is nonzero.
\end{proof}

In light of Lemma \ref{lem:existence_of_comult}, we will write $\nu_{\beta_{\bullet}} : \Theta^{\scrL}_{\pi (\beta_{\bullet})} \to \Theta^{\scrL}_{\star \beta_{\bullet}}$ for the unique map which satisfies the lemma.
At times, we may simplify the notation and write $\nu = \nu_{\beta_{\bullet}}$.

\begin{proposition}\label{prop:coalg_axioms}
    Let $\scrL_1, \scrL_2, \scrL_3, \scrL_4 \in \fr{o}$ and $\beta \in \uW{\scrL_1}{\scrL_2}$,  $\gamma \in \uW{\scrL_2}{\scrL_3}$,  $\delta \in \uW{\scrL_3}{\scrL_4}$.
    The following diagrams commute
    \[\begin{tikzcd}
        \Theta^{\scrL_1}_{\beta \gamma \delta} \arrow[r, "\nu"] \arrow[d, "\nu"]                   & \Theta^{\scrL_1}_{\beta} \star \Theta^{\scrL_2}_{\gamma \delta}  \arrow[d, "\id \star \nu"]  & \Theta^{\scrL_1}_{\beta \gamma} \arrow[d, "\nu"] \arrow[r, "\nu"] \arrow[rrd, "\id"]     & \Theta^{\scrL_1}_{\beta} \star \Theta^{\scrL_2}_{\gamma} \arrow[r, "\id \star \epsilon"] & \Theta^{\scrL_1}_{\beta} \star \theta^{\scrL_2}_{\gamma} \arrow[d, "\sim"] \\
        \Theta^{\scrL_1}_{\beta \gamma} \star \Theta_{\scrL_3}^{\delta} \arrow[r, "\nu \star \id"] & \Theta^{\scrL_1}_{\beta} \star \Theta^{\scrL_2}_{\gamma} \star \Theta^{\scrL_3}_{\delta},     & \Theta^{\scrL_1}_{\beta} \star \Theta^{\scrL_2}_{\gamma} \arrow[r, "\epsilon \star \id"] & \theta^{\scrL_1}_{\beta} \star \Theta^{\scrL_2}_{\gamma} \arrow[r, "\sim"]               & \Theta^{\scrL_1}_{\beta \gamma}.                                           
        \end{tikzcd}\]
\end{proposition}
\begin{proof}
    The first diagram commutes by the uniqueness of $\nu_{(\beta, \gamma, \delta)}$ from Lemma \ref{lem:existence_of_comult}.
   For the second diagram, we can compose 
   \begin{equation}\label{eq:coalg_axioms_1}
    \Theta^{\scrL_1}_{\beta \gamma} \stackrel{\nu}{\to} \Theta^{\scrL_1}_{\beta} \star \Theta^{\scrL_2}_{\gamma} \stackrel{\epsilon_{\beta} \star \id}{\to} \theta^{\scrL_1}_{\beta} \star \Theta^{\scrL_2}_{\gamma} \cong \Theta^{\scrL_1}_{\beta \gamma}
   \end{equation}
   with $\epsilon_{\beta \gamma}$.
   The defining property (\ref{eq:existence_of_comult_1}) for $\nu$ ensures that this composition is $\epsilon_{\beta \gamma}$.
   Since $\End (\Theta^{\scrL_1}_{\beta \gamma}) \cong \K$, we may conclude that (\ref{eq:coalg_axioms_1}) is the identity map. The other triangle in the second diagram commutes by a similar argument.
\end{proof}

\subsection{Soergel Functor}

We keep the assumption from the previous section that $\k = \K$.

Let $\scrF, \scrG \in \DEE{\scrL}{\scrL'} (G, \K)$. The internal sheaf Hom object $\RHom (\scrF, \scrG)$ is naturally a $B$-biequivariant sheaf on $G$,
Moreover, there is an isomorphism
\[\Hom_{\DEE{\scrL}{\scrL'} (G, \K)}^{\bullet} (\scrF, \scrG) \coloneq \bigoplus_{n \in \Z} \Hom_{\DEE{\scrL}{\scrL'} (G, \K)} (\scrF, \scrG [n]) \cong H_{T \times T}^{\bullet} (\eFl, \RHom (\scrF, \scrG)).\]
As a result, there is an action of $R_{\K} \cong H_T^{\bullet} (\pt; \K)$ on both the left and right of $\Hom^{\bullet} (\scrF, \scrG)$ which makes $\Hom^{\bullet} (\scrF, \scrG)$ into a graded $R_{\K}$-bimodule.

Let $\scrL, \scrL' \in \fr{o}$, and write $\Theta^{[\scrL, \scrL']} \coloneq \bigoplus_{\beta \in \uW{\scrL}{\scrL'}} \Theta^{\scrL}_{\beta}.$
We define the Soergel $\H$-functor,
\[\H : \DEE{\scrL}{\scrL'} (G, \K) \to \grbim{R_{\K}}, \qquad \scrF \mapsto \Hom_{\DEE{\scrL}{\scrL'} (G, \K)}^{\bullet} (\Theta^{[\scrL, \scrL']}, \scrF).\]

We will also consider a ``right-equivariant'' version of the Soergel functor. 
\[\tilde{\H} : \DME{\scrL}{\scrL'} (G, \K) \to \grrmod{R_{\K}}, \qquad \scrF \mapsto \Hom_{\DME{\scrL}{\scrL'} (G, \K)}^{\bullet} (\ForME{\scrL} (\Theta^{[\scrL, \scrL']}), \scrF).\] 

\subsubsection{Behavior on Objects}

We recall some facts about the $\H$-functor from \cite{LY} or \cite{Sandvik}.

\begin{lemma}[{\cite[Lemma 7.3]{LY}}]\label{lem:H_on_min_ICs}
    There is a canonical isomorphism in $\grbim{R_{\K}}$
    \[\H (\theta^{\scrL}_{\beta}) \cong R_{w^{\beta}}\]
    under which the canonical map $\epsilon^{\beta} : \Theta^{\scrL}_{\beta} \to \theta_{\beta}^{\scrL}$ corresponds to $1 \in R_{w^{\beta}}$.
\end{lemma}

If $\scrL s = \scrL$, by \cite[Lemma 4.5.6]{Sandvik}, then there is a unique map $\xi_s : \Theta^{\scrL}_{\circ} \to \IC_s^{\scrL} [-1]$ making the following diagram commute,
\begin{equation}\label{eq:char_of_theta_s}
    \begin{tikzcd}
        \Theta^{\scrL}_{\circ} \arrow[rr, "\xi_s"] \arrow[rd, "\epsilon_{\scrL}"'] &               & \IC_s^{\scrL} [-1] \arrow[ld] \\
                                                                              & \IC_e^{\scrL}, &                         
        \end{tikzcd}
\end{equation}
where the right downwards arrow is induced from the unit of the adjunction $\id \to j_{e*} j_e^*$.

\begin{lemma}[{\cite[Lemma 7.4]{LY}}]\label{lem:H_conv_by_ICs}
    Let $s \in S$ and $\scrL \in \fr{o}$ such that $\scrL s = \scrL$.
    \begin{enumerate}
        \item Let $\scrF \in \DEE{\scrL}{\scrL'} (G, \K)$.
            There is a canonical isomorphism in $\grbim{R_{\K}}$,
            \begin{equation}\label{eq:H_conv_by_ICs_1}
                \H (\scrF) \otimes_{R^s} R(1) \stackrel{\sim}{\to} \H (\scrF \star \IC_s^{\scrL}).
            \end{equation}
        \item There is a canonical isomorphism in $\grbim{R_{\K}}$
        \begin{equation}\label{eq:H_conv_by_ICs_3}
            \H (\IC_s^{\scrL}) \cong C_s
        \end{equation}
        under which the map $\xi_s : \Theta_{\scrL}^{\circ} \to \IC_s^{\scrL} [-1]$ corresponds to $u_s \in C_s$. 
    \end{enumerate}
\end{lemma}

\subsubsection{2-Cateorical Structure}

Let $\grbim{R_{\K}}_{\fr{o}}$ denote the 2-category whose object set is $\fr{o}$ and where each morphism category is a copy of $\grbim{R_{\K}}$.
The horizontal composition is given by the tensor product of graded $R_{\K}$-bimodules.
By Proposition \ref{prop:coalg_axioms}, the Soergel $\H$-functor extends to a lax 2-functor
\[\H : \DEE{}{} (G, \K, \fr{o}) \to \grbim{R_{\K}}_{\fr{o}}\]

It is useful to unpack what the (lax) 2-functor condition means here.
Let $\scrF \in \DEE{\scrL}{\scrL'} (G, \K)$ and $\scrG \in \DEE{\scrL'}{\scrL''} (G, \K)$.
We have maps
\begin{align*}
  \Hom (\Theta^{[\scrL, \scrL']}, \scrF [i]) \times \Hom (\Theta^{[\scrL', \scrL'']}, \scrG [j]) &\stackrel{\star}{\to} \Hom (\Theta^{[\scrL, \scrL']} \star \Theta^{[\scrL', \scrL'']}, \scrF \star \scrG [i+j]) \\
  &\stackrel{(-) \circ \nu}{\longrightarrow} \Hom (\Theta^{[\scrL, \scrL'']}, \scrF \star \scrG [i+j]),
\end{align*}
where $\nu : \Theta^{[\scrL, \scrL'']} \to \Theta^{[\scrL, \scrL']} \star \Theta^{[\scrL', \scrL'']}$ is given by summing over all
\[\nu_{(\beta, \gamma)} : \Theta_{\beta \gamma}^{\scrL} \to \Theta_{\beta}^{\scrL} \star \Theta_{\gamma}^{\scrL'}\]
for $\beta \in \uW{\scrL}{\scrL'}$ and $\gamma \in \uW{\scrL'}{\scrL''}$.
This map factors through a morphism of graded $R_{\K}$-bimodules,
\[ c ( \scrF, \scrG) : \H (\scrF) \otimes_R \H (\scrG) \to \H (\scrF \star \scrG)\]
which produce natural transformations
\[c : \H (-) \otimes_R \H (-) \implies \H (- \star -) : \DEE{\scrL}{\scrL'} (G, \K) \times \DEE{\scrL'}{\scrL''} (G, \K) \to \grbim{R_{\K}}.\]
From here, one can observe that $\H$ is a lax 2-functor from Proposition \ref{prop:coalg_axioms} (cf., \cite[\S 7.5]{LY}).
Our goal is to show that on parity sheaves, the natural transformation $c$ is an equivalence.
The following is a combination of \cite[Lemma 7.6]{LY} and \cite[Lemma 7.7]{LY} (also see \cite[Lemma 4.5.8]{Sandvik}).

\begin{lemma}\label{lem:lax_morphisms_props}
    Let $\scrF \in \DEE{\scrL}{\scrL'} (G, \K)$, $\gamma \in \uW{\scrL'}{\scrL''}$, and $s \in S$.
    \begin{enumerate}
        \item The morphism $c (\scrF, \theta^{\scrL'}_{\gamma})$ is an isomorphism.
        \item If $\scrL' s = \scrL'$, then the morphism $c (\scrF, \IC_s^{\scrL'})$ is an isomorphism.
    \end{enumerate}
\end{lemma}

\begin{corollary}\label{cor:H_functor_is_monoidal}
    The restriction of the lax 2-functor $\H$ to parity sheaves,
    \[\H : \PEE{}{} (G, \K, \fr{o}) \to \grbim{R_{\K}}_{\fr{o}},\]
    is a (strong) 2-functor.
\end{corollary}
\begin{proof}
    By induction using Lemma \ref{lem:lax_morphisms_props}, the morphisms $c (\scrF, \scrG)$ are isomorphisms whenever $\scrF$ and $\scrG$ are Bott--Samelson parity complexes.
    The result then follows by taking direct summands using the classification of parity sheaves from \cite[Theorem 3.4.1]{Sandvik}.
\end{proof}

\subsubsection{Equivalence with Soergel Bimodules}

Define a 2-category $\grAmon{}{} (\fr{h}_{\K}, W, \fr{o})$ whose object set is $\fr{o}$. The 1-morphisms in $\grAmon{}{} (\fr{h}_{\K}, W, \fr{o})$ are given by symbols $M(n)$ where $M$ is a 1-morphism in $\Amon{}{} (\fr{h}_{\K}, W, \fr{o})$ and $n \in \Z$.
The 2-morphisms from $M(m)$ to $M(n)$ are the 2-morphisms $M \to N$ of degree $n-m$ in $\Amon{}{} (\fr{h}_{\K}, W, \fr{o})$. The horizontal and vertical compositions are those inherited from $\Amon{}{} (\fr{h}_{\K}, W, \fr{o})$.
Note that on morphism categories, using the notation of Appendix \ref{apdx:mdc}, $\grAmon{\scrL}{\scrL'} (\fr{h}_{\K}, W, \fr{o}) = \left( \Amon{\scrL}{\scrL} (\fr{h}_{\K}, W, \fr{o}) \right)^{\circ}$.

Since $\K$ is a characteristic 0 field, it contains a copy of $\Q$. The Kac--Moody realization of $W$ over $\Q$ is a faithful $W$-representation, and hence, we must have that $\fr{h}_{\K}$ is also faithful.
As a result, by Lemma \ref{lem:for_ff_on_faithful_realizations}, the forgetful 2-functor
\[\grAmon{}{} (\fr{h}_{\K}, W, \fr{o}) \to \grbim{R_{\K} }_{\fr{o}}\] 
is locally fully faithful. By Lemma \ref{lem:H_on_min_ICs}, Lemma \ref{lem:H_conv_by_ICs}, and Corollary \ref{cor:H_functor_is_monoidal}, we then have that $\H$-functor factors through a 2-functor
\[\H : \PEE{}{} (G, \K, \fr{o}) \to \grAmon{}{} (\fr{h}_{\K}, W, \fr{o}).\]

We are ready to state the main result of the section.

\begin{theorem}\label{thm:H_functor_properties}
  The $\H$-functor restricts to an equivalence of 2-categories
  \[\H : \PEE{}{} (G, \K, \fr{o}) \to \grAmon{}{} (\fr{h}_{\K}, W, \fr{o})\]
  such that $\H (\scrE_{\uw}^{\scrL} [n]) \cong B_{\uw}^{\scrL} (n)$ for all $\uw \in \Exp (W)$, $\scrL \in \fr{o}$, and $n \in \Z$.
\end{theorem}

\begin{corollary}\label{cor:re_H_functor_properties}
  Let $\scrL, \scrL' \in \fr{o}$. The $\tilde{\H}$-functor restricts to a fully faithful functor
  \[\tilde{\H} : \PME{\scrL}{\scrL'} (G, \K) \to \grrmod{R_{\K}}\]
  such that $\H (\scrE_{\uw}^{\scrL} [n]) \cong B_w^{\scrL} (n)$ for all $\uw \in \Exp (W)$ and $n \in \Z$ such that $\scrL \uw = \scrL'$. 
  Moreover, we have that $\tilde{\H}$ is related to $\H$ by the following diagram which commutes up to natural isomorphism,
  \begin{equation}\label{eq:re_H_functor_properties_1}
    \begin{tikzcd}
\PEE{\scrL}{\scrL'} (G, \K) \arrow[d, "\ForME{\scrL}"'] \arrow[r, "\H"] & \grAmon{\scrL}{\scrL'} (\fr{h}_{\K}, W, \fr{o}) \arrow[d, "\K \otimes_{R_{\K}} (-)"] \\
\PME{\scrL}{\scrL'} (G, \K) \arrow[r, "\tilde{\H}"]                     & \grrmod{R_{\K}}.                                        
\end{tikzcd}
  \end{equation}
\end{corollary}
\begin{proof}
  The first part of the corollary follows from Theorem \ref{thm:H_functor_properties} using the same argument as \cite[Proposition 4.7.1]{Sandvik}.
  Diagram (\ref{eq:re_H_functor_properties_1}) commutes by \cite[Lemma 3.6.5]{Sandvik} and Theorem \ref{thm:abe_dl_are_basis}.
\end{proof}

It is clear from Lemma \ref{lem:H_on_min_ICs} and Lemma \ref{lem:H_conv_by_ICs} that $\H$ is locally essentially surjective.
As a result, to prove Theorem \ref{thm:H_functor_properties}, it suffices to show that $\H$ is locally fully faithful.
Before proving full faithfulness, we need some preliminary lemmas.

\begin{lemma}\label{lem:H_functor_ff_base_case}
  Let $\scrG \in \PEE{\scrL}{\scrL'} (G, \K)$. The natural map
  \[ \Hom (\IC_e^{\scrL}, \scrG) \to \Hom (R, \H (\scrG))\]
  induced by $\H$ is an isomorphism.
\end{lemma}
\begin{proof}
  The lemma follows from the proof of \cite[Lemma 7.10]{LY} after the obvious modifications are made.
\end{proof}

\begin{lemma}\label{lem:compat_of_self_adjointness}
  Let $s \in S$, $\scrF\in \PEE{\scrL}{\scrL'} (G, \K)$, and $\scrG\in \PEE{\scrL}{\scrL' s} (G, \K)$. There is a natural commutative diagram
  \begin{equation}\label{eq:compat_of_self_adjointness_1}
    \begin{tikzcd}
      {\Hom (\scrF \star \scrE_s^{\scrL'}, \scrG)} \arrow[rr, "\sim"] \arrow[d, "{\H}"]  &  & {\Hom (\scrF, \scrG \star \scrE_s^{\scrL' s})} \arrow[d, "{\H}"]                                                                               \\
      {\Hom (\H (\scrF \star \scrE_s^{\scrL'}), \H (\scrG))} \arrow[d, "{(-)\circ c (\scrF, \scrE_s^{\scrL'})}"] &  & {\Hom (\H (\scrF ), \H (\scrG \star \scrE_s^{\scrL' s}) )} \arrow[d, "{c (\scrG, \scrE_s^{\scrL' s})^{-1} \circ (-)}"] \\
      {\Hom (\H (\scrF) \otimes_R B_s^{\scrL'}, \H (\scrG))} \arrow[rr, "\sim"]   &  & {\Hom (\H (\scrF) , \H (\scrG) \otimes_R B_s^{\scrL' s})}
    \end{tikzcd}
  \end{equation}
  where the horizontal maps are given by the self-adjunction of $(-) \star \scrE_s^{\scrL'}$ and $(-) \otimes_R B_s^{\scrL'}$. 
\end{lemma}
\begin{proof}
  We can factor the diagram (\ref{eq:compat_of_self_adjointness_1}) as follows:
  \[\adjustbox{scale=0.95, center}{ % scale
      \begin{tikzcd}
        {\Hom (\scrF \star \scrE_s^{\scrL'}, \scrG)} \arrow[d, "\H"'] \arrow[r, "(-) \star \scrE_s^{\scrL' s}"] & {\Hom (\scrF \star \scrE_{(s,s)}^{\scrL'} , \scrG \star \scrE_s^{\scrL' s})} \arrow[r] \arrow[d, "\H"]  & {\Hom (\scrF, \scrG \star \scrE_s^{\scrL' s})} \arrow[d, "\H"] \\
        {\Hom (\H (\scrF \star \scrE_s^{\scrL'}), \H (\scrG))} \arrow[d, "{(-)\circ c (\scrF, \scrE_s^{\scrL'})}"'] \arrow[r, "\H (- \star \scrE_s^{\scrL' s})"] & {\Hom (\H (\scrF \star \scrE_{(s,s)}^{\scrL' }), \H ( \scrG \star \scrE_s^{\scrL' s}))} \arrow[r] \arrow[d] & {\Hom (\H (\scrF ), \H (\scrG \star \scrE_s^{\scrL' s}) )} \arrow[d, "{c (\scrG, \scrE_s^{\scrL's})^{-1} \circ (-)}"] \\
        {\Hom (\H (\scrF) \otimes_R B_s^{\scrL'}, \H (\scrG))} \arrow[r, "(-) \otimes_R B_s^{\scrL' s}"] & {\Hom (\H (\scrF) \otimes_R B_{(s,s)}^{\scrL'} , \H (\scrG) \otimes_R B_s^{\scrL' s})} \arrow[r] & {\Hom (\H (\scrF) , \H (\scrG) \otimes_R B_s^{\scrL' s})}
  \end{tikzcd}}\]
  where the horizontal morphisms on the right are given by unit morphisms of adjunctions.
  The top two squares commute by functorality of $\H$. The bottom two squares commute by Corollary \ref{cor:H_functor_is_monoidal}.
\end{proof}

\begin{midsecproof}{Theorem \ref{thm:H_functor_properties}}
  Let $\scrL, \scrL' \in \fr{o}$ and $\scrF, \scrG \in \PEE{\scrL}{\scrL'} (G, \K)$.
  Consider the morphism
  \[\alpha (\scrF, \scrG) : \Hom (\scrF, \scrG) \to \Hom (\H (\scrF), \H (\scrG))\]
  Since every sheaf in $\PEE{\scrL}{\scrL'} (G, \K)$ is in the idempotent completion of the additive hull of the Bott--Samelson sheaves it suffices to
  prove that $\alpha
  (\scrF, \scrG)$ is an isomorphism for $\scrF = \scrE_{\uw}^{\scrL}$ for all expressions $\uw = (s_1, \ldots, s_k) \in \Exp (W)$ with $\scrL \uw = \scrL'$.
  We will prove $\alpha (\scrE_{\uw}^{\scrL}, \scrG)$ is an isomorphism by induction on $k$. The base case of $w = e$ is covered by Lemma
  \ref{lem:H_functor_ff_base_case}.

  Let $s = s_k$.
  By Lemma \ref{lem:compat_of_self_adjointness}, we get a commutative diagram
  \begin{equation}
    \begin{tikzcd}
      {\Hom (\scrE_{\uw_{\leq k-1}}^{\scrL} \star \scrE_s^{\scrL' s}, \scrG)} \arrow[rr, "\sim"] \arrow[d, "{\H}"] &  & {\Hom (\scrE_{\uw_{\leq k-1}}^{\scrL}, \scrG \star \scrE_s^{\scrL' })} \arrow[d, "{\H}"] \\
      {\Hom (\H (\scrE_{\uw_{\leq k-1}}^{\scrL} \star \scrE_s^{\scrL' s}), \H (\scrG))} \arrow[d, "\sim", "\ref{cor:H_functor_is_monoidal}"'] &  & {\Hom (\H (\scrE_{\uw_{\leq k-1}}^{\scrL} ), \H (\scrG \star \scrE_s^{\scrL'}) )} \arrow[d, "\sim", "\ref{cor:H_functor_is_monoidal}"'] \\
      {\Hom (\H (\scrE_{\uw_{\leq k-1}}^{\scrL}) \otimes_R B_s^{\scrL' s}, \H (\scrG))} \arrow[rr, "\sim"] &  & {\Hom (\H (\scrE_{\uw_{\leq k-1}}^{\scrL}) , \H (\scrG) \otimes_R B_s^{\scrL'})}.
    \end{tikzcd}
  \end{equation}
  The top right vertical map is an isomorphism by induction; therefore, we must have that $\alpha (\scrE_{\uw}^{\scrL}, \scrG)$
  is an isomorphism. We can then conclude that $\H$ is locally fully faithful.
\end{midsecproof}

\subsection{Constructing the Riche--Williamson Functor}

\begin{lemma}\label{lem:eos_and_torsion_loc_sys}
  Let $\k, \k'$ be noetherian domains of finite global dimension, and let $\varphi : \k \to \k'$ be a ring homomorphism.
  Let $\scrL \in \Ch^{\mu} (T, \k)$.
  For all $w \in W$, one has that $\scrL w = \scrL$ if and only if $\k' (\scrL) w = \k' (\scrL)$.
\end{lemma}
\begin{proof}
  If $\scrL w = \scrL$, then clearly $\k' (\scrL) w = \k' (\scrL)$ since $\k' (-)$ commutes with the $W$-action.
  Now suppose that $\k' (\scrL) w = \k' (\scrL)$
  There exists some $n \in \Z$ invertible in $\k$ such that $\scrL$ is defined by a group homomorphism
  \[\rho_{\scrL} : \bfY \to \mu_n (\k)\]
  where $\mu_n (\k)$ is the group of $n$-th roots of unity in $\k$. By assumption, we have that for all $\lambda \in \bfY$,
  \[\varphi (\rho_{\scrL} (\lambda \cdot w - \lambda)) = 1.\]
  In other words, $\rho_{\scrL} (\lambda \cdot w - \lambda)$ is in the kernel of the group homomorphism $\mu_n (\k) \to \mu_n (\k')$.
  We claim that $\mu_n (\k) \to \mu_n (\k')$ is injective. Let $\zeta \in \mu_n (\k)$ such that $\varphi (\zeta) = 1$.
  Assume that $\zeta \neq 1$. Since $\k$ is a domain and $\zeta^n - 1 =0$, we must have that $\zeta^{n-1} + \zeta^{n-2} + \ldots + 1 = 0$.
  We can then compute
  \[0 = \varphi  (\zeta^{n-1} + \zeta^{n-2} + \ldots + 1) = \varphi(\zeta)^{n-1} + \varphi (\zeta)^{n-2} + \ldots + \varphi (1) = n.\]
  However, $n$ is assumed to be invertible in $\k$ and hence also in $\k'$, so this contradicts our assumption that $\zeta \neq 1$.
  We conclude that $\rho_{\scrL} (\lambda \cdot w - \lambda) = 1$ which translates into the condition that $\scrL w = \scrL$.
\end{proof}

We return to the general case where $\k \in \{\K, \O, \F\}$ with ring homomorphism $\Z' \to \k$.

For each $\scrL \in \fr{o}$, we can find some $\scrL_{\O} \in \Ch^{\mu} (T, \O)$ such that $\k (\scrL_{\O}) \cong \scrL$.
Indeed, if $\k = \F$, this follows from the existence of a Teichmüller lift $\F^{\times} \to \O^{\times}$. For $\k = \K$, this can be constructed using the fact that the roots of unity of $\K$ are contained in $\O$.
Define $\fr{o}_{\O} = \scrL_{\O} \cdot W$ for some $\scrL \in \fr{o}$. We can then set $\fr{o}_{\K} = \K (\scrL_{O}) \cdot W$ and $\fr{o}_{\F} = \F (\scrL_{\O})$.
Note that $\fr{o}_{\k} = \fr{o}$.
For $\scrL \in \fr{o}$, the sheaf $\scrL_{\O}$ is torsion, and hence $\fr{o}_{\k} \subseteq \Ch^{\mu} (T, \k)$.
By Lemma \ref{lem:eos_and_torsion_loc_sys}, extension of scalars $\k (-) : \fr{o}_{\O} \to \fr{o}_{\k}$ gives a $W$-equivariant bijection.
We can then abuse notation and simply write $\fr{o}$ to refer to any $\fr{o}_{\k}$ with $\k \in \{\K, \O, \F\}$ under these bijections. 
Likewise, for any $\scrL \in \fr{o}$, we will simultaneously view it as an object in $\fr{o}_{\k}$ for any such $\k$.
Moreover, extension of scalars defines a 2-functor
\[\k (-) : \PEE{}{}^{\BS} (G, \O, \fr{o}) \to  \PEE{}{}^{\BS} (G, \k, \fr{o}).\]

Recall the 2-category $\grEWmon{}{}^{\BS} (\fr{h}_{\k}, W, \fr{o})$ constructed in \S\ref{subsubsec:cat_of_diag_hecke}.
The goal of this section is to construct a 2-functor
\[\Upsilon_{\textnormal{RW}}^{\k} :  \grEWmon{}{}^{\BS} (\fr{h}_{\k}, W, \fr{o}) \to \PEE{}{}^{\BS} (G, \k, \fr{o}).\]
The construction is analogous to the non-monodromic version constructed by Riche and Williamson \cite{RW}.
We will first define $\Upsilon_{\RW}^{\O}$. We will then check that the defining relations in $\grEWmon{}{}^{\BS} (\fr{h}_{\K}, W, \fr{o})$ are satisfied by extending scalars of $\Upsilon_{\RW}^{\O}$ to $\K$-coefficients.
Finally, we will extend scalars to construct $\Upsilon_{\RW}^{\F}$.

It is easy to define $\Upsilon_{\RW}^{\O}$ on 1-morphisms. For an expression $\uw \in \Exp (W)$, $\scrL \in \fr{o}$, and $n \in \Z$, we define
\[\Upsilon_{\RW}^{\O} (B_{\uw}^{\scrL} (n)) \coloneq \scrE_{\uw}^{\scrL} [n] \in \grEWmon{\scrL}{\scrL \uw} (\fr{h}_{\O}, W, \fr{o}).\]
We then must specify the image of the generating 2-morphisms in $\grEWmon{}{}^{\BS} (\fr{h}_{\O}, W, \fr{o})$.

\subsubsection{Polynomials}

Let $\scrL \in \fr{o}$.
Since $\scrE_{\emptyset}^{\scrL} = \IC_e^{\scrL} = \Delta_e^{\scrL}$, by adjunction, we have a canonical isomorphism,
\[\xi : \End^{\bullet} (\scrE_{\emptyset}^{\scrL}) \stackrel{\sim}{\to} H_T^{\bullet} (\pt; \O) = R_{\O}.\]
We then define
\[\Upsilon_{\RW}^{\O} \left( 
  % \tikzsetnextfilename{#1}
  \tikzstyle{every picture}=[tikzfig]
  \input{./figures/poly2.tikz}
\right) = \xi^{-1} (f)\]
for $f \in R_{\O}$.

\subsubsection{1-color morphisms}

Let $s \in S$ and $\scrL \in \fr{o}$ such that $\scrL s = \scrL$.
By \cite[Lemma 4.3.7]{Sandvik}, the object $\scrE_s^{\scrL}$ carries a graded Frobenius algebra structure with structure maps
\[\mu_s : \scrE_s^{\scrL} \star \scrE_s^{\scrL} \to \scrE_s^{\scrL} [-1], \qquad \nu_s : \scrE_s^{\scrL}  \to \scrE_s^{\scrL}  \star \scrE_s^{\scrL}[-1],\]
\[\eta_s : \IC_e^{\scrL} \to \scrE_s^{\scrL} [1], \qquad \epsilon_s : \scrE_s^{\scrL} \to \IC_e^{\scrL} [1],\]
where $\epsilon_s$ is the map defined in \S\ref{subsec:ew_par} and $\eta_s$ is defined by the adjunction counit $j_{e*} j_e^! \to \id$.
We can then define  
\begin{align*}
 \Upsilon_{\RW}^{O} \left(
  % \tikzsetnextfilename{#1}
  \tikzstyle{every picture}=[tikzfig]
  \input{./figures/enddot2.tikz}
 \right) &\coloneq \eta_s , & \Upsilon_{\RW}^{\O} \left(
  % \tikzsetnextfilename{#1}
  \tikzstyle{every picture}=[tikzfig]
  \input{./figures/startdot2.tikz}
 \right) &\coloneq \epsilon_s,  & \Upsilon_{\RW}^{\O} \left(
  % \tikzsetnextfilename{#1}
  \tikzstyle{every picture}=[tikzfig]
  \input{./figures/mult2.tikz}
 \right) &\coloneq  \mu_s, & \Upsilon_{\RW}^{\O} \left(
  % \tikzsetnextfilename{#1}
  \tikzstyle{every picture}=[tikzfig]
  \input{./figures/comult2.tikz}
 \right) = \nu_s.
\end{align*}

Now let $s \in S$ and $\scrL \in \fr{o}$ such that $\scrL s \neq \scrL$. By Proposition \ref{prop:min_IC_functor}, we have a canonical isomorphism of minimal IC sheaves $\cap_s : \scrE_s^{\scrL} \star \scrE_s^{\scrL s} \to \IC_e^{\scrL}$ with inverse $\cup_s$.
We can then define  
\[\Upsilon_{\RW}^{\O} \left(
  % \tikzsetnextfilename{#1}
  \tikzstyle{every picture}=[tikzfig]
  \input{./figures/cap2.tikz}
 \right) \coloneq \cap_s \qquad\text{and}\qquad \Upsilon_{\RW}^{\O} \left(
  % \tikzsetnextfilename{#1}
  \tikzstyle{every picture}=[tikzfig]
  \input{./figures/cup2.tikz}
 \right) \coloneq \cup_s.\]

\subsubsection{2\texorpdfstring{$m_{s,t}$}{mst}-valent vertices}

Let $s, t \in S$ be distinct simple reflections and $\scrL \in \fr{o}$. Define $m = m_{s,t}$ and $w = {}_s m = {}_t m$. 

\begin{lemma}\label{lem:behavior_of_2m_valent_verts_geom}
    We use the same notation as above. The pullback $j_w^*$ induces an isomorphism of $\O$-modules
    \begin{equation}\label{eq:behavior_of_2m_valent_verts_geom_1}
        \Hom (\scrE_{{}_s \underline{m}}^{\scrL}, \scrE_{{}_t \underline{m}}^{\scrL}) \stackrel{\sim}{\to} \Hom (j_w^* \scrE_{{}_s \underline{m}}^{\scrL}, j_w^* \scrE_{{}_t \underline{m}}^{\scrL}).
    \end{equation}
    Moreover, both sides of (\ref{eq:behavior_of_2m_valent_verts_geom_1}) are free $\O$-modules of rank 1.
\end{lemma}
\begin{proof}
    The lemma is essentially a variation of \cite[Lemma 5.4.4]{Sandvik}. By the Soergel Hom formula for the monodromic Hecke category (see \cite[Proposition 3.8.13]{Sandvik}), there is an isomorphism $\Hom (\scrE_{{}_s \underline{m}}^{\scrL}, \scrE_{{}_t \underline{m}}^{\scrL}) \cong \O$.
    By the same argument as in \cite[Theorem 3.4.1]{Sandvik}, we have an isomorphism $j_w^* \scrE_{{}_s \underline{m}}^{\scrL} \cong \scrK_w^{\scrL} [\ell_{\scrL} (w)] \cong j_w^* \scrE_{{}_s \underline{m}}^{\scrL}$. Moreover, $\End (\scrK_w^{\scrL}) \cong \O$.
    Since $\eFl_w$ is open in the support of $\scrE_{{}_s \underline{m}}^{\scrL}$ and $\scrE_{{}_t \underline{m}}^{\scrL}$, we have that the map (\ref{eq:behavior_of_2m_valent_verts_geom_1}) is surjective.
    Both sides of (\ref{eq:behavior_of_2m_valent_verts_geom_1}) are free of rank 1; therefore, (\ref{eq:behavior_of_2m_valent_verts_geom_1}) is an isomorphism.
\end{proof}

Define the map $g_{s,t} : \scrE_{{}_s \underline{m}}^{\scrL} \to \scrE_{{}_t \underline{m}}^{\scrL}$ as the pre-image under (\ref{eq:behavior_of_2m_valent_verts_geom_1}) of an isomorphism $f_{s,t}^w :j_w^* \scrE_{{}_s \underline{m}}^{\scrL} \stackrel{\sim}{\to} j_w^* \scrE_{{}_t \underline{m}}^{\scrL}$.
Likewise, by switching the roles of $s$ and $t$, there is a map $g_{t,s} : \scrE_{{}_t \underline{m}}^{\scrL} \to \scrE_{{}_s \underline{m}}^{\scrL}$ lifted from an isomorphism $g_{t,s}^w :j_w^* \scrE_{{}_t \underline{m}}^{\scrL} \stackrel{\sim}{\to} j_w^* \scrE_{{}_s \underline{m}}^{\scrL}$. 
We pick $g_{s,t}^w$ and $g_{t,s}^w$ such that $g_{t,s}^w \circ g_{s,t}^w = \id_{j_w^* \scrE_{{}_s \underline{m}}^{\scrL}}$.

To define the image of the $2m_{s,t}$-valent vertices, we will need another lemma. 

\begin{lemma}\label{lem:hom_generated_by_counit_soergel}
  We use the same notation as above. Let $\beta \in \uW{\scrL}{\scrL w}$ be the block containing $w$. For $u \in \{s,t\}$, we have
  \[\Hom_{\DME{\scrL}{\scrL w} (G, \O)} (\scrE_{{}_u \underline{m}}^{\scrL}, \theta_{\beta}^{\scrL} [m]) = \O \cdot \ForME{\scrL} (\epsilon_{{}_u \underline{m}}).\]
\end{lemma}
\begin{proof}
    By \cite[Lemma 3.6.5]{Sandvik} and \cite[Proposition 3.8.13]{Sandvik}, the space in question is free of rank 1 over $\O$.
    It then suffices to show that $\ForME{\scrL} (\epsilon_{{}_u \underline{m}})$ is nonzero after extension of scalars to any field. This is clear since extension of scalars commutes with convolution and the unit of the adjunction $\id \to j_{s*} j_s^*$ for any $s \in S_{\scrL}^{\circ}$.
\end{proof}

By Lemma \ref{lem:hom_generated_by_counit_soergel}, we have 
\begin{equation}\label{eq:cst_constants_soergel}
  \ForME{\scrL} (\epsilon_{{}_t \underline{m}} \circ g_{s,t}) = c_{s,t} \cdot \left( \ForME{\scrL} (\epsilon_{{}_s \underline{m}})\right) \qquad\text{and}\qquad \ForME{\scrL} (\epsilon_{{}_s \underline{m}} \circ g_{t,s}) = c_{t,s} \cdot \left( \ForME{\scrL} (\epsilon_{{}_t \underline{m}}) \right)
\end{equation}
for some $c_{s,t}, c_{t,s} \in \O$. We can now set
\[f_{s,t} \coloneq c_{t,s} g_{s,t} : \scrE_{{}_s \underline{m}}^{\scrL} \to \scrE_{{}_t \underline{m}}^{\scrL} \qquad\text{and}\qquad f_{t,s} \coloneq c_{s,t} g_{t,s} : \scrE_{{}_t \underline{m}}^{\scrL} \to \scrE_{{}_s \underline{m}}^{\scrL},\]
and then define these to be the image of the $2m$-valent vertices under $\Upsilon_{\RW}^{\O}$:
\[\Upsilon_{\RW}^{\O} \left( 
  % \tikzsetnextfilename{#1}
  \tikzstyle{every picture}=[tikzfig]
  \input{./figures/mst_valent2.tikz}
\right) \coloneq f_{s,t} \qquad\text{and}\qquad\Upsilon_{\RW}^{\O} \left( 
  % \tikzsetnextfilename{#1}
  \tikzstyle{every picture}=[tikzfig]
  \input{./figures/mst_valent2_flipped.tikz}
\right) \coloneq f_{t,s}.\]
 
\subsection{Verifying Relations}

In order to check that $\Upsilon_{\RW}^{\O}$ is well-defined, one must check that the image of a morphism under $\Upsilon_{\RW}^{\O}$ is invariant under the defining relations in $\grEWmon{}{}^{\BS} (\fr{h}_{\O}, W, \fr{o})$.
By \cite[Lemma 3.6.5]{Sandvik} and Theorem \ref{thm:dll_is_a_basis}, it suffices to check this after extending scalars to $\K$-coefficients.
Write $\Upsilon_{\RW}^{\K}$ for the extension of $\Upsilon_{\RW}^{\O}$ along $\O \to \K$.
By composing with $\H$ and using Theorem \ref{thm:H_functor_properties}, we can further reduce to checking that the following diagram strictly commutes
\[\begin{tikzcd}
    {\grEWmon{}{}^{\BS} (\fr{h}_{\K}, W, \fr{o})} \arrow[rr, "\Upsilon_{\RW}^{\K}"] \arrow[rd, "\Upsilon_{\Abe}"'] &                                       & \PEE{}{}^{\BS} (G, \K, \fr{o}) \arrow[ld, "\H"] \\
                                                                                        & {\grAmon{}{} (\fr{h}_{\K}, W,  \fr{o})} &                                     
    \end{tikzcd}\]

We can now calculate the image of the generating 2-morphisms under $\H$.
\begin{enumerate}
    \item \emph{Polynomials:} $\H (\Upsilon_{\RW}^{\K} (f))$ for $f \in R_{\K}$ is given by multiplication by $f$ on $R_{\K}$.
    \item \emph{The upper dot:} Let $s \in S$ and $\scrL \in \fr{o}$ such that $\scrL s = \scrL$. 
            Any morphism of graded $R_{\K}$-bimodules $C_s \to R_{\K} (1)$ is a scalar multiple of $\epsilon^s : C_s \to R_{\K} (1)$.
            Recall from (\ref{eq:char_of_theta_s}) that there is a unique morphism $\xi_s : \Theta_{\circ}^{\scrL} \to \scrE_s^{\scrL} [-1]$ such that $(\epsilon_s [-1]) \circ \xi_s = \epsilon_{\scrL}$. 
            Moreover, by Lemma \ref{lem:H_conv_by_ICs}, $\xi_s$ identifies with $u_s \in C_s $ under the isomorphism $\H (\scrE_s^{\scrL}) \cong C_s$. 
            Similarly, by Lemma \ref{lem:H_on_min_ICs}, $\epsilon_{\scrL}$ identifies with $1 \in R$ under the isomorphism $\H (\IC_e^{\scrL}) \cong R$.
            As a result, the relation $(\epsilon_s [-1]) \circ \xi_s = \epsilon_{\scrL}$ becomes $(\H (\epsilon_s)) (u_s) = 1$. Hence, $\H (\epsilon_s) = \epsilon^s$.
    \item \emph{The lower dot:} 
            Let $s \in S$ and $\scrL \in \fr{o}$ such that $\scrL s = \scrL$. 
            The map
            \[\IC_e^{\scrL} [-1] = j_{e*} j_e^! \scrE_s^{\scrL} \stackrel{\eta_s}{\to} \scrE_s^{\scrL} \stackrel{\epsilon_s}{\to} j_{e*} j_e^* \scrE_s^{\scrL} [1] \in \Hom (\IC_e^{\scrL} [-1] , \IC_e^{\scrL} [1]) \cong H_T^2 (\pt; \K)\]
            is given by $\lambda \cdot \id_{\IC_e^{\scrL} [-1]}$ for some $\lambda \in H_T^2 (\pt; \K)$.  
            Let $L_s$ be the rank 1 Levi subgroup of $G$ corresponding to $s$ and write $j : T \hookrightarrow L_s$ for the inclusion.
            Since $\scrL s = \scrL$, there is a unique multiplicative local system $\scrL^s$ on $L_s$ which extends $\scrL$.
            There is an isomorphism of varieties $\eFl_{\leq s} \cong U_s \backslash L_s$ where $U_s = U \cap L_s$.
            Under this isomorphism, the simple perverse sheaf $\scrE_s^{\scrL}$ is canonically identified with $\scrL^s$. 
            We can then see under this isomorphism and adjunction that the natural morphism of $(T \times T, \scrL \boxtimes \scrL^{-1})$-equivariant sheaves
            \begin{equation}\label{eq:lower_dot_reln_1}
              \scrL [-1] = j^! \scrL^s [1] \to j^* \scrL^s [1] = \scrL [1]
            \end{equation}
            is given by $\lambda \cdot \id_{\scrL [-1]}$.
            After applying $(-) \otimes^L \scrL^{-1}$ to (\ref{eq:lower_dot_reln_1}), we see that the natural morphism of $T\times T$-equivariant sheaves
            \[
              \underline{\K}_{T} [-1] = j^! \underline{\K}_{L_s} [1] \to j^* \underline{\K}_{L_s} [1] = \underline{\K}_T [1]
            \]
            is also given by $\lambda \cdot \id_{\underline{\K}_T [-1]}$. It then follows from the non-monodromic case (cf., \cite[\S10.5]{RW}) that $\lambda = \alpha_s$, and as a result, $\H (\epsilon_s \circ \eta_s) = \alpha_s$.
            Note that $\eta^s$ is uniquely determined by $\epsilon^s \circ \eta^s = \alpha_s$. As a consequence, we deduce from the upper dot computation that $\H (\eta_s) = \eta^s$. 
    \item \emph{Cups and caps:} Let $s \in S$ and $\scrL \in \fr{o}$ such that $\scrL s \neq \scrL$. Let $\beta \in \uW{\scrL}{\scrL s}$ be the block containing $s$.
            By Lemma \ref{lem:existence_of_comult}, there is a commutative diagram
            \[\begin{tikzcd}
                \Theta_{\scrL}^{\circ}  \arrow[d, "\epsilon_{\scrL}"] \arrow[r, "\nu"] & \Theta_{\scrL}^{\beta} \star \Theta_{\scrL s}^{\beta^{-1}} \arrow[d, "\epsilon_{\beta} \star \epsilon_{\beta^{-1}}"] \\
                \scrE_{\emptyset}^{\scrL} \arrow[r, "\cup_s"]                          & \scrE_s^{\scrL} \star \scrE_s^{\scrL s}.                                                                          
                \end{tikzcd}\]
                By Lemma \ref{lem:H_on_min_ICs}, under the canonical isomorphisms $\H (\scrE_s^{\scrL} \star \scrE_s^{\scrL s}) \cong (R_{\K})_s \otimes_{R_{\K}} (R_{\K})_s$ and $\H (\scrE_{\emptyset}^{\scrL}) \cong R_{\K}$, the above diagram translates to the condition that $\H (\cup_s) (u_{(s,s)}) = 1$.
                Therefore, $\H (\cup_s) = \cup^s$. Since $\cap_s$ is the inverse of $\cup_s$, we must also have that $\H (\cap_s) = \cap^s$.
    \item \emph{The trivalent vertices: } Let $s \in S$ and $\scrL \in \fr{o}$ such that $\scrL s = \scrL$. We have isomorphisms of vector spaces
        \[\Hom (C_s, C_s \otimes_{R_{\K}} C_s (-1)) \cong \K \qquad\text{and}\qquad \Hom (C_s \otimes_{R_{\K}} C_s, C_s (-1)) \cong \K.\]
        These morphism spaces are generated by the maps $\nu^s : C_s \to C_s \otimes_{R_{\K}} C_s (-1)$ and $\mu^s : C_s \otimes_{R_{\K}} C_s \to C_s (-1)$ respectively. Therefore, $\H (\nu_s)$ (resp. $\H (\mu_s)$) is a scalar multiple of $\nu^s$ (resp. $\mu^s$). 
        Note that $\nu^s$ (resp. $\mu^s$) is uniquely characterized by the identity
        \[(\epsilon^s (-1) \otimes_{R_{\K}} \id_{C_s}) \circ \nu^s = \id_{C_s} \qquad \text{resp.} \qquad \mu^s (1) \circ (\eta^s \otimes_{R_{\K}} \id_{C_s}) = \id_{C_s}.\]
        Therefore, $\H (\nu_s) = \nu^s$ and $\H (\mu_s) = \mu^s$ follow from the upper and lower dot computations and the Frobenius algebra structure on $\scrE_s^{\scrL}$.
    \end{enumerate}

Computing the image of the $2m_{s,t}$-valent vertices takes some work. Let $s,t \in S$ be distinct simple reflections and let $\scrL \in \fr{o}$. Denote the longest element in $\langle s, t \rangle$ by $w = {}_s m = {}_t m$.
We will write $d = \ell_{\scrL} (w)$ and denote the block in $\uW{\scrL}{\scrL w}$ containing $w$ by $\beta$.
Recall from \S\ref{subsec:alg_dll} that we have morphisms
\[\epsilon^{{}_u \underline{m}} : B_{{}_u \underline{m}}^{\scrL} \to (R_{\K})_{w^{\beta}} (d)\]
for $u \in \{s,t\}$.
By Lemma \ref{lem:uniqueness_of_beta} (2), $\beta_{s,t}$ is the unique morphism of Bott--Samelson bimodules such that there is an equality of maps 
\begin{equation}
  \K \otimes_{R_{\K}} (\epsilon^{{}_t \underline{m}} \circ \beta_{s,t}) = \K \otimes_{R_{\K}} \epsilon^{{}_s \underline{m}} : \K \otimes_{R_{\K}} B_{{}_s \underline{m}}^{\scrL} \to \K (d).
\end{equation}

\begin{lemma}\label{lem:cst_constants_inverses_soergel}
  The constants $c_{s,t}, c_{t,s} \in \O$ defined by (\ref{eq:cst_constants_soergel}) satisfy
  \[c_{s,t} c_{t,s} = 1.\]
\end{lemma}
\begin{proof}
  It follows from the definitions of $g_{s,t}$ and $g_{t,s}$ that
  \[g_{s,t} \circ g_{t,s} \circ g_{s,t} = g_{s,t}.\]
  Applying $\ForME{\scrL} (\epsilon_{{}_t \underline{m}} \circ (-))$ both sides and using (\ref{eq:cst_constants_soergel}), we deduce that
  $c_{s,t} c_{t,s} c_{s,t} = c_{s,t}$, or alternatively that $c_{s,t} (c_{t,s} c_{s,t} - 1) = 0$.
  
  It remains to prove that $c_{s,t} \neq 0$. 
  Since $g_{s,t}$ is a generator for $\Hom (\scrE_{{}_s \underline{m}}^{\scrL}, \scrE_{{}_t \underline{m}}^{\scrL})$, it suffices to show that the map
  \begin{align*}
    \Hom_{\DEE{\scrL}{\scrL} (G, \O)} (\scrE_{{}_s \underline{m}}^{\scrL}, \scrE_{{}_t \underline{m}}^{\scrL}) &\to \Hom_{\DME{\scrL}{{\scrL}} (G, \O)} (\scrE_{{}_s \underline{m}}^{\scrL}, \IC_{w^{\beta}}^{\scrL} [d]) \\
    f &\mapsto \ForME{\scrL} (\epsilon_{{}_t \underline{m}} \circ f)
  \end{align*}
  is nonzero. It suffices to prove that this map is nonzero after applying $\K (-)$.

  Let $f : \scrE_{{}_s \underline{m}}^{\scrL} \to \scrE_{{}_t \underline{m}}^{\scrL}$. By the upper dot relation we have already verified and (\ref{eq:re_H_functor_properties_1}), we have a commutative diagram
  \[\begin{tikzcd}
{\Hom (\scrE_{{}_s \underline{m}}^{\scrL}, \scrE_{{}_t \underline{m}}^{\scrL})} \arrow[d, "\H"] \arrow[r, "\epsilon_{{{}_t \underline{m}}} \circ f"] & {\Hom (\scrE_{\underline{m}}^{\scrL}, \IC_{w^{\beta}}^{\scrL} [d])} \arrow[d, "\H"] \arrow[r, "\ForME{\scrL}"] & {\Hom_{\DME{\scrL}{\scrL w} (G, \K)} (\scrE_{{}_s \underline{m}}^{\scrL}, \IC_{w^{\beta}}^{\scrL} [d])} \arrow[d, "\tilde{\H}"] \\
{\Hom (B_{{}_s \underline{m}}^{\scrL}, B_{{}_t \underline{m}}^{\scrL})} \arrow[r, "\epsilon^{{{}_t \underline{m}}} \circ \H (f)"]                                               & {\Hom (B_{{}_s \underline{m}}^{\scrL}, (R_{\K})_{w^{\beta}} (d))} \arrow[r, "\K \otimes_{R_{\K}} (-)"]                                                      & {\Hom_{\grrmod{R_{\K}}} (\K \otimes_{R_{\K}} B_{{}_s \underline{m}}^{\scrL}, \K (d))}.                                                                                             
\end{tikzcd}\]
By Theorem \ref{thm:H_functor_properties} and Corollary \ref{cor:re_H_functor_properties}, it suffices to prove that 
$\K \otimes_{R_{\K}} (\epsilon^{{}_s \underline{m}} \circ \varphi) \neq 0$ for some $\varphi : B_{{}_s \underline{m}}^{\scrL} \to B_{{}_t \underline{m}}^{\scrL}$, but the existence of such a morphism follows from Lemma \ref{lem:uniqueness_of_beta}.
\end{proof}

We can now compute the image of the $2m_{s,t}$-valent vertices.
\begin{enumerate}
  \item[5.] \emph{$2m_{s,t}$-valent vertices: } Let $m = m_{s,t}$.
  By our definition for $f_{s,t}$ and Lemma \ref{lem:cst_constants_inverses_soergel}, we have $\ForME{\scrL} (\epsilon_{{}_t \underline{m}} \circ f_{s,t}) = \ForME{\scrL} (\epsilon_{{}_s \underline{m}})$.
  We can now apply $\tilde{\H}$ and use (\ref{eq:re_H_functor_properties_1}) to deduce that 
  \[\K \otimes_{R_{\K}} \tilde{\H} (\epsilon_{{}_t \underline{m}}) \circ \tilde{\H} (f_{s,t}) = \K \otimes_{R_{\K}} \tilde{\H} (\epsilon_{{}_s \underline{m}}).\]
  Since $\tilde{\H} (\epsilon_{{}_u \underline{m}}) = \epsilon^{{}_u \underline{m}}$ for $u \in \{s,t\}$, we conclude from Lemma \ref{lem:uniqueness_of_beta} (2) that $\H (f_{s,t}) = \beta_{s,t}$.
  By a symmetric argument, we also have that $\H (f_{t,s}) = \beta_{t,s}$
\end{enumerate}

As a result, we have 2-functors
\[\Upsilon_{\RW}^{\O} : \grEWmon{}{}^{\BS} (\fr{h}_{\O}, W, \fr{o}) \to \PEE{}{}^{\BS} (G, \O, \fr{o}) \qquad\text{and}\qquad \Upsilon_{\RW}^{\K} : \grEWmon{}{}^{\BS} (\fr{h}_{\K}, W, \fr{o}) \to \PEE{}{}^{\BS} (G, \K, \fr{o}).\]
Since all the relevant Hom spaces in are graded free $R$-modules (Theorem \ref{thm:dll_is_a_basis} and \cite[Lemma 3.6.3]{Sandvik}), we can extend scalars from $\O$ to $\F$ to produce a 2-functor
\[\Upsilon_{\RW}^{\F} : \grEWmon{}{}^{\BS} (\fr{h}_{\F}, W, \fr{o}) \to \PEE{}{}^{\BS} (G, \F, \fr{o}).\]

\subsection{Monodromic Riche--Williamson Theorem} 

\begin{theorem}\label{thm:mon_RW_thm}
    Let $\k \in \{\K, \O, \F\}$ and $\fr{o} \subseteq \Ch^{\mu} (T, \k)$ be a $W$-orbit. The 2-functor
    \[\Upsilon_{\RW}^{\k} : \grEWmon{}{}^{\BS} (\fr{h}_{\k}, W, \fr{o}) \to \PEE{}{}^{\BS} (G, \k, \fr{o})\]
    is an equivalence of 2-categories.
\end{theorem}

\begin{remark}
    We expect that it should be possible to construct $\Upsilon_{\textnormal{RW}}^{\k}$ when $G$ is a Kac--Moody group which extends the neutral block equivalence from \cite[Theorem 5.3.1]{Sandvik}.
    The main problem is that $\IC_s^{\scrL}$ for $\scrL s \neq \scrL$ has no canonical representative and the Whittaker model is too large to rigidify these choices.
\end{remark}

\begin{remark}
    Theorem \ref{thm:diag_endoscopy} and Theorem \ref{thm:mon_RW_thm} can be combined to give another proof of the monodromic-endoscopic equivalence for the geometric Hecke categories of parity sheaves (\cite[Theorem 5.11.1]{Sandvik}).
\end{remark}

\begin{remark}
  \begin{enumerate}
    \item While Theorem \ref{thm:mon_RW_thm} is only stated for $\k \in \{\K, \O, \F\}$, it is easy to deduce a version for $\k = \overline{\Q}_{\ell}$ or $\k = \overline{\F}_{\ell}$.
  In these cases, the torsion constraint on $\scrL \in \fr{o}$ ensures that we can find some $\ell$-modular triple $(\K, \O, \F)$ such that $\scrL$ (as well as any member of its $W$-orbit) descends to an $\O$-local system along either the morphism $\O \to \K \to \overline{\Q}_{\ell}$ or $\O \to \F \to \overline{\F}_{\ell}$.
  From here, one can apply extension-of-scalars to deduce the $\overline{\Q}_{\ell}$- and $\overline{\F}_{\ell}$-variants from Theorem \ref{thm:mon_RW_thm}.
    \item Instead of taking $\fr{o}$ to be a $W$-orbit in $\Ch^{\mu} (T, \k)$, we could instead take $\fr{o} = \Ch^{\mu} (T, \k)$. In this generality, Theorem \ref{thm:mon_RW_thm} still holds.
    This can be easily deduced from the decomposition of $\grEWmon{}{}^{\BS} (\fr{h}_{\k}, W, \Ch^{\mu} (T, \k))$ and $\PEE{}{}^{\BS} (G, \k, \Ch^{\mu} (T, \k))$ into sub-2-categories,
    \[\grEWmon{}{}^{\BS} (\fr{h}_{\k}, W, \Ch^{\mu} (T, \k)) \cong \bigsqcup_{W \text{-orbits } \fr{o} \subseteq \Ch^{\mu} (T, \k)} \grEWmon{}{}^{\BS} (\fr{h}_{\k}, W, \fr{o}),\]
    \[\PEE{}{}^{\BS} (G, \k, \Ch^{\mu} (T, \k)) \cong \bigsqcup_{W \text{-orbits } \fr{o} \subseteq \Ch^{\mu} (T, \k)} \PEE{}{}^{\BS} (G, \k, \fr{o}).\]
\end{enumerate}  
\end{remark}

We will prove Theorem \ref{thm:mon_RW_thm} using the same strategy invoked in \cite[Theorem 5.3.1]{Sandvik}. 
It is clear from its definition that $\Upsilon_{\RW}^{\k}$ is a bijection on 0- and 1-morphisms.
Therefore, we only need to show that the induced functor on morphism categories
\[\Upsilon_{\RW}^{\k} : \grEWmon{\scrL}{\scrL'}^{\BS} (\fr{h}_{\k}, W, \fr{o}) \to \PEE{\scrL}{\scrL'}^{\BS} (G, \k)\]
is fully faithful. 

Define a category $\grEWmonME{\scrL}{\scrL'}^{\BS} (\fr{h}_{\k}, W, \fr{o})$ obtained from $\grEWmon{\scrL}{\scrL'}^{\BS} (\fr{h}_{\k}, W, \fr{o})$ by applying $\k \otimes_{R_{\k}} (-)$ to the graded Hom spaces.
This category is well-defined by Theorem \ref{thm:dll_is_a_basis}.
By Theorem \ref{thm:dll_is_a_basis} and \cite[Lemma 3.6.3]{Sandvik}, the functor $\Upsilon_{\RW}^{\k}$ induces a functor
\[\tilde{\Upsilon}_{\RW}^{\k} : \grEWmonME{\scrL}{\scrL'}^{\BS} (\fr{h}_{\k}, W, \fr{o}) \to \PME{\scrL}{\scrL'}^{\BS} (G, \k).\]
Along the way towards proving Theorem \ref{thm:mon_RW_thm}, we will prove the following result.

\begin{proposition}\label{prop:L_mon_RW_thm}
    Let $\k \in \{\K, \O, \F\}$ and $\fr{o} \subseteq \Ch^{\mu} (T, \k)$ be a $W$-orbit. For all $\scrL, \scrL' \in \fr{o}$, the functor
    \[\tilde{\Upsilon}_{\RW}^{\k} : \grEWmonME{\scrL}{\scrL'}^{\BS} (\fr{h}_{\k}, W, \fr{o}) \to \PME{\scrL}{\scrL'}^{\BS} (G, \k)\]
    is an equivalence of categories.
\end{proposition}

\subsubsection{Reduction to Fields}

We will prove that $\Upsilon_{\RW}^{\k}$ and $\tilde{\Upsilon}_{\RW}^{\k}$ are equivalences assuming the following special case.

\begin{claim}\label{claim:field_case}
     $\tilde{\Upsilon}_{\RW}^{\F}$ is an equivalence.
\end{claim}

We delay the proof of Claim \ref{claim:field_case} until \S\ref{subsubsec:field_case_soergel}.

\begin{midsecproof}{Proposition \ref{prop:L_mon_RW_thm} assuming Claim \ref{claim:field_case}}
  Let $\scrL, \scrL' \in \fr{o}$ and $\ux, \uy \in \Exp (W)$ such that $\scrL \ux = \scrL' = \scrL \uy$.
  Consider the morphism of free $\O$-modules of finite rank
  \[\alpha : \Hom_{\grEWmonME{\scrL}{\scrL'}^{\BS} (\fr{h}_{\O}, W, \fr{o})} (B_{\ux}^{\scrL}, B_{\uy}^{\scrL} (n)) \to  \Hom_{\PME{\scrL}{\scrL'}^{\BS} (G, \O)} (\scrE_{\ux}^{\scrL}, \scrE_{\uy}^{\scrL} [n])\]
  given by $\tilde{\Upsilon}_{\RW}^{\O}$.
  By the Nakayama lemma, $\alpha$ is an isomorphism if and only if $\F \otimes_{\O} \alpha$ is an isomorphism.
  Note that by definition, $\F \otimes_{\O} \alpha$ is induced by $\tilde{\Upsilon}_{\RW}^{\F}$. We then have that $\F \otimes_{\O} \alpha$ is an isomorphism by Claim \ref{claim:field_case}.
  Finally, we deduce the full faithfulness of $\tilde{\Upsilon}_{\RW}^{\K}$ from $\tilde{\Upsilon}_{\RW}^{\O}$ by extension-of-scalars.
\end{midsecproof}

\begin{midsecproof}{Theorem \ref{thm:mon_RW_thm} assuming Proposition \ref{prop:L_mon_RW_thm}}
  Let $\scrL, \scrL' \in \fr{o}$ and $\ux, \uy \in \Exp (W)$ such that $\scrL \ux = \scrL' = \scrL \uy$.
  Consider the morphism of graded left $R_{\k}$-modules
  \[\beta : \Hom_{\grEWmon{\scrL}{\scrL'}^{\BS} (\fr{h}_{\k}, W, \fr{o})}^{\bullet} (B_{\ux}^{\scrL}, B_{\uy}^{\scrL}) \to  \Hom_{\PEE{\scrL}{\scrL}^{\BS} (G, \k)}^\bullet (\scrE_{\ux}^{\scrL}, \scrE_{\uy}^{\scrL})\]
  given by $\Upsilon_{\RW}^{\k}$.
  As graded left $R_{\k}$-modules, the left-hand side of $\beta$ is free of finite rank by Theorem \ref{thm:dll_is_a_basis} and the right-hand side of $\beta$ is free of finite rank by \cite[Lemma 3.6.3]{Sandvik}.
  By the graded Nakayama lemma, to show that $\beta$ is an isomorphism it suffices to show that $\k \otimes_{R_{\k}} \beta$ is an isomorphism.
  Note that $\k \otimes_{R_{\k}} \beta$ is the map induced from $\tilde{\Upsilon}_{\RW}^{\k}$.
  It follows from Proposition \ref{prop:L_mon_RW_thm} that $\k \otimes_{R_{\k}} \beta$ is an isomorphism.
\end{midsecproof}

\subsubsection{Mixed Derived Categories}\label{subsubsec:field_case_soergel}

Our goal is to prove Claim \ref{claim:field_case}. This will be proved by first passing to mixed derived categories. 
Let $\scrL, \scrL' \in \fr{o}$.
We consider the geometrically-defined biequivariant and right equivariant mixed derived categories,
\[\DGEE{\scrL}{\scrL'} (G, \F) \coloneq K^b \PEE{\scrL}{\scrL'} (G, \F) \qquad\text{and}\qquad \DGME{\scrL}{\scrL'} (G, \F) \coloneq K^b \PME{\scrL}{\scrL'} (G, \F).\]
These categories were introduced in \cite[\S 3.9]{Sandvik}. 
They share many similar features to the mixed derived category of \S\ref{subsec:mixed_diagrammatic_cat}.
In particular, there are standard objects $\underline{\Delta}_w^{\scrL}$ and costandard objects $\underline{\nabla}_w^{\scrL}$ defined for any $w \in \W{\scrL}{\scrL'}$.

The geometrically-defined biequivariant mixed derived categories can be assembled into a 2-category $\DGEE{}{} (G, \F, \fr{o})$ with object set $\fr{o}$ and morphism categories from $\scrL$ to $\scrL'$ given by $\DGEE{\scrL}{\scrL'} (G, \F)$.
The 2-functor  $\Upsilon_{\RW}^{\F} : \grEWmon{}{} (\fr{h}_{\F}, W, \fr{o}) \to \PEE{}{} (G, \F, \fr{o})$ induces a 2-functor
\[ \underline{\Upsilon}_{\RW}^{\F} :  \DDBE{}{} (\fr{h}_{\F}, W, \fr{o}) \to \DGEE{}{} (G, \F, \fr{o}).\]

\begin{lemma}\label{lem:Psi_on_stds}
  For all $w \in W$ and $\scrL \in \fr{o}$, there are isomorphisms
  \[\underline{\Upsilon}_{\RW}^{\F} (\Delta_w^{\scrL}) \cong \uDelta_w^{\scrL} \qquad \text{and} \qquad \underline{\Upsilon}_{\RW}^{\F} (\nabla_w^{\scrL}) \cong \unabla_w^{\scrL}.  \]
\end{lemma}
\begin{proof}
  We will just prove the first isomorphism. The second isomorphism follows from a similar argument.
  By 2-functorality of $\underline{\Upsilon}_{\RW}^{\F}$, Proposition \ref{prop:mixed_conv_rules}, and \cite[Proposition 3.9.6]{Sandvik}, it suffices to consider the case of $s \in S$.
  
  The case when $\scrL s \neq \scrL$ is obvious since $\Delta_s^{\scrL} = B_s^{\scrL}$ and $\uDelta_s^{\scrL} = \scrE_s^{\scrL}$, so we may assume that $\scrL s = \scrL$.
  In $\DDBE{\scrL}{\scrL} (\fr{h}_{\F}, W, \fr{o})$ there is a distinguished triangle
  \begin{equation}\label{eq:Psi_on_stds_1}
    \Delta_s^{\scrL} \to B_s^{\scrL} \to \Delta_e^{\scrL} (1) \to.
  \end{equation}
  We can then apply $\underline{\Upsilon}_{\RW}^{\F}$ to (\ref{eq:Psi_on_stds_1}) to produce a distinguished triangle
  \begin{equation}\label{eq:Psi_on_stds_2}
    \underline{\Upsilon}_{\RW}^{\F} \left( \Delta_s^{\scrL} \right) \to \scrE_{s}^{\scrL} \to \uDelta_e^{\scrL} (1) \to,
  \end{equation}
  where the second map is given by the unit map $\id \to j_{e*} j_e^*$ in the non-mixed category.
  Since $j_x^* \scrE_x^{\scrL} = 0$ for all $x \notin \{e,s\}$, the distinguished triangle (\ref{eq:Psi_on_stds_2}) coincides with the open-closed distinguished triangle for the pair ($\eFl_{e}, \eFl \setminus \eFl_e$).
  As a result, we have produced an isomorphism $ \underline{\Upsilon}_{\RW}^{\F} \left( \Delta_s^{\scrL} \right) \cong \uDelta_s^{\scrL}$ as desired.
\end{proof}

Let $\scrL, \scrL' \in \fr{o}$.
The functor $\tilde{\Upsilon}_{\RW}^{\F} : \grEWmonME{\scrL}{\scrL'} (\fr{h}_{\F}, W, \fr{o}) \to \PME{\scrL}{\scrL'} (G, \F)$ induces a functor
\[\underline{\tilde{\Upsilon}}_{\RW}^{\F} : \DDME{\scrL}{\scrL'} (\fr{h}_{\F}, W) \to \DGME{\scrL}{\scrL'} (G, \F)\]
By Lemma \ref{lem:Psi_on_stds}, we also have isomorphisms
\begin{equation}\label{eq:ME_Psi_on_stds}
  \underline{\tilde{\Upsilon}}_{\RW}^{\F} (\Delta_w^{\scrL}) \cong \uDelta_w^{\scrL} \qquad \text{and} \qquad \underline{\tilde{\Upsilon}}_{\RW}^{\F} (\nabla_w^{\scrL}) \cong \unabla_w^{\scrL} 
\end{equation}
for all $\scrL \in \fr{o}$ and $w \in \W{\scrL}{\scrL'}$.

\begin{midsecproof}{Claim \ref{claim:field_case}}
  It suffices to show that $\underline{\tilde{\Upsilon}}_{\RW}^{\F}$ is fully faithful.
  Let $\scrL \in \fr{o}$ and $x,y \in \W{\scrL}{\scrL'}$.
  By Lemma \ref{lem:hom_vanishing_for_stds_and_costds}, we have that 
  \begin{equation}\label{eq:field_case_1}
    \Hom_{\DDME{\scrL}{\scrL'}} (\Delta_x^{\scrL} , \nabla_y^{\scrL} (m) [n]) \cong \begin{cases} \F & \text{if } x=y \text{ and } n=m=0, \\ 0 & \text{otherwise.}\end{cases}
  \end{equation}
  Similarly, by \cite[Lemma 2.3.4]{Sandvik}, we have an isomorphism
  \begin{equation}\label{eq:field_case_2}
    \Hom_{\DGME{\scrL}{\scrL'} (\F)} (\uDelta_x^{\scrL} , \unabla_y^{\scrL} (m) [n]) \cong \begin{cases} \F & \text{if } x=y \text{ and } n=m=0,\\ 0 & \text{otherwise.}\end{cases}.
  \end{equation}
   In view of this, (\ref{eq:ME_Psi_on_stds}), and Beĭlinson's lemma (\cite[Lemma 3.9.3]{ABG}), it suffices to prove that for all $w \in \W{\scrL}{\scrL'}$, there exists a morphism $f  : \Delta_w^{\scrL} \to \nabla_w^{\scrL}$ such that $\underline{\tilde{\Upsilon}}_{\RW}^{\F} (f) : \uDelta_w^{\scrL} \to \unabla_w^{\scrL}$ is nonzero. 

   Let $\uw \in \Exp (W)$ be a reduced expression of $w$. 
   By Proposition \ref{prop:mixed_conv_rules} and the chain complex description of $\Delta_s^{\scrL}$ (resp. $\nabla_s^{\scrL}$) for $s \in S$ (Example \ref{ex:stds_and_costds_for_simples}), the complex $\Delta_w^{\scrL}$ (resp. $\nabla_w^{\scrL}$) is isomorphic to a complex $\scrF_{\uw}^{\diag, \scrL}$ (resp. $\scrG_{\uw}^{\diag, \scrL}$)
   such that 
   \[(\scrF_{\uw}^{\diag, \scrL})^n = \begin{cases} B_{\uw}^{\scrL} & n = 0, \\ 0 & n < 0, \end{cases} \qquad\text{and}\qquad (\scrG_{\uw}^{\diag, \scrL})^n = \begin{cases} B_{\uw}^{\scrL} & n = 0, \\ 0 & n > 0. \end{cases} \]
   Similarly, by \cite[Proposition 3.9.6]{Sandvik} and the chain complex description of $\uDelta_s^{\scrL}$ (resp. $\unabla_s^{\scrL}$) for $s \in S$ (\cite[(3.9.1)]{Sandvik}), the complex $\uDelta_w^{\scrL}$ (resp. $\unabla_w^{\scrL}$) is isomorphic to a complex $\scrF_{\uw}^{\geom, \scrL}$ (resp. $\scrG_{\uw}^{\geom, \scrL}$)
   such that
  \[(\scrF_{\uw}^{\geom, \scrL})^n = \begin{cases} \scrE_{\uw}^{\scrL} & n = 0, \\ 0 & n < 0, \end{cases} \qquad\text{and}\qquad (\scrG_{\uw}^{\geom, \scrL})^n = \begin{cases} \scrE_{\uw}^{\scrL} & n = 0, \\ 0 & n > 0. \end{cases} \]

   Consider the map $f : \Delta_w^{\scrL} \to \nabla_w^{\scrL}$ defined by the morphism of chain complexes
   \[\begin{tikzcd}[row sep=small]
    \vdots                                                & \vdots                                    \\
    {(\scrF_{\uw}^{\diag, \scrL})^1} \arrow[r, "0"] \arrow[u] & 0 \arrow[u]                               \\
    {B_{\uw}^{\scrL}} \arrow[r, "\id"] \arrow[u]   & {B_{\uw}^{\scrL}} \arrow[u]        \\
    0 \arrow[r, "0"] \arrow[u]                            & {(\scrG_{\uw}^{\diag, \scrL})^{-1}} \arrow[u] \\
    \vdots \arrow[u]                                      & \vdots \arrow[u]                         
    \end{tikzcd}\]
   Under the isomorphisms $\uDelta_w^{\scrL} \cong \scrF_{\uw}^{\geom, \scrL}$ and $\unabla_w^{\scrL} \cong \scrG_{\uw}^{\geom, \scrL}$, the map $\scrF_{\uw}^{\geom, \scrL} \to \scrG_{\uw}^{\geom, \scrL}$ induced by $\underline{\tilde{\Upsilon}}_{\RW}^{\F} (f)$ corresponds to the chain map given by 
   \[ (\scrF_{\uw}^{\geom, \scrL})^0 = \scrE_{\uw}^{\scrL} \stackrel{\id}{\to} \scrE_{\uw}^{\scrL} = (\scrG_{\uw}^{\geom, \scrL})^0 \qquad\text{and} \qquad (\scrF_{\uw}^{\geom, \scrL})^n \stackrel{0}{\to} (\scrG_{\uw}^{\geom, \scrL})^n\]
   for all $n \neq 0$.
   In particular, $\underline{\tilde{\Upsilon}}_{\RW}^{\F} (f)$ is nonzero.
\end{midsecproof} 
    
    \appendix
    \section{Three-Color Relations}\label{apdx:three_color}
    This appendix serves to fill in the details needed to complete the proof of Claim \ref{claim:Phi_well_defined}.
The goal is to verify the three-color relations need to construct the functor
\[\Upsilon_{\Abe} : \EWmon{}{} (W, \fr{h}, \fr{o}) \to \Amon{}{}^{\BS} (W, \fr{h}, \fr{o}).\]

We will first explain the general strategy. We assume that $W$ is a rank 3 finite Coxeter group.
Let $\Gamma_L : B_{\ux}^{\scrL} \to B_{\uy}^{\scrL}$ and $\Gamma_R : B_{\ux}^{\scrL} \to B_{\uy}^{\scrL}$ denote the (linear combinations of) monodromic Elias--Williamson graphs for the left and right sides of a 3-color relation, respectively. 
Note that by definition $\Upsilon_{\Abe} (\Gamma_L)$ and $\Upsilon_{\Abe} (\Gamma_R)$ both take 1-tensors to 1-tensors. As a result, if the space of degree 0 homomorphisms $\Hom_{\grAmon{}{}^{\BS}} (B_{\ux}^{\scrL}, B_{\uy}^{\scrL})$ is a rank 1 $\k$-module, then we are done.
It turns out that this happens in almost every case with the following exceptions:
\begin{enumerate}
    \item[(Case 1)] In type $A_3$, when $W_{\scrL}^{\circ} = W$, the degree 0 Hom space has rank 2. 
    \item[(Case 2)] In type $B_3$, when $W_{\scrL}^{\circ} = W$, the degree 0 Hom space has rank $\rk_{\k} (\fr{h}) + 68$. 
    \item[(Case 3)] In type $B_3$, when $W_{\scrL}^{\circ}$ is type $A_3$, the degree 0 Hom space has rank 3.
    \item[(Case 4)] In type $H_3$, when $W_{\scrL}^{\circ} = W$, the degree 0 Hom space has rank
    \[  \binom{\rk_{\k} (\fr{h}) +3}{4} + 34 \binom{\rk_{\k} (\fr{h}) +2}{3} + 578 \binom{\rk_{\k} (\fr{h}) +1}{2} + 6644 \rk_{\k} (\fr{h}) + 71160.\] 
\end{enumerate}
Cases 1 and 2 follow from the well-definedness of the Soergel functor in the non-monodromic setting. This was already proved using localization in \cite{EW}.
Case 4 is omitted by Assumption \ref{ass:no_H3_copy}.
As a result, the only case that requires a meaningful amount of work is Case 3. 

In this section, we will first explain how we can enumerate all the possible labelings of $\Gamma_L$ and $\Gamma_R$ by monodromy parameters.
It is impossible to enumerate all the $W$-sets; however, we only have to restrict our attention to the small subset which may appear on a given Elias--Williamson graph.
Given a monodromic Elias--Williamson graph $\Gamma : B_{\ux}^{\scrL} \to B_{\uy}^{\scrL}$, the relations in the Hecke category only depend on whether two adjacent faces are labelled the same along an edge.
As a result, we will be able to iterate over labelings of the faces of an Elias--Williamson graph where we only keep track of whether faces are equal or not.
This gives us a finite parameter space. 

Next, we will use this parameter space to compute the rank of the degree 0 Hom spaces for types $A_3, B_3$, and $H_3$ using a case-by-case analysis on the parameter space and the defect formula (Corollary \ref{cor:soergel_hom_defects}).
Using the argument given earlier, this computation will show that $\Upsilon_{\Abe} (\Gamma_L) = \Upsilon_{\Abe} (\Gamma_R)$ in all of these cases except when $W$ is of type $B_3$ and $W_{\scrL}^{\circ}$ is of type $A_3$.
We will then handle this case separately. We will use the algebraic monodromic-endoscopic equivalence (Proposition \ref{prop:abe_endo}) to handle this case.

Finally, in types $A_1 \times I_2 (m)$, since there are infinitely many cases, we are unable to give an exhaustion of the parameter space. Instead, we will appeal to a Coxeter theoretic argument to show using the defect formula that the degree 0 Hom spaces are rank 1.  

The computations in this appendix crucially rely on computer assisted computation. The author has implemented the necessary algorithms in Python and SageMath-- the code can be found here: 
\begin{equation}\label{eq:github}
    \textnormal{\href{https://github.com/colton5007/Monodromic-Three-Color-Relations}{https://github.com/colton5007/Monodromic-Three-Color-Relations}}
\end{equation}

\noindent It includes the exhaustion of the monodromy parameter spaces and the defect calculations used to compute the rank of the degree 0 Hom spaces.

\subsection{Finding Possible Cases}

\begin{definition}
    Let $\Gamma$ be an Elias--Williamson graph.
    Let $\scrF$ be a subset of the faces of $\Gamma$.
    A \emph{$\scrF$-partial monodromy labeling} of an Elias--Williamson graph $\Gamma$ is a surjective map
    \[c : \scrF \twoheadrightarrow \{1 , \ldots, n\}\]
    for some $n \in \Z_{\geq 1}$ such that for all $F \in \scrF$ and $G_1, G_2 \in \scrF$ faces adjacent to $F$, if the edge separating $F$ and $G_1$ and the edge separating $F$ and $G_2$ are both colored by some $s \in S$, then $c (G_1) = c(G_2)$.

    Two $\scrF$-partial monodromy labelings $c, c' : \scrF \to \{1, \ldots, n\}$ are said to be equivalent if there exists a permutation $\sigma$ of $\{1, \ldots, n\}$ such that $\sigma \circ c' = c$. 
    We write $\abs{c} = n$ and say that $n$ is the \emph{number of parameters} for $c$.
    If $\scrF$ is the set of all faces of $\Gamma$, then we say that $c$ is a \emph{complete monodromy labeling}.
\end{definition}

Of course, any monodromic Elias--Williamson graph automatically comes equipped with a complete monodromy labeling.
It is not always the case that a complete monodromy labeling arises via a group action $W$ on $\{1, \ldots, n\}$. 
Nonetheless, a complete monodromy labeling of an Elias--Williamson graph is enough to determine the monodromic subexpressions of the domain and codomain.
In particular, a complete monodromy labeling is sufficient for computing right-hand side of the defect formula (Corollary \ref{cor:soergel_hom_defects}) even if the left-hand side is not well-defined.

The following algorithm provides a means of computing the parameter space of complete monodromy labelings. 

\begin{algorithm}\label{algo:monodromy_labelings}
    Let $\Gamma$ be an Elias--Williamson graph.
    Let $F_1, \ldots, F_n$ be an enumeration of the faces of $\Gamma$. Define $\scrF_i = \{F_1, \ldots, F_i\}$ for each $1 \leq i \neq n$.
    \begin{enumerate}
        \item There is a unique $\scrF_1$-partial monodromy labeling $c_1$ taking $F_1$ to $1$. Define $P_1 = \{c_1\}$.
        \item Recursively, assume we have defined the set $P_i$ of $\scrF_i$-partial monodromy labelings.
            For each $c \in P_i$ and $1 \leq j \leq \abs{c} + 1$, we can attempt to define a $\scrF_{i+1}$-partial monodromy labeling $c'$ by extending $c$ on $\scrF_i$ and labeling $F_{i+1}$ by $j$.
            
            We will then check that $c'$ is a valid partial monodromy labeling as follows. 
            We iterate through each $1 \leq k \leq i+1$ and each pair of adjacent faces $F, F'$ of $F_k$ such that the edges separating $F_k$ from $F$ and $F'$ are both colored by $s$.
            If $F, F' \in \scrF_{i+1}$ and $c' (F) \neq c' (F')$, then $c'$ is not a valid partial monodromy labeling and can be discarded from consideration. Otherwise, continue until all such $k$ and pairs $(F, F')$ are checked.
            If the process completes without $c'$ being rendered invalid, then $c'$ is a valid $\scrF_{i+1}$-partial monodromy labeling.
            
            We then define a list $P_{i+1}$ consisting of the $\scrF_{i+1}$-partial monodromy labelings produced as above.
        \item After $n$ iterations, the above procedure produces a set $P_n$ consisting of complete monodromy labelings.
    \end{enumerate}
\end{algorithm}

It is not hard to check that every complete monodromy labeling (up to equivalence) arises via Algorithm \ref{algo:monodromy_labelings}. 
We keep the notation from the above.
Given a complete monodromy labeling $c : \{F_1, \ldots, F_n\} \twoheadrightarrow \{1, \ldots, k\}$, we can replace $c$ by an equivalent complete monodromy labeling such that $c (F_i) \leq \max_{1 \leq j \leq i-1} (c (F_j) + 1)$.
Let $c_i$ be the $\scrF_i$-partial monodromy labeling obtained by restricting $c$ to $\scrF_i$.
It is clear that $\scrF_1 \in P_1$. Arguing by induction, if $c_i \in P_i$, then our condition that $c (F_i) \leq \max_{1 \leq j \leq i-1} (c (F_j) + 1)$ ensures that $c_{i+1} \in P_{i+1}$.

\begin{remark}
    Algorithm \ref{algo:monodromy_labelings} is not very optimized. It is quick enough for the purposes of the present work, but grows quite quickly as the number of faces and size of $W$ increases.
    The algorithm's efficiency also highly depends on choice of face enumeration. When choosing an enumeration at the $(i+1)$-st step, it is best to choose a face adjacent to the first $i$ faces with maximal valence.
\end{remark}

\subsection{Types \texorpdfstring{$A_3$}{A3}, \texorpdfstring{$B_3$}{B3}, and \texorpdfstring{$H_3$}{H3}}

While Algorithm \ref{algo:monodromy_labelings} is sufficient for labeling one diagram, we actually need to label two diagrams simultaneously, both $\Gamma_L$ and $\Gamma_R$.
To handle this, we will simply apply the algorithm to the horizontal concatenation $\Gamma_{R}' \star \Gamma_{L}$ where $\Gamma_{R}'$ is the horizontal reflection of $\Gamma_{R}$. 
\begin{enumerate}
    \item In type $A_3$, there are 30 complete monodromy labelings of $\Gamma_L$ and $\Gamma_R$.
    \item In type $B_3$, there are 98 complete monodromy labelings of $\Gamma_L$ and $\Gamma_R$.
    \item In type $H_3$, there are 164 complete monodromy labelings of $\Gamma_L$ and $\Gamma_R$. 
\end{enumerate}
As discussed before, one can compute the graded ranks of the Hom spaces in which $\Gamma_L$ and $\Gamma_R$ reside using Corollary \ref{cor:soergel_hom_defects} and the code provided in (\ref{eq:github}). 
All these graded ranks are 1 except in 3 cases: the non-monodromic cases in type $A_3$ and $B_3$ and the monodromic case in type $B_3$ where the endoscopic group has type $A_3$. When the graded rank of the Hom space is 1, then the $1$-tensor argument from before is enough to show that $\Upsilon_{\Abe} (\Gamma_L) = \Upsilon_{\Abe} (\Gamma_R)$. 
The non-monodromic cases in type $A_3$ and $B_3$ follow from the non-monodromic case considered in \cite{EW}.
As a result, we only have to check one case (type $B_3$ with endoscopic type $A_3$) which is depicted below.
\begin{equation}\label{eq:B3_A3_1}
    
  % \tikzsetnextfilename{#1}
  \tikzstyle{every picture}=[tikzfig]
  \input{./figures/B3_A3_three_color.tikz}

\end{equation}
To show that $\Upsilon_{\Abe} (\Gamma_L) = \Upsilon_{\Abe} (\Gamma_R)$ in this case, we will appeal to algebraic endoscopy. Let $w_0$ denote the longest element of $W$.
Note that $\scrL w_0 \neq \scrL$. Instead, $w_0 \in \W{\scrL}{\scrL w_0}$ is in the block containing $u$. 
By Proposition \ref{prop:abe_endo}, it suffices to show that
\[\Psi_{\scrL}^{\alg} \Upsilon_{\Abe} (\Gamma_L \star \id_{B_u^{\scrL w_0}}) = \Psi_{\scrL}^{\alg} \Upsilon_{\Abe} (\Gamma_R \star \id_{B_u^{\scrL w_0}}).\]
We will compute the left-hand side and right-hand side separately.

We fix an endosimple expansion datum $\iota$ as follows: $s' = utu$, $t' = s$, and $u' = t$.
For each Elias--Williamson graph $\Gamma$ appearing in (\ref{eq:B3_A3_1}) the strategy for computing $\Psi_{\scrL}^{\alg} \Upsilon_{\Abe} (\Gamma \star \id_{B_u^{\scrL w_0}})$ will be the same.
First, we will construct an Elias--Williamson graph $\Gamma'$ for the endoscopic Coxeter group $W_{\scrL}^{\circ}$.
We will then compute $\Psi_{\scrL}^{\alg, -1} (\Gamma') : M \to N$ after fixing 1-tensor preserving isomorphisms $M \cong B_{\iota (\ux')}^{\scrL}$ and $N \cong B_{\iota (\uy')}^{\scrL}$ where $\ux' = (t', u',t', s',t',u')$ and $\uy' = (s',t',u',t',s',t')$.
Consider the monodromic Elias--Williamson graphs
\[
  % \tikzsetnextfilename{#1}
  \tikzstyle{every picture}=[tikzfig]
  \input{./figures/endotransfer_1.tikz}
 \qquad\text{and}\qquad
  % \tikzsetnextfilename{#1}
  \tikzstyle{every picture}=[tikzfig]
  \input{./figures/endotransfer_2.tikz}
.\]
Via $\Upsilon_{\Abe}$, these give rise to isomorphisms $\xi_{\ux u, \iota (\ux')} : B_{\ux u}^{\scrL} \stackrel{\sim}{\to} B_{\iota (\ux')}^{\scrL}$ and $\xi_{\uy u, \iota (\uy')} : B_{\uy u}^{\scrL} \stackrel{\sim}{\to} B_{\iota (\uy')}^{\scrL}$ in $\Amon{\scrL}{\scrL}^{\oplus}$.
We will then show that 
\[\Upsilon_{\Abe} (\Gamma) = \xi_{\uy u, \iota (\uy')}^{-1} \circ \Psi_{\scrL}^{\alg, -1} (\Gamma') \circ \xi_{\ux u, \iota (\ux')}.\]
This will allow us to check the 3-color relation holds after using endoscopy to transfer the relation to a relation in $\Amon{1}{1}^{\oplus} (\fr{h}, W_{\scrL}^{\circ}, 1)$. 

We will use morphisms in the Elias--Williamson diagrammatic category $\EWmon{1}{1}^{\BS} (\fr{h}, W_{\scrL}^{\circ}, 1)$ to denote morphisms in $\Amon{1}{1}^{\BS} (\fr{h}, W_{\scrL}^{\circ}, 1)$.
This is justified by \cite[Theorem 5.6]{Abe19}. 

The following examples are illustrative and will be used in the proof. 
\begin{example}\label{ex:B3_A3_endosimple_expansions}
    \begin{enumerate}
        \item Let
        \[\Gamma = 
  % \tikzsetnextfilename{#1}
  \tikzstyle{every picture}=[tikzfig]
  \input{./figures/endosimp_ex1.tikz}
 \qquad\text{and}\qquad \Gamma' = 
  % \tikzsetnextfilename{#1}
  \tikzstyle{every picture}=[tikzfig]
  \input{./figures/endosimp_ex2.tikz}
.\]
        A simple rank computation using Corollary \ref{cor:soergel_hom_defects} shows that 
        \[\rk_{\k} \Hom_{\grAmon{}{}^{\oplus}} (B_{(s,u,t,u,s)}^{\scrL}, B_{(u,t,u,s,u,t,u)}^{\scrL}) = 1.\]
        Moreover, $\Upsilon_{\Abe} (\Gamma)$ and $\Psi_{\scrL}^{\alg, -1} (\Gamma')$ both preserve 1-tensors. As a result, we must have that $\Upsilon_{\Abe} (\Gamma) = \Psi_{\scrL}^{\alg, -1} (\Gamma')$.
        \item Let
        \[\Gamma = 
  % \tikzsetnextfilename{#1}
  \tikzstyle{every picture}=[tikzfig]
  \input{./figures/endosimp_ex3.tikz}
 \qquad\text{and}\qquad \Gamma' = 
  % \tikzsetnextfilename{#1}
  \tikzstyle{every picture}=[tikzfig]
  \input{./figures/endosimp_ex4.tikz}
.\]
        The uniqueness of the Frobenius algebra structure on $B_{(u,t,u)}^{\scrL}$ implies that $\Upsilon_{\Abe} (\Gamma) = \Psi_{\scrL}^{\alg, -1} (\Gamma')$.
    \end{enumerate}
\end{example}

Let 
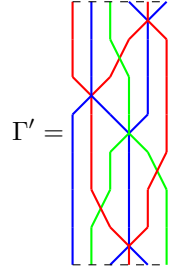
\[\Gamma' = 
  % \tikzsetnextfilename{#1}
  \tikzstyle{every picture}=[tikzfig]
  \input{./figures/A3_B3_calc_1.tikz}
 \]
By writing $\Gamma'$ as a composition of generators, we can calculate analogous to Example \ref{ex:B3_A3_endosimple_expansions} (1) that
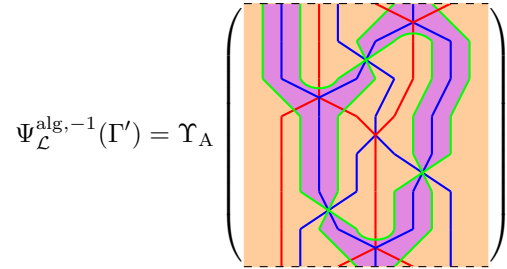
\[\Psi_{\scrL}^{\alg, -1} (\Gamma') = \Upsilon_{\Abe} \left( 
  % \tikzsetnextfilename{#1}
  \tikzstyle{every picture}=[tikzfig]
  \input{./figures/A3_B3_calc_2.tikz}
 \right)\]
It then follows by a series of block minimality, 2-color associativity, and the Elias--Jones--Wenzl relations that
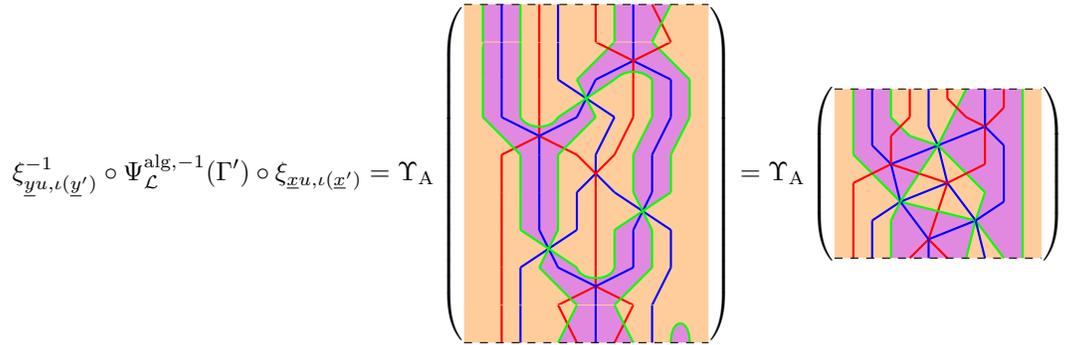
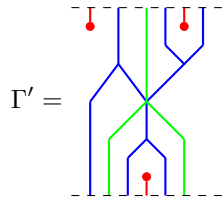
\[\xi_{\uy u, \iota (\uy')}^{-1} \circ \Psi_{\scrL}^{\alg, -1} (\Gamma') \circ \xi_{\ux u, \iota (\ux')} = \Upsilon_{\Abe} \left( 
  % \tikzsetnextfilename{#1}
  \tikzstyle{every picture}=[tikzfig]
  \input{./figures/A3_B3_calc_3.tikz}
 \right) = \Upsilon_{\Abe} \left(  
  % \tikzsetnextfilename{#1}
  \tikzstyle{every picture}=[tikzfig]
  \input{./figures/A3_B3_calc_4.tikz}
 \right).\]

Let 
\[\Gamma' = 
  % \tikzsetnextfilename{#1}
  \tikzstyle{every picture}=[tikzfig]
  \input{./figures/A3_B3_calc_5.tikz}
\]
By writing $\Gamma'$ as a composition of generators, we can calculate analogous to Example \ref{ex:B3_A3_endosimple_expansions} that
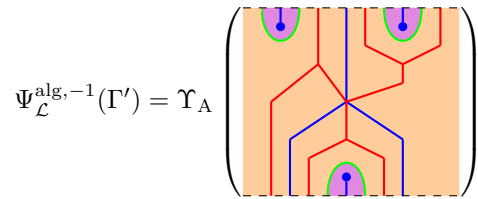
\[\Psi_{\scrL}^{\alg, -1} (\Gamma') = \Upsilon_{\Abe} \left( 
  % \tikzsetnextfilename{#1}
  \tikzstyle{every picture}=[tikzfig]
  \input{./figures/A3_B3_calc_6.tikz}
\right)\]
It then follows by a series of block minimality and 2-color associativity relations that
\[\xi_{\uy u, \iota (\uy')}^{-1} \circ \Psi_{\scrL}^{\alg, -1} (\Gamma') \circ \xi_{\ux u, \iota (\ux')} = \Upsilon_{\Abe} \left( 
  % \tikzsetnextfilename{#1}
  \tikzstyle{every picture}=[tikzfig]
  \input{./figures/A3_B3_calc_7.tikz}
\right) = \Upsilon_{\Abe} \left( 
  % \tikzsetnextfilename{#1}
  \tikzstyle{every picture}=[tikzfig]
  \input{./figures/A3_B3_calc_8.tikz}
 \right).\]

There are two remaining monodromic Elias--Williamson graphs in (\ref{eq:B3_A3_1}) that must also be considered. The arguments are entirely similar to the two graphs considered previously and will be omitted.
We conclude that there are equalities of morphisms $B_{\ux u}^{\scrL} \to B_{\uy u}^{\scrL}$ in $\Amon{\scrL}{\scrL}^{\oplus}$.
\[\left( \xi_{\uy u, \iota (\uy')}^{-1} \circ \Psi_{\scrL}^{\alg, -1} \right) \left(  
  % \tikzsetnextfilename{#1}
  \tikzstyle{every picture}=[tikzfig]
  \input{./figures/A3_B3_calc_9.tikz}
\right) \circ \xi_{\ux u, \iota (\ux')} = \Upsilon_{\Abe} \left( 
  % \tikzsetnextfilename{#1}
  \tikzstyle{every picture}=[tikzfig]
  \input{./figures/A3_B3_calc_10.tikz}
 \right),\]
and
\[\left( \xi_{\uy u, \iota (\uy')}^{-1} \circ \Psi_{\scrL}^{\alg, -1} \right) \left( 
  % \tikzsetnextfilename{#1}
  \tikzstyle{every picture}=[tikzfig]
  \input{./figures/A3_B3_calc_11.tikz}
 \right) \circ \xi_{\ux u, \iota (\ux')} = \Upsilon_{\Abe} \left( 
  % \tikzsetnextfilename{#1}
  \tikzstyle{every picture}=[tikzfig]
  \input{./figures/A3_B3_calc_12.tikz}
 \right).\]

It then suffices to show the following equality holds in $\EWmon{1}{1}^{\BS} (\fr{h}, W_{\scrL}^{\circ}, 1)$,
\begin{equation}\label{eq:B3_A3_2}
    
  % \tikzsetnextfilename{#1}
  \tikzstyle{every picture}=[tikzfig]
  \input{./figures/A3_B3_calc_13.tikz}

\end{equation}
Equation (\ref{eq:B3_A3_2}) can be checked using the Zamolodchikov relation and the Elias--Jones--Wenzl relation for $A_3$. 
Alternatively, it can be checking via localization using the ASLoc software package \cite{GJW23}.

\subsection{Type \texorpdfstring{$A_1 \times I_2 (m)$}{A1 x I2(m)}}

Any reflection subgroup of a product of Coxeter groups can be written as a product of reflection subgroups of the individual factors.
As a result, the possible reflection subgroups of the Coxeter group of type $A_1 \times I_2 (m)$ are of the following types: $A_1 \times I_2 (v)$ and $I_2 (v)$ where $0 \leq v \leq m$ divides $m$ (by definition, we take $I_2 (0) = 1$ and $I_2 (1) = A_1$).
Since $\Upsilon_A (\Gamma_L)$ and $\Upsilon_A (\Gamma_R)$ both take 1-tensor to 1-tensors, the desired relation will follow from the following lemma.

\begin{lemma}\label{lem:homs_for_A1_I2m}
    Write $S = \{s,t,u\}$ where $m_{s,t} = m$, $m_{s,u} = 2$, and $m_{t,u}= 2$. 
    The $\k$-module of degree 0 homomorphisms
    \[ \Hom_{\grAmon{}{}^{\BS}} (B_{{}_s \underline{m} u}^{\scrL}, B_{u {}_t \underline{m}}^{\scrL}) \]
    is rank 1.
\end{lemma}
\begin{proof}
    Let $w_0$ denote the maximal element in $W$.
    By Corollary \ref{cor:soergel_hom_defects}, it suffices to show that for all proper (i.e., not all 1's) $\scrL$-monodromic subexpressions $\ue$ of ${}_s \underline{m} u$ and $\uf$ of $u {}_t \underline{m}$ that
    $d ({}_s \underline{m} u, \ue) + d (u {}_t \underline{m}, \uf) > 0$.
    We will show something even stronger, namely, that 
    \[d ({}_s \underline{m} u, \ue) > 0 \qquad \text{and}\qquad d (u {}_t \underline{m}, \uf) > 0.\]
    We will just show the first inequality as the other follows from a similar argument.
    Note that $\dec_1 ({}_s \underline{m} u, \ue) = \dec_{m+1} (u {}_t \underline{m}, \uf)$ and can only be equal to $U1$ or $U0$.
    Let $\ue' = \ue_{\geq 2}$ which is a subexpression of ${}_s \underline{m}$.
    A simple computation for dihedral groups shows that $d ({}_s \underline{m}, \ue') \geq 0$ with equality if and only if $\ue'$ is the all 1's expression.
    We are then done unless $\ue'$ is the all 1's subexpression. Since we are assuming that $\ue$ is proper, this means that $e_1 = U0$ which implies that $d ({}_s \underline{m} u, \ue) > 0$.
\end{proof}

    \section{Mixed Derived Categories}\label{apdx:mdc}
    Achar--Riche--Vay introduced a mixed derived category associated to the Elias--Williamson diagrammatic Hecke category \cite{ARV}.
Their mixed derived category was later used in \cite{RV} to extend the Koszul duality equivalence of \cite{AMRW2} to arbitrary Coxeter groups.
One of the main reasons for studying the mixed derived category is that it comes equipped with a recollement formalism.
This gives rise to a theory of standard and costandard objects which generate the mixed derived category. Moreover, the Hom spaces between standards and costandards are simple to compute. 

In fact, most of the constructions in \cite{ARV} hold in larger generality. Namely, the key properties used are about the double leaves basis. 
We can then generalize their construction to produce a mixed derived category for any object-adapted cellular category.
The principal goal of this section is to develop this theory with two key examples in mind: (1) the category of monodromic Bott--Samelson bimodules and (2) the monodromic diagrammatic Hecke category.
We omit many of the proofs in this section and refer to \cite{ARV} for detailed arguments.

\subsection{Object-Adapted Cellular Categories}

Object-adapted cellular categories were first introduced in \cite{EL}. A gentle treatment of the subject can be found in \cite{EMTW} whose terminology we closely follow.

Let $\k$ be a noetherian domain of finite global dimension.
Let $R$ be a graded commutative ring such that $R^0 \cong \k$. 

\begin{definition}\label{def:OACC}
    An \emph{object-adapted cellular category}\footnote{In \cite{EL}, these are called strictly object-adapted cellular categories.} (abbv. OACC) is a graded $R$-linear category $\scrC$ equipped with the following data:
    \begin{itemize}
        \item There is a set $\Lambda$ of objects of $\scrC$ and a partial order $\leq$ on $\Lambda$.
        \item There is an $R$-linear contravariant involution $\DD$ on $\scrC$.
        \item For each $X \in \scrC$ and each $\lambda \in \Lambda$, there are two finite families
                \[(\LL_T)_{T \in E(X, \lambda)} \subset \Hom_{\scrC} (X, \lambda) \qquad\text{and}\qquad (\overline{\LL}_S)_{S \in \overline{E}(\lambda, X)} \subset \Hom_{\scrC} (\lambda, X)\]
                of homogenous morphisms, indexed by finite sets $E(X, \lambda)$ and $\overline{E} (\lambda, X)$.
        \item For each object $X \in \scrC$ and $\lambda \in \Lambda$, there are pairwise inverse maps between the indexing sets $E(X, \lambda)$ and $\overline{E} (\lambda, X)$ that we denote by $\overline{\:\cdot\:}$ in either direction. 
    \end{itemize}
    
    This data is subject to the following requirements:
    \begin{enumerate}
        \item[(CC1)]\customlabel{oacc:cc1}{CC1} Define $\dLL_{S,T}^{\lambda} \coloneq \overline{LL}_S \circ \LL_T \in \Hom_{\scrC} (X, Y)$ for $T \in E(X, \lambda)$ and $S \in \overline{E}(\lambda, Y)$.
        Consider the family 
        \[ \dLL^{\lambda} (X, Y) \coloneq \{ \dLL_{S,T}^{\lambda} \colon T \in E(X, \lambda), S \in \overline{E}(\lambda, Y)\} \subseteq \Hom_{\scrC} (X,Y). \]
        Then the family $\dLL (X, Y) \coloneq \bigcup_{\lambda \in \Lambda} \dLL^{\lambda} (X,Y)$ is a graded $R$-basis of $\Hom_{\scrC} (X,Y)$.
        \item[(CC2)]\customlabel{oacc:cc2}{CC2} For all $T \in E(X, \lambda)$, we have that $\DD (\LL_T) = \overline{\LL}_{\overline{T}}$.
        \item[(CC3)]\customlabel{oacc:cc3}{CC3} For each $\lambda \in \Lambda$, the sets $E(\lambda, \lambda)$ and $\overline{E} (\lambda, \lambda)$ are just singletons $\{*\}$ and
        \[\LL_{*} = \id_{\lambda} = \overline{\LL}_{*}.\]
        \item[(CC4)]\customlabel{oacc:cc4}{CC4} For each $\lambda \in \Lambda$, the collection of morphisms
            \[ \Hom_{\leq \lambda} (X,Y) \coloneq \textnormal{span}_R \left( \bigcup_{\mu \leq \lambda} \dLL^{\mu} (X,Y) \right) \subset \Hom_{\scrC} (X,Y) \]
            forms a two-sided ideal $\scrC_{\leq \lambda}$ in $\scrC$, i.e., it is stable under pre- and post-composition with arbitrary morphisms in $\scrC$.
    \end{enumerate}
\end{definition}

\begin{definition}
    Let $\scrC$ be an object-adapted cellular category. A map $\pi : \Ob (\scrC) \to \Lambda$ is called an \emph{evaluation} if
    \begin{enumerate}
        \item $\lambda \leq \pi (X)$ whenever $E(X, \lambda) \neq \emptyset$,
        \item  $\pi (\lambda) = \lambda$ for all $\lambda \in \Lambda$.
    \end{enumerate}

    An object $X \in \scrC$ is said to be \emph{reduced} with respect to $\pi$ if 
    \begin{enumerate}
        \item $E(X, \pi (X))$ is a singleton $\{*\}$,
        \item the coefficient of $\id_{\pi (X)}$ in $\LL_* \circ \overline{\LL}_{*}$ is 1.
    \end{enumerate}
\end{definition}

\begin{remark}\label{rem:OACC_with_evals}
    Let $\scrC$ be an object-adapted cellular category with evaluation $\pi$.
    \begin{enumerate}
        \item For all $\lambda \in \Lambda$, by (\ref{oacc:cc3}), we have that $\lambda$ is reduced.
        \item If $X, Y$ are reduced and $\pi (X) = \lambda = \pi (Y)$, then there is a unique map $\dLL^{\lambda} : X \to Y$ in $\dLL^{\lambda} (X,Y)$.
    \end{enumerate}
\end{remark}

\subsection{Homological Algebra Preliminaries}

Let $\scrC$ be a graded $R$-linear category. 
We can define a category $\scrC^{\circ}$ whose objects are symbols $X(n)$ where $X \in \Ob (\scrC)$ and $n \in \Z$, and whose morphisms are defined by
\[\Hom_{\scrC^{\circ}} (X (n), Y(m)) \coloneq \Hom_{\scrC}^{m-n} (X,Y).\]
We will also define a graded $R$-module
\[\Hom_{\scrC^{\circ}}^{\bullet} (X (n), Y(m)) \coloneq \bigoplus_{k \in \Z} \Hom_{\scrC^{\circ}} (X (n), Y(m + k)) = \Hom_{\scrC} (X, Y) (m-n).\]
The category $\scrC^{\circ}$ admits a natural autoequivalence $(1) : \scrC^{\circ} \to \scrC^{\circ}$ defined by $X(n) \mapsto X(n+1)$.
For $j \in \Z_{\geq 0}$, we can define $(j)$ as the $j$-th power of $(1)$ whose inverse is denoted by $(-j)$. 
We define $\scrC^{\circ, \oplus}$ as the additive hull of $\scrC^{\circ}$. It inherits a grading shift autoequivalence $(n)$ for each $n \in \Z$ from $\scrC^{\circ}$.

We can view $R$ as a graded $R$-linear category itself with one object $*$. Observe that $R^{\circ}$ is a monoidal category with $* (n) \otimes * (m) = * (n+m)$. 
Note that the additive hull of $R^{\circ}$ is equivalent to the monoidal category of finitely generated graded free $R$-modules $\Free^{\fg, \Z} (R)$.
Since $\scrC$ is enriched over graded $R$-modules, the category $\scrC^{\circ}$ (resp. $\scrC^{\circ, \oplus}$) is naturally an $R^{\circ}$-module (resp. $\Free^{\fg, \Z} (R)$-module).
In particular, for $X \in \scrC^{\circ, \oplus}$ and $M \in \Free^{\fg, \Z} (R)$, we have the internal tensor object, denoted $X \underline{\otimes}_R M$, defined as the object representing the functor
\[Y \mapsto (\Hom_{\scrC^{\circ, \oplus}}^{\bullet} (X,Y) \otimes_R M)^{0}.\]
We then deduce that there is a natural isomorphism
\begin{equation}\label{eq:internal_tensor_defining}
    \Hom_{\scrC^{\circ, \oplus}}^{\bullet} (Y, X \underline{\otimes}_R M) \cong \Hom_{\scrC^{\circ, \oplus}}^{\bullet} (Y, X) \otimes_R M.
\end{equation}

Now, let $X, Y \in \scrC^{\circ, \oplus}$ and assume that $\Hom_{\scrC^{\circ, \oplus}} (X, Y)$ is a graded free as an $R$-module.
We will define a canonical morphism
\begin{equation}\label{eq:canon_mor_internal_tensor}
    X \underline{\otimes}_R \Hom_{\scrC^{\circ, \oplus}}^{\bullet} (X, Y) \to Y.
\end{equation}
If $(\varphi_i)_{i \in I}$ is a graded basis of the $R$-module $\Hom_{\scrC^{\circ, \oplus}}^{\bullet} (X, Y)$, then this choice allows us to identify $ X \underline{\otimes}_R \Hom_{\scrC^{\circ, \oplus}}^{\bullet} (X, Y)$ with $\bigoplus_{i \in I} X (-\deg (\varphi_i))$.
We then define (\ref{eq:canon_mor_internal_tensor}) by $\bigoplus_{i \in I} \varphi_i (- \deg (\varphi_i))$.
It is not hard to check that this morphism is independent of the choice of basis. For any $Z \in \scrC^{\circ, \oplus}$, the induced morphism
\[\Hom_{\scrC^{\circ, \oplus}}^{\bullet} (Z, X \underline{\otimes}_R \Hom_{\scrC^{\circ, \oplus}}^{\bullet} (X, Y)) \to \Hom_{\scrC^{\circ, \oplus}}^{\bullet} (Z, Y)\]
is identified, via (\ref{eq:internal_tensor_defining}), with the morphism
\[\Hom_{\scrC^{\circ, \oplus}}^{\bullet} (Z, X) \otimes_R \Hom_{\scrC^{\circ, \oplus}}^{\bullet} (X, Y) \to \Hom_{\scrC^{\circ, \oplus}}^{\bullet} (Z, Y)\]
given by morphism composition.

\subsection{Categories Attached to Locally Closed Subsets}

\subsubsection{Topology on \texorpdfstring{$\Lambda$}{Lambda}}

Let $(\Lambda, \leq)$ be a poset. We say that a subset $I \subseteq \Lambda$ is \emph{closed} if for any $\lambda,\mu \in \Lambda$ with $\mu \in I$ and $\lambda \leq \mu$, then $\lambda \in I$.
A subset $J \subseteq \Lambda$ is called \emph{open} if $\Lambda \setminus J$ is closed. Given any subset $K \subseteq \Lambda$, we can consider its \emph{closure}
\[\overline{K} \coloneq \{ \lambda \in \Lambda \mid \exists \mu \in K \textnormal{ such that } \lambda \leq \mu \},\]
which is the smallest closed set containing $K$.
Finally, we say that $K \subseteq X$ is \emph{locally closed} if $K$ is open in $\overline{K}$, or equivalently if $K$ is closed in $\Lambda \setminus (\overline{K} \setminus K)$.

Let $\lambda \in \Lambda$, we will write 
\[\{\leq \lambda\} \coloneq \{ \mu \in \Lambda \mid \mu \leq \lambda \} \qquad\text{and}\qquad \{< \lambda\} \coloneq \{ \mu \in \Lambda \mid \mu < \lambda \}.\]

\subsubsection{Locally Closed Subsets}

Let $\scrC$ be an OACC with evaluation $\pi$ and fix a closed subset $I \subseteq \Lambda$.
Define a category $\scrC_I^{\circ}$ as the full subcategory of $\scrC^{\circ}$ whose objects are of the form $X (n)$ for $n \in \Z$ and $X$ reduced in $\scrC$ such that $\pi (X) \in I$.
We also write $\scrC_{I}^{\oplus, \circ}$ for the additive hull of $\scrC_I^{\circ}$. Note that  $\DD$ stabilizes $\scrC_{I}^{\oplus, \circ}$ by (\ref{oacc:cc2}).

\begin{remark}\label{rem:failure_of_additive_generation}
    When $I = \Lambda$, it is generally not true that $\scrC_{\Lambda}^{\oplus, \circ}$ is equivalent to the additive hull of $\scrC^{\circ}$.
\end{remark}

Let $X, Y \in \scrC^{\circ}$, we will denote by
\[\mathfrak{J}_{I} (X,Y) \subseteq \Hom_{\scrC^{\circ, \oplus}}^{\bullet} (X,Y)\]
the $R$-submodule of morphisms which factor through $\scrC_{I}^{\oplus, \circ}$.

\begin{lemma}[{\cite[Lemma 4.2]{ARV}}]\label{lem:OACC_ideals}
    Let $X,Y \in \scrC$, then
    \begin{equation}\label{eq:OACC_ideals}
        \mathfrak{J}_{I} (X,Y) = \textnormal{span}_R \left( \bigcup_{\lambda \in I} \dLL^{\lambda} (X,Y) \right).
    \end{equation}
\end{lemma}
\begin{proof}
    It is clear from the definition that if $\lambda \in I$, then $\dLL_{S,T}^{\lambda} \in \mathfrak{J}_{I} (X,Y)$. As a result, the right-hand side of (\ref{eq:OACC_ideals}) is contained in $\mathfrak{J}_{I} (X,Y)$.
    
    Let $Z$ be reduced with $\pi (Z) \in I$ and $f : X \to Y$ be a morphism in $\scrC$ which factors through $\pi (Z)$.
    By (\ref{oacc:cc1}), we can replace $f$ by a composition of the form
    \[X \stackrel{\dLL_{S,T}^{\lambda}}{\longrightarrow} Z \to Y\]
    for some $\lambda \in \Lambda$, $S \in E(X, \lambda)$, and $T \in E(\lambda, Z)$. 
    Since $\pi$ is an evaluation, we must have that $\lambda \leq \pi (Z)$, and hence $\lambda \in I$. We are then done by (\ref{oacc:cc4}).
\end{proof}

Let $I_0 \subset \Lambda$ and $I_1 \subseteq I_0$ be closed subsets.
We define a category
\[\scrC_{I_0, I_1}^{\circ, \oplus} \coloneq \scrC_{I_0}^{\circ, \oplus} \naivequotient \scrC_{I_1}^{\circ, \oplus},\]
where the quotient appearing on the right-hand side is the ``naive'' quotient. I.e., its objects are the same as those in $\scrC_{I_0}^{\circ, \oplus}$, and its morphisms are defined by
\[\Hom_{\scrC_{I_0, I_1}^{\circ, \oplus}} (X,Y) = \left( \Hom_{\scrC_{I_0}^{\circ, \oplus}}^{\bullet} (X,Y) / \mathfrak{J}_{I_1} (X,Y) \right)^0\]
for $X,Y \in \scrC_{I_0}^{\circ, \oplus}$.

The shift functor (1) induces an autoequivalence on $\scrC_{I_0, I_1}^{\circ, \oplus}$. Likewise, $\DD$ on $\scrC_{I_0}^{\circ, \oplus}$ restricts to a contravariant autoequivalence of $\scrC_{I_0, I_1}^{\circ, \oplus}$. 
By Lemma \ref{lem:OACC_ideals}, we have that $\bigcup_{\lambda \in I_0 \setminus I_1} \dLL^{\lambda} (X,Y)$ forms a graded $R$-linear basis for $\Hom_{\scrC_{I_0, I_1}^{\circ, \oplus}}^{\bullet} (X,Y)$.
It then follows (see \cite[Lemma 4.3]{ARV}) that $\scrC_{I_0, I_1}^{\circ, \oplus}$ only depends on $I_0 \setminus I_1$ up to a canonical equivalence of categories.

We can then define, for any locally closed subset $I \subseteq \Lambda$, the category 
\[\scrC_{I}^{\circ, \oplus} \coloneq \scrC_{I_0, I_1}^{\circ, \oplus}\]
where $I_1 \subset I_0$ are closed subsets of $\Lambda$ such that $I = I_0 \setminus I_1$. Of course, when $I$ itself is closed, the notation agrees with definition of the category associated to a closed subset from earlier.

\subsubsection{Case of a Singleton}

We consider the special case when $I = \{\lambda\}$ for some $\lambda \in \Lambda$.

\begin{lemma}\label{lem:singleton_identification}
    Let $X \in \scrC$ be reduced and $\pi (X) = \lambda$. By definition $E (X,\lambda) = \{*\}$ is a singleton. 
    The image of the canonical map $\LL_* : X \to \lambda$ in $\scrC_{\{\lambda\}}^{\circ, \oplus}$ is an isomorphism.
\end{lemma}
\begin{proof}
   Since $X$ is reduced, we must have that $\LL_* \circ \overline{\LL}_{\overline{*}} = \id_{\lambda}$ in $\scrC_{\{\lambda\}}^{\circ, \oplus}$.
   On the other hand, by Lemma \ref{lem:OACC_ideals}, we have an isomorphism $\End_{\scrC_{\{\lambda\}}^{\circ, \oplus}} (X) \cong R^0$.
   As a result, $\id_X = \alpha \dLL_{*, \overline{*}}^{\lambda}$ for some $\alpha \in R^0$. 
   By applying $\LL_{*} \circ (-)$ to both sides, we get that $\LL_* = \alpha \LL_{*}$. Therefore, $\alpha = 1$.
\end{proof}

By Lemma \ref{lem:singleton_identification}, for any reduced objects $X, Y \in \scrC$ with $\pi (X) = \lambda = \pi (Y)$, there is a canonical isomorphism $X \stackrel{\sim}{\to} Y$ in $\scrC_{\{\lambda\}}^{\circ, \oplus}$.
As a result, any object in $\scrC_{\{\lambda\}}^{\circ, \oplus}$ is canonically isomorphism to a direct sum of shifts of $\lambda$.
Moreover, by Lemma \ref{lem:OACC_ideals}, we know that $\End_{\scrC_{\{\lambda\}}^{\circ, \oplus}}^{\bullet} (\lambda) = R$.
From this, we deduce the following lemma.

\begin{lemma}\label{lem:single_stratum_equivalence}
    There is a canonical equivalence of categories
    \[\gamma : \scrC_{\{\lambda\}}^{\circ, \oplus} \stackrel{\sim}{\to} \Free^{\fg, \Z} (R)\]
    such that $\gamma (\lambda) = R$. Moreover, under $\gamma$, the autoequivalence (1) identifies with the shift of grading autoequivalence (1) on $\Free^{\fg, \Z} (R)$.
\end{lemma}

\subsubsection{Open and Closed Inclusions}

Let $I \subseteq \Lambda$ be a locally closed subset. Write $I = I_0 \setminus I_1$ for some closed subsets $I_1 \subset I_0 \subseteq \Lambda$.
Let $J \subseteq I$ be a closed subset of $I$. We can find some closed subset $J_0 \subseteq I_0$ such that $J = J_0 \setminus (J_0 \cap I_1)$. 
The natural embedding $\scrC_{J_0}^{\circ, \oplus} \hookrightarrow \scrC_{I_0}^{\circ, \oplus}$ induces a functor
\[j_{J*}^I : \scrC_{J}^{\circ, \oplus} = \scrC_{J_0}^{\circ, \oplus} \naivequotient \scrC_{J_0\cap I_1}^{\circ, \oplus} \to \scrC_{I_0}^{\circ, \oplus} \naivequotient \scrC_{I_1}^{\circ, \oplus} = \scrC_I^{\circ, \oplus}.\]
It is not hard to check from Lemma \ref{lem:OACC_ideals} that $j_{J*}^I$ functor is fully faithful and independent of the choices of $I_0$ and $J_0$.
Note that $j_{J*}^I$ commutes with $\DD$ and this construction is functorial with respect to compositions of closed inclusions in the obvious way.

Let $K \subset I$ be open in $I$.
Let $J = I \setminus K$ be the closed complement. We can write $J = J_0 \setminus (J_0 \cap I_1)$ as above. Then $K = I_0 \setminus K_1$ where $K_1 = J_0 \cup I_1$.
We then have a natural functor
\[j_{K}^{I*} : \scrC_I^{\circ, \oplus} = \scrC_{I_0}^{\circ, \oplus} \naivequotient \scrC_{I_1}^{\circ, \oplus} \to \scrC_{I_0}^{\circ, \oplus} \naivequotient \scrC_{K_1}^{\circ, \oplus} =  \scrC_K^{\circ, \oplus}.\]
It is not hard to check from the morphism space computations earlier that $j_{K}^{I*}$ is full and does not depend on the choices of $I_0$ and $J_0$.
Note that $j_{K}^{I*}$ commutes with $\DD$ and this construction is functorial with respect to compositions of open inclusions in the obvious way.

\subsection{Recollement}

\subsubsection{Mixed Derived Category}

Let $\scrC$ be an OACC with evaluation $\pi$.

\begin{definition}
    Let $I \subseteq \Lambda$ be a locally closed subset.
    We define the \emph{mixed derived category} of $\scrC$ supported on $I$ as the bounded homotopy category
    \[D_I^m (\scrC) \coloneq K^b \scrC_{I}^{\circ, \oplus}.\]
\end{definition}

The mixed derived category admits two ``shift'' functors:
\begin{enumerate}
    \item the internal shift inherited from $\scrC_{I}^{\circ, \oplus}$ which is denoted by $(1)$;
    \item the cohomological shift which is denoted by $[1]$.
\end{enumerate} 
We also define the \emph{Tate twist} by $\langle 1 \rangle \coloneq [1] (-1)$.
The contravariant involution $\DD$ on $\scrC_{I}^{\circ, \oplus}$ also induces a contravariant involution on $D_I^m (\scrC)$ which will also be denoted by $\DD$.

If $J \subseteq I$ is a closed subset and $K \subseteq I$ is an open subset. The functors $j_{J*}^I$ and $j_{K}^{I*}$ induce similarly denoted functors on bounded homotopy categories
\[j_{J*}^I : D_J^m (\scrC) \to D_I^m (\scrC) \qquad\text{and}\qquad j_{K}^{I*} : D_I^m (\scrC) \to D_K^m (\scrC).\]
Both $j_{J*}^I$ and $j_{K}^{I*}$ commute with $\DD$.

\subsubsection{Recollement}

\begin{proposition}[{\cite[Proposition 5.6]{ARV}}]\label{prop:recollement}
    Let $I \subseteq \Lambda$ be a locally closed subset, and $J \subset I$ be a finite closed subset.
    Then the functor $j_{I \setminus J}^{I*} : D_I^m (\scrC) \to D_{I \setminus J}^m (\scrC)$ admits a left adjoint  $j_{I \setminus J !}^{I}$ and a right adjoint $j_{I \setminus J *}^{I}$.
    Similarly, the functor $j_{J*}^I : D_J^m (\scrC) \to D_I^m (\scrC)$ admits a left adjoint $j_{J}^{I*}$ and a right adjoint $j_J^{I!}$.
    Together, these functors constitute a recollement diagram
    \[\begin{tikzcd}
        {D_J^m (\scrC)} \arrow[rr, "j_{J*}^{I}" description] &  & {D_I^m (\scrC)} \arrow[rr, "j_{I \setminus J}^{I*}" description] \arrow[ll, "j_J^{I*}"', bend right, shift left] \arrow[ll, "j_J^{I!}", bend left] &  & {D_{I \setminus J}^m (\scrC).} \arrow[ll, "j_{I \setminus J !}^{I}"', bend right] \arrow[ll, "j_{I \setminus J *}^{I}", bend left]
        \end{tikzcd}\]
\end{proposition}

We will not provide a proof of Proposition \ref{prop:recollement} instead referring to \cite{ARV}.
The argument provided in \emph{loc. cit.} is rather technical; however, it translates using the facts provided in the previous subsections.

\begin{example}\label{ex:min_rec_dts}
It is enlightening to see how $j_{I \setminus \{\lambda\} *}^I$ is defined on objects when $I$ is locally closed and $\lambda \in I$ is minimal (see \cite[\S 5.2]{ARV}).
Let $X \in \scrC$ be reduced such that $\pi (X) \in I \setminus \{\lambda\}$. The underlying complex of $j_{I \setminus \{\lambda\} *}^I X$ is given by
\[\ldots \to 0 \to \lambda \underline{\otimes}_R \Hom_{\scrC_{I}^{\circ, \oplus}}^{\bullet} (\lambda, X) \to X \to 0 \to \ldots,\]
where $X$ is in cohomological degree $0$. The only nontrivial differential is given by the canonical morphism from (\ref{eq:canon_mor_internal_tensor}),
\[\lambda \underline{\otimes}_R \Hom_{\scrC_{I}^{\circ, \oplus}}^{\bullet} (\lambda, X) \to X.\]
Moreover, there is a canonical distinguished triangle
\begin{equation}
    X \to j_{I \setminus \{\lambda\} *}^I X \to \lambda \underline{\otimes}_R \Hom_{\scrC_{I}^{\circ, \oplus}}^{\bullet} (\lambda, X) [1] \to
\end{equation}
in $D_I^m (\scrC)$. 
\end{example}

\subsubsection{Pushforward and Pullback under Locally Closed Inclusions}

Let $J \subseteq I \subseteq \Lambda$ be finite locally closed subsets.
We can write $J = J_0 \setminus J_1$ for some closed subsets $J_1 \subset J_0 \subset I$.
We define the pushforward and pullback functors
\begin{align*}
    j_{J*}^I &\coloneq j_{J_0 *}^I \circ j_{J*}^{J_0} : D_J^m (\scrC) \to D_I^m (\scrC), & j_{J!}^I &\coloneq j_{J_0 !}^I \circ j_{J!}^{J_0} : D_J^m (\scrC) \to D_I^m (\scrC), \\
j_J^{I*} &\coloneq j_J^{I \setminus J_1 *} \circ j_{I \setminus J_1}^{I*} : D_I^m (\scrC) \to D_J^m (\scrC), & j_J^{I!} &\coloneq j_J^{I \setminus J_1 !} \circ j_{I \setminus J_1}^{I!} : D_I^m (\scrC) \to D_J^m (\scrC).
\end{align*}
It follows from \cite[Lemma 5.10]{ARV} that these functors are naturally independent of the choices of $J_0$ and $J_1$.  
Moreover, $\left( j_{J!}^I, j_{J}^{I!} \right)$ and $\left( j_{J}^{I*}, j_{J*}^I\right)$ are adjoint pairs of functors.
We also have canonical isomorphisms
\begin{equation}\label{eq:verdier_duality_and_pushforwards}
    \DD \circ j_{J*}^{I} \cong j_{J!}^{I} \circ \DD \qquad\text{and}\qquad \DD \circ j_{J}^{I!} \cong j_{J}^{I*} \circ \DD.
\end{equation}
By Proposition \ref{prop:recollement}, the adjunction morphisms
\[j_{J}^{I*} \circ j_{J*}^{I} \to \id \qquad\text{and}\qquad \id \to j_{J}^{I!} \circ j_{J!}^I \]
are isomorphisms. Finally, we note that when $J \subset I$ is closed, the $*$-pushforward and $!$-pushforward coincide.
Similarly, when $J \subseteq I$ is open, the $*$-pullback and $!$-pullback coincide.

The four functors are also functorial with respect to compositions of locally closed subsets.
\begin{lemma}
    Let $I \subseteq \Lambda$ be a finite locally closed subset, and let $J \subset I$ and $K \subset J$ be locally closed subsets.
    Then there exist canonical isomorphisms
    \begin{align*}
        j_{J*}^I \circ j_{K*}^J &\cong j_{K*}^I, & j_{J!}^I \circ j_{K!}^J &\cong j_{K!}^I, \\
        j_{K}^{J*} \circ j_J^{I*} &\cong j_K^{I*}, & j_{K}^{J!} \circ j_J^{I!} &\cong j_K^{I!}.
    \end{align*}
\end{lemma}

Let $I \subseteq \Lambda$ be a finite locally closed subset.
If $\lambda \in I$, then we will simplify the notation for the pullbacks and pushforwards corresponding to the inclusion $\{ \lambda \} \subset I$ by replacing $\{ \lambda \}$ with $\lambda$.

\subsection{Standards and Costandards}\label{subsub:stds_and_costds}

\subsubsection{Generation}

\begin{definition}
    The \emph{mixed derived category} of $\scrC$ is the bounded homotopy category
\[D^m (\scrC) \coloneq K^b \scrC^{\circ, \oplus}.\]
\end{definition}

As in Remark \ref{rem:failure_of_additive_generation}, $\scrC^{\circ, \oplus}$ is not generally equivalent to $\scrC_{\Lambda}^{\circ, \oplus}$.
The situation is somewhat rectified after passing to mixed derived categories provided the following generation assumption holds.

\begin{assumption}\label{ass:generation_assumption}
    Every object $X \in \scrC$ when viewed as an object in $D^m (\scrC)$ is isomorphic to an object in $D_I^m (\scrC)$ for some finite locally closed subset $I \subseteq \Lambda$.
\end{assumption}

Whenever Assumption \ref{ass:generation_assumption} holds, the canonical embedding
\[D_{\Lambda}^m (\scrC) \hookrightarrow D^m (\scrC)\]
is an equivalence of categories. Indeed, the assumption implies that $D^m (\scrC)$ is generated as a triangulated category by reduced objects in $\scrC$. Therefore, the embedding is essentially surjective.

\begin{remark}
    Assumption \ref{ass:generation_assumption} is fairly mild. For example, this is known to hold whenever $\k$ is a complete local ring or a field by \cite[Proposition 2.24]{EL}.
    Additionally, it holds without any restriction on $\k$ when $\scrC$ is the Hecke category of a Coxeter group \cite[Lemma 6.2]{ARV}; although, their proof crucially uses properties of the Hecke category which are not featured in a general OACC.
    It would be interesting to reinterpret Assumption \ref{ass:generation_assumption} as a condition on the cellular structure. 
    This could lead to interesting generalizations of categorification results that historically have been beholden to constraints on $\k$ (see Proposition \ref{prop:OACC_and_categorification}).   
\end{remark}

\subsubsection{Definitions of Standards and Costandards}

Let $I \subseteq \Lambda$ be a finite locally closed subset.
For each $\lambda \in I \subseteq \Lambda$, we define the \emph{standard} and \emph{costandard} objects in $D_I^m (\scrC)$ respectively by
\[\Delta_{\lambda}^I \coloneq j_{\lambda !}^I \lambda \qquad\text{and}\qquad \nabla_{\lambda}^I \coloneq j_{\lambda *}^I \lambda.\]
When $I = \Lambda$, we will sometimes omit the superscript in the notation.

The following lemmas are simple consequences of recollement (Proposition \ref{prop:recollement}), Lemma \ref{lem:single_stratum_equivalence}, and other basic properties of pushforwards/pullbacks along locally closed subsets.

\begin{lemma}[{\cite[Lemma 6.6]{ARV}}]\label{lem:hom_vanishing_for_stds_and_costds}
    Let $I \subseteq \Lambda$ be a finite locally closed subset. Let $\lambda, \mu \in I$.
    Then we have
    \[\Hom_{D_I^m (\scrC)} (\Delta_{\lambda}^I, \nabla_{\mu}^I \langle n \rangle [m]) \cong \begin{cases} R^m & \text{if } \lambda=\mu\text{ and }m=-n, \\ 0 & \text{otherwise.}\end{cases}\]
\end{lemma}

\begin{lemma}[{\cite[Lemma 6.8]{ARV}}]\label{lem:pushforwards_and_pullbacks_of_stds}
    Let $I \subseteq \Lambda$ be a finite locally closed subset, and let $J \subset I$ be a locally closed subset.
    Then for any $\lambda \in J$, we have
    \[j_{J!}^I \Delta_{\lambda}^J \cong \Delta_{\lambda}^I \qquad\text{and}\qquad j_{J*}^I \nabla_{\lambda}^J \cong \nabla_{\lambda}^I. \]
    Likewise for any $\mu \in I$, we have
    \[j_J^{I*} \Delta_{\mu}^I \cong \begin{cases} \Delta_{\mu}^J & \text{if } \mu \in J, \\ 0 & \text{otherwise,}\end{cases} \qquad\text{and}\qquad j_J^{I!} \nabla_{\mu}^I \cong \begin{cases} \nabla_{\mu}^J & \text{if } \mu \in J, \\ 0 & \text{otherwise.}\end{cases}\]
\end{lemma}

Suppose $I \subseteq \Lambda$ is locally closed (but possible infinite). 
Let $\lambda \in I$. We can pick $J \subseteq \Lambda$ finite and closed such that $\lambda \in J$.
We can define $\Delta_{\lambda}^I$ and $\nabla_{\lambda}^I$ by
\[\Delta_{\lambda}^I \coloneq j_{J !}^I \Delta_{\lambda}^J \qquad\text{and}\qquad \nabla_{\lambda}^I \coloneq j_{J !}^I \nabla_{\lambda}^J.\]
These objects are independent of the choice of $J$.

\begin{lemma}[{\cite[Lemma 6.9]{ARV}}]\label{lem:stds_generate}
    For any locally closed subset $I \subseteq \Lambda$, the category $D_I^m (\scrC)$ is generated, as a triangulated category, by objects of the form $\Delta_{\lambda}^I (m)$ for $\lambda \in I$ and $m \in \Z$.
    Alternatively, $D_I^m (\scrC)$ is also generated by objects of the form $\nabla_{\lambda}^I (m)$ for $\lambda \in I$ and $m \in \Z$.
\end{lemma}

\subsection{Grothendieck Groups}

We $\scrA$ is an essentially small triangulated category (resp. additive category), we denote by $[\scrA]_{\Delta}$ (resp. $[\scrA]_{\oplus}$), the Grothendieck group of $\scrA$ (resp. the split Grothendieck group of $\scrA$).

\begin{lemma}[{\cite[Theorem 1.1]{Rose}}]\label{lem:red_to_tri_groth_gps}
    Let $\scrA$ be an essentially small additive category. The natural group homomorphism
    \[[\scrA]_{\oplus} \to [K^b \scrA]_{\Delta}\]
    is an isomorphism.
\end{lemma}

\begin{proposition}\label{prop:OACC_and_categorification}
    Assume that $\scrC$ is an OACC with evaluation $\pi$ satisfying Assumption \ref{ass:generation_assumption}. 
    \begin{enumerate}
        \item There exists an isomorphism of $\Z [v^{\pm}]$-modules
        \[\Z [v^{\pm}] [\Lambda] \stackrel{\sim}{\to} [D^m (\scrC)]_{\Delta} \]
        sending $v^n \lambda$ to $\Delta_{\lambda} (n)$ for all $\lambda \in \Lambda$ and $n \in \Z$.
        \item There exists an isomorphism of $\Z [v^{\pm}]$-modules
        \[\Z [v^{\pm}] [\Lambda] \stackrel{\sim}{\to} [\scrC^{\circ, \oplus}]_{\oplus}. \]
    \end{enumerate}
\end{proposition}
\begin{proof}
    The proof is essentially the same as \cite[Theorem 6.13]{ARV}.

    By Lemma \ref{lem:red_to_tri_groth_gps}, the second isomorphism will follow from the first isomorphism.
    By Assumption \ref{ass:generation_assumption}, we have that $[D^m (\scrC)]_{\Delta} \cong [D_{\Lambda}^m (\scrC)]_{\Delta}$.
    We can construct a $\Z [v^{\pm}]$-module homomorphism
    \[\Z[v^{\pm}] [\Lambda] \to [D_{\Lambda}^m (\scrC)]_{\Delta}\]
    by $v^n \lambda \mapsto [\Delta_{\lambda} (n)]$ for $\lambda \in \Lambda$ and $n \in \Z$.
    By Lemma \ref{lem:stds_generate}, the elements $\{ [\Delta_{\lambda} (n)]\}_{\lambda \in \Lambda, n \in \Z}$ form a $\Z$-spanning set of $[D_{\Lambda}^m (\scrC)]_{\Delta}$.
    In order to prove this spanning set is linearly independent over $\Z$, we will argue by contradiction. Suppose there is a relation
    \[\sum_{\substack{\lambda \in K_1 \\ m \in \Z}} c_{\lambda, m} [\Delta_{\lambda} (m)] = \sum_{\substack{\mu \in K_2 \\ n \in \Z}} c_{\mu, n} [\Delta_{\mu} (n)]\]
    for some disjoint finite subsets $K_1, K_2 \subset \Lambda$ and integers $c_{\lambda, n} \in \Z_{\geq 0}$ such that (1) $K_1 \neq \emptyset$, (2) $c_{\lambda, m} \neq 0$ for some $\lambda \in K_1$ and $m \in \Z$, and (3) $c_{\lambda, m}$ and $c_{\mu, n}$ are only nonzero for finitely many $m$ and $n$'s. 
    We can then define
    \[X_1 = \bigoplus_{\substack{\lambda \in K_1 \\ m \in \Z}} \left( \Delta_{\lambda} (m)\right)^{\oplus c_{\lambda, m}} \qquad\text{and}\qquad X_2 = \bigoplus_{\substack{\mu \in K_2 \\ n \in \Z}} \left( \Delta_{\mu} (n)\right)^{\oplus c_{\mu, n}}.\]
    By \cite[Lemma 2.6]{Th}, there exists objects $Y, Y', Y'' \in D_{\Lambda}^m (\scrC)$ and distinguished triangles
    \[Y \oplus X_1 \to Y' \to Y'' \to \qquad\text{and}\qquad Y \oplus X_2 \to Y' \to Y'' \to.\]
    We can find a finite closed subset $I \subseteq \Lambda$ such that all the above objects belong to $D_I^m (\scrC)$.
    Choose $\lambda \in K_1$ such that $c_{\lambda, m} \neq 0$ for at least one $m$. By Lemma \ref{lem:pushforwards_and_pullbacks_of_stds}, we obtain distinguished triangles
    \[ j_{\lambda}^{I*} Y \oplus j_{\lambda}^{I*} X_1 \to j_{\lambda}^{I*} Y' \to j_{\lambda}^{I*} Y'' \to, \qquad\qquad j_{\lambda}^{I*} Y  \to j_{\lambda}^{I*} Y' \to j_{\lambda}^{I*} Y'' \to.\]
    Therefore, $[j_{\lambda}^{I*} X_1 ] = 0$ in $[D_{\{\lambda\}}^m (\scrC)]_{\Delta}$.
    However, by Lemma \ref{lem:single_stratum_equivalence} and Lemma \ref{lem:red_to_tri_groth_gps}, there is an isomorphism of $\Z[v^{\pm}]$-modules, 
    \[ \Z[v^{\pm}] \stackrel{\sim}{\to} [D_{\{\lambda\}}^m (\scrC)]_{\Delta}\]
    defined by $v^n \mapsto \lambda (n)$ for $n \in \Z$. Under this equivalence, we see that $c_{\lambda, m} = 0$ for all $m$ which is a contradiction.
\end{proof}

\end{document}